\documentclass[12pt]{article}

\usepackage[english]{babel}
\usepackage[T1]{fontenc}
\usepackage[utf8x]{inputenc}
\usepackage{multicol}
\usepackage{enumitem}
\usepackage{amsbsy}
\usepackage{graphicx}
\usepackage[font=small,textfont=it,labelfont=bf,labelsep=period]{caption}
\usepackage{stmaryrd}\SetSymbolFont{stmry}{bold}{U}{stmry}{m}{n}

\usepackage[textwidth=18mm]{todonotes}
\graphicspath{ {./img/} }

\newcommand{\MEMS}[2]{#2}
\newcommand{\textgras}[1]{\textbf{#1}}

\RequirePackage{xcolor}
\RequirePackage{mathtools,amsthm,amssymb,amsfonts,latexsym,mathrsfs}
\DeclareFontFamily{U}{dutchcal}{\skewchar\font=45 }
\DeclareFontShape{U}{dutchcal}{m}{n}{<-> s*[1.0] dutchcal-r}{}
\DeclareMathAlphabet{\mathlcal}{U}{dutchcal}{m}{n}
\DeclareMathAlphabet{\mathpzc}{OT1}{pzc}{m}{it}

\numberwithin{equation}{section}
\numberwithin{figure}{section}

\DeclareMathOperator{\cov}{Cov}
\DeclareMathOperator{\dis}{dis}
\DeclareMathOperator{\lgth}{length}
\DeclareMathOperator{\CVS}{CVS}
\DeclareMathOperator{\Aut}{Aut}

\newcommand{\BHP}{\mathsf{BHP}}
\newcommand{\BP}{\mathsf{BP}}
\newcommand{\cEP}{\mathsf{EP}}
\newcommand{\Qd}{\mathsf{Qd}}
\newcommand{\sfS}{\mathsf{S}}
\newcommand{\BS}[2]{\sfS^{[#1]}_{#2}}
\newcommand{\Sl}{\mathsf{Sl}}
\newcommand{\sfa}{\mathsf{a}}
\newcommand{\sfh}{\mathsf{h}}
\newcommand{\sfl}{\mathsf{l}}
\newcommand{\sfx}{\mathsf{x}}

\newcommand{\Half}{\mathbf{Half}}
\newcommand{\Plane}{\mathbf{Plane}}
\newcommand{\Quad}{\mathbf{Quad}}
\newcommand{\FQuad}{\mathbf{FQuad}}
\newcommand{\Slice}{\mathbf{Slice}}
\newcommand{\FSlice}{\mathbf{FSlice}}

\newcommand{\rmL}{\mathrm{L}}
\newcommand{\EP}{\mathrm{EP}}
\newcommand{\Qn}{Q_n}
\newcommand{\Qdn}{\mathrm{Qd}_n}
\newcommand{\Sln}{\mathrm{Sl}_n}
\newcommand{\Mark}{\mathrm{Mark}}

\newcommand{\E}{\mathbb{E}}
\newcommand{\M}{\mathbb{M}} 
\newcommand{\N}{\mathbb{N}} 
\renewcommand{\P}{\mathbb{P}}
\newcommand{\R}{\mathbb{R}}
\newcommand{\Rp}{\mathbb{R}_{\ge 0}}
\newcommand{\Rm}{\mathbb{R}_{\le 0}}
\newcommand{\Z}{\mathbb{Z}} 
\newcommand{\Zp}{\mathbb{Z}_{\ge 0}}
\newcommand{\Zm}{\mathbb{Z}_{\le 0}}

\newcommand{\CC}{\mathcal{C}}
\newcommand{\WW}{\mathcal{W}}
\newcommand{\FF}{\mathcal{F}}
\newcommand{\GG}{\mathcal{G}}
\newcommand{\ch}{\mathlcal{h}}
\newcommand{\cL}{\mathcal{L}}
\newcommand{\RR}{\mathcal{R}}
\newcommand{\binRR}{\mathbin{\mathcal{R}}}
\newcommand{\binRRn}{\mathbin{\mathcal{R}_n}}
\newcommand{\cT}{\mathcal{T}}
\newcommand{\X}{\mathcal{X}}
\newcommand{\Y}{\mathcal{Y}}
\newcommand{\ZZ}{\mathcal{Z}}

\newcommand{\bA}{\mathbf{A}}

\newcommand{\bB}{\mathbf{B}}
\newcommand{\bd}{\mathbf{d}}
\newcommand{\be}{\mathbf{e}}
\newcommand{\bff}{\mathbf{f}}
\newcommand{\bH}{\mathbf{H}}
\newcommand{\bh}{{\boldsymbol h}}
\newcommand{\bl}{{\boldsymbol l}}

\newcommand{\bLambda}{{\boldsymbol \Lambda}}
\newcommand{\blambda}{{\boldsymbol \lambda}}
\newcommand{\bL}{{\boldsymbol L}}
\newcommand{\bsl}{\mathbf{sl}}
\newcommand{\bqd}{\mathbf{qd}}
\newcommand{\bm}{\mathbf{m}}
\newcommand{\bM}{\mathbf{M}}
\newcommand{\RbM}{\smash{\overset{\rightarrow}{\vphantom{o}\smash{\bM}}}}
\newcommand{\bp}{\mathbf{p}}
\newcommand{\bq}{\mathbf{q}}
\newcommand{\bQ}{\mathbf{Q}}
\newcommand{\RbQ}{\smash{\overset{\rightarrow}{\vphantom{o}\smash{\bQ}}}}
\newcommand{\RbQnl}{{\RbQ}_{n,\bl}^{[g]}}
\newcommand{\bQnl}{{\bQ}_{n,\bl}^{[g]}}
\newcommand{\bQnln}{{\bQ}_{n,\bl_n}^{[g]}}
\newcommand{\bQl}{{\bQ}_{\bl}^{[g]}}
\newcommand{\RbQnof}[2]{{\RbQ}_{n,#1}^{[#2]}}
\newcommand{\RbQnln}{{\RbQ}_{n,\bl_n}^{[g]}}
\newcommand{\RbQnlnO}{{\RbQ}_{n,\bl_n 0}^{[g]}}
\newcommand{\RbQnzero}{{\RbQ}_{n,\varnothing}^{[0]}}
\newcommand{\bs}{\mathbf{s}}
\newcommand{\bS}{\mathbf{S}}
\newcommand{\RbS}{\smash{\overset{\rightarrow}{\vphantom{o}\smash{\bS}}}}
\newcommand{\SR}{\smash{\overset{\rightarrow}{\vphantom{o}\smash{\mathrm{S}}}}}
\newcommand{\bt}{\mathbf{t}}
\newcommand{\bbr}{\mathbf{r}}
\newcommand{\bT}{\mathbf{T}}
\newcommand{\bx}{{\boldsymbol x}}
\newcommand{\by}{{\boldsymbol y}}
\newcommand{\Sgb}{\mathbf\Sigma_{b}^{[g]}}
\newcommand{\Sof}[2]{\mathbf\Sigma_{#2}^{[#1]}}

\newcommand{\bpi}{{\boldsymbol\pi}}
\newcommand{\bmu}{{\boldsymbol\mu}}
\newcommand{\bnu}{{\boldsymbol\nu}}
\newcommand{\bzero}{{\boldsymbol 0}}

\newcommand{\pzl}{\mathpzc{l}}
\newcommand{\Mi}{\mathscr{M}}
\newcommand{\scS}{\mathscr{S}}

\newcommand{\fX}{\mathfrak{X}}
\newcommand{\fW}{\mathfrak{W}}
\newcommand{\fD}{\mathfrak{D}}
\newcommand{\fS}{\bar{\mathfrak{T}}}
\newcommand{\fT}{\mathfrak{T}}
\newcommand{\fC}{\mathfrak{C}}
\newcommand{\fL}{\mathfrak{L}}

\newcommand{\dGH}{\mathrm{d}_{\mathrm{GH}}}
\newcommand{\dGHP}{\mathrm{d}_{\mathrm{GHP}}}
\newcommand{\dist}{\mathrm{dist}}
\newcommand{\Leb}{\mathrm{Leb}}
\newcommand{\ind}{\mathbf{1}}
\newcommand{\build}[3]{\mathrel{ \mathop{\kern 0pt#1}\limits_{#2}^{#3}}}
\newcommand{\ton}{\build{\longrightarrow}{n\to\infty}{}}	\newcommand{\tod}{\build{\longrightarrow}{n\to\infty}{(d)}}
\newcommand{\eps}{\varepsilon}

\newcommand{\rst}{{\raisebox{-.15em}{$|$}}{}}
\newcommand{\rsts}{{\text{\scriptsize\raisebox{-.15em}{$|$}}}{}}
\newcommand{\binR}{\mathbin{R}}
\DeclareMathAlphabet\mathbfcal{OMS}{cmsy}{b}{n}
\newcommand\sO{\scalebox{.5}[.62]{$\mathbfcal{O}$}}
\newcommand\bO{\mathcal{O}}
\newcommand{\cllbracket}{\mathclose{\llbracket}}
\newcommand{\sas}{{\smash{*}}}
\newcommand{\kk}{k}
\newcommand{\bb}{b}
\newcommand{\pp}{p}
\newcommand{\comp}{\bowtie}
\newcommand{\tp}{{\mathclap{\,\cdot}\circ}}

\newcommand{\br}{\bar}

\newcommand{\wt}[1]{\mkern 1mu\widetilde{\mkern-1mu#1\mkern-1mu}\mkern 1mu}
\newcommand{\un}[1]{\mkern 1mu\underline{\mkern-1mu#1\mkern-1mu}\mkern 1mu}
\newcommand{\dun}[1]{\mkern 1mu\smash{\underline{\vphantom{o_o}\smash{\underline{\mkern-1mu#1\mkern-1mu}}}}\mkern 1mu}
\newcommand{\dunderline}[1]{\smash{\underline{\vphantom{o_o}\smash{\underline{#1}}}}}
\usepackage{yhmath}
\newcommand{\dov}[1]{\mkern 1mu\widehat{\mkern-1mu#1\mkern-1mu}\mkern 1mu}
\newcommand{\overl}[1]{\mkern 3mu\overline{\mkern-3mu#1\mkern-1mu}\mkern 1mu}

\newtheorem*{clm}{Claim}
\newtheorem{thm}{Theorem}[section]
\newtheorem{lmm}[thm]{Lemma}
\newtheorem{prp}[thm]{Proposition}
\newtheorem{defn}[thm]{Definition}
\newtheorem{crl}[thm]{Corollary}

\newtheorem{rem}[thm]{Remark}

\theoremstyle{definition}

\newtheorem*{ack}{Acknowledgment}

\usepackage{fancyhdr}
\usepackage[outerbars]{changebar}
\usepackage[colorlinks=true,pdftex,urlcolor=magenta!60!black,citecolor=green!50!black,linkcolor=red!90!black,bookmarksnumbered=true,pdfstartview=FitH,pdfauthor={Jeremie Bettinelli, Gregory Miermont}, pdftitle={Compact Brownian surfaces II}, linktocpage=true]{hyperref}

\renewcommand{\d}{\mathrm{d}}

\title{Compact {B}rownian surfaces {II}. Orientable surfaces}
\author{J\'er\'emie Bettinelli\thanks{LIX, CNRS, \'Ecole polytechnique, Institut Polytechnique de Paris}
\and
Gr\'egory Miermont\thanks{UMPA, ENS de Lyon, and Institut Universitaire de France}}

\begin{document}

\maketitle

\begin{abstract}
Fix an arbitrary compact orientable surface with a boundary and consider a uniform bipartite random quadrangulation of this surface with~$n$ faces and boundary component lengths of order~$\sqrt n$ or of lower order. Endow this quadrangulation with the usual graph metric renormalized by $n^{-1/4}$, mark it on each boundary component, and endow it with the counting measure on its vertex set renormalized by $n^{-1}$, as well as the counting measure on each boundary component renormalized by $n^{-1/2}$. We show that, as $n\to\infty$, this random marked measured metric space converges in distribution for the Gromov--Hausdorff--Prokhorov topology, toward a random limiting marked measured metric space called a \emph{Brownian surface}. 

This extends known convergence results of uniform random planar quadrangulations with at most one boundary component toward the \emph{Brownian sphere} and toward the \emph{Brownian disk}, by considering the case of quadrangulations on general compact orientable surfaces. Our approach consists in cutting a Brownian
surface into elementary pieces that are naturally related to the Brownian sphere and the Brownian disk and their noncompact analogs, the Brownian plane and the Brownian half-plane, and to prove convergence results for these elementary pieces, which are of independent interest. 
\end{abstract}

\begin{ack}
We thank G.\ Chapuy for stimulating 
discussions during the elaboration of this work, and in particular for
his encouragements to deal with general compact surfaces with a
boundary. Thanks to J.\ Bouttier for bringing our attention to the fact
that the semigroup property of discrete slices might have an
interesting continuum counterpart.
Thanks are also due to X.\ Sun for discussions around the problem of
defining Brownian surfaces in the context of LQG and the question of
the random modulus of these surfaces, and also for sharing the work
\cite{ARS22mod}. We finally thank N.\ Holden and J.\ Miller for their
comments on a first version of this paper, as well as anonymous referees for their thorough reading. 
\end{ack}

\tableofcontents

\section{Introduction}\label{sec:introduction}

\subsection{Context}

Random maps, seen as discrete models of random $2$-dimensional
geometries, have generated a sustained interest in the last couple of
decades. 
An important instance of this line of research are the results by Le~Gall~\cite{legall11} and
the second author~\cite{miermont11}, showing that a uniform random
quadrangulation of the sphere with~$n$ faces, seen as a random finite
metric space by endowing its vertex set with the usual graph metric
renormalized by~$n^{-1/4}$, converges in distribution toward the
so-called \emph{Brownian sphere}, or \emph{Brownian map}. The aim of the present work is to
generalize this result to the case of general compact
orientable surfaces.
Let us start with some elements of context.

\paragraph{Random surfaces as scaling limits of random maps.}
While the idea that continuum random geometries should be obtained as
scaling limits of random maps originates from the physics literature
on 2-dimensional quantum gravity \cite{Pol81,david85,KnPoZa88}, this question was first
approached in the mathematical literature in the pioneering work
of Chassaing and Schaeffer~\cite{CSise}, who studied the model of
uniformly chosen random quadrangulations of
the sphere, and found in particular
that the proper scaling factor in this case was~$n^{-1/4}$.
Marckert
and Mokkadem~\cite{mm01} then constructed a candidate limiting space
today called the \textit{Brownian sphere}, and showed the convergence toward it
in another topology than the Gromov--Hausdorff
topology. Le~Gall~\cite{legall06} later showed that the sequence of
rescaled metric spaces associated with uniform random quadrangulations
of the sphere was relatively compact. Finally, Le~Gall~\cite{legall11}
and the second author~\cite{miermont11} showed by two independent approaches
that the previous sequence converges toward the Brownian sphere.

It is known that the Brownian sphere arises as a 
universal scaling limit for
many models of planar maps \emph{that are uniformly chosen in a certain class, given
their face degrees}, and provided that face 
degrees are typically all of the same order of magnitude; see
\cite{legall11,BeLG,ABAl17,BeJaMi14,abr14,CuLG19,ABAl21,marzouk22}. See
also~\cite{LGMi09} for models of maps that fall out of  
this universality class. 

The scaling limits of quadrangulations on surfaces that are more
general than the sphere were considered by the first author
in~\cite{bettinelli11b,Bet16geo}, who showed similar results to 
the above, but only up to extraction of appropriate subsequences,
leaving a gap that amounts to uniquely characterize the limit. This
gap was filled in the particular case of the disk topology in our
previous work~\cite{BeMi17}. In particular, we showed that a uniform
quadrangulation of genus~$0$ with one boundary component having~$n$
internal faces and perimeter~$2l_n$ weakly converges, once scaled by
the factor $n^{-1/4}$ and when $l_n\sim L\sqrt{2n}$, toward a
random metric space called the \emph{Brownian disk of
perimeter~$L$}.
Two
alternate constructions of Brownian disks were proposed by Le~Gall
\cite{LGa19dis,legall22}, allowing in particular to show that
Brownian disks arise as connected components of the complement of metric balls in the Brownian
sphere, conditionally given their areas and boundary lengths. See also
\cite{MiSh210,BeCuKo18,LGRi20}.

Besides the case of the sphere and the disk, only a few results have been
obtained for maps on compact surfaces. Namely, it has been shown that uniform
quadrangulations of a given compact surface with a boundary exhibit
scaling limits~\cite{bettinelli10,bettinelli11,Bet16geo}, all of
the same topology as the considered surface, and geodesics to a
uniformly chosen points were studied~\cite{Bet16geo}. More
recently, it was shown that uniformly distributed essentially simple
toroidal triangulations (that is, triangulations of the torus without
contractible loops or double edges forming cycles that are homotopic
to~$0$) also exhibit scaling limits~\cite{BeHuLe19arX}, which are
believed to be the same as for random quadrangulations. See also
\cite{ARS22mod} for a scaling limit result of Boltzmann random maps with annular
topology. 

There has also been a growing
interest in noncompact versions of these models, especially as they
bridge some Brownian surfaces with so-called \emph{uniform infinite
  random maps}, which are maps with infinitely many faces that first arose in a work by Angel and Schramm~\cite{AnSc03}, as \emph{local limits} of random finite maps. Three main models of
noncompact Brownian surfaces have been identified: the \emph{Brownian
  plane}~\cite{CuLG12Bplane}, the \emph{Brownian
  half-plane}~\cite{GwMi17,BaMiRa}, and the \emph{infinite-volume Brownian
  disk}~\cite{BaMiRa}, which can be thought of as noncompact versions
of the Brownian sphere and Brownian disks, either with unbounded or
bounded boundary. See~\cite{LGRi21} for a framework unifying those
objects. The first two of these models will play an important role in the current
work.

This whole line of research crucially depends on strong combinatorial
techniques, and in particular on bijective approaches \cite{schaeffer98,BdFGmobiles,AlPo15} that
allow to give very detailed quantitative information on the geodesic
paths in random maps and their scaling limits. The present work is
no exception.
See for instance \cite{legall08,miertess,AnKoMi17,MiQi21arX,LGa22sta} for results related to the structure of geodesics in the 
Brownian sphere, 
\cite{LeGallBrownianGeometry} for a recent survey, and
\cite{CurienStFlour} for another approach called \emph{peeling}. We
note, however, that, so far, these methods are restricted to models of maps chosen uniformly, conditionally
given their face degrees, as alluded to above.

\paragraph{Random surfaces via Liouville quantum gravity metrics.}
A line of research parallel to the above consists in building the
limiting spaces directly as continuum random metrics in planar domains
or Riemann surfaces. This approach also finds its roots in the physics
theory of Liouville quantum gravity \cite{Pol81}. In the case of Brownian surfaces, this has first been
implemented by Miller and Sheffield in a series of works \cite{MiSh210,
MiSh21I,MiSh21II,MiSh21III}, where they use a growth model called \emph{Quantum Loewner Evolution} (QLE) to define a random metric on the plane, whose metric balls are
described by QLE, and whose law as an abstract metric space is equal to that of the Brownian
plane. Local variants of the construction allow to define the Brownian
sphere in this way.
The Miller--Sheffield metric is in fact a special element of
a one-parameter family of \emph{Liouville Quantum Gravity} (LQG) metrics,
that have been defined as scaling limits of first-passage percolation
models in mollified exponentiated Gaussian free fields landscapes
\cite{DiDuDuFa20,GwMi21b}. See \cite{DiDuGw23} for an overview of LQG metrics.  

These constructions operate entirely in the continuum, and naturally
ask whether canonical embeddings of random maps in the sphere are
compatible with the convergence toward the Brownian sphere, in the
sense that the metrics induced by the embedding converge to
the random metric of Miller--Sheffield. Such a 
result was recently obtained by Holden and Sun~\cite{HoSu23} (which is
the last piece of a vast research project, described in details in
this reference), who showed the joint convergence of the metric and the area measure
generated by a uniform plane triangulation embedded via the Cardy--Smirnov
embedding in an equilateral triangle. We refer to the overview
article~\cite{GwHoSu23}.

The existence of a canonical conformal structure for
Brownian surfaces was also approached in a more direct way by Gwynne,
Miller and Sheffield in~\cite{GwMiSh20,GwMiSh22}. Their method, which
has been implemented so far for the plane, 
half-plane, sphere and disk topologies, consists in taking limits of discrete embeddings
obtained directly from the continuum limit by considering 
Poisson--Voronoi tessellations with a finer and finer mesh, and showing
that the random walk on the discrete approximation converges to
Brownian motion in the plane. In passing, this allows one to define
Brownian motion on the Brownian surfaces under consideration.

\paragraph{Random surfaces and conformal field theories.}
While the definition of LQG metrics applies to any field that
``locally looks like'' the Gaussian free field, the exact law of the
latter is of crucial importance to obtain the exact law of random
surfaces that arise as scaling limits of maps, and this law can be
obtained from Liouville conformal field theory \cite{Pol81}. Here,
rather than dealing with random metrics, one is rather interested in
the computation of partition functions defined from the field, and it
has been shown recently in a rich body of work -- see
\cite{DaKuRhVa16,GuRhVa19,KuRhVa21} and references therein -- that this theory has a
probabilistic interpretation in terms of Gaussian multiplicative chaoses, which are
random measures defined in terms of the Gaussian free field. This
approach has unveiled fundamental integrability properties for
planar Gaussian multiplicative chaoses, which can be
used to provide exact distributions for various quantities related to
the LQG metrics, hence to the scaling limits of random maps. 
For instance, in~\cite{ARS22mod}, the authors compute the law of the
conformal modulus of a Brownian annulus, which is a member of the
family of Brownian surfaces described in the present work. 

The interplay between these approaches 
provides a wealth of methods to prove various properties of
random surfaces \cite{sheffieldICM}, and the geometric properties of
the Brownian surfaces, as well as the other LQG metrics, are the object of intensive
current research.

\subsection{Generalities and terminology on maps}\label{sec:generalities-maps}

\paragraph{Surface with a boundary.}
Recall that a \emph{surface with a boundary} is a nonempty Hausdorff
topological space in which every point has an open neighborhood
homeomorphic to some open subset of~$\R\times\Rp$. Its
\emph{boundary} is the
set of points having a neighborhood homeomorphic to a neighborhood of $(0,0)$
in~$\R\times\Rp$. When it is nonempty, this set forms a
$1$-dimensional topological manifold. 
In this
work, we will only consider \textgras{orientable} compact connected surfaces with
a (possibly empty) boundary. By the classification theorem, these are characterized up
to homeomorphisms by two nonnegative integers, the genus~$g$ and the
number~$\bb$ of connected components of the boundary. We denote
by~$\Sgb$ the compact orientable surface
of genus~$g$ with~$\bb$ boundary components, which is unique up to
homeomorphisms. It can be obtained from the
connected sum of~$g$ tori, or from the sphere in the case $g=0$, by
removing~$\bb$ disjoint open disks whose boundaries are pairwise
disjoint circles.

\paragraph{Map.}
A \emph{map} is a proper cellular embedding of a finite graph, possibly with
multiple edges and loops, into a compact connected orientable surface
\emph{without} boundary.
Here, the word \emph{proper} means that edges can intersect
only at vertices, and \emph{cellular} means that the
connected components of the complement of the edges, which are called the
\emph{faces} of the map, are homeomorphic to $2$-dimensional open disks.
Maps will always be considered up to orientation-preserving
homeomorphisms of the surface into which they are embedded. The \emph{genus} of a map is defined as the genus of the
surface into which it is embedded; we speak of \emph{plane} maps when the genus is~$0$. We call \emph{half-edge} an oriented edge in a map. With every
half-edge, we may associate in a one-to-one way a \emph{corner}, which
is the angular sector lying to its
left at the origin of the half-edge. Note that this makes sense because the surfaces we are considering are orientable. We say that a corner, or the corresponding half-edge, is \emph{incident} to a face~$f$ if it lies into~$f$. We also say that the face is incident to the corner or the half-edge in this case. The number of half-edges (or equivalently, of corners) incident to a face is called its \emph{degree}.

A map is \emph{rooted} if it comes with a distinguished corner -- or,
equivalently, a half-edge -- called the \emph{root}. Rooting is a
very useful notion as it allows to kill the symmetries of a map. In
fact, when dealing with nonrooted maps, we will systematically count
them by weighting each map~$\bm$ by a factor $1/\Aut(\bm)$, where $\Aut(\bm)$
denotes the number of automorphisms of~$\bm$. The latter is also equal to
$2\,|E(\bm)|/R(\bm)$, where $E(\bm)$ is the edge set of~$\bm$, and $R(\bm)$ 
is the number of distinct rooted
maps that can be obtained from the nonrooted map~$\bm$. Therefore, with
this convention, the weighted number of nonrooted maps in a given
family of maps with a given number $e$  
of edges is simply the cardinality of the set of rooted maps from this
same family, divided by $2e$.

\paragraph{Map with holes.}
We will consider maps with pairwise distinct distinguished elements, generically denoted by~$\ch_1$, $\ch_2$, \dots, $\ch_\kk$, that can be either \textgras{faces} or \textgras{vertices}. %
These distinguished elements are called the \emph{holes} of the map, a
given hole being called either an \emph{external face} or an
\emph{external vertex}, depending on its nature. The nondistinguished
faces and vertices are called the \emph{internal faces} and 
\emph{internal vertices}. The degree of a hole, also called its \emph{perimeter},
is defined as~$0$ for an external vertex or as the degree of the face
for an external face. Beware that the boundaries of the external faces
are in general neither simple curves, nor pairwise disjoint. As a
result, the object obtained by removing them from the surface in which
the map is embedded is not necessarily a surface. Note 
that, however, removing from every external face an open disk whose
closure is included in the (open) face results in a surface with a boundary. 

\paragraph{Bipartite map.}
Finally, we say that a map is \emph{bipartite} if its vertex set can be partitioned into two subsets such that no edge links
two vertices of the same subset.

\paragraph{Tuples.}
The many tuples considered in this work will conventionally be denoted by a
boldface font letter (possibly with a subscript) and their coordinates with the same letter
in a normal font, with the index written as a superscript, as in
$\bx=(x^1,\ldots,x^r)$ for instance. When~$\bx$ is a tuple of real
nonnegative numbers, we set $\|\bx\|=\sum_{i=1}^r x^i$. We denote by $\bx\by$ the concatenation of~$\bx$ with~$\by$. Finally, when concatenating with a $1$-tuple, we often identify it with its unique coordinate, writing for instance $\bx 0=(x^1,\ldots,x^r,0)$.

\subsection{The Gromov--Hausdorff--Prokhorov topology}\label{secGHP}

In this paper, a \emph{metric measure space} is a triple
$(\X,d_\X,\mu_\X)$, where $(\X,d_\X)$ is a nonempty \textgras{compact} metric space and
$\mu_\X$ is a finite Borel measure on $\X$. We say that two metric
measure spaces $(\X,d_\X,\mu_\X)$ and $(\Y,d_\Y,\mu_\Y)$ are
\emph{isometry-equivalent} if there exists an isometry~$\phi$ from
$(\X,d_\X)$ onto $(\Y,d_\Y)$ such that $\mu_\Y=\phi_*\mu_\X$. This
defines an equivalence relation on the class of all metric measure
spaces. If $(\X,d_\X,\mu_\X)$ and $(\Y,d_\Y,\mu_\Y)$
are two metric measure spaces, the 
\emph{Gromov--Hausdorff--Prokhorov metric} (\emph{GHP metric} for short)
is defined by 
\begin{equation}\label{eqGHP}
 \dGHP\big((\X,d_\X,\mu_\X),(\Y,d_\Y,\mu_\Y)\big)=
\inf_{\substack{\hphantom{\psi}\llap{\scriptsize$\phi$}\,:\,\X\to\ZZ\\\psi\,:\,\rlap{\scriptsize$\Y$}\hphantom{\X}\to\ZZ}}
 \left\{d_\ZZ^{\mathrm{H}}(\phi(\X),\psi(\Y))\vee
   d_\ZZ^{\mathrm{P}}(\phi_*\mu_\X,\psi_*\mu_\Y)\right\},
\end{equation}
where the infimum is taken over all choices of compact metric spaces
$(\ZZ,d_\ZZ)$, and all isometric maps~$\phi$, $\psi$ from~$\X$, $\Y$ to~$\ZZ$,
where~$d_\ZZ^{\mathrm{H}}$ is the \emph{Hausdorff metric} on compact subsets of~$\ZZ$, and $d_\ZZ^{\mathrm{P}}$ is the \emph{Prokhorov metric} on finite
positive measures on~$\ZZ$, defined as follows. First, for any $\eps>0$ and any closed subset $A\subseteq\ZZ$, we denote by
\[A^\eps = \Big\{ z\in \ZZ\,:\, \inf_{y\in A} d_\ZZ(z,y) < \eps\Big\}\]
its $\eps$-enlargement. Then, for any compact subsets~$A$, $B\subseteq\ZZ$,
\[d_\ZZ^{\mathrm{H}}(A,B)=\inf \{ \eps >0\,:\, A\subseteq B^\eps\text{ and }B\subseteq A^\eps \},\]
and, for any finite Borel measures~$\mu$, $\nu$ on~$\ZZ$,
\MEMS{
\begin{multline*}
d_\ZZ^{\mathrm{P}}(\mu,\nu)=\\\inf \big\{\eps >0\,:\, \text{for all closed } A\subseteq\ZZ,\,
	\mu(A)\le \nu(A^\eps) + \eps \text{ and } \nu(A)\le \mu(A^\eps)+\eps\big\}.
\end{multline*}
}{%
\[d_\ZZ^{\mathrm{P}}(\mu,\nu)=\inf \big\{\eps >0\,:\, \text{for all closed } A\subseteq\ZZ,\,
	\mu(A)\le \nu(A^\eps) + \eps \text{ and } \nu(A)\le \mu(A^\eps)+\eps\big\}.\]
}

Equation~\eqref{eqGHP} defines a metric on the
set~$\M$ of isometry-equivalence classes of metric measure spaces, making
it a complete and separable metric space. The references \cite[Chapter~27]{villani09} as well as \cite{AbDeHo13,LGa19dis} discuss
relevant aspects of the GHP topology, with some variations, as the exact
definition of the metric may differ from place to place. 

More generally, for $\ell$, $m\geq
0$, we will
consider \emph{$\ell$-marked, $m$-measured metric spaces} of the form
$(\X,d_\X,\bA,\bmu_\X)$, where
\begin{itemize}
\item $(\X,d_\X)$ is a nonempty compact metric space,
\item $\bA$ is an $\ell$-tuple, called \emph{marking}, of nonempty compact subsets of~$\X$,
called \emph{marks},
\item
$\bmu_\X$ is an $m$-tuple of finite Borel measures on $\X$.
\end{itemize}
We often consider
marks that are singletons; in this case, we identify the singleton
with the point it contains. We define the \emph{$\ell$-marked, $m$-measured Gromov--Hausdorff--Prokhorov metric} (still \emph{GHP metric} for short) on such spaces by
\begin{multline*}
\dGHP^{(\ell,m)}\big((\X,d_\X,\bA,\bmu_\X),(\Y,d_\Y,\bB,\bmu_\Y)\big)\\=
	\inf_{\substack{\hphantom{\psi}\llap{\scriptsize$\phi$}\,:\,\X\to\ZZ\\\psi\,:\,\rlap{\scriptsize$\Y$}\hphantom{\X}\to\ZZ}}
		\Big\{d_\ZZ^{\mathrm{H}}(\phi(\X),\psi(\Y)) 
			\vee \max_{1\leq i\leq \ell}d_\ZZ^{\mathrm{H}}(\phi(A^i),\psi(B^i))
			\vee \max_{1\leq j\leq m}d_\ZZ^{\mathrm{P}}(\phi_*\mu^j_\X,\psi_*\mu^j_\Y) \Big\},
\end{multline*}
where the infimum is taken over the same family as in~\eqref{eqGHP}. Again, this defines a complete and separable metric on the
set~$\M^{(\ell,m)}$ of isometry-equivalence classes of $\ell$-marked,
$m$-measured metric spaces, where $(\X,d_\X,\bA,\bmu_\X)$ and
$(\Y,d_\Y,\bB,\bmu_\Y)$ are isometry-equivalent if there exists
an isometry~$\phi$ from~$\X$ onto~$\Y$ such that $\phi(A^i)=B^i$ for
$1\leq i\leq \ell$ and $\phi_*\mu^j_\X=\mu^j_\Y$
for $1\leq j\leq m$. Note that we have
$\dGHP^{(0,1)}=\dGHP$. 
Finally, the space $(\M^{(\ell)},\dGH^{(\ell)})=(\M^{(\ell,0)},\dGHP^{(\ell,0)})$ of $\ell$-marked compact metric
spaces without measures is the so-called \emph{$\ell$-marked
Gromov--Hausdorff metric} (\emph{GH metric} for short).

As a first example, we will sometimes use in the present work the \emph{point space} $\{\varrho\}$ consisting of a single point, seen as the element $(\{\varrho\},(\varrho,\dots,\varrho),(0,\dots,0)\}\in \M^{(\ell,m)}$ for any values of~$\ell$ and~$m$.

In what follows, we will often simply use the terminology ``marked''
or ``measured'' instead of ``$\ell$-marked'' or ``$m$-measured'' if the
numbers~$\ell$ or~$m$ are clear from the context. Furthermore, when $m\ge 2$, we will often single out the first measure by writing it as a separate coordinate, writing $(\X,d_\X,\bA,\mu_\X,\bnu_\X)$ for instance. The reason is that this first measure will often be an \emph{area} measure whereas the other will be \emph{boundary} measures, and these have different natural scales, as we will see shortly.

\subsection{The main convergence result}\label{sec:main-conv-result}

\paragraph{Brownian surfaces.}
For $\kk\ge 0$, a \emph{quadrangulation with~$\kk$ holes} is a \textgras{bipartite} map having~$\kk$ holes~$\ch_1$, \dots, $\ch_\kk$ and whose internal faces are all of degree~$4$. For\footnote{We will write $\Zp=\{0,1,2,\dots\}$ the set of nonnegative integers, as well as $\N=\{1,2,3,\dots\}$ the set of positive integers.} $n\in\Zp$ and $\bl=(l^1,\dots,l^\kk) \in (\Zp)^\kk$ (with the convention that $(\Zp)^0= \{\varnothing\}$), we define the set~$\RbQnl$ of all genus~$g$ \emph{rooted} quadrangulations with~$\kk$ holes having~$n$
internal faces, and whose holes~$\ch_1$, \dots, $\ch_{\kk}$ are of respective degrees $2l^1$, \dots, $2l^{\kk}$; 
see Figure~\ref{exmapc} for an example.

\begin{figure}[htb!]
	\centering\includegraphics[width=88mm]{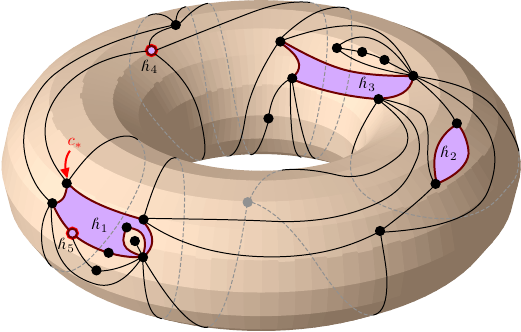}
	\caption{A quadrangulation from $\RbQ_{19,(4,1,2,0,0)}^{[1]}$.
          The root is the corner~$c_*$. Here, $\ch_1$, $\ch_2$, and~$\ch_3$ are external faces, while~$\ch_4$ and~$\ch_5$ are
          external vertices.
          }
	\label{exmapc}
\end{figure}

Likewise, we denote by~$\bQnl$ the set of \emph{nonrooted} quadrangulations
of genus~$g$ with~$n$ internal faces and half-perimeters given by~$\bl$. Since maps
are counted with an inverse factor given by the number of
automorphisms, the weighted cardinality of this set is 
\begin{equation}\label{eqnrbQ} 
  \sum_{\bq\in\bQnl} \frac1{\Aut(\bq)} = \frac{\big|\RbQnl\big|}{4n+2\|\bl\|}, 
\end{equation}
where $|\RbQnl|$ is the cardinality of $\RbQnl$, and
$4n+2\|\bl\|$ is the number of oriented edges, hence of
potential roots, in any element of~$\bQnl$. 

It will be useful to notice for further reference that the quadrangulations with~$\kk$ holes in~$\RbQnl$ or in~$\bQnl$ all have the same number of internal vertices, namely
\begin{equation}\label{nbintvert}
n+\|\bl\|+2-2g-\kk
\end{equation} 
Indeed, let us consider such a map, and denote by~$v$, $e$, $f$, its number of vertices, edges, faces. The number of external faces is thus $f-n$ so that the desired number is $v-\kk+f-n$. Furthermore, counting the corners yields $2e=4n+2\|\bl\|$, and the result follows from the Euler characteristic formula $v-e+f=2-2g$.

If~$\bq$ is a quadrangulation with~$\kk$ holes, we can view it as a
$\kk$-marked, $\kk+1$-measured metric space, in the following way.
We let $V(\bq)$ be the vertex set of~$\bq$, and $d_\bq$ the graph
metric on this set. We let
\[
\partial \bq=\big(V(\ch_1),\ldots,V(\ch_\kk)\big), 
\]
where for $1\le i \le \kk$,
$V(\ch_i)$ is either~$\{\ch_i\}$ if~$\ch_i$ is an external vertex, or
the set of vertices incident to~$\ch_i$ if it is an external face.  We
let~$\mu_\bq$ and $\bnu_{\partial\bq}$ be the measures on $V(\bq)$ and
the elements of $\partial \bq$ defined by:
\[\mu_\bq =\sum_{v\in V(\bq)}\delta_v,\qquad
  \nu_{\partial\bq}^i=\sum_{v\in V(\ch_i)}m_v\,\delta_v,\]
where $m_v$, the \emph{ multiplicity} of $v$, is the number of corners of the face~$\ch_i$ that are
incident to~$v$ (by convention, we set $m_{\ch_i}= 1$ for an external vertex~$\ch_i$). These are respectively called the \emph{area measure} and
\emph{boundary measures}. While we believe that our results also hold when 
$\nu_{\partial\bq}^i$ is replaced by the counting 
measure on~$V(\ch_i)$ (without multiplicities), it turns out that the above definition makes 
matter simpler.  
We associate with the quadrangulation~$\bq$ the space
\[\big(V(\bq),d_\bq,\partial \bq,\mu_\bq,\bnu_{\partial\bq}\big)\in \M^{(\kk,\kk+1)}.\]

Our main result exhibits a family 
\[\BS{g}{\bL},\qquad g\ge 0, \quad \bL\in \bigsqcup_{\kk\geq 0}[0,\infty)^\kk,\]
of random marked measured metric spaces, where $\BS{g}{\bL}$ will be called the
\emph{Brownian surface of genus~$g$ with boundary perimeter 
vector~$\bL$ and unit area}. The latter family describes the scaling limits
of uniform random elements of~$\RbQnln$, in the following sense. 
Define the scaling operator~$\Omega_n$ by 
\begin{equation}\label{defomega}
\Omega_n(\bq)=\Bigg(V(\bq),\Big(\frac{9}{8n}\Big)^{1/4}d_{\bq},\partial
\bq,\frac{1}{n}\mu_{\bq},\frac{1}{\sqrt{8n}}\bnu_{\partial\bq}\Bigg).
\end{equation}
The scaling constants $(8/9)^{1/4}$ and $\sqrt{8}$ are here to make
the upcoming description of $\BS{g}{\bL}$ simpler in Sections~\ref{sec:pfthm1} 
to~\ref{sec:univ}. Our main result is the following.

\begin{thm}\label{mainthm}
Fix $g$, $\kk\geq 0$. Let $\bL=(L^1,\dots,L^\kk)$ be a $\kk$-tuple of nonnegative real numbers and,
for $n\geq 1$, let $\bl_n=(l_n^1,\dots,l_n^\kk) \in (\Zp)^\kk$ be such that
$l_n^i/{\sqrt{2n}} \to L^i$ as $n\to\infty$, for $1\le i \le \kk$. Let $Q_n$ be a random
variable that is uniformly distributed over~$\RbQnln$. Then
\[\Omega_n(Q_n)\tod\BS{g}{\bL}\]
where the convergence holds in distribution in the space $\big(\M^{(\kk,\kk+1)},\dGHP^{(\kk,\kk+1)}\big)$.
\end{thm}

By our discussion on nonrooted maps, note that the same statement holds
if~$Q_n$ is rather distributed over the set $\bQnln$ of nonrooted
maps, with a probability proportional to the inverse of the number of
automorphisms. Note however that this automorphism number is equal to~$1$ for the vast majority of maps \cite{RiWo95}, so
we expect that our results also hold for genuine uniform random nonrooted maps. 

If $\BS{g}{\bL}=(\X,d_\X,\bA,\mu_\X,\bnu_\X)$, we will call $\mu_\X$
the \emph{area measure}, and $\bnu_\X$ the \emph{boundary measures}. Note that
$\mu_\X$ is a probability measure, since~\eqref{nbintvert} implies that $|V(Q_n)|\sim n$ as $n\to\infty$, while $\nu^i_\X$ has total mass
$L^i$ for $1\leq i\leq \kk$, so $\nu^i_\X$ is the trivial zero measure if $L^i=0$. 

Note that, for $(g,\kk)=(0,0)$, the above result amounts to the aforementioned convergence of plane
quadrangulations to the Brownian sphere \cite{legall11,miermont11},
while for $(g,\kk)=(0,1)$ with $L^1>0$, it corresponds to the convergence of
quadrangulations with a boundary to the Brownian disk~\cite{BeMi17}. Note however that the statement of the present paper is
slightly stronger, since it is formulated in terms of the marked
GHP topology rather than the weaker
GH topology. In the case $(g,\kk)=(0,0)$ of the Brownian sphere, it
amounts to the GHP topology since there are no marks and only one measure; this
stronger forms appears for instance in \cite[Theorem~1.2]{ABW17} and
\cite[Theorem~7]{LGa19dis}.

\paragraph{Topology and Hausdorff dimension.}
Let us also list some basic properties of the limiting metric spaces,
which justify the terminology of \emph{Brownian surfaces}. We say that a metric space is \emph{locally of Hausdorff dimension~$d$} if any nontrivial ball has Hausdorff dimension~$d$.

\begin{prp}\label{proptopo}
Let $\bL=(L^1,\dots,L^\kk)$ be fixed and let $\bb$ denote the number of positive coordinates of~$\bL$. Almost surely, the random metric space $\BS{g}{\bL}$ is
homeomorphic to~$\Sof{g}{\bb}$, is locally of Hausdorff dimension~$4$, and, if $\bb>0$, each of
the~$\bb$ connected components of its boundary, considered as a metric
space by restriction of the metric on $\BS{g}{\bL}$, is locally of Hausdorff
dimension~$2$.
\end{prp}

This statement is an immediate corollary of a result from~\cite{Bet16geo}, showing that, in the case $\bb=\kk$, any subsequential
limit in distribution of $n^{-1/4}Q_n$ satisfies the stated
properties. The case $\bb<\kk$ is easily obtained from there by the
observation concerning null perimeters at the end of this
section. However, our method of proof of Theorem~\ref{mainthm} will
also provide an alternative and rather transparent proof of
Proposition~\ref{proptopo}, once an analogous statement has
been established for the noncompact analogs of the cases of the sphere
and disk, namely, the Brownian plane and the Brownian half-plane
\cite{CuLG12Bplane,GwMi17,BaMiRa} (see Section~\ref{sectopo}). We also
mention that the case $(g,\kk,\bb)=(0,1,0)$ was obtained
in~\cite{bettinelli11b} (see also \cite{BoGu09}).

\paragraph{A comment on notation.} Throughout this paper, we will often work in fixed topology and consistently use the following pieces of notation, as in the above statement:
\begin{itemize}
	\item $g$ for the genus of the surface;
	\item $\kk$ for the size of the boundary perimeter vector, that is the number of holes in the discrete maps;
	\item $\bb$, as in \emph{boundary}, for the number of nonzero coordinates in the boundary perimeter vector;
	\item $\pp$, as in \emph{puncture}, for the number of null coordinates in the boundary perimeter vector.
\end{itemize}
Beware that the latter two numbers do not always correspond to the
numbers of external faces and external vertices in the discrete maps,
since we only require that $\bl_n/\sqrt{2n}\to\bL$. However, for $n$
sufficiently large, the~$\bb$ holes corresponding to the~$\bb$ nonzero
coordinates in the boundary perimeter vector are external faces; each
of the~$\pp$ remaining holes can be either a vertex or a face but, in
the latter case, it should be thought of as a ``small face'' in the
sense that its perimeter is of order $\sO(\sqrt n)$, and we will see that this implies a diameter of order $\sO(n^{1/4})$.

\paragraph{Method of proof.}
We prove Theorem~\ref{mainthm} by some
surgical methods, and from the known cases $g=0$ and $\kk\in \{0,1\}$.
Heuristically, we will cut~$\Qn$ along well-chosen
geodesics into a \emph{finite number} of elementary pieces of planar
topology, to which we can apply a variant of the cases
$(g,\kk)\in\{(0,0),(0,1)\}$ of Theorem~\ref{mainthm}. The idea of
cutting quadrangulations along geodesics into so-called slices appears
in Bouttier and Guitter \cite{BoGu09,BoGu12}. The use of these slices
and the study of their scaling limits play an important role in Le
Gall's proof~\cite{legall11} of the uniqueness of the Brownian sphere 
(they are called \emph{maps with a piecewise geodesic
  boundary} in this reference) and are crucial to our
study~\cite{BeMi17} in the case of the disk. More specifically, in
the latter reference, we view Brownian disks as a continuum
version of the slice decomposition.

The proof of Theorem~\ref{mainthm}
relies on similar but yet different ideas, and 
will require the introduction of other types of surgeries on objects
that we call \emph{(composite) slices} and \emph{quadrilaterals (with geodesic sides)}. 
The core of the proof of Theorem~\ref{mainthm} consists in
showing scaling limit results for these elementary pieces, as stated
in Theorems~\ref{thmslslice} and~\ref{thmslquad}. We believe
that these results are of independent interest, as elementary pieces
and their scaling limits might serve as building blocks in other
models of random surfaces. In order to prove this result, it turns out
that it is simpler to view the discrete and continuum elementary
pieces as embedded into non-compact version of the Brownian sphere and
disk, namely the Brownian plane and half-plane defined in
\cite{CuLG12Bplane,GwMi17,BaMiRa}. 
We stress that the description of the Brownian half-plane in terms of
gluing of composite slices considered in Section~\ref{secBHP} below is related
to the \emph{slice decomposition of metric bands} property used by
Miller and Qian~\cite{MiQi21arX} for studying geodesic stars in the
Brownian sphere.

Theorem~\ref{mainthm} generalizes the case of the sphere at two
different levels, one given by the positive genus and one given by the
addition of a boundary. Although these two levels of generalization
rely to some extent on similar ideas, the difficulties that 
they generate are of quite different nature. The case of the disk,
which was the focus of~\cite{BeMi17}, relied on relatively
well-understood objects, but 
required gluing an infinite number of such objects, which in principle
could create problems in the limit. On the other hand, the surgery
involved in the general case consists in gluing a
bounded number of objects, but the objects themselves will turn out to
be of a more complicated nature. 

\paragraph{Null perimeter coordinates.}
We end this section by the following observation relating Brownian
surfaces in case of null perimeter coordinates. The operations of adding or removing a mark used in the following proposition are given by Lemmas~\ref{lemrandmark} and~\ref{lem1lip} in Section~\ref{secmarkglue}. 

\begin{prp}\label{propnull}
Let $\bL=(L^1,\dots,L^\kk)\in [0,\infty)^k$, and~$\bL 0=(L^1,\ldots,L^\kk,0)$.
Then the space $\BS{g}{\bL 0}$ has
same distribution as the space $\BS{g}{\bL}$, where, denoting by~$\mu$ the area measure of the latter space, a 
random $\mu$-distributed point has been added to the set of
marks of $\BS{g}{\bL}$ in $(\kk+1)$-th position (and the zero measure
has been added as a trivial $(\kk+1)$-th boundary measure).

Consequently, if  $L^i=0$ for
some given~$i\in \{1,2,\ldots,\kk\}$, and if~$\hat{\bL}$ denotes the
vector~$\bL$ with $i$-th coordinate removed, then
$\BS{g}{\hat{\bL}}$ has same distribution as the space
$\BS{g}{\bL}$ with its $i$-th mark and (trivial) $i$-th boundary measure removed. 
\end{prp}

\begin{proof}
Let us fix 
$\bl_n=(l_n^1,\dots,l_n^{\kk}) \in (\Zp)^{\kk}$ such that
$l_n^j\sim \sqrt{2n}\,L^j$ for $1\leq j\leq \kk$, and let~$Q_n$
be uniformly distributed over~$\RbQnln$. Setting $\bl_n0=(l_n^1,\dots,l_n^{\kk},0)$, a uniformly distributed random variable~$Q_n'$
in~$\smash{\RbQnlnO}$ may be obtained by choosing an extra
distinguished external vertex $\ch_{\kk+1}$ uniformly at random among the
internal vertices of~$Q_n$, that is, according to the measure
$\mu_{Q_n}$ conditioned on the set of internal vertices. Since
the number of distinguished vertices in~$Q_n$ is at most~$\kk$, while the total number of vertices is
asymptotically equivalent to $n$, the GHP limit of the
quadrangulation $Q_n'$ rescaled as in Theorem~\ref{mainthm} is the same
as if we had chosen $\ch_{\kk+1}$ uniformly at random among the set of all
vertices of $Q_n$. By Theorem~\ref{mainthm} applied to~$Q_n$ 
and Lemma~\ref{lemrandmark} below, we obtain the result. 
The second part of the statement is obtained by permuting or removing
marks and measures appropriately, as discussed in Lemma~\ref{lem1lip} below. 
\end{proof}

As an example, the Brownian
sphere $\BS{0}{\varnothing}$ can be seen as $\BS{0}{(0,0)}$ by
forgetting its two marks. Anticipating on the construction of the
Brownian surfaces in Section~\ref{sec:pfthm1}, this provides a
nontrivially equivalent construction of the Brownian sphere as the
gluing of one quadrilateral with geodesic sides, rather than the one
from \cite{legall11,miermont11}.

\subsection{Scaling limits of Boltzmann
  quadrangulations}\label{secboltz}

We may also consider scaling limits for models of
quadrangulations with holes in
which the area and perimeters are not fixed, but rather weighted
by Boltzmann factors. We introduce the following sets of nonrooted maps:
\begin{align*}
\bQl &= \bigsqcup_{n\ge 0} \bQnl,		
		&&\text{ for }g\ge 0, \ \bl\in \bigsqcup_{\kk\geq 0}(\Zp)^\kk,\\
\shortintertext{and}
	\bQ^{[g]}(\bb,\pp)&= \bigsqcup_{\bl\in \N^\bb} \bQ^{[g]}_{\bl\bzero^\pp},
		&&\text{ for $g$, $\bb$, $\pp \ge 0$},
\end{align*}
where $\bl\bzero^\pp$ stands for the sequence~$\bl$ to which we append~$\pp$ terms equal to~$0$.

We then let~$\WW$ be the $\sigma$-finite measure on the set of
nonrooted quadrangulations with an arbitrary number of holes and arbitrary genus,
given by
\[\WW(\bq)=\frac{1}{\Aut(\bq)}12^{-|\bq|}\,8^{-\|\partial
  \bq\|},
\]
where~$|\bq|$ is the number of internal faces of~$\bq$, and~$\|\partial\bq\|$ is the sum of the perimeters of its holes. The
reason for the choice of the weights $1/12$ and $1/8$ for the internal
faces and perimeters comes from the following enumeration result,
which will be proved in an extended form in Proposition~\ref{cardRbQ}, in Appendix~\ref{appdata}.
 
\begin{prp}\label{enumQnln}
Fix $\bb\geq 0$ and $\bL\in (0,\infty)^\bb$. Let $(\bl_n,n\geq 0)$ be
a sequence  of integers such that $l_n^i\sim\sqrt{2n}\,L^i$ as
$n\to\infty$ for $1\leq i\leq \bb$. Then there exist a continuous function~$t_g$ of~$\bL$, such that
\[\WW\big(\bQ^{[g]}_{n,\bl_n}\big) \build{\sim}{n\to\infty}{}
	t_g(\bL)\, n^{\frac{5g-7}{2}+\frac{3\bb}{4}}.\]
\end{prp}

The function $t_g(\bL)$ is related to the so-called
\emph{double scaling limit} of maps, as described in 
\cite[Chapter~5]{eynard16}, and its Laplace transform can be computed
by solving Eynard and Orantin's 
topological recursion. The method presented in
Appendix~\ref{appdata} is based on the bijections presented in Section
\ref{secCVS}. 

For any $g\geq 0$, $\pp\geq 0$, $\bL\in [0,\infty)^\pp$ and $A>0$, if
$(\X,d,\bA,\mu)$ is a random variable with same law as
$\smash{\BS{g}{\bL/\sqrt{A}}}$, we define the Brownian surface of
genus~$g$ with boundary perimeter vector~$\bL$ and area~$A$ as a random variable
$\BS{g}{A,\bL}$ with same law as $(\X,A^{1/4}d,\bA,A\mu)$. 
If $\bL\in \bigsqcup_{\bb\geq 0}(0,\infty)^\bb$ and $\pp\geq 0$, we let $\bL\bzero^\pp\in [0,\infty)^{\bb+\pp}$ be the sequence~$\bL$ to which we append~$\pp$ terms equal to~$0$. 

For integers $g$, $\bb$, $\pp\geq 0$, and for $\bL\in (0,\infty)^\bb$, setting $\kk=\bb+\pp$, we define a
$\sigma$-finite measure on $\M^{(\kk,\kk+1)}$ by the formula 
\[\scS^{[g]}_{\bL,\pp}(\cdot)=\int_{(0,\infty)} \d A \,
A^{\frac{5g-7}{2}+\frac{3\bb}{4}+\pp}\, 
	t_g\big(\bL/\sqrt{A}\big)\, \P\big(\BS{g}{A,\bL\bzero^\pp}\in
\cdot\,\big).\]

The measure $\scS^{[g]}_{\bL,\pp}$ is a $\sigma$-finite measure that
``randomizes'' the area measure of the Brownian
surface of genus~$g$ with~$\bb$ boundary components of lengths given by~$\bL$, as well as~$\pp$ marked vertices, in
the sense that its conditional law given having total area~$A$ is that
of~$\BS{g}{A,\bL \bzero^\pp}$.

Recall that the scaling operator~$\Omega_n$ is defined by~\eqref{defomega}; here, we use it for any $n\in(0,\infty)$.

\begin{thm}\label{thmboltz}
Let $g$, $\bb$, $\pp\in \Zp$, let $\kk=\bb+\pp$, let $\bL\in (0,\infty)^\bb$, let $K>0$, and let $F:\M^{(\kk,\kk+1)}\to \R$ be a continuous and bounded
function that is supported on the set of spaces
$(\X,d_\X,\bA,\mu_\X,\bnu_{\X})$ such that $\mu_\X(\X)\in [1/K,K]$.
Let $(\bl_a,a>0)$ be a family where $\bl_a\in\N^\bb$ is such that $l_a^i\sim
\sqrt{2/a}\,L^i$ for $1\leq i\leq \bb$. Then, it holds that
\[
a^{\frac{5(g-1)}{2}+\frac{3\bb}{4}+\pp}\, 
	\WW\left(F\left(\Omega_{a^{-1}}(Q)\right)
		\ind_{\bQ^{[g]}_{\bl_a\bzero^\pp}} \right)
	\build{\longrightarrow}{a\downarrow 0}{} \scS^{[g]}_{\bL,\pp}(F). 
\]
\end{thm}

Note that our main result, Theorem~\ref{mainthm}, can be seen as a
``local limit'' version of Theorem~\ref{thmboltz}, in the sense that it gives the
conditional statement of this last result given
$\smash{\bQ^{[g]}_{a^{-1},\bl_a\bzero^\pp}}$, taking $a=1/n$. 
There is also a version of this theorem where the perimeters given 
by~$\bL$ are left free as well. For~$g$, $\bb$, $\pp\geq 0$, we define the
$\sigma$-finite measure
\[\scS^{[g]}_{\bb,\pp}(\cdot)=\int_{(0,\infty)^{\bb}} \d \bL\,
\scS^{[g]}_{\bL,\pp}(\cdot).\]

\begin{crl}\label{corboltz}
Let $g$, $\bb$, $\pp\in \Zp$, $\kk=\bb+\pp$, $K>0$, and $F:\M^{(\kk,\kk+1)}\to \R$ be a continuous and bounded
function that is supported on the set of spaces
$(\X,d_\X,\bA,\mu_\X,\bnu_{\X})$ such that $\mu_\X(\X)$ and
$\nu^i_\X(A^i), 1\leq i\leq \bb$, all lie in $[1/K,K]$. 
Then it holds that
\[
2^{\frac{\bb}{2}} a^{\frac{5(g-1)}{2}+\frac{5\bb}{4}+\pp}\, 
	\WW\left(F\left(\Omega_{a^{-1}}(Q)\right)
		\ind_{\bQ^{[g]}(\bb,\pp)}
		\right)
	\build{\longrightarrow}{a\downarrow 0}{}  \scS^{[g]}_{\bb,\pp}(F). 
\]
\end{crl}

Interestingly, the measure $\scS^{[g]}_{\bL,\pp}$ is finite in the
particular cases $g=0$, $\bb=1$ and $\pp\in \{0,1\}$, or $g=0$, $\bb=2$ and
$\pp=0$; it can be checked that it is infinite in all other cases. By
computing the functions~$t_0(\bL)$ in the case $\bb\in \{1,2\}$, we obtain three probability distributions
by normalizing the measures $\scS^{[0]}_{(L),0}$,
$\scS^{[0]}_{(L),1}$, $\scS^{[0]}_{(L,L'),0}$. Those are the law of the \emph{free Brownian disk} of
perimeter $L\in (0,\infty)$:
\[\mathrm{FBD}_L=\int_0^\infty \d A\, \frac{L^3}{\sqrt{2\pi
    A^5}}\exp\left(-\frac{L^2}{2A}\right)\, 
\P\big(\BS{0}{(L),A}\in \cdot\,\big),\]
the law of the \emph{free pointed Brownian disk} of
perimeter $L\in (0,\infty)$:
\[\mathrm{FBD}^\bullet_L=\int_0^\infty  \d A\, \frac{L}{\sqrt{2\pi
    A^3}}\exp\left(-\frac{L^2}{2A}\right)\, 
\P\big(\BS{0}{(L,0),A}\in \cdot\,\big),\]
and the law of the \emph{free Brownian annulus} of boundary perimeters~$L$, $L'$\,:
\[\mathrm{FBA}_{L,L'}=
\int_0^\infty  \d A\, \frac{(L+L')}{\sqrt{2\pi
    A^3}}\exp\left(-\frac{(L+L')^2}{2A}\right) \,
\P\big(\BS{0}{(L,L'),A}\in \cdot\,\big).\]
Note in particular that $\mathrm{FBD}^\bullet_L=\lim_{\eps\downarrow
  0}\mathrm{FBA}_{L,\eps}$. 
These laws, as well as the associated $\sigma$-finite measures
$\scS^{[0]}_{1,0}$,
$\scS^{[0]}_{1,1}$, $\scS^{[0]}_{2,0}$, play an important
role in~\cite{ARS22mod}. 

In the case $\bb=0$, the two previous statements are in fact the
same, since 
$\scS^{[g]}_{\varnothing,\pp}=\scS^{[g]}_{0,\pp}$. This measure
describes the scaling limit of quadrangulations with no boundary, $\pp$
marked vertices, and free area measure. In this case, the quantity
$t_g(\varnothing)$
is equal to the
classical universal constant $t_g$ arising in map enumeration; see
\cite{BenCan86,LaZv04}. Explicitly, the numbers
$\tau_g=2^{5g-2}\Gamma(\frac{5g-1}{2})t_g$ satisfy $\tau_0=-1$ and the recursion
\[\tau_{g+1}=\frac{(5g+1)(5g-1)}{3}\tau_g+\frac{1}{2}\sum_{h=1}^g\tau_{h}\tau_{g+1-h}\,
,\qquad g\geq 0.\]
In this case, we thus have the following formula 
\[\scS^{[g]}_{0,\pp}(\cdot)=t_g\int_{(0,\infty)} \d A \,
A^{\frac{5g-7}{2}+\pp}\, \P(\BS{g}{A,\bzero^\pp}\in
\cdot\,).\]

\subsection{Perspectives}
A natural question, which we plan to investigate in future works,
is to derive the analog of Theorem~\ref{mainthm} for bipartite
quadrangulations on \emph{nonorientable compact surfaces}, using the
bijective techniques developed in \cite{ChapuyDolega,Bet22}. The first step of showing the existence of subsequential limits for nonorientable
quadrangulations without boundary has been taken in~\cite{ChapuyDolega}. Addressing this question would
complete the catalog of compact Brownian surfaces. 
 
As mentioned in the first section of this introduction, an important
aspect is that of \emph{universality} of the spaces 
$\BS{g}{\bL}$. In fact, we expect these spaces to be the scaling
limits of many other models of random maps on surfaces. In the case
of the Brownian sphere $\BS{0}{\varnothing}$, this was indeed verified
for several models; see the references mentioned above. In the case of
Brownian disks, we showed in~\cite{BeMi17} 
that the spaces $\BS{0}{(L^1)}$ appear as scaling limits of many
conditioned Boltzmann models. This approach to 
universality should generalize to our context, at the price of some
specific technicalities. We will not address this question here, but will comment
more on this in Section~\ref{sec:univ}. 

It would be most interesting to complete the bridge between Brownian
surfaces and LQG metrics and CFT. As was pointed to us by J.\
Miller, in order to define a canonical conformal
structure and Brownian motion on Brownian surfaces, it would
be natural to investigate whether the construction of general
Brownian surfaces given in the present paper, by gluing elementary
pieces of disk topologies along geodesic boundaries, can be made
compatible with the approach of
\cite{GwMiSh20,GwMiSh22} mentioned in the introduction. Knowing that
such a structure exists, one can try to delve even further 
into its integrability properties. The works
\cite{DaKuRhVa16,GuRhVa19} state precise conjectures linking
Liouville CFT with scaling limits of the area measure of random maps
(without boundary) after suitable 
uniformization. In a nutshell, the LQG metrics are local objects
that can be defined globally by using charts and atlases on general
Riemann surfaces. However, fixing a surface  amounts to fixing the
conformal modulus of the LQG metric, while Brownian
surfaces have a random modulus. Hence, the
computation of the law of this modulus is an important question, which 
has been solved by~\cite{ARS22mod} in the case of the annular
topology. It seems that the case of general compact surfaces should be
approachable as well given the recent developments on conformal
bootstrap in Liouville CFT~\cite{GuKuRhVa21,Wu22}. 

We also mention that random surfaces with boundaries of the type studied
in this paper are related
to the study of self-avoiding paths in random geometries. See
\cite{GwMi19a,GwMi21a} for more on this in
the case of the gluing of two Brownian half-planes or disks. It would
be interesting to explicitly describe the scaling limits of self-avoiding paths and
loops on maps of fixed topologies as gluings of Brownian surfaces along
boundaries. As N.\ Holden
pointed to us, this would involve presumably difficult computations of the partition functions
for self-avoiding loops in fixed classes of the fundamental group of
the surface, although this problem simplifies in the case of the
self-avoiding loop on a Brownian sphere \cite{AnHoSu23}.

\subsection{Organization of the paper}

In Section~\ref{secCVS}, we present the extension of the famous
Cori--Vauquelin--Schaeffer bijection allowing to encode a
quadrangulation with a simpler tree-like structure carrying integer
labels on its vertices. We also present a variant
of the bijection, which leads to the definition of the elementary
pieces into which we decompose a
quadrangulation. We finally state the relevant scaling limit results for
these elementary pieces. In Section~\ref{secmarkglue}, we present the surgical
operation we need in order to reconstruct a metric space from its
elementary pieces, namely gluing along geodesic segments. The proof of
Theorem~\ref{mainthm} and Proposition~\ref{proptopo} are
given in Section~\ref{sec:pfthm1}. In
Sections~\ref{seccvcs} and~\ref{seccvquad}, we present the metric
spaces forming the continuum elementary pieces into consideration and explain how they are
natural building blocks of the Brownian plane and half-plane, which
are the noncompact analogs of the Brownian sphere and disk, and we
tweak known convergence results to these noncompact Brownian surfaces
to prove that the continuum elementary pieces are the scaling limits
of the discrete elementary pieces. Finally, we give in
Section~\ref{sec:univ} an alternate description of Brownian surfaces
that does not involve gluing operations and that is closer to the
usual definition of the Brownian sphere and disks.

\section{Variants of the Cori--Vauquelin--Schaeffer bijection}\label{secCVS}

As is customary when studying on scaling limits of maps, this work strongly
relies on powerful encodings of discrete maps by tree-like objects. We
now present variants of the famous Cori--Vauquelin--Schaeffer (CVS)
bijection \cite{CoVa,schaeffer98} between plane quadrangulations and so-called \emph{well-labeled trees}, and its generalizations
by Chapuy--Marcus--Schaeffer \cite{ChMaSc} for higher genera and by Bouttier--Di Francesco--Guitter~\cite{BdFGmobiles} for plane maps with faces of arbitrary degrees. We only give the constructions from the
encoding objects to the considered maps and refer the reader to the
aforementioned works for converse constructions and proofs.

\subsection{Basic construction}\label{sec:basic-construction}

Let $\bm$ be a map, rooted or not, and~$f$ be a face of~$\bm$. Starting from a choice of a corner~$c_0$ in~$f$, we index the
subsequent corners of~$f$ in counterclockwise order as $(c_i,i\in \Z)$ (forming a periodic sequence). Let $\lambda:V(\bm)\to\Z$ be a labeling of the vertices of~$\bm$ by integers. We extend the definition of~$\lambda$ to the corners of the map by setting $\lambda(c)=\lambda(v)$ if~$v$ is the vertex incident to the corner~$c$. In what follows, we will
either consider that~$\lambda$ is defined up to addition of a
constant, or that the value of~$\lambda$ at some corner is fixed, for
instance that $\lambda(c_0)=0$. 

We say that $(\bm,\lambda)$ is \emph{well labeled} inside~$f$ if
$\lambda(c_{i+1})\geq \lambda(c_i)-1$ for every $i\geq 0$. In
particular, if $(\bm,\lambda)$ is well labeled inside~$f$ and~$e$ is
a half-edge of~$\bm$ such that both~$e$ and its reverse~$\bar{e}$ are incident to~$f$, then
$|\lambda(e^+)-\lambda(e^-)|\leq 1$, where~$e^-$, $e^+$ denote
the origin and end of~$e$. Note that this will be the case for
every edge when~$\bm$ is a map with a single face.

Let $(\bm,\lambda)$ be well labeled inside~$f$. With the above
notation, let us define $s(i)=\inf\{j>i:\lambda(c_j)=\lambda(c_i)-1\}\in \Z\cup \{\infty\}$ and the \emph{successor} of~$c_i$ as $s(c_i)=c_{s(i)}$, where~$c_\infty$ is
by convention the unique corner incident to a vertex~$v_*$ that is
added in the interior of~$f$, and which naturally carries the label
$\lambda(v_*)=\min\{\lambda(c_i),i\in \Z\}-1$.  Clearly, $s(c_i)$ is
then well defined for all corners (distinct from~$c_\infty$), and only
depends on the corner~$c_i$ and not on the particular choice of
the index~$i$. The \emph{CVS construction inside the face~$f$} consists 
in
\begin{itemize}
\item 
linking by an arc every corner~$c$ incident to~$f$ to its successor
$s(c)$, in such a way that arcs do not cross, which is always possible
due to the well labeling condition, 
\item
deleting all the edges of $\bm$. 
\end{itemize}
This construction results in an embedded graph\footnote{In general,
  this embedded graph is not a map of the surface into
  consideration. In all the constructions we will use in this work, it
  will, however, always turn out to be a map.} denoted by $\CVS(\bm,\lambda;f)$, whose vertex set consists of~$v_*$ and the vertices
of~$\bm$ incident to~$f$, and whose edges are the arcs between the
corners of $f$ and their successors. By construction, the edges of $\CVS(\bm,\lambda;f)$ are in bijection with the corners of~$\bm$ incident to~$f$. If~$\bm$ is rooted inside~$f$, say at the corner~$c_i$, then $\CVS(\bm,\lambda;f)$ naturally inherits a root at the corner preceding the arc linking~$c_i$ to~$s(c_i)$. 
Note that the well labeling condition, as well as the output $\CVS(\bm,\lambda;f)$,
are invariant under addition of a constant to~$\lambda$, as they should. 

\bigskip
We will also need an \emph{interval variant} of this construction, where
we fix a sequence, also referred to as an \emph{interval}, of subsequent corners $I=\{c_0,c_1,\ldots,c_r\}$ of a
face~$f$ of~$\bm$, and only ask that $(\bm,\lambda)$ is \emph{well
 labeled} on~$I$ in the sense that $\lambda(c_{i+1})\geq
\lambda(c_i)-1$ for $0\le i \le r-1$. In this case, we set $\lambda_*=\min\{\lambda(c_i),0\leq i\leq r\}-1$ and
$\ell=\lambda(c_r)-\lambda_*$. Instead of a single extra corner~$c_\infty$, we introduce inside~$f$ a
sequence of distinct consecutive corners
$c_{r+1}$, $c_{r+2}$, \ldots, $c_{r+\ell}$, incident to new vertices
$v_{r+1}$, $v_{r+2}$, \ldots, $v_{r+\ell}$ with labels
$\lambda(c_r)-1$, $\lambda(c_r)-2$, \ldots, $\lambda_*$. The successor
mapping~$s$ is then defined for all corners except~$c_{r+\ell}$. We let
$\CVS(\bm,\lambda;I)$ be the resulting (nonrooted) embedded graph whose edges are the
arcs. In this embedded graph, the following are of particular interest:
\begin{enumerate}[label=(\textit{\arabic*})]
	\item the \emph{apex}~$v_{r+\ell}$, which will usually be
          denoted with a subscript $*$; 
	\item the \emph{maximal geodesic}, which is the chain of arcs linking~$c_0$, $s(c_0)$, $s(s(c_0))$,
\ldots, $c_{r+\ell}$, and which will always be denoted with the
letter~$\gamma$ and depicted in red \MEMS{(darker color) }{}in the figures;
	\item the \emph{shuttle}, which is the chain of arcs linking
          $c_r$, $c_{r+1}$, \ldots, $c_{r+\ell}$, and which will
          always be denoted with the letter~$\xi$ and depicted in
          green \MEMS{(lighter color) }{}in the figures.
\end{enumerate}
Note that the two latter are paths from the first and last corners of~$I$ to the apex.

The construction generalizes to several intervals~$I$, $J$, \ldots\
that pairwise share at most one extremity. In the case of a shared
extremity, say $I=\{c_0,c_1,\ldots,c_r\}$ and
$J=\{c'_0,c'_1,\ldots,c'_{r'}\}$ with $c_r=c'_0$, one first duplicates
the common corner before applying the construction, in the sense that
the copy~$c_r$ is used in the shuttle of~$I$ and~$c'_0$ is used in the
maximal geodesic of~$J$; see Figure~\ref{figintCVS}. In this construction, each interval
yields a distinct apex, maximal geodesic, and shuttle, and the
construction results in an embedded graph denoted by
$\CVS(\bm,\lambda;I,J,\ldots)$. Plainly, the ordering of the intervals
does not affect the construction. 

\begin{figure}[ht!]	
	\centering\includegraphics[width=.95\linewidth]{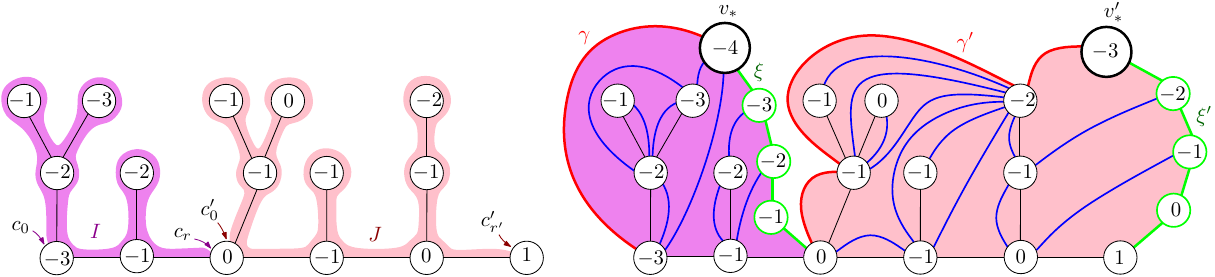}
	\caption{Performing the interval variant of the
          Cori--Vauquelin--Schaeffer bijection with two intervals
          sharing an extremity. The interval~$I$ consists of the
          corners in the purple (darker) area, starting with~$c_0$ and ending
          with~$c_r$, while~$J$ consists of the corners in the red
          (lighter) area, starting with~$c'_0$ and ending
          with~$c'_{r'}$. The interval~$I$ yields the apex~$v_*$,
          maximal geodesic~$\gamma$, and shuttle~$\xi$, while~$J$
          yields respectively~$v'_*$, $\gamma'$, and~$\xi'$. As will
          be the case in all the figures, the maximal geodesics are in
          red \MEMS{(darker colored boundary) }{}and the shuttles in green\MEMS{ (lighter colored boundary)}.}
	\label{figintCVS}
\end{figure}

We make the important observation that any chain 
$c$, $s(c)$, \ldots, $s^i(c)$ of consecutive successors induces
a geodesic chain for the graph metric in the resulting embedded graph
$\CVS(\bm,\lambda;I,J,\dots)$, that is, a path of minimal length between its extremities. This is
simply because, by construction, any arc of the resulting embedded graph links
two vertices~$u$ and~$v$ such that $|\lambda(u)-\lambda(v)|=1$, and
because~$\lambda$ decreases by~$1$ at every step on a chain of consecutive
successors. In particular, the maximal geodesics and shuttles of
$\CVS(\bm,\lambda;I,J,\ldots)$ are geodesic chains.

\subsection{The generalized Chapuy--Marcus--Schaeffer bijection}\label{secCMS}

\paragraph{Encoding quadrangulations.}
As a first example, let us perform this construction on a particular class of maps. For $n\in\Zp$ and $\bl=(l^1,\dots,l^\kk) \in (\Zp)^\kk$, we
let~$\RbM^{[g]}_{n,\bl}$ be the set of labeled rooted maps $(\bm,\lambda)$
satisfying the following properties:
\begin{itemize}
	\item $\bm$ is a map of genus~$g$ with $n+\sum_{i=1}^\kk l^i$ edges, one internal face~$f_*$ and~$\kk$ holes~$\ch_1$, \dots, $\ch_\kk$, rooted at a corner of its internal face~$f_*$;
	\item for all~$i$, the hole~$\ch_i$ is of degree~$l^i$; if it is an external face, then it has a simple boundary\footnote{A face has a \emph{simple boundary} if it is incident to as many vertices as its degree.};
	\item for any $i\neq j$, if~$\ch_i$ and~$\ch_j$ are faces, then they are not incident to any common edge;
	\item $(\bm,\lambda)$ is well labeled inside $f_*$. 
\end{itemize}
We similarly define the set $\bM^{[g]}_{n,\bl}$ of labeled nonrooted maps. Setting $\bl0=(l^1,\dots,l^\kk,0)$, the CVS construction applied to the internal face~$f_*$ provides a bijection between $\smash{\bM^{[g]}_{n,\bl}}$ and~$\smash{\bQ^{[g]}_{n,\bl0}}$, through which the~$k$ first holes correspond, while the extra hole~$\ch_{\kk+1}$ of the quadrangulation is the extra vertex~$v_*$ of the construction. In case of rooted maps, it yields a one-to-two correspondence\footnote{The factor~$2$ comes from the fact that the corners of~$f_*$ correspond to the edges of the resulting map, each edge corresponding to~$2$ half-edges. We refer the
interested reader to \cite[Section~3.1]{Bet16geo} for a presentation of the reverse mapping.} between $\smash{\RbM^{[g]}_{n,\bl}}$ and~$\smash{\RbQ^{[g]}_{n,\bl0}}$.

\paragraph{Decomposition into elementary pieces.}
Let us now perform the construction on the same set of maps~$\smash{\bM^{[g]}_{n,\bl}}$ but with well-chosen intervals. We will decompose a map of~$\bM^{[g]}_{n,\bl}$ into a collection of labeled forests indexed by an underlying structure called the \emph{scheme}. 
For the remainder of this section, we exclude the cases $(g,\kk)\in \{(0,0),(0,1)\}$ leading to encoding objects not entering the upcoming framework. We fix $(\bm,\lambda)\in\bM^{[g]}_{n,\bl}$.

\begin{figure}[htb!]
	\centering\includegraphics[width=.95\linewidth]{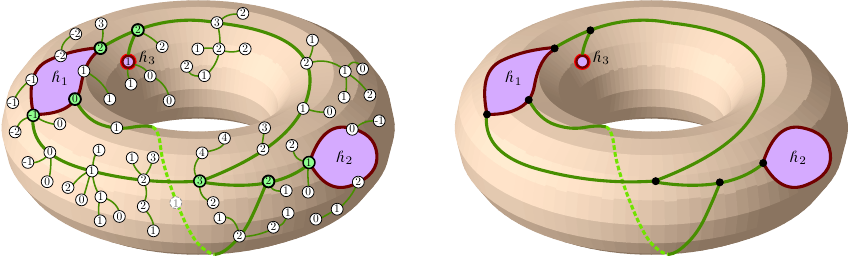}
	\caption{\textup{\textgras{Left.}} A labeled map from $\bM^{[1]}_{63,(6,3,0)}$. The outlined vertices are its nodes and the thicker edges correspond to the map~$\wt\bm$. \textup{\textgras{Right.}} The corresponding scheme.}
	\label{scheme}
\end{figure}

Let~$\wt\bm$ be the nonrooted map obtained from~$\bm$ by iteratively removing all its vertices of degree~$1$ that are not holes. The resulting map~$\wt\bm$ may be seen as a submap of~$\bm$: the map~$\bm$ is obtained from~$\wt\bm$ by appending rooted labeled trees at its corners. 
We call \emph{nodes} of~$\bm$ the following vertices:
\begin{itemize}
	\item the external vertices of~$\bm$;
	\item the vertices of~$\bm$ having degree~$3$ or more in~$\wt\bm$.
\end{itemize}
These nodes are linked in~$\wt\bm$ by maximal
chains of edges not containing any nodes other than their
extremities. 
Replacing every such chain with a single edge yields a nonrooted map~$\bs$,
called the \emph{scheme}~of $\bm$. 
It has one internal face, still denoted by~$f_*$, and~$\kk$ holes, still denoted by~$\ch_1$, \ldots, $\ch_{\kk}$; see Figure~\ref{scheme}.

We denote by~$\vec E(\bs)$ the set of half-edges incident to the
internal face of~$\bs$; this set is partitioned into the set~$\vec
I(\bs)$ of half-edges whose reverses belong to~$\vec E(\bs)$ as well, and the
set~$\vec B(\bs)$ of half-edges whose reverses do not belong to~$\vec
E(\bs)$. (We used the letter~$I$ for \emph{internal} and~$B$ for \emph{boundary}.)
The set $\vec{B}(\bs)$ is further partitioned as
\begin{equation*}
  \vec{B}(\bs)=\bigsqcup_{1\leq r\leq \kk}\vec{B}_r(\bs)
\end{equation*}
where $\vec{B}_r(\bs)$ is either empty if~$\ch_r$ is a vertex, or the set of half-edges of
$\vec{E}(\bs)$ whose reverse are incident to~$\ch_r$ if it is a face.
We consider $e\in\vec E(\bs)$. It corresponds to a chain~$e_1$, \dots,
$e_j$ of half-edges in~$\bm$. Let us denote by~$c_e$ and~$c'_e$ the
corners of~$\wt\bm$ preceding~$e_1$ and succeeding~$e_j$ in the
contour order. In~$\bm$, there are several corners that make up~$c_e$
and~$c'_e$. The corner interval~$I_e$ is the interval of corners
of~$\bm$ from the \textgras{first} corner corresponding to~$c_e$ to the \textgras{first}
corner corresponding to~$c'_e$. Observe that, in~$\bm$, the tree
grafted at~$c_e$ is thus covered by~$I_e$, whereas the tree grafted
at~$c'_e$ is not. 

By construction, $\bigcup_{e\in\vec E(\bs)} I_e$ is equal to the set
of corners of~$f_*$ and each extremity of these intervals is shared by
exactly two such intervals. More precisely, the intervals~$I_e$,
$e\in\vec E(\bs)$, with their last corner removed give a partition of
the corners of~$f_*$. 
Applying the
interval CVS construction $\CVS(\bm,\lambda; \{I_e, e\in\vec
E(\bs)\})$ gives a natural decomposition of the quadrangulation
$(\bq,v_*)=\CVS(\bm,\lambda;f_*)$ into submaps, whose study, starting
in the next section, are the key to this work; see
Figure~\ref{fig:decompq}. These submaps are called the \emph{elementary
 pieces} of $(\bq,v_*)$ and are of two types: 
the ones corresponding to half-edges of~$\vec B(\bs)$ are called
\emph{(composite) slices} and the ones corresponding to half-edges
of~$\vec I(\bs)$ are called \emph{quadrilaterals (with geodesic sides)}. 
They are not rooted and come with distinguished
vertices on their boundaries that will be discussed later on. 

\begin{figure}[htb!]
	\centering\includegraphics[width=.95\linewidth]{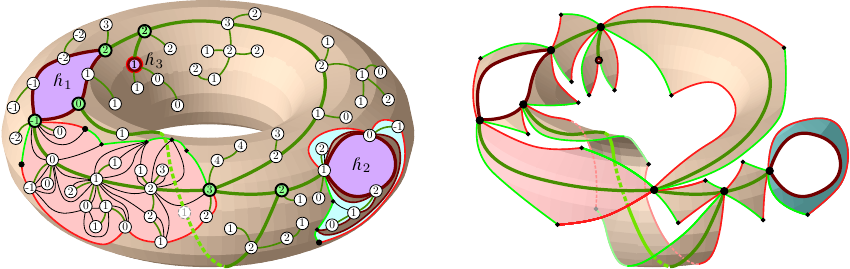}
	\caption{Performing the interval bijection on the labeled map
          from Figure~\ref{scheme}. \textup{\textgras{Left.}} Two
          elementary pieces are represented: one quadrilateral with
          geodesic sides in red\MEMS{ (light color, bottom left)}, and one composite slice
          in blue\MEMS{ (even lighter color, around $\ch_2$)}. \textup{\textgras{Right.}} The interval bijection yields a decomposition into~4 composite slices and 7 quadrilateral with geodesic sides. Here, only the maximal geodesics and shuttles are depicted. We let the edges of the original scheme figure on this output map, but these are neither edges nor chains of edges of this output map (remember that the edges of the original map are never edges of the output map).}
	\label{fig:decompq}
\end{figure}

The elementary piece corresponding to the half-edge $e\in\vec E(\bs)$ is encoded by the part of the labeled map $(\bm,\lambda)$ corresponding to
\begin{itemize}
	\item either the interval~$I_e$ if $e\in\vec B(\bs)$,
	\item or the union $I_e\cup I_{\bar e}$ if $e\in\vec I(\bs)$, where~$\bar e$ denotes the reverse of~$e$.
\end{itemize}
These encoding parts are depicted on Figure~\ref{fig:decompm}. Note that, when $e\in\vec I(\bs)$, the elementary pieces corresponding to~$e$ and to its reverse~$\bar e$ are the same map; only the distinguished vertices on the boundary will differ (more precisely be given in a different order). We refer the reader
to \cite[Section~3.4.1]{Bet16geo} for more on this decomposition, keeping in mind that, in the latter reference, the maps are rooted and the root is encoded in the scheme, which essentially amounts in seeing the root of the map as an extra external vertex. 

\begin{figure}[htb!]
	\centering\includegraphics[width=.95\linewidth]{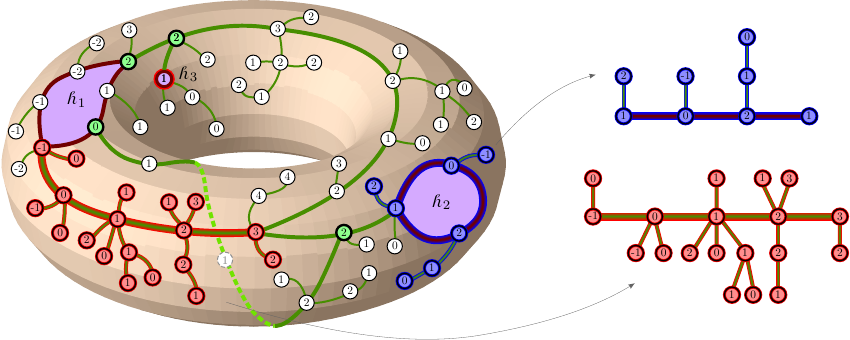}
	\caption{The parts of the labeled map from Figure~\ref{scheme}
          encoding the elementary pieces. The two parts corresponding
          to the quadrilateral with geodesic sides and the composite
          slice from Figure~\ref{fig:decompq} are extracted. The red
          and blue colors match.}
	\label{fig:decompm}
\end{figure}

\paragraph{Finiteness of the number of schemes.}
We will elaborate more on elementary pieces in the next two sections and end this one with a simple combinatorial lemma. We say that a map with holes is a \emph{scheme} if it has one internal face, all its external faces have a simple boundary and do not share a common incident edge, and all its internal vertices have degree~$3$ or more.

\begin{lmm}\label{lemnbschemes}
For fixed values of $(g,\kk)\notin\{(0,0),(0,1)\}$, there are finitely many genus~$g$ schemes with~$\kk$ holes 
and these have at most $3(2g+\kk-1)$ edges and $2(2g+\kk-1)$ vertices.
\end{lmm}

\begin{proof}
As there is a finite number of maps with a given number of edges, the bound on the number of edges yields the finiteness of the considered set.

Let $v$, $e$, $f$ be the number of vertices, edges and faces of a given scheme as in the statement, and let~$\bb$ be its number of external faces (so that $\pp=\kk-\bb$ are external vertices). By construction, we have $f=\bb+1$ and the vertices are all of degree at least~$3$, except possibly up to~$\pp$ of them, which have degree at least~$1$. The sum of the degrees of the vertices being twice the number of edges, we obtain $2e\geq 3(v-\pp)+\pp$, and we see 
by the Euler characteristic formula $v-e+f=2-2g$ that the considered scheme has at most $6g+3\kk-\pp-3\leq 3(2g+\kk-1)$ edges, and at most $2(2g+\kk-1)$ vertices,
using that $\kk=\pp+\bb$.
\end{proof}

\subsection{Composite slices}\label{sec:comp-sli}

We call \emph{plane forest} a collection
$\bff=(\bt^0, \dots, \bt^{l-1}, \rho^{l})$, for some $l\ge 1$, of rooted plane trees (the
last one being reduced to the vertex-tree), 
which we view systematically as a map by taking an embedding of every~$\bt^i$ in the upper half-plane $\R\times \Rp$, with root~$\rho^i$ at the point $(i,0)$, and in which~$\rho^i$ is linked to~$\rho^{i-1}$ by the line segment between $(i,0)$ and $(i-1,0)$, for
$1\leq i\leq l$. The union of these line segments is called the \emph{floor}
of the forest. 
The resulting embedded graph, which we still denote by~$\bff$ (see e.g.\ the left of Figure~\ref{fig:bdry_tri}) is a nonrooted plane map coming with the two distinguished vertices $\rho=\rho^0$ and $\br\rho=\rho^l$; it is in fact a plane tree, but we insist on calling
it a forest. 

We let~$a$ be the total number of edges of~$\bt^0$, \ldots, $\bt^{l-1}$, and $I=\{c_0,c_1,\ldots,c_{2a+l}\}$ be the interval of corners of~$\bff$
that are incident to the upper half-plane (hence excluding the corners
that are ``below'' the floor), starting from the root corner of~$\bt^0$ and ending with the only corner incident to~$\rho^l$, arranged in the usual contour order. 
We now equip~$\bff$ with an integer-valued labeling
function~$\lambda:V(\bff)\to \Z$, again defined up to addition of a
constant, that we require to satisfy the well labeling condition in the
interval~$I$. In this case, it means that
\begin{itemize}
\item $\lambda(u)-\lambda(v)\in \{-1,0,1\}$ whenever $u$ and $v$ are
  neighboring vertices of the same tree;
\item for $1\le i \le l$, we have
  $\lambda(\rho^{i})\ge \lambda(\rho^{i-1})-1$. 
\end{itemize}

\begin{figure}[htb!]
 \centering
\includegraphics[width=.95\linewidth]{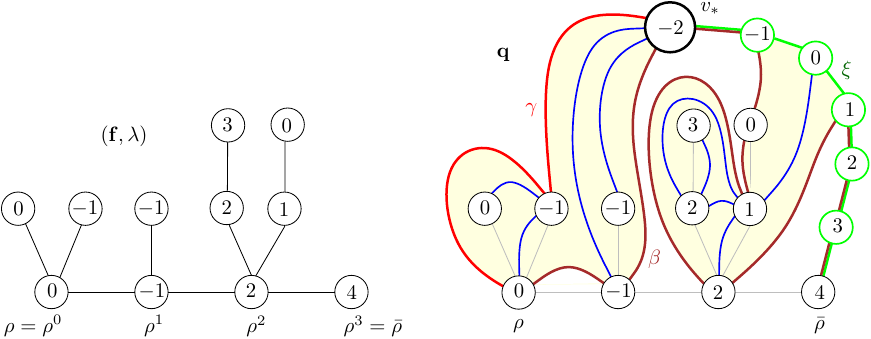} 
\caption{The interval Cori--Vauquelin--Schaeffer bijection giving
  composite slices. On this example, the forest has $a=7$ edges and
  $l=3$ trees (remember that the last vertex-tree does not count as a
  ``real'' tree). The boundary of~$\bq$ has three parts: the maximal
  geodesic~$\gamma$ (in red), the shuttle~$\xi$ (in green) and the
  base~$\beta$ (in burgundy). Its tilt is~$4$.}
 \label{fig:bdry_tri}
\end{figure}

The map $\bsl=\CVS(\bff,\lambda;I)$ is then a nonrooted plane quadrangulation with
one hole having~$a$ internal faces. Setting $\lambda_*=\min\{\lambda(v):v\in V(\bff)\}-1$, we see that the boundary of~$\bsl$ has length
$2(\lambda(\rho^{ l})-\lambda_*+ l)$. It contains three distinguished
vertices: $\rho$, the apex~$v_*$ (the extra vertex with
label~$\lambda_*$), and~$\br\rho$, as well as three distinguished paths:
\begin{enumerate}[label=(\textit{\arabic*})]
\item 
the maximal geodesic~$\gamma$, which has length
$\lambda(\rho^0)-\lambda_*$;
\item the
shuttle~$\xi$, which has length $\lambda(\rho^{
  l})-\lambda_*$;
\item the remaining boundary segment, called the \emph{base} and
  denoted by~$\beta$, consisting in the arcs connecting the root
  vertices of the 
trees. More precisely, if~$c_j$ denotes the last corner of the
tree~$\bt^i$, then the part of the boundary of~$\bsl$ between~$\rho^i$
and~$\rho^{i+1}$ consists in the arc linking~$c_j$ to~$s(c_j)$ and the
successive arcs linking~$c_{j+1}$, $s(c_{j+1})$, $s(s(c_{j+1}))$,
\ldots, $s(c_j)$. As a result, this base has
length $\lambda(\rho^{l})-\lambda(\rho^0)+2 l$. Moreover, any vertex
of the base is at distance at most $\max_{1\leq i\leq
  l}|\lambda(\rho^i)-\lambda(\rho^{i-1})|+1$ from some element of
the set $\{\rho^0,\ldots,\rho^l\}$. 
\end{enumerate}

Note that, as is the case in Figure~\ref{fig:bdry_tri}, the base may
overlap with the other distinguished paths. Furthermore, as noted at
the end of Section~\ref{sec:basic-construction}, the maximal geodesic
and the shuttle are geodesic chains. On the contrary, the base is not
a geodesic in general.

\begin{defn}
A map obtained by this construction will be called a \emph{discrete composite slice}, or simply \emph{slice} for short: its \emph{area} is the integer~$a$,
its \emph{width} is the integer~$l$ and its \emph{tilt} is
defined as the integer
\[\delta=\lambda(\br\rho)-\lambda(\rho).\]
\end{defn}

The terminology of composite slices, width and tilt are borrowed
from~\cite{BouttierHDR}; however, the reader should mind that our 
exact definitions
differ slightly from those in that reference\footnote{In particular, in \cite{BouttierHDR}, the width is the length
of the base, equal to $2l+\delta$ in our notation,  
and the tilt is the opposite~$-\delta$ of what we call the tilt in
this paper.}. Note also that, in the present work, we use the simplified terminology of \emph{slice} in order to designate a composite slice. Beware that these have not to be confused with similar objects existing in the literature, in particular in our previous work \cite{BeMi17}, called \emph{elementary slices} or also slices for short; they actually correspond to composite slices of width~$0$, objects that we do not consider here.

We record the following useful counting result. 
\begin{prp}\label{countslice}
The number of slices with area~$a$, width~$l$ and tilt~$\delta$ is equal to
\[3^a\, \frac{l}{2a+l}\, \binom{2a+l}{a}\, \binom{2l+\delta-1}{l-1},\]
which can also be recast as
\[12^a \, 8^l\,  2^\delta \, Q_l(2a+l)\, P_l(\delta),\]
where $Q_\ell(u)$ is the probability that a simple random walk hits~$-\ell$ for the first time at time~$u$, and $P_\ell(j)=\P(G_1+\dots+G_\ell=j)$, where $G_1$, $G_2$, \dots\ are
independent random variables with shifted Geometric(1/2) law, i.e.,
such that $\P(G_1=j)=2^{-j-2}$ for $j\geq -1$.
\end{prp}

\begin{proof}
The term $\frac{l}{2a+l}\binom{2a+l}{a}$ is the number of
forests with~$l$ trees and~$a$ non-floor edges, the term
$\binom{2l+\delta-1}{l-1}$ counts the number of possible
ways to well label the roots, and the term~$3^a$ counts the number of
ways to well label the other vertices, since it amounts to choosing a
label difference in $\{-1,0,1\}$ along each edge.

The probabilistic form is a simple exercise using the encoding of forests and geometric walks
by simple walks, yielding $\frac{l}{2a+l}\binom{2a+l}{a}=2^{2a + l}\, Q_l(2a+l)$ and $\binom{2l+\delta-1}{l-1}=2^{2l+\delta} \, P_l(\delta)$. See Section~\ref{secUIHPQ} and \cite[Lemma~6]{bettinelli11b}. 
\end{proof}

An important feature of the construction is that the labels on $V(\bsl)$ inherited from those on~$V(\bff)$ are exactly the relative distances to~$v_*$ in~$\bsl$:
\[d_\bsl(v,v_*)=\lambda(v)-\lambda_*,\qquad v\in V(\bsl),\]
and that the following bound holds:
\begin{equation}\label{dd0bound}
d_\bsl(c_i,c_j)\le \lambda(c_i)+\lambda(c_j)-2\min_{i\le r \le j}\lambda(c_r)+2,\qquad i\le j .
\end{equation}

If $\bsl$ is a slice, using a slightly different convention from that
of Section~\ref{sec:main-conv-result}, we view it as the marked
measured metric space in $\M^{(5,2)}$ given by 
\begin{equation}\label{mmsslice}
\big(V(\bsl),d_\bsl,\partial \bsl,\mu_\bsl,\nu_{\beta}\big)\qquad\text{ with }\qquad \partial\bsl=\big(\beta,\rho,\gamma,\br{\rho},\xi\big),
\end{equation}
where each boundary part is identified with the vertices it contains,
where~$\mu_\bsl$ is the counting measure on the vertices of~$\bsl$
that \textgras{do not belong to the shuttle}, and where~$\nu_\beta$ is
the counting measure (with multiplicities) on $\beta\setminus\{\br\rho\}$. The measures~$\mu_\bsl$ and~$\nu_{\beta}$ are respectively called the \emph{area measure} and the \emph{base measure} of the slice. It might be surprising at this point to include~$\rho$ and $\br{\rho}$ in the marking as these can be found from the other three marks; they are here to enter the framework of geodesic marks introduced in Section~\ref{sec:geodesicmarks}. The idea is that the data of $(\rho,\gamma)$ suffice to recover the maximal geodesic as an \emph{oriented} path, whereas the data of~$\gamma$ (as a set of vertices) do not give the orientation of the path.

\subsection{Quadrilaterals with geodesic sides}\label{sec:quadr-with-geod}

Consider a \emph{double forest}, that is, a pair $(\bff,\br\bff)$ of plane forests with the same number of trees. Let $h\ge 1$ denote this common number of trees and recall that this means that~$\bff$ and~$\br\bff$ have~$h$ trees plus an additional vertex-tree. Similarly to the previous section, we represent it by letting
\begin{itemize}
	\item the floors be both sent to the chain linking the points $(i,0)\in \R^2$, where $0\leq i\leq h$,
	\item the trees of~$\bff$ be contained in the upper half-plane $\R\times \Rp$, the $i$-th tree attached to $(i-1,0)$, for $1\leq i\leq h$,
	\item and the trees of~$\br\bff$ be contained in the lower half-plane, the $i$-th tree attached to $(h-i+1,0)$, for $1\leq i\leq h$.
\end{itemize}
We obtain a nonrooted plane map, which we denote by~${\bff}\cup{\br\bff}$, coming with the two
distinguished vertices $\rho=(0,0)$ and $\br\rho=(h,0)$. Here also, it
is in fact a plane tree having two distinguished vertices. 

We let $I=\big\{c_0, c_1, \ldots, c_{2a+h}\big\}$ be the interval of corners of~${\bff}\cup{\br\bff}$, in
facial order, that are incident to the upper half-plane, and $\br
I=\big\{\br c_0,\br c_1,\ldots,\br c_{2\br a+h}\big\}$ the interval of
those incident to the lower half-plane, where~$a$ (resp.~$\br a$) is
the number of edges in the trees of the upper (resp.\ lower) half-plane. As mentioned during Section~\ref{sec:basic-construction}, we use the slightly unusual convention
that $c_{2a+h}\neq \br c_0$ (and similarly $\br c_{2\br a+h}\neq c_0$): this means that the first corner incident to~$\rho$ is ``split'' in two corners, one in the upper half-plane and
one in the lower half-plane.

Finally, assume that, in its unique face, the map~${\bff}\cup{\br\bff}$ is well labeled by an integer function~$\lambda:V({\bff}\cup{\br\bff})\to\Z$ defined up to addition of a constant: this simply means that $\lambda(u)-\lambda(v)\in \{-1,0,1\}$ whenever~$u$ and~$v$
are neighboring vertices. See Figure~\ref{fig:quadrilateral} for an
example. Note that, equivalently, a well-labeled double forest
$((\bff,\bar{\bff}),\lambda)$ can be seen as a well-labeled \emph{vertebrate}, that is a well-labeled tree with two distinct
distinguished vertices~$\rho$, $\bar{\rho}$, where the interval~$I$
corresponds to the consecutive corners in the contour order from~$\rho$
to~$\bar{\rho}$, which contains all the corners incident to~$\rho$ and stops
at the first corner incident to~$\bar{\rho}$, and~$\bar{I}$ is defined
similarly with the roles of~$\rho$, $\bar{\rho}$ exchanged. 

\begin{figure}[htb!]
 \centering\includegraphics[width=.95\linewidth]{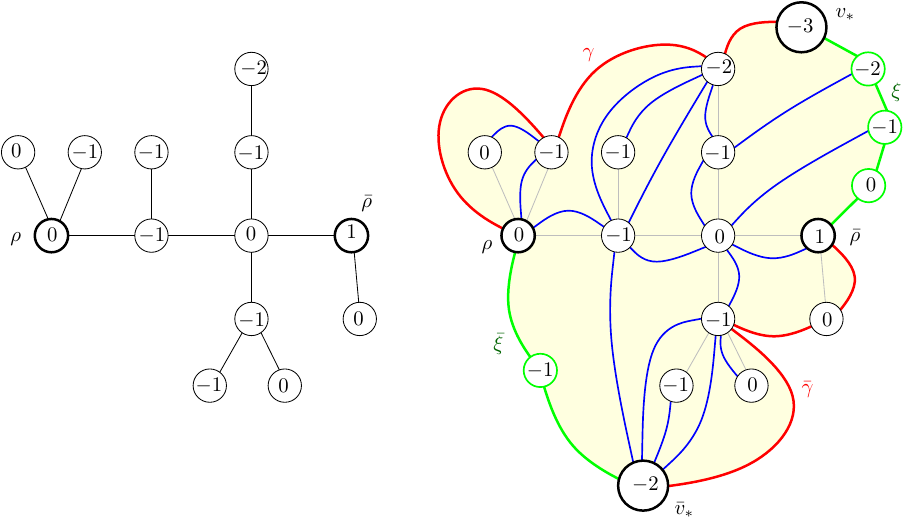} 
\caption{The interval Cori--Vauquelin--Schaeffer bijection giving quadrilaterals with geodesic sides. The quadrilateral with geodesic sides has half-areas~$5$ and~$4$, width~$3$, and tilt~$1$.}
 \label{fig:quadrilateral}
\end{figure}

The map $\bqd=\CVS({\bff}\cup{\br\bff},\lambda;I,\br I)$ is then a nonrooted plane quadrangulation with
one hole having~$a+\br a+h$ internal faces. Its boundary contains the four distinguished vertices~$\rho$, the apex~$v_*$ associated with~$I$, $\br\rho$, and the apex~$\br v_*$ associated with~$\br I$, as well as the maximal geodesics~$\gamma$, $\br\gamma$ and shuttles~$\xi$, $\br\xi$, with obvious notation.

\begin{defn}
The \emph{quadrilateral with geodesic sides}, or simply \emph{quadrilateral} for short, associated with $((\bff,\br\bff),\lambda)$ is by definition $\bqd=\CVS({\bff}\cup{\br\bff},\lambda;I,\br I)$. Its \emph{width}, \emph{half-areas}, and \emph{tilt} are respectively the numbers
\[h,\qquad
a\ \text{ and }\ \br a, \qquad
\lambda(\br\rho)-\lambda(\rho).\] 
\end{defn}

Observe that the parameters of a quadrilateral can be recovered from the map~$\bqd$ and the distinguished vertices~$\rho$, $\br\rho$. Furthermore, the quadrilateral associated with $((\bar{\bff},\bff),\lambda)$ is obtained from the one associated with $((\bff,\bar{\bff}),\lambda)$ simply by switching the distinguished elements $\rho$ with~$\br{\rho}$, $\gamma$ with~$\br\gamma$, and~$\xi$ with~$\br\xi$. It has the same width, its half-areas are switched and its tilt is reversed.

\begin{prp}\label{countquad}
The number of quadrilaterals with half-areas~$a$, $\br a$, width~$h$ and tilt~$\delta$ is equal to
\[12^{a+\br a+h}\, Q_h(2a+h)\, Q_h(2\br a+h)\, M_h(\delta),\]
where $Q_\ell$ has been defined in Proposition~\ref{countslice}, and $M_\ell(j)=\P(U_1+\ldots+U_\ell=j)$, where $U_1$, $U_2$, \dots\ are
independent uniform random variables in $\{-1,0,1\}$. 
\end{prp}

\begin{proof}
As in the proof of Proposition~\ref{countslice}, the number of forests with~$h$ floor edges and~$\alpha$ non-floor edges is $2^{2\alpha+h}Q_h(2\alpha+h)$, so the number of double forests with proper parameters is $4^{a+\br a +h} Q_h(2a+h)Q_h(2\br a+h)$ . Then, the number of possible labelings of the floor vertices is the number of walks with~$h$ steps in $\{-1,0,1\}$ going form~$0$ to~$\delta$,
which equals $3^h M_h(\delta)$. The final term
$3^{a+\br a}$ counts the possible labelings of the
non-root vertices in the double forest.
\end{proof}

If $\bqd$ is a quadrilateral, we will view it as a marked measured metric space in $\M^{(6,1)}$ given by
\begin{equation}\label{mmsquad}
\big(V(\bqd),d_\bqd,\partial \bqd,\mu_\bqd\big)\qquad\text{ with }\qquad \partial\bqd=\big(\rho,\gamma, \xi,\br{\rho},\br\gamma,\br\xi\big),
\end{equation}
where each
boundary part is identified with the vertices it contains, and where~$\mu_\bqd$ is the counting measure on the vertices of $\bqd$ that \textgras{do not belong to the shuttles}. We call this measure~$\mu_\bqd$ the \emph{area measure} of the quadrilateral.

\subsection{Scaling limits of elementary pieces}\label{sec:scaling-limits}

In this section, we state two important results that will be crucial
in 
the proof of Theorem~\ref{mainthm}. These show that, under appropriate
hypotheses, random discrete slices and quadrilaterals converge in
distribution in the GHP topology toward ``continuum analogs'' of
these objects. 

We first fix three sequences $(a_n)\in(\Zp)^\N$, $(l_n)\in\N^\N$ and~$(\delta_n)\in\Z^\N$ such that
\begin{equation}\label{anlndn}
\frac{a_n}{n}\ton A>0\ , \qquad
\frac{l_n}{\sqrt{2n}}\ton L>0\qquad\text{ and }\qquad \left(\frac{9}{8n}\right)^{1/4}\delta_n\ton \Delta\in\R.
\end{equation}
Recall that a slice is seen as an element of~$\M^{(5,2)}$ given by~\eqref{mmsslice} and that~$\Omega_n$ is the scaling operator defined in~\eqref{defomega}.

\begin{thm}\label{thmslslice}
Let~$\Sln$ be uniformly distributed among composite slices with
area~$a_n$, width~$l_n$ and tilt~$\delta_n$. Then we have the
convergence
\[\Omega_n(\Sln)\tod\Sl_{A,L,\Delta},\]
in distribution in the space $\big(\M^{(5,2)},\dGHP^{(5,2)}\big)$. 
The limit is called a
\emph{(continuum composite) slice with area~$A$, width~$L$ and
tilt~$\Delta$}.
\end{thm}

This theorem will be proved in Section~\ref{seccvcs}, where a detailed
characterization of the limiting object will be given. For the time
being, this theorem should be taken as a definition of the spaces $\Sl_{A,L,\Delta}$. 

The following statement deals with the case
of vanishing areas and widths, and will be useful in Section~\ref{sec:gluing-piec-togeth} below. 

\begin{crl}\label{slslice0}
Let the sequences $(a_n)\in(\Zp)^\N$ and $(l_n)\in(\Zp)^\N$ satisfy 
$l_n=\sO(\sqrt{n})$ and 
$a_n+l_n=\Theta((l_n)^2)$. 
Let~$\Sln$ be the vertex map whenever $l_n=0$, or be uniformly distributed among slices with 
area~$a_n$, width~$l_n$ and tilt~$0$ otherwise. Then we have the 
convergence toward the point space 
\[\Omega_n(\Sln)\tod\{\varrho\},\]
in distribution in the space $\big(\M^{(5,2)},\dGHP^{(5,2)}\big)$. 
\end{crl}

Let us sketch in a few lines why this is indeed a consequence of Theorem~\ref{thmslslice}. By the assumption that $a_n+l_n=\Theta((l_n)^2)$, 
the sequence $((a_n+l_n)/(l_n)^2)$, restricted to the values of~$n$
for which $l_n \ne 0$, is bounded away from~$0$ and~$\infty$.  This compact way of writing this property covers in
fact the two following situations. If $(l_n)$ is a bounded integer
sequence, it simply means that $(a_n)$ is a bounded integer sequence, in which case the statement becomes trivial. If $(l_n)$ is
unbounded, then it means that $(a_n/(l_n)^2)$ is bounded away from~$0$
and~$\infty$. In this case, Theorem~\ref{thmslslice} easily implies that the diameters of $\Omega_{a_n}(\Sln)$ form a tight family of random variables. Since $a_n=\sO(n)$, the conclusion follows. Note that, up to extracting subsequences, we may always assume that we are in one of the two situations discussed above. 

We will derive Theorem~\ref{thmslslice} from the known
convergence of the uniform 
infinite half-planar quadrangulation toward the Brownian
half-plane. The former naturally contains a family of slices
and the latter contains a continuous ``flow'' of continuum slices. These consist in free versions of the objects considered here
so that we will need to finish with a conditioning argument.  

\bigskip
We now turn to quadrilaterals, which are seen as elements of~$\M^{(6,1)}$ given by~\eqref{mmsquad}. We consider four sequences $(a_n)$, $(\br a_n)\in(\Zp)^\N$, $(h_n)\in\N^\N$ and~$(\delta_n)\in\Z^\N$ such that, as $n\to\infty$,
\begin{equation}\label{2anhndn}
\frac{a_n}{n}\to A>0, \quad
\frac{\br a_n}{n}\to \br A>0, \quad 
\frac{h_n}{\sqrt{2n}}\to H>0,\quad 
\left(\frac{9}{8n}\right)^{1/4}\delta_n\to \Delta\in\R.
\end{equation}

\begin{thm}\label{thmslquad}
Let~$\Qdn$ be a random variable uniformly distributed among quadrilaterals with half-areas~$a_n$ and~$\br a_n$, width~$h_n$ and
tilt~$\delta_n$. Then we have the
convergence
\[\Omega_n(\Qdn)\tod\Qd_{A,\br A,H,\Delta},\]
in distribution in the space $\big(\M^{(6,1)},\dGHP^{(6,1)}\big)$. 
The limit is 
called a \emph{continuum quadrilateral with half-areas~$A$ and~$\br A$, width~$H$ and
tilt~$\Delta$}. 
\end{thm}

As for slices, the proof of this result is postponed, to
Section~\ref{seccvquad}, where a detailed characterization of the
limiting object will be given. The idea of the proof will be similar
to that of Theorem~\ref{thmslslice}, using the uniform infinite planar
quadrangulation and Brownian plane as reference
spaces instead of the half-planar versions mentioned above.

\section{Marking and gluing along geodesics}\label{secmarkglue}

In our previous work~\cite{BeMi17}, we proved Theorem~\ref{mainthm}
in the case of disks (for the GH topology) by writing~$Q_n$ and~$\BS{0}{(L^1)}$ as gluings 
of appropriate subspaces along geodesic segments, namely so-called 
\emph{slices} in the discrete setting and their scaling limits in the 
continuum. The fact that the number of gluings needed was infinite 
caused some difficulties (which we mainly overcame by noticing that any 
geodesic between two typical points may be broken down to a finite 
number of pieces lying in different such subspaces). In contrast, in 
this work, we will only need to consider gluings of a \emph{finite 
  number} of subspaces along geodesic segments. As this operation is 
well behaved in a more general setting, we present it in this 
section. But first, we collect a number of useful lemmas on the GHP topology. 

We will use the following notation. If~$\mu_\X$ is a finite positive measure on a set~$\X$, we let 
$\br{\mu}_\X=\mu_\X/\mu_\X(\X)$ be the normalized probability measure. If $\mu_\X=0$, we use the 
convention $\br{\mu}_\X=0$. If $\bmu=(\mu^1,\ldots,\mu^m)$ is a finite 
family of nonnegative measures, we let 
$\br{\bmu}=(\br{\mu}^1,\ldots,\br{\mu}^m)$.

\subsection{Useful facts on the GHP topology and markings}\label{sec:adding-measures}

Recall the definitions of $\big(\M^{(\ell,m)},\dGHP^{(\ell,m)}\big)$ and $\big(\M^{(\ell)},\dGH^{(\ell)}\big)$ from Section~\ref{secGHP}. 
If the space $(\X,d_\X,\bA,\bmu_\X)$ is an element of $\M^{(\ell,m)}$ and $\bbr\in(\Zp)^m$ is such that $r^j=0$ whenever $\mu_\X^j=0$, 
we may consider the variable
$(\X,d_\X,\bA(x^1_1,\dots,x^1_{r^1})\dots(x ^m_1,\dots,x^m_{r^m}))$
taking values in $\M^{(\ell+\|\bbr\|)}$, where, for each $j\in\{1,\dots,m\}$, the points~$x^j_1$, \dots,
$x^j_{r_j}$ are i.i.d.\ sampled random variables with law~$\br{\mu}^j_\X$ (if the latter measure is~$0$, then this still makes sense since $r^j=0$); we denote by
$\Mark_\bbr((\X,d_\X,\bA,\bmu_\X),\cdot)$ the law of this
random marked metric space. 
Some care is actually needed
here since we are considering isometry classes of metric measure
spaces. See~\cite{miertess} for an accurate definition of this notion, which is
immediately generalized to our setting where we incorporate the extra
marks given by~$\bA$, and several measures. The following lemma states that one can formulate
the GHP convergence entirely in terms of the GH convergence of randomly marked spaces. 

\begin{lmm}\label{lemrandmark}
Let $(\X_n,d_{\X_n},\bA_n,\bmu_{\X_n})$, $n\geq 1$, and
$(\X,d_\X,\bA,\bmu_\X)$ be elements of $\M^{(\ell,m)}$. The following statements are equivalent.
\begin{enumerate}[label=(\textit{\roman*})]
\item\label{lemrandmarki} The space $(\X_n,d_{\X_n},\bA_n,\bmu_{\X_n})$ converges to\MEMS{ the space}{} $(\X,d_\X,\bA,\bmu_\X)$ in $\big(\M^{(\ell,m)},\dGHP^{(\ell,m)}\big)$. 
\item\label{lemrandmarkii} One has $\bmu_{\X_n}(\X_n)\to
\bmu_\X(\X)$ coordinatewise as $n\to\infty$ and, for every
$\bbr\in (\Zp)^m$ such that $r^j=0$ whenever $\mu^j_\X=0$, it
holds that 
\[\Mark_\bbr\big((\X_n,d_{\X_n},\bA_n,\bmu_{\X_n}),\cdot\,\big)\ton\Mark_\bbr\big((\X,d_\X,\bA,\bmu_\X),\cdot\,\big)\]
for weak convergence of probability measures on
$\big(\M^{(\ell+\|\bbr\|)},\dGH^{(\ell+\|\bbr\|)}\big)$. 
\end{enumerate}
\end{lmm}

\begin{proof}
The implication $\text{\ref{lemrandmarki}}\implies
\text{\ref{lemrandmarkii}}$ is an easy generalization of known
results. See
\cite[Proposition~10]{miertess} for the case where the measures are
probability measures, and \cite[Section~2.2]{LGa19dis} for a generalized
context with finite measures; our extended context of marked measured metric
spaces adds no difficulty. To show the converse implication, we
argue as follows. By taking the trivial case $\bbr=\bzero^m$ of~\ref{lemrandmarkii}, we obtain that $\{(\X_n,d_{\X_n},\bA_n),n\geq
1\}$ is  relatively compact in $\big(\M^{(\ell)},\dGH^{(\ell)}\big)$. Together with the fact that the sequences
$(\mu^j_{\X_n}(\X_n),n\geq 1)$ are bounded, this implies that $\{(\X_n,d_{\X_n},\bA_n,\bmu_{\X_n}),n\geq
1\}$ is relatively compact in
$\big(\M^{(\ell,m)},\dGHP^{(\ell,m)}\big)$. So let
$(\X',d_{\X'},\bA',\bmu_{\X'})$ be a limit in $\M^{(\ell,m)}$
 along some subsequence of $(\X_n,d_{\X_n},\bA_n,\bmu_{\X_n})$. By using the
implication
$\text{\ref{lemrandmarki}}\implies\text{\ref{lemrandmarkii}}$, we
obtain that, for every $\bbr$ such that $r^j=0$ whenever $\mu_\X^j=0$, 
\[\Mark_\bbr\big((\X,d_\X,\bA,\bmu_\X),\cdot\,\big)=\Mark_\bbr\big((\X',d_{\X'},\bA',\bmu_{\X'}),\cdot\,\big)\,
.\]
Now, let~$m'$ be the number of nonzero elements of $\bmu_\X$, fix
$r>0$ and set $r^j=r\ind_{\{\mu_\X^j\neq 0\}}$. We let 
$(\X,d_\X,(A^1,\ldots,A^\ell,x^1_1,\dots,x^1_{r^1},\ldots,x^m_1,\ldots,x^m_{r^m}))$ be the $(\ell+rm')$-marked metric space with law $\Mark_\bbr((\X,d_\X,\bA,\bmu_\X),\cdot)$,
and  set $\boldsymbol{\theta}_r=(\theta^j_r,1\leq j\leq m)$,
where $\theta^j_r=r^{-1}\sum_{i=1}^{r}\delta_{x^j_i}$ if
$\mu_\X^j\neq 0$ and $\theta^j_r=0$ if $\mu_\X^j=0$. 
It is a consequence of the law of large numbers that
$(\X,d_\X,\bA,\boldsymbol{\theta}_r)$ converges almost surely in
$\M^{(\ell,m)}$, as $r\to\infty$, 
to $(\X,d_\X,\bA,\br{\bmu}_\X)$; see
for instance \cite[Lemma~5]{LGa19dis}. 
Applying this same result to
$(\X',d_{\X'},\bA',\bmu_{\X'})$ âllows to show that $(\X',d_{\X'},\bA',\br{\bmu}'_{\X'})$ is
isometry-equivalent to $(\X,d_\X,\bA,\br{\bmu}_\X)$.  Since
$\bmu_\X(\X)=\bmu_{\X'}(\X')$ is the limit of $\bmu_{\X_n}(\X_n)$, we conclude. 
\end{proof}

We also recall that, often, the most useful way to estimate
GH distances is via the notion of \emph{distortion of a correspondence}.
A \emph{correspondence} between two sets~$\X$ and~$\Y$
is a subset $\RR \subseteq \X\times \Y$ whose
coordinate projections are~$\X$ and~$\Y$. We will often write $x\binRR y$
instead of $(x,y)\in \RR$. 
If~$\X$ and~$\Y$ are endowed
with the metrics~$d_\X$ and~$d_\Y$, the \emph{distortion} of
the correspondence~$\RR$ is the
number
\[\dis(\RR)=\sup\big\{|d_\X(x,x')-d_\Y(y,y')|\,:\,x\binRR y,\ x'\binRR y'\big\}.\]
If $\bA=(A^1,\ldots,A^\ell)$ and $\bB=(B^1,\ldots,B^\ell)$ are
markings of~$\X$ and of~$\Y$, we say that the correspondence $\RR$
between $\X$ and $\Y$ is \emph{compatible} with the markings if for every $1\le i\le \ell$, $\RR \cap (A^i\times B^i)$ is a correspondence between~$A^i$ and $B^i$. 

\begin{lmm}[{\cite[Section~6.4]{miertess}}]\label{lemcorresGH} 
It holds that
\[\dGH^{(\ell)}\big((\X,\bA,d_\X),(\Y,\bB,d_\Y)\big)
=\frac 12 \inf_\RR \dis(\RR),\]
where the infimum is taken over correspondences compatible with the markings. 
\end{lmm}

Correspondences are also useful for estimating GHP distances when used
together with the notion of couplings, which are measures on the
product of the two spaces to be compared. The following is a
direct adaptation of \cite[Lemma~4]{LGa19dis}, which treats the
case of~$\M^{(0,1)}$.  

\begin{lmm}\label{lemcorresP}
Let $(\X,d_\X,\bA,\bmu_\X)$ and $(\Y,d_\Y,\bB,\bmu_\Y)$ be elements
of~$\M^{(\ell,m)}$ for some $\ell$, $m\geq 0$. Let $\eps>0$, and~$\RR$ be a correspondence between~$\X$ and~$\Y$ compatible with the markings and of distortion bounded
above by~$\eps$. For $1\leq j\leq m$, let~$\nu^j$ be a finite measure on the product
$\X\times \Y$ such that $\nu^j(\RR^c)< \eps$ and, letting~$p_\X$, $p_\Y$
be the coordinate projections onto~$\X$ and~$\Y$, 
\[d_\X^{\mathrm{P}}(\mu^j_\X,(p_\X)_*\nu^j)\vee
d_\Y^{\mathrm{P}}(\mu^j_\Y,(p_\Y)_*\nu^j)< \eps.\]
Then
$\dGHP^{(\ell,m)}((\X,d_\X,\bA,\bmu_\X),(\Y,d_\Y,\bB,\bmu_\Y))\leq 3\eps$.
\end{lmm}

Finally, we state an elementary lemma whose proof is straightforward
and omitted.

\begin{lmm}\label{lem1lip}
The mappings
\begin{align*}
(\X,d_\X,(A^1,\ldots,A^\ell),\bmu_{\X})&\longmapsto (\X,d_\X,(A^1\cup
                                         A^2,A^3,\ldots,A^{\ell}),\bmu_\X),\\
  (\X,d_\X,(A^1,\ldots,A^\ell),\bmu_{\X})&\longmapsto (\X,d_\X,(A^1,
                                          \ldots,A^{\ell-1}),\bmu_\X)
  \intertext{are $1$-Lipschitz from $\big(\M^{(\ell,m)},\dGHP^{(\ell,m)}\big)$ to
$\big(\M^{(\ell-1,m)},\dGHP^{(\ell-1,m)}\big)$; the mappings}
  (\X,d_\X,\bA,(\mu^1_{\X},\ldots,\mu^m_\X)) &\longmapsto  (\X,d_\X,\bA,(\mu^1_{\X}+\mu^2_\X,\mu^3_\X,\ldots,\mu^m_\X)) \\
   (\X,d_\X,\bA,(\mu^1_{\X},\ldots,\mu^m_\X)) &\longmapsto  (\X,d_\X,\bA,(\mu^1_{\X},\ldots,\mu^{m-1}_\X) )
\intertext{are\MEMS{}{ respectively} $2$-Lipschitz and $1$-Lipschitz from $\big(\M^{(\ell,m)},\dGHP^{(\ell,m)}\big)$ to
$\big(\M^{(\ell,m-1)},\dGHP^{(\ell,m-1)}\big)$; and, for every
                                               permutation~$\sigma$ of
                                               $\{1,2,\ldots,\ell\}$
                                               and~$\tau$ of $\{1,2,\ldots,m\}$,}
\MEMS{%
\big(\X,d_\X,\bA,\bmu_\X\big)
	&\longmapsto \big(\X,d_\X,\big(A^{\sigma(1)},\ldots,A^{\sigma(\ell)}\big), \big(\mu^{\tau(1)}_\X,\ldots,\mu^{\tau(m)}_\X\big)\big),
}{%
\big(\X,d_\X,(A^1,\ldots,A^\ell),(\mu^1_\X,\ldots,\mu^m_\X)\big)
	&\longmapsto \big(\X,d_\X,\big(A^{\sigma(1)},\ldots,A^{\sigma(\ell)}\big), \big(\mu^{\tau(1)}_\X,\ldots,\mu^{\tau(m)}_\X\big)\big),
}
\end{align*}
is an isometry from $\big(\M^{(\ell,m)},\dGHP^{(\ell,m)}\big)$ onto itself.
\end{lmm}

\subsection{Geodesics in metric spaces}\label{sec:geod-metr-spac}

We now discuss the important notion of geodesics in metric spaces, as
well as its relations with GHP limits. 

In a metric space $(\X,d_\X)$, compact or not, a \emph{geodesic} is a mapping
$\chi:[0,\pzl]\to \X$ defined on some compact interval\footnote{We allow $\pzl=0$ in this definition.} $[0,\pzl]$ and that is isometric, i.e.,
satisfies
\begin{equation}
  \label{eqgeod}
  d_{\X}(\chi(s),\chi(t))=|t-s|,\qquad 0\leq s,t\leq \pzl.
\end{equation}
The points~$\chi(0)$, $\chi(\pzl)$ are called the \emph{extremities} of~$\chi$, and\MEMS{}{ the quantity} $\pzl=d_\X(\chi(0),\chi(\pzl))$ is the \emph{length} of the geodesic, denoted by $\lgth_{d_\X}(\chi)$, or
simply $\lgth(\chi)$ when there is little risk of ambiguity.
The space $(\X,d_\X)$ is called a \emph{geodesic space} if,
for every pair of points $x$, $y\in \X$, there exists a
geodesic with extremities~$x$, $y$. 

The range $\chi([0,\lgth(\chi)])$ of a geodesic path is called a \emph{geodesic segment}. An
\emph{oriented geodesic segment} is a pair
$(\chi(0),\chi([0,\lgth(\chi)]))$ made 
of a geodesic segment and a 
distinguished extremity, called its \emph{origin}. Note that an
oriented geodesic segment uniquely determines the geodesic~$\chi$,
since~$\chi(t)$ is the unique point at distance~$t$ away from the
origin. For this reason, we will systematically
identify geodesics with oriented geodesic segments and use the same
piece of notation for both of them. 

\medskip
In a marked measured metric space
$(\X,d_\X,\bA,\bmu_\X)$, some pairs $(A^i,A^j)$ of
marks might be oriented
geodesic segments; such pairs are called \emph{geodesic marks}.

\paragraph{Geodesic marks and GHP limits.}\label{sec:geodesicmarks}
The
following proposition states that geodesic marks nicely pass to the
limit in the GHP topology.

\begin{prp}\label{limgeod}
Let $(\X_n,d_{\X_n},\bA_n,\bmu_{\X_n})$, $n\ge 1$, be a sequence of
marked measured compact metric spaces that converges to some 
limit $(\X,d_\X,\bA,\bmu_\X)$ in the GHP topology.
Suppose that~$i$, $j$ are fixed and that, for every~$n$, the pair of marks 
$(A_{n}^i,A_{n}^j)=\gamma_n$ is a geodesic mark. Then the 
pair of marks $(A^i,A^j)=\gamma$ of~$\bA$ is also a geodesic mark. Moreover, it
holds that
\[\lgth(\gamma)=\lim_{n\to\infty}\lgth(\gamma_n).\]
\end{prp}

\begin{proof}
The wanted property deals only with the marks and not with the
measures, so it suffices to establish the proposition in the space
$\M^{(\ell)}$ of marked, nonmeasured spaces. Without loss of generality
(by Lemma~\ref{lem1lip}), we may and will assume that $i=1$ and $j=2$. 

By Lemma~\ref{lemcorresGH}, we may find a sequence of
correspondences $\RR_n$ between $\X_n$ and $\X$ that is compatible
with the markings $\bA_n$ and $\bA$, and whose
distortion $\eps_n:=\mathrm{dis}(\RR_n)$ goes to zero. From now on,
we will never need to refer to marks other than the first two. 

Let $y$, $z\in A^1$. Since $A_{n}^1$ contains a single point, which we denote by $x_n=\gamma_n(0)$, we have $x_n\binRRn y$
and $x_n\binRRn z$, so that $d_\X(y,z)\leq \eps_n$ for every $n\geq
1$, entailing $y=z$. So $A^1$ is a singleton, which we denote by $A^1=\{x\}$.

Next, let $a\in A^2$ and $a_n\in A_{n}^2$ be such that $a_n\binRRn a$. Then $|d_\X(x,a)-d_{\X_n}(x_n,a_n)|\leq \eps_n$, which implies
that $d_{\X_n}(x_n,a_n)\to d_\X(x,a)$ as $n\to\infty$, and in
particular, $d_{\X}(x,a)\leq \liminf_{n\to\infty}
\lgth(\gamma_n)$, and therefore 
\[\max_{a\in A^2}d_\X(x,a)\leq \liminf_{n\to\infty}
\lgth(\gamma_n).\]
 
In the other direction, let $t\le\limsup _{n\to\infty}\lgth(\gamma_n)$. We claim that there exists at least a point $c_t\in A^2$ such that $d_\X(x,c_t)=t$. This will entail that $\lgth(\gamma_n)$ converges to $\pzl=\max_{a\in A^2}d_\X(x,a)$. To see the claim, observe that, at least along a suitable extraction, there exists a sequence
$t_n\to t$ such that 
$\lgth(\gamma_n)\geq t_n$. Along this
extraction, let $\gamma_n(t_n)$ be the
unique point of $\gamma_n$ such that $t_n=d_{\X_n}(x_n,\gamma_n(t_n))$,
and let $g_n$ be an element of~$A^2$ such that $\gamma_n(t_n)\binRRn
g_n$. Then $|d_\X(x,g_n)-t_n|\leq \eps_n$, so that, possibly by further
extracting, $(g_n)$ converges to a limit
$c_t\in A^2$. It then holds that $d_\X(x,c_t)=t$, as claimed.

Now fix $s$, $t\in [0,\pzl]$ with $s\leq t$, and let $a$, $b\in A^2$ be such that
$d_\X(x,a)=s$ and $d_\X(x,b)=t$. By the triangle inequality, we have
$d_\X(a,b)\geq t-s$, and, on the other hand, if $a_n\binRRn a$ and
$b_n\binRRn b$ with $a_n$, $b_n\in A_{n}^2$, then
\begin{align*}
 d_\X(a,b)\leq d_{\X_n}(a_n,b_n)+\eps_n
&=|d_{\X_n}(x_n,b_n)-d_{\X_n}(x_n,a_n)|+\eps_n\\
&\leq |d_\X(x,b)-d_\X(x,a)|+3\eps_n\\
&=t-s+3\eps_n
\end{align*}
where in the second line, we have used the fact that $a_n$, $b_n$ lie
on a geodesic having $x_n$ as one of its extremities. Letting $n\to\infty$, this shows that
$d_\X(a,b)=t-s$, and in particular, taking $t=s$ shows that the
point $c_t$ of the preceding paragraph is the unique point of~$A^2$ at distance~$t$ from~$x$. We
conclude that $\gamma$ is an oriented geodesic segment with length~$\pzl$ and origin~$x$.  
\end{proof}

\paragraph{Maps as compact geodesic metric spaces. }\label{sec:comp-geod}
So far, we have been seeing maps as finite metric spaces. We may also interpret a 
map~$\bm$ as a compact geodesic metric space, by viewing each edge as 
isometric to a real segment of length~$1$ (this is called the metric 
graph~\cite{burago01} associated with~$\bm$). Note that the 
restriction of the metric to the subset corresponding to the vertex 
set of~$\bm$ is the graph metric, so that the two metric spaces 
corresponding to~$\bm$ are at $\dGH$-distance less than~$1/2$. In the 
scaling limit, this bears no effects. 

With this point of view on maps, note that, in the notation of
Sections~\ref{sec:comp-sli} and~\ref{sec:quadr-with-geod},  
\begin{itemize}
	\item $(\rho,\gamma)$ and $(\br{\rho},\xi)$ are geodesic marks of~$\bsl$; 
	\item $(\rho,\gamma)$, $(\br{\rho},\xi)$, $(\br{\rho},\br\gamma)$, and $(\rho,\br\xi)$ are geodesic marks of~$\bqd$. 
\end{itemize}

\subsection{Gluing along geodesics}\label{secgluing}

\paragraph{Quotient pseudometrics. }
Let $(\X,d)$ be a pseudometric space, that is, a set equipped with a
symmetric function $d:\X^2\to \Rp\sqcup\{\infty\}$ that vanishes on
the diagonal and satisfies the triangle inequality. Then $\{d=0\}$ is
an equivalence relation on~$\X$, and the quotient set $\X/\{d=0\}$
equipped with  the function induced by~$d$ (still denoted by~$d$ for
simplicity), is a true metric space, meaning that~$d$ is also separated. 

Let~$R$ be an equivalence relation on~$\X$. Let $d/R$ be the largest pseudometric on~$\X$ such that $d/R\leq d$ and that satisfies
$d/R(x,y)=0$ as soon as $x\binR y$. By \cite[Theorem 3.1.27]{burago01},
it is given by the formula
\begin{equation}\label{eq:quotient}
d/R(x,y)=\inf\left\{
\sum_{i=1}^m d(x_i,y_i):\!
\begin{array}{l}
  m\ge 1, x_1,\ldots, x_m,y_1,\ldots,y_m\in \X, \\
	x_1=x,y_m=y, y_i \binR x_{i+1} \mbox{ for }i\in \{1,\ldots,m-1\}
\end{array}
\right\}\! .
\end{equation}
In this setting, the set $\{d/R=0\}$ is another equivalence relation
on~$\X$ that contains~$R$, possibly strictly. We let $(\X,d)/R=(\X/\{d/R=0\},d/R)$ and call it the gluing of $(\X,d)$
along $R$.

A simple observation is that if $R_1$, $R_2$ are two equivalence relations
on~$\X$, then we have the equality of pseudometrics on~$\X$
\begin{equation}
  \label{dR1R2}
  (d/R_1)/R_2=(d/R_2)/R_1=d/R,
\end{equation}
where $R$ is the coarsest equivalence relation containing $R_1\cup
R_2$. This expression is indeed a direct consequence of~\eqref{eq:quotient} and the fact that
$x \binR y$ if and only if there exists some integer~$m$ and points
$x_0=x$, $x_1$, \dots, $x_m=y$ such that $(x_{i-1},x_i)\in R_1\cup R_2$ for
every $i\in \{1,2,\ldots,m\}$.

\paragraph{Gluing two spaces along geodesics. }

Let $(\X,d_\X)$, $(\Y,d_\Y)$ be two pseudometric spaces and~$\gamma$, $\xi$
be two geodesics in~$\X$ and~$\Y$,
respectively, where the definition~\eqref{eqgeod} of a geodesic is
naturally extended to pseudometric spaces.
The pseudometric of the disjoint union $\X\sqcup \Y$ is
defined by
\[d_{\X\sqcup \Y}(x,y)=\begin{cases}
	d_\X(x,y) & \text{ if }x, y\in \X\\
    d_\Y(x,y) & \text{ if }x, y\in \Y\\
    \infty & \text{ otherwise }
  \end{cases}.\]
We define the metric gluing of~$\X$
and~$\Y$ along~$\gamma$ and~$\xi$ by letting $\pzl=\lgth(\gamma)\wedge
\lgth(\xi)$ and by setting
\begin{equation}\label{GXY}
  G(\X,\Y;\gamma,\xi)=(\X\sqcup \Y,d_{\X\sqcup \Y})/R,
\end{equation}
where~$R$ is the coarsest equivalence relation satisfying
$\gamma(t) \binR \xi(t)$ for every $t\in [0,\pzl]$. 

In this particular case, the fact that~$\gamma$ and~$\xi$ are
geodesics greatly simplifies~\eqref{eq:quotient}. Indeed, $y_i \binR
x_{i+1}$ and $y_{i+1} \binR x_{i+2}$ imply that $d_{\X\sqcup
  \Y}(y_i,x_{i+2})=d_{\X\sqcup \Y}(x_{i+1},y_{i+1})$. In other words,
using twice the relation~$R$ does not create shortcuts. As a result,
the pseudometric of the gluing is the function~$d_G$ whose restrictions to $\X\times \X$ and $\Y\times \Y$ are~$d_\X$ and~$d_\Y$ respectively, and
\begin{equation}\label{dG}
d_G(x,y)=d_G(y,x)=\inf_{t\in [0,\pzl]}\big\{d_\X(x,\gamma(t))+d_\Y(\xi(t),y)\big\} \qquad
\text{ if }x\in \X,\, y\in \Y.
\end{equation}

\begin{rem}
In fact, Equation~\eqref{dG} holds in the more general setting of gluing along isometric subspaces~\cite[Chapter~I.5]{BrHa99}, the underlying isometry in our context being $\gamma(t) \mapsto \xi(t)$. We will not need this level of generality here.
\end{rem}

If~$(\X,d_\X,\bA,\bmu_\X)\in \M^{(\ell,m)}$
and~$(\Y,d_\Y,\bB,\bmu_\Y)\in \M^{(\ell',m')}$ are marked measured metric spaces, we may view
$G(\X,\Y;\gamma,\xi)$ as an element of $\M^{(\ell+\ell',m+m')}$ by
assigning marks and measures
\[
(\bp(A^1),\ldots,\bp(A^\ell),\bp(B^1),\ldots,\bp(B^{\ell'}))\quad\text{ and }\quad
(\bp_*\mu_\X^1,\ldots,\bp_*\mu_\X^m,\bp_*\mu_\Y^1,\ldots,\bp_*\mu_\Y^{m'}),
\]
where $\bp:\X\sqcup \Y\to G(\X,\Y;\gamma,\xi)$ is the canonical
projection. With a slight abuse of notation, we will keep denoting
these by~$\bA\bB$ and $\bmu_\X\bmu_\Y$. Observe that~$\gamma$ and~$\xi$ may themselves be part of the marking, in which case
they induce the same marks $\bp(\gamma)=\bp(\xi)$ in the
glued space. 
Observe also that geodesic marks in~$\bA$ or in~$\bB$ remain geodesic
marks in~$\bA\bB$, due to the fact that, by definition, 
$(\X,d_\X)$ and $(\Y,d_\Y)$ are isometrically embedded in
$G(\X,\Y;\gamma,\xi)$.

Finally observe that the gluing of the point space as an element of~$\M^{(\ell,m)}$ with $(\Y,d_\Y,\bB,\bmu_\Y)$ along~$\xi$ only has the effect of prepending $\ell$ times $\xi(0)$ to~$\bB$ and $m$ times the zero measure to~$\bmu_\Y$.

\paragraph{Gluing two geodesics in the same space.}
A similar gluing procedure\footnote{In fact, since we are allowing
points at infinite distance, the gluing $G(\X,\Y;\gamma,\xi)$ could
be seen as a particular case of gluing of a single space along
two geodesics, but we refrain to do so as we will mostly be
interested in gluing true metric spaces.} can be defined for two geodesics~$\gamma$, $\xi$ in the same pseudometric space
$(\X,d_\X)$. We again set $\pzl=\lgth(\gamma)\wedge
\lgth(\xi)$ and then define
\begin{equation}
G(\X;\gamma,\xi)=(\X,d_\X)/R,
\end{equation}
where~$R$ is the coarsest equivalence
relation satisfying $\gamma(t) \binR \xi(t)$ for every $t\in
[0,\pzl]$. The quotient pseudometric $d_G(x,y)$ may be condensed into 
\begin{equation}\label{eq:22}
d_\X(x,y)\wedge
	\inf_{t\in [0,\pzl]}\big\{d_\X(x,\gamma(t))+d_\X(\xi(t),y)\big\}\wedge
	\inf_{t\in
          [0,\pzl]}\big\{d_\X(x,\xi(t))+d_\X(\gamma(t),y)\big\}.
\end{equation}

Similarly to the above, the space $G(\X;\gamma,\xi)$ naturally
inherits the marking~$\bA$ and measures~$\bmu_\X$ that~$\X$ may be endowed with, simply
by pushing those forward by the canonical projection $\X\to
G(\X;\gamma,\xi)$; by a slight abuse of notation, we keep the piece of
notation~$\bA$, $\bmu_\X$ for these inherited objects. Note however
that it
is \emph{not true in general} that geodesic marks in $(\X,d_\X)$
remain geodesic marks in $G(\X;\gamma,\xi)$. Let us state a useful
comparison result between~$d_\X$ and~$d_G$. 

\begin{lmm}\label{lemdistglue}
Let $(\X,d_\X)$ be a pseudometric space with
two distinguished geodesics~$\gamma$, $\xi$. Denote by~$d_G$ the pseudometric on
$G(\X;\gamma,\xi)$ as in~\eqref{eq:22}.
\begin{enumerate}[label=(\textit{\roman*})]
	\item\label{distgluei} For every $x$, $y\in \X$,
\[d_G(x,y)\leq d_\X(x,y)\leq
d_G(x,y)+R(\gamma,\xi),\]
where $R(\gamma,\xi)$ is the Hausdorff distance in the space
$(\X,d_\X)$ between the initial segments of $\gamma$, $\xi$ that are
glued together, i.e., of length $\pzl$.
	\item\label{distglueii} For every $\eps>0$ and $x$, $y\in \X$,
if $d_\X(x,y)<\eps$ and $d_\X(x,\gamma)
\wedge d_\X(y,\gamma)>\eps$ (or if $d_\X(x,\xi)\wedge
\d_\X(y,\xi)>\eps$), it holds that 
\[d_G(x,y)=d_\X(x,y).\]
\end{enumerate}
\end{lmm}

\begin{proof}
Let us first prove~\ref{distgluei}. The first inequality is a direct consequence of~\eqref{eq:22}. To 
prove the other bound, simply observe that
for every  $t\in [0,\pzl]$ we have
\begin{align*}
	d_\X(x,y)&\leq d_\X(x,\gamma(t))+d_\X(\gamma(t),\xi(t))+d_\X(\xi(t),y)\\
	&\leq  d_\X(x,\gamma(t))+d_\X(\xi(t),y)+R(\gamma,\xi). 
\end{align*}
Taking the infimum over $t$, and then applying the same reasoning with the
roles of~$\gamma$, $\xi$ interchanged, we obtain the result
by~\eqref{eq:22}.

\medskip
The proof of~\ref{distglueii} is even more straightforward. Under our assumptions, it holds
that both $d_\X(x,\gamma(t))+d_\X(y,\xi(t))$ and
$d_\X(x,\xi(t))+d_\X(y,\gamma(t))$ are greater than $\eps>d_\X(x,y)$, for every
choice of $t$, so that $d_G(x,y)$ must be  equal to $d_\X(x,y)$. 
\end{proof}

\paragraph{Gluing and GHP limits. }
From now on, we mostly focus on compact geodesic spaces. The gluing of one or 
two geodesic spaces along geodesics is again a geodesic space by general results presented
in~\cite{burago01}. The gluing operation also preserves the
compactness of the spaces that are glued together. Furthermore, if a
compact metric space is the Gromov--Hausdorff limit of a sequence of
compact geodesic spaces, then it is also a compact geodesic
space \cite[Theorem~7.5.1]{burago01}. 

The next result shows that the gluing operations behave well with
respect to the GHP metric. For simplicity, we state it with the first marks of the markings but it obviously holds up to index permutations, using Lemma~\ref{lem1lip} for instance.

\begin{prp}\label{compgluing}
Let $(\X_n,d_{\X_n},\bA_n,\bmu_{\X_n})$,
$(\Y_n,d_{\Y_n},\bB_n,\bmu_{\Y_n})$\MEMS{}{, $n\ge 0$,} be geodesic
marked measured metric spaces that converge in the marked
GHP topology to
$(\X,d_\X,\bA,\bmu_{\X})$, $(\Y,d_\Y,\bB,\bmu_{\Y})$.
Assume that the first pairs of marks $(A_{n}^1,A_{n}^2)=\gamma_n$
and $(B_{n}^1,B_{n}^2)=\xi_n$ of~$\X_n$ and of~$\Y_n$ are
geodesic marks for every $n\geq
0$. Then the first two marks of~$\bA$ and of~$\bB$ are
geodesic marks~$\gamma$, $\xi$, and 
\[G(\X_n,\Y_n;\gamma_{n},\xi_{n})\ton G(\X,\Y;\gamma,\xi)\]
in the GHP topology. 

Similarly, if we now assume that the first four marks are such that $(A_{n}^1,A_{n}^2)=\gamma_n$ and
$(A_{n}^3,A_{n}^4)=\xi_n$ are geodesic marks, then the same holds for the first marks~$\gamma$, $\xi$ of~$\bA$, and 
\[G(\X_n;\gamma_{n},\xi_{n})\ton G(\X;\gamma,\xi)\]
in the GHP topology. 
\end{prp}

In order to prove this proposition, we are first going to state and prove a
useful lemma that allows one to deal only with the situations where
the geodesics along which the spaces of interest are glued have the
same lengths. While Lemma~\ref{lem1lip} showed that the
operation of merging two marks is continuous on
$(\M^{(\ell,m)},\dGHP^{(\ell,m)})$, this lemma states that
in the case of geodesic marks, the natural splitting operation is
continuous. 
If~$\gamma$
is a geodesic mark and $r\in [0,\lgth(\gamma)]$, the \emph{splitting}
of~$\gamma$ at level~$r$ is the two geodesic marks $(\gamma(0),\{\gamma(t):0\leq t\leq
r\})$, $(\gamma(r),\{\gamma(t),r\leq t\leq \lgth(\gamma)\})$. 

\begin{lmm}\label{lemananasplit}
Let $(\X_n,d_{\X_{n}},\bA_n,\bmu_{\X_n})$, $n\ge 0$, be a sequence of geodesic marked measured metric spaces
converging in the GHP topology toward a geodesic marked measured metric
space $(\X,d_\X,\bA,\bmu_\X)$, and assume that, say, the first
pairs of marks, are geodesic marks~$\gamma_n$ and~$\gamma$. Denote by $\pzl_n=\lgth(\gamma_n)$ and $\pzl=\lgth(\gamma)$ their lengths. 
Let $r_n\in(0,\pzl_n)$ be real numbers such that $r_n\to r\in(0,\pzl)$. 
Then the convergence $\X_n\to \X$ still holds in the GHP topology after replacing the marks~$\gamma_n$ and~$\gamma$
in~$\bA_n$ and~$\bA$, with their splittings $\gamma_{n}'$, $\gamma_{n}''$ and
$\gamma'$, $\gamma''$, respectively at levels~$r_n$ and~$r$. 
\end{lmm}

\begin{proof}
Since the desired property does not involve the measures, it suffices to establish it in the space of marked, nonmeasured spaces. 
Let $\RR$ be a correspondence between~$\X_n$ and~$\X$ compatible with
the markings. We fix $\eps>\dis(\RR)$ and consider the enlarged
correspondence
\[\RR^\eps=\big\{(x,y)\in \X_n\times \X:\exists (x',y')\in \RR,\,
d_{\X_n}(x,x')\vee \d_\X(y,y')<\eps\big\}.\]
By the triangle inequality, the distortion of~$\RR^\eps$ is at most $\dis(\RR)+4\eps$. Moreover,
we claim that for $s\in [0,\pzl_n]$ and $t\in [0,\pzl]$ such that
$|t-s|< \eps-\dis(\RR)$, it holds that
$\gamma_n(s)\binRR^\eps\gamma(t)$. Indeed, since $\RR$ is compatible with
the markings, for every~$s$, $t$ as above, there exists~$u$ such
that $\gamma_n(s)\binRR\gamma(u)$, so
\[|s-u|=\big|d_{\X_n}(\gamma_n(s),\gamma_n(0))-d_\X(\gamma(u),\gamma(0))\big|\leq
\dis \RR,\]
and therefore $|d_\X(\gamma(t),\gamma(u))|=|t-u|\leq
|t-s|+|s-u|<\eps$, as wanted. From this, we conclude that $\RR^\eps$
is compatible with the geodesic marks~$\gamma_n'$, $\gamma'$ and~$\gamma_n''$, $\gamma''$, as soon as $\eps>|r_n-r|\vee |\pzl_n-\pzl|$.
Choosing a sequence of correspondences $\RR_n$ with vanishing
distortion and taking $\eps_n=(\dis(\RR_n)\vee|r_n-r|\vee |\pzl_n-\pzl|)+1/n$ yields the result.
\end{proof}

\begin{proof}[Proof of Proposition~\ref{compgluing}]
With the help of Lemma~\ref{lemananasplit}, we may and will assume that
$\gamma_n$, $\xi_n$ have same length for every $n$. 
The fact that the limiting marks are geodesics marks with same length comes from Proposition~\ref{limgeod}. We will first
establish the result for marked, nonmeasured spaces, from which we will deduce the full result thanks to Lemma~\ref{lemrandmark}. 

Let~$\RR$ and~$\RR'$ be two correspondences\MEMS{}{ respectively} between~$\X_n$ and~$\X$ and between~$\Y_n$ and~$\Y$, compatible with the considered markings. Identifying~$\X_n$ and~$\Y_n$ (resp.~$\X$ and~$\Y$) with their canonical embeddings into $G(\X_n,\Y_n;\gamma_n,\xi_n)$ (resp.\ into $G(\X,\Y;\gamma,\xi)$), we see $\RR''=\RR\cup\RR'$ as a correspondence between $G(\X_n,\Y_n;\gamma_n,\xi_n)$ and $G(\X,\Y;\gamma,\xi)$, obviously compatible with the other marks. In order to bound its distortion, let us take $x_n\binRR x$ and $y_n\binRR' y$.

We let $\pzl_n=\lgth(\gamma_n)=\lgth(\xi_n)$ and $\pzl=\lgth(\gamma)=\lgth(\xi)$ be the lengths of the considered geodesics and we denote by~$d_n$ and~$d$ the metrics in the previous gluings. From~\eqref{dG} and by compactness, there exists $t\in [0,\pzl]$ such that
\[d(x,y)=d_\X(x,\gamma(t))+d_\Y(\xi(t),y).\]
Then there exist $t_n$, $t'_n\in[0,\pzl_n]$ such that $\gamma_n(t_n)\binRR\gamma(t)$ and $\xi_n(t'_n)\binRR'\xi(t)$. As a result,
\begin{align*}
d_n(x_n,y_n)&\le d_{\X_n}\big(x_n,\gamma_n(t_n)\big)	+d_n\big(\gamma_n(t_n),\xi_n(t'_n)\big)
              +d_{\Y_n}\big(\xi_n(t'_n),y_n\big)\\
  &\le d_\X(x,\gamma(t))+\dis(\RR) + d_n\big(\gamma_n(t_n),\xi_n(t'_n)\big) + d_\Y(\xi(t),y)+\dis(\RR')\\
						&\le d(x,y)+\dis(\RR)+\dis(\RR')+d_n\big(\gamma_n(t_n),\xi_n(t'_n)\big).
\end{align*}
Using the facts that~$\gamma_n$, $\xi_n$, $\gamma$, $\xi$ are geodesics and $\gamma_{n}(0)\binRR\gamma(0)$, $\xi_{n}(0)\binRR'\xi(0)$, we easily obtain
\begin{align*}
d_n\big(\gamma_n(t_n),\xi_n(t'_n)\big)&=\big|d_{\X_n}\big(\gamma_{n}(0),\gamma_n(t_n)\big)-d_{\Y_n}\big(\xi_{n}(0),\xi_n(t'_n)\big)\big|\\
                                          &\leq |d_\X(\gamma(0),\gamma(t))-d_\Y(\xi(0),\xi(t))|+ \dis(\RR)+\dis(\RR')\\
                                           &= \dis(\RR)+\dis(\RR').
\end{align*}
Using a symmetric argument, we \MEMS{obtain}{see that} $|d_n(x_n,y_n)- d(x,y)|\le
2\,(\dis(\RR)+\dis(\RR'))$. Adding to this the simpler cases where the pairs of points we compare belong both to~$\RR$ or both to~$\RR'$, we obtain
\[\dis(\RR'')\le 2\,(\dis(\RR)+\dis(\RR'))\]
and the first statement easily follows for the GH topology (without the measures).

Let us show that the result still holds when considering the
measures. We assume for simplicity that the terms of~$\bmu_\X$, $\bmu_\Y$ are
all nonzero, since the case of a vanishing measure, say~$\mu^j_\X$, is
equivalent to the fact that $\mu^j_{\X_n}(\X_n)\to0$. 
Denote by~$m$, $m'$ the numbers of coordinates of~$\bmu_\X$, $\bmu_\Y$ and sample $r(m+m')$ independent points $\bx=(x^j_i,1\leq i\leq
r,1\leq j\leq m)$ and $\by=(y_i^j,1\leq i\leq r,1\leq j\leq m')$ where
$x_i^j$ has law $\br{\mu}^j_\X$ and $y_i^j$ has law
$\br{\mu}^j_\Y$. 
We identify these points with their images in the glued
space by the canonical projection $\X\sqcup \Y\to
G(\X,\Y;\gamma,\xi)$. 
We assume that $\mu_{\X_n}(\X_n)>0$ and $\mu_{\Y_n}(\Y_n)>0$, which
hold for~$n$ sufficiently large since, as $n\to\infty$,
$\mu_{\X_n}(\X_n)\to \mu_{\X}(\X)>0$ and $\mu_{\Y_n}(\Y_n)\to
\mu_{\Y}(\Y)>0$. We proceed similarly to sample~$r(m+m')$ random
points $\bx_n=(x_{n,i}^j)$, $\by_n=(y_{n,i}^j)$ in 
$G(\X_n,\Y_n;\gamma_n,\xi_n)$ with laws
$\br{\mu}^j_{\X_n}$ and $\br{\mu}^j_{\Y_n}$ as appropriate. Lemma~\ref{lemrandmark}
guarantees that the marked spaces
$(\X_n,d_{\X_n},\bA_n\bx_n)$ and
$(\Y_n,d_{\Y_n},\bB_n\by_n)$
converge to $(\X,d_{\X},\bA\bx)$ and
$(\Y,d_{\Y},\bB\by)$ in
distribution in the GH topology. Applying the
result of Proposition~\ref{compgluing} proved above in 
the case without measures, we obtain the convergence in distribution of the glued
space $G(\X_n,\Y_n;\gamma_{n},\xi_{n})$ with markings
$\bA_n\bB_n\bx_n\by_n$ to $G(\X,\Y;\gamma,\xi)$ with the marking
$\bA\bB\bx\by$. Since~$\bx_n$, $\by_n$ and~$\bx$, $\by$ are also independent
samples from the renormalized measures~$\br\bmu_{\X_n}$, $\br\bmu_{\Y_n}$
and~$\br\bmu_\X$, $\br\bmu_\Y$ viewed as measures on the glued spaces, an
application of the converse implication of Lemma~\ref{lemrandmark}
implies the result.  

\medskip
The second part of the statement dealing with metric spaces that are glued along two marked
geodesics is shown in a similar fashion. We leave the details to the
reader.
\end{proof}

\subsection[Proof of Theorem 1.1]{Proof of Theorem~\ref{mainthm}}\label{sec:pfthm1}

Let $g$, $\kk\in\Zp$ be fixed; as in Section~\ref{secCMS}, we exclude the cases $(g,\kk)\in \{(0,0),(0,1)\}$ of the sphere, the pointed sphere, or the disk. Recall that the case $(g,\kk)=(0,0)$ of the sphere is already known. The case $(g,\bb,\kk)=(0,1,1)$ of the disk is partially known but has not been treated in the complete setting of Theorem~\ref{mainthm}. In fact, it may actually enter the following framework: in this case, the decomposition of Section~\ref{secCMS} yields only one well-labeled forest (and thus one unique slice) indexed by a degenerate scheme with one external face having one unique vertex with a self-loop edge. This creates a small ambiguity coming from the choice of the first tree in the forest, which can be overcome by randomization when considering random maps. The case $(g,\bb,\kk)=(0,0,1)$ of the pointed sphere would yield an even more degenerate decomposition with a one-vertex map as a scheme and a unique composite slice of width~$0$, object that is not introduced in this work.

Instead of considering these extensions and objects, we rather obtain these cases by using Proposition~\ref{propnull} at the end of Section~\ref{sec:main-conv-result}. More precisely, provided Theorem~\ref{mainthm} holds for $(g,\kk)=(0,2)$, we infer the case $(g,\kk)=(0,1)$ as follows. Let $l_n^1\in \Zp$ be such that $l_n^1/{\sqrt{2n}} \to L^1$ as $n\to\infty$ and $Q_n$ be uniform in~$\smash{{\RbQ}_{n,(l_n^1)}^{[0]}}$. Then let us choose a vertex uniformly at random among the internal vertices of~$Q_n$ and denote by~$Q_n^\bullet$ the map obtained from~$Q_n$ by declaring the chosen vertex as a second hole. Since~$Q_n^\bullet$ is clearly uniform in $\smash{{\RbQ}_{n,(l_n^1,0)}^{[0]}}$, the case $(g,\kk)=(0,2)$ of Theorem~\ref{mainthm} implies the convergence of the corresponding metric measure space toward~$\smash{\BS{0}{(L^1,0)}}$, and finally the convergence of the metric measure space corresponding to~$Q_n$ toward the space~$\BS{0}{(L^1)}$, defined as~$\BS{0}{(L^1,0)}$ with its second mark and second boundary measure forgotten.

\subsection{Gluing quadrangulations from elementary
  pieces}\label{sec:gluing-quadr-from} 

We start by interpreting the observations of Section~\ref{secCMS} in
the light of the previous section for a deterministic map. Let
$n\in\Zp$ and $\bl \in \N^\kk$ be fixed, let $\bq\in \bQ^{[g]}_{n,\bl}$ and $v_*\in V(\bq)$. The CVS construction
being one-to-one, there is a unique labeled map
$(\bm,\lambda)\in\bM^{[g]}_{n,\bl}$ that corresponds
to~$(\bq,v_*)$. We denote by~$\bs$ the scheme of~$\bm$ and by
$(\EP^e,e\in \vec E(\bs))$ the collection of elementary pieces
of~$(\bq,v_*)$. We emphasize that the decomposition strongly depends
on the distinguished point~$v_*$ and not only on~$\bq$.

If $e\in \vec B(\bs)$, we let $\gamma^e$, $\xi^e$ and $\beta^e$ be the maximal
geodesic, shuttle, and base of the slice $\EP^e$,
and we let~$\mu^e$, $\nu^e$ be the associated area and base measures defined at~\eqref{mmsslice}. If
$e\in \vec I(\bs)$, we let~$\gamma^e$, $\xi^e$, $\bar{\gamma}^e$, $\bar{\xi}^e$ be the
maximal geodesics and shuttles of~$\EP^e$, where the first two correspond to~$I_e$ and the latter two to~$I_{\br e}$, in the
notation of Section~\ref{secCMS}; we also let~$\mu^e$ be the
associated area measure defined at~\eqref{mmsquad}. 
Note that~$\EP^{\br e}$ yield the same map as~$\EP^e$ with the same measure; only the marks are ordered differently, namely $\bar{\gamma}^e$, $\bar{\xi}^e$, $\gamma^e$, $\xi^e$. 

The construction of~$\bq$ from $(\bm,\lambda)$ consists in connecting
every corner of~$\bm$ to its successor, and the paths following
consecutive successors are geodesic paths all aiming toward~$v_*$.
On the other hand, the
construction of the elementary pieces from $(\bm,\lambda)$ consists in
performing the interval CVS bijection on every interval $I_e$, in the
notation of Section~\ref{secCMS}. The only difference between these
constructions lies on the shuttles of these elementary pieces: if $c$
is a corner in some interval~$I_e$ whose successor in the interval bijection belongs to the
shuttle~$\xi^e$, then in order to obtain~$\bq$ we should rather connect $c$ to
its successor $s(c)$ in the contour order around~$f_*$. This successor
will belong to some interval $I_{e'}$ arriving later in
contour order around $f_*$ -- note that this interval can be~$I_e$
itself. Note also that this successor $s(c)$ belongs to the maximal geodesic~$\gamma^{e'}$. 
Moreover, interpreting~$\bq$ and its
elementary pieces as compact geodesic 
marked metric spaces, it is straightforward to see that the identifications
correspond to metric gluings along geodesics. 

\paragraph{Iterative gluing procedure.}
In order to reconstruct $\bq$ from $\EP^e$, $e\in
\vec{E}(\bs)$, 
rather than connecting the shuttle
vertices to their actual successors all at once, we will proceed
progressively by first connecting only those whose successors belong to the
maximal geodesic of the elementary piece that arrives immediately
after in contour order around~$f_*$. 

We now formalise this idea. Let~$\kappa$ denote the
cardinality of~$\vec E(\bs)$. We arrange the half-edges~$e_1$, \dots, $e_\kappa$ incident to the internal face of~$\bs$ according to the contour
order, starting at an arbitrarily chosen half-edge. 
While following the contour of the internal
face~$f_*$ of~$\bm$, we successively visit the 
elementary pieces~$\EP^{e_i}$, $1\leq i\leq
\kappa$, which are themselves viewed as marked measured geodesic metric spaces. The reconstruction of~$\bq$ will be done recursively in~$\kappa$ steps, resulting in a
sequence of marked $\kk+1$-measured metric spaces~$\bq_0$, \dots, $\bq_\kappa$. At the $i$-th step, $\bq_{i+1}$ will be obtained from~$\bq_i$ by gluing~$\EP^{e_i}$ along (part of) its marked maximal geodesic~$\gamma^{e_i}$. At the same time, we will do some operations on the markings and measures, namely reorderings, unions of marks and sums of measures, which are all continuous by Lemma~\ref{lem1lip}.

We need to keep track of the boundary marks\MEMS{and}{, as well as} the geodesics yet to be glued as marks. More precisely, the marking of~$\bq_i$ is
$(\gamma_i^0,\xi_i^0,\gamma_i^1,\xi_i^1,\ldots,\gamma_i^{u_i},\xi_i^{u_i},\beta_i^1,\ldots,\beta_i^\kk)$, where
\begin{itemize}
\item $\gamma_i^j$, $\xi_i^j$, $0\leq j\leq u_i$, are
  geodesic marks,
  \item $\beta_i^1$, \ldots, $\beta_i^\kk$ are called the \emph{boundary
    marks}. By convention, certain of these marks may be empty, in
    which case they are simply discarded from the marks. 
\end{itemize}
The mark~$\xi_i^0$ is the mark along which the subsequent gluing
producing~$\bq_{i+1}$ will occur, and~$u_i$ represents the number of
quadrilaterals that have been involved only once in the gluing
procedures up to the $i$-th step. Each of these quadrilateral yields
two marks~$\gamma_i^j$, $\xi_i^j$ for some
$j\in\{1,\dots,u_i\}$, corresponding to the unvisited half of the
quadrilateral, which will have to be glued at a further step. 
Finally, each~$\bq_i$ will come with measures~$\mu_i$, $\bnu_i$ where~$\mu_i$ is called the \emph{area measure} and~$\bnu_i=(\nu_i^1,\ldots,\nu_i^\kk)$ is the $\kk$-tuple of \emph{boundary measures}. 

\medskip
We initiate the construction by letting $u_0=0$, $\bq_0\in\M^{(2,\kk+1)}$ be the point space with the two marks $\gamma_0^0$, $\xi_0^0$ being the unique point, and measures $\mu_0=0$, $\bnu_0=\bzero^\kk$. We also let all the boundary marks be empty.

Next, provided that~$\bq_{i-1}$ has been constructed for some $i\in \{1,\ldots,\kappa\}$, we define~$\bq_i$ by considering the following cases, depicted on Figures~\ref{fig:glueboundary} to~\ref{fig:glueself}. 
\begin{itemize}
\item 
If $e_i\in \vec B_r(\bs)$ for some $r\in \{1,\ldots,\kk\}$, meaning in particular that~$\EP^{e_1}$ is a
slice, we set
\[\bq_i=G\big(\bq_{i-1},\EP^{e_i};\xi_{i-1}^0,\gamma^{e_i}\big),\]
and mark it as follows. We update the boundary marks by setting
\[\beta_i^r=\beta_{i-1}^r\cup \beta^{e_i}\, 
,\qquad \beta_i^{r'}=\beta_{i-1}^{r'}\ \text{ for } r'\in \{1,\ldots,\kk\}\setminus
\{r\}.\]
We update the geodesic marks by letting\footnote{In~\eqref{eq:20}, we use the convention set in Section~\ref{secgluing}
for marks in a glued space: in particular, after gluing, one of the marks
$\gamma^{e_i}$, $\xi_{i-1}^0$ is contained in the other,
depending on which is longest.}
\begin{equation}
  \label{eq:20}
    \gamma_i^0=\gamma_{i-1}^0\cup \big(\gamma^{e_i}\setminus
  \xi_{i-1}^0\big),\qquad \xi_i^0=\xi^{e_i}\cup \big(\xi_{i-1}^0
  \setminus \gamma^{e_i}\big),
\end{equation}
and, setting $u_i=u_{i-1}$, we let $\gamma_i^j=\gamma_{i-1}^j$ and
$\xi_i^j=\xi_{i-1}^j$ for $1\leq j\leq u_i$.
Finally, we update the measures by
\[\mu_i=\mu_{i-1}+\mu^{e_i},\qquad
\nu_i^r=\nu_{i-1}^r+\nu^{e_i},\qquad
\nu_i^{r'}=\nu_{i-1}^{r'}\ \text{ for } r'\in \{1,\ldots,\kk\}\setminus
\{r\}.\]
This case is illustrated in Figure~\ref{fig:glueboundary}.
\item
If $e_i\in \vec I(\bs)$, meaning in particular that $\EP^{e_1}$ is a
quadrilateral, we keep the boundary marks unchanged by setting
$\beta_i^r=\beta_{i-1}^r$ for $1\leq r\leq \kk$, we set
$\bnu_i=\bnu_{i-1}$, and 
consider the following two possible situations. 
\begin{itemize}
\item
If $e_i\notin \{\bar{e}_j,\, 1\le j < i\}$, that is, if the
unoriented edge corresponding to $e_i$ is visited for the \textgras{first time}, 
we let again
\[\bq_i=G\big(\bq_{i-1},\EP^{e_i};\xi_{i-1}^0,\gamma^{e_i}\big),\]
and update its geodesic marks as follows. We update the first two geodesic marks by~\eqref{eq:20}. We set $u_i=u_{i-1}+1$ and let $\gamma_i^j=\gamma_{i-1}^j$ and $\xi_i^j=\xi_{i-1}^j$ for $1\leq j\leq u_i-1$. Finally, we set
$\gamma_i^{u_i}=\bar{\gamma}^{e_i}$,
$\xi_i^{u_i}=\bar{\xi}^{e_i}$, and $\mu_i=\mu_{i-1}+\mu^{e_i}$. This case is illustrated in
Figure~\ref{fig:gluequad}.
\item
If $e_i\in \{\bar{e}_j,\, 1\le j < i\}$, say $e_i=\bar e_\ell$, that is, if the
unoriented edge corresponding to~$e_i$ is visited for the \textgras{second time}, then $\gamma^{e_i}=\br\gamma^{e_\ell}$ is a mark of~$\bq_{i-1}$: it is the mark $\gamma^{e_i}=\gamma_\ell^{u_\ell}$ of~$\bq_\ell$ and stays a mark of the subsequent spaces $\bq_{\ell+1}$, \dots, $\bq_{i-1}$. Similarly, $\xi^{e_i}=\br\xi^{e_\ell}$ is a mark of~$\bq_{i-1}$. We let 
\[\bq_i=G\big(\bq_{i-1};\xi_{i-1}^0,\gamma^{e_i}\big),\]
we update the first two geodesic marks by~\eqref{eq:20},
and, setting $u_i=u_{i-1}-1$, we let $(\gamma_i^j,\xi_i^j,1\leq j\leq u_i)$ be the
the sequence $(\gamma_{i-1}^j,\xi_{i-1}^j,1\leq j\leq u_{i-1})$ from 
which the terms~$\gamma^{e_i}$ and~$\xi^{e_i}$ have been
removed. Finally, we set $\mu_{i}=\mu_{i-1}$. This case is illustrated in
Figure~\ref{fig:glueself}.
\end{itemize}
\end{itemize}

\begin{figure}
\centering
\includegraphics{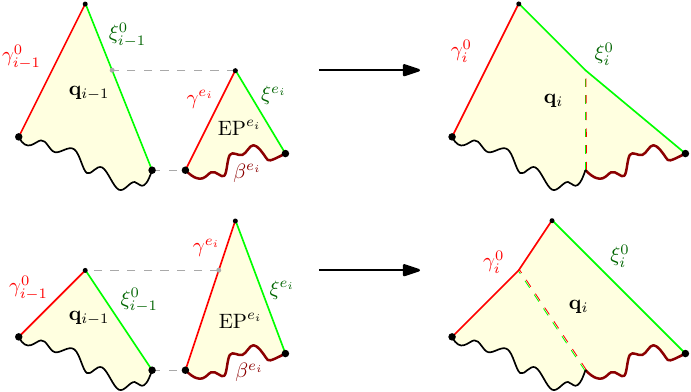}
\caption{The gluing procedure in the case where $\EP^{e_i}$ is a slice. In this picture and the following ones, the black wiggly
curve depicts all the marks different from
$\gamma_{i-1}^0$, $\xi_{i-1}^0$. The reader should bear in mind that, in
general, $\bq_{i-1}$ has no reason to present a planar topology as
in these pictures.  
The boundary $\xi_{i-1}^0$ of $\bq_{i-1}$ is
glued to the maximal geodesic $\gamma^{e_i}$, and the base of $\EP^{e_i}$ is added to the $r$-th boundary mark of~$\bq_{i-1}$
whenever~$\bar{e}_i$ is incident to~$\ch_r$. The first two geodesic marks are
updated according to~\eqref{eq:20}, which leads to the two
alternative situations described in this figure, depending on which
of~$\xi_{i-1}^0$ and~$\gamma^{e_i}$ is the longest: the
unglued part of these geodesics becomes part of~$\xi_i^0$ or of~$\gamma_i^0$.}
\label{fig:glueboundary}
\end{figure}

\begin{figure}
\centering
\includegraphics{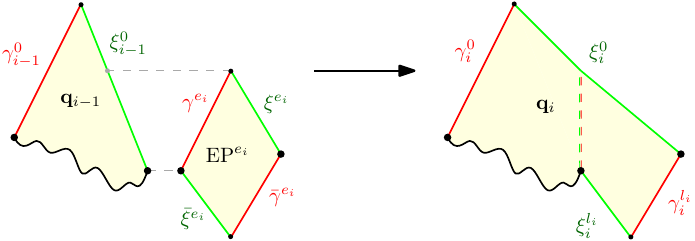}
\caption{The gluing procedure in the case where $\EP^{e_i}$ is a
quadrilateral that was not involved previously in the construction.
The boundary $\xi_{i-1}^0$ of $\bq_{i-1}$ is
glued to the maximal geodesic~$\gamma^{e_i}$. The first two geodesic marks are
again updated according to~\eqref{eq:20}, leading to two possible
situations depending on which
of~$\xi_{i-1}^0$ and~$\gamma^{e_i}$ is the longest. Only one
of these situations is represented on this figure. In this
case, the two geodesic boundary marks $\bar{\gamma}^{e_i}$
and~$\bar{\xi}^{e_i}$ are added to the marking; they will be involved in a
later construction step.}
\label{fig:gluequad}
\end{figure}

\begin{figure}
\centering
\includegraphics{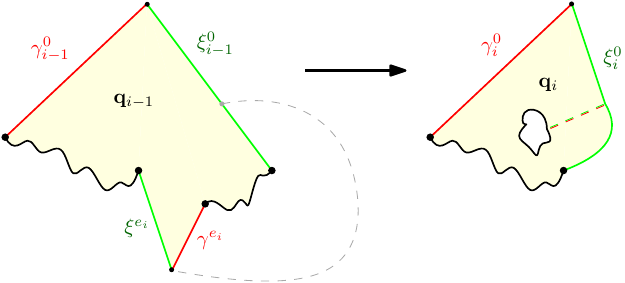}
\caption{The gluing procedure in the case where $\EP^{e_i}$ is a
quadrilateral, one side of which was already involved in a previous
construction step.
The boundary $\xi_{i-1}^0$ of $\bq_{i-1}$ is
glued to the maximal geodesic~$\gamma^{e_i}$, which had been
introduced as a geodesic mark in this previous construction
step. The first two geodesic marks are
again updated according to~\eqref{eq:20}, and the geodesic
marks~$\gamma^{e_i}$, $\xi^{e_i}$ are removed from the remaining marks.}
\label{fig:glueself}
\end{figure}

It is important to notice that, in~$\bq_i$, all the marks~$\gamma_i^j$, $\xi_i^j$, $0\leq j\leq u_i$, are geodesic marks. Indeed, each of these paths always take the form of a
chain of consecutive successors, which therefore must be a geodesic;
more precisely, these are the maximal geodesics and shuttles of the
interval CVS bijection on the intervals $I_{e_1}\cup\ldots\cup I_{e_i}$ and $I_{\bar e_j}$ for each $j\le i$ such that $e_j\in\vec I(\bs)\setminus\{\bar e_1,\dots,\bar e_i\}$.

At the end of this inductive procedure, we have connected all shuttle corners
of some interval $I_{e_i}$, to their actual successors in $\bm$
whenever these lie on some $I_{e_{i'}}$ with $1\leq i<i'\leq \kappa$. It
remains to connect the shuttle corners in some $I_{e_i}$
whose actual successor in $\bm$ lies in some $I_{e_{i'}}$ with $1\leq
i'\leq i\leq \kappa$. But one can observe that $u_\kappa=0$, 
so that $\bq_\kappa$ carries  exactly two geodesic marks~$\gamma_\kappa^0$, $\xi_\kappa^0$. The shuttle corners yet to be connected are exactly
those of $\xi_\kappa^0$, and should be matched to the successive corners of~$\gamma_\kappa^0$. 
Therefore, as marked metric spaces, we have $\bq=G(\bq_\kappa;\gamma_\kappa^0,\xi_\kappa^0)$, with
marks $\beta_\kappa^r$, $1\leq r\leq \kk$, which are precisely the connected components of the boundary~$\partial \bq$, ordered as they should. 

It is also possible to view all these gluing operations at once, as shown on Figure~\ref{gluings}.

\begin{figure}[htb!]
	\centering\includegraphics[width=.6\linewidth]{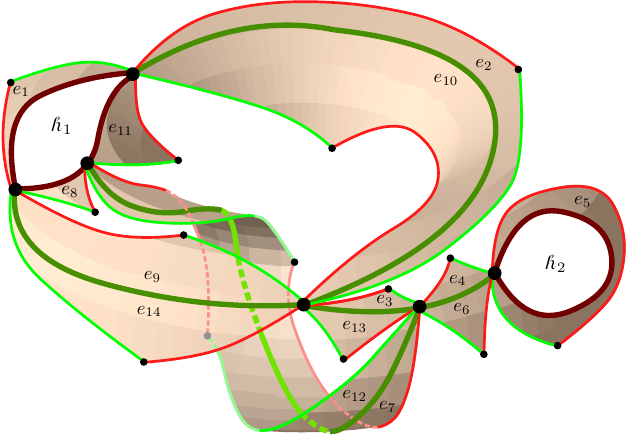}\vspace{9mm}
	
	\includegraphics[width=.95\linewidth]{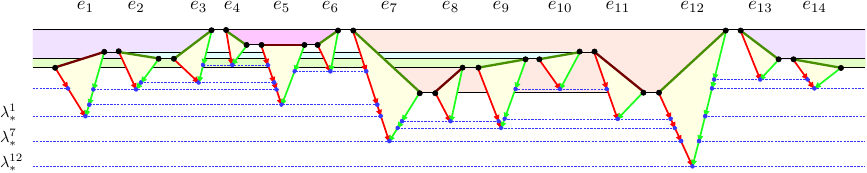}
	\caption{Reconstructing~$\bq$ by gluing its elementary pieces along geodesics. On this example, we have $\kappa=14$. Although we used the same scheme as in Figure~\ref{scheme} without~$\ch_3$ (for lower complexity), beware that the labels here do not match those from the left of Figure~\ref{scheme}. On the top, the half-edges of $\vec E(\bs)$ are arranged according to the contour order. With each one of them corresponds a ``triangle,'' which is either a slice (here, with~$e_1$, $e_5$, $e_8$, $e_{11}$) or ``half'' a quadrilateral, depending whether the half-edge belongs to~$\vec B(\bs)$ or~$\vec I(\bs)$. The five matchings of the corresponding ``halves'' of quadrilaterals, that is, the matchings of half-edges in~$\vec I(\bs)$, are represented with light colors. The vertices of these triangles are represented at a height corresponding to their label, where $\lambda_*^i=\min_{I_{e_i}} \lambda-1$. The geodesic marks of the pieces may be involved in multiple gluings. For instance, the shuttle of the leftmost triangle is involved in three gluings; it is split in three parts, corresponding to the triangles that can be ``seen'' to its right (those corresponding to~$e_2$, $e_5$, and~$e_7$).}
	\label{gluings}
\end{figure}

\paragraph{Measures.}
We claim that the previous equality $\bq=G\big(\bq_\kappa;\gamma_\kappa^0,\xi_\kappa^0\big)$ only holds as $\kk$-marked $\kk+1$-measured metric spaces up to a difference in the supports consisting of a bounded number of vertices. When considering rescaled measures through the operator~$\Omega_n$, this small difference 
will be of no importance in our limiting arguments. 

First of all, the boundary measures of the external faces match. This is because the boundary of a given external face of~$\bq$ is made of the bases, say $\beta^{e_{i_1}}$, \dots, $\beta^{e_{i_j}}$, of several slices that satisfy $\beta^{e_{i_\ell}}\cap\beta^{e_{i_{\ell+1}}}=\{\br\rho^{e_{i_\ell}}\}$ for $1\le \ell \le j$ with $i_{j+1}=i_1$ and where we denoted by~$\br\rho^{e}$ the final point of the base of~$\EP^e$. Since the base measure of the slice~$\EP^e$ is the counting measure on $\beta^e\setminus\{\br\rho^e\}$, the resulting measure in the gluing is the counting measure on the boundary of the considered external face of~$\bq$, as desired.

For the boundary measures of the external vertices, the measure in~$\bq$ is the counting measure on the singleton consisting of the external vertex, while the corresponding boundary measure in the gluing is the zero measure.

For the area measure, by convention, we decided that in the elementary 
pieces, the measures were taken on vertices \emph{outside of 
  the shuttles}. In doing so, after each gluing operation where a 
piece of a shuttle is glued to a piece of a maximal geodesic, the 
corresponding vertices are counted only once, as they should, \emph{except possibly for the vertices~$\rho$, $\bar{\rho}$ of the 
  quadrilaterals}, since they lie both on a maximal geodesic and on a 
shuttle, and are therefore not part of the counting measures by
convention. Therefore, the final gluing is naturally equipped with an area measure that is the counting measure on 
all but at most $2(2g+\kk-1)$ vertices, which is an upper bound on the 
number of vertices of the scheme by Lemma~\ref{lemnbschemes}.

\paragraph{Conclusion.}
We finally observe that the number~$\kappa$ of gluing operations necessary
to obtain~$\bq$ is uniformly bounded in~$n$. Indeed, the number of
edges of~$\bs$ is, by Lemma~\ref{lemnbschemes}, smaller than
$3(2g+\kk-1)$.

As a result, Theorem~\ref{mainthm} will directly follow from
subsequent applications of Proposition~\ref{compgluing} once we will
have shown that, after a proper scaling, the elementary pieces of a
uniform quadrangulation jointly converge in distribution in the GHP topology.

\subsection{Scaling limit of the collection of elementary pieces}\label{secslEP}

We now fix $\bL=(L^1,\dots,L^\kk)\in [0,\infty)^\kk$. We let $\bb$ and $\pp$ be the numbers of
indices~$i$ such that $L^i>0$ and $L^i=0$ respectively. In order to ease
notation, we assume that $L^1$, \dots, $L^{\bb}>0$ while
$L^{\bb+1}$, \dots, $L^{\kk}=0$.

\paragraph{Limiting measure for size parameters.}
We denote by $\RbS^\star$ the set of rooted genus~$g$ schemes with~$\kk$ holes, $\ch_1$, \dots, $\ch_\bb$ being faces, $\ch_{\bb+1}$, \dots, $\ch_\kk$ being vertices of degree~$1$, and whose internal vertices are all of degree exactly~$3$. These are called \emph{dominant} schemes. Let~$\bs\in \RbS^\star$ be fixed and denote its root vertex by~$v_0$. We let~$\cT_\bs$ be the set of tuples
\[\Big((\sfa^e)_{e\in \vec E(\bs)}, (\sfh^e)_{e\in \vec I(\bs)},(\sfl^e)_{e\in \vec B(\bs)}, (\lambda^v)_{v\in V(\bs)} \Big)
	\MEMS{}{\in (\Rp)^{\vec E(\bs)} \times(\Rp)^{\vec I(\bs)}\times \R^{\vec B(\bs)} \times \R^{V(\bs)}}\] 
\MEMS{in $(\Rp)^{\vec E(\bs)} \times(\Rp)^{\vec I(\bs)}\times \R^{\vec B(\bs)} \times \R^{V(\bs)}$ }{}such that
\begin{multicols}{2}\begin{itemize}
	\item $\sum_{e\in\vec E(\bs)} \sfa^e=1$,
	\item $\sfh^{\bar e}=\sfh^e$, for all $e \in \vec I(\bs)$, 
	\item $\sum_{e\in\vec B_i(\bs)} \sfl^{e}=L^i$, for $1\le i \le \bb$,
	\item $\lambda^{v_0}=0$.
\end{itemize}\end{multicols}

There is a natural Lebesgue measure $\cL_{\bs}$ on $\cT_\bs$
defined as follows. First, if $J$ is a finite set, and $L>0$ a positive real number, we let
$\Delta_J^L$ be the Lebesgue measure on the simplex $\{(\sfx^j,j\in J)\in(\Rp)^J:\sum_{j\in J}\sfx^j=L\}$. The latter measure can be defined as the image
of the measure 
 $\bigotimes_{j\in J'}\d \sfx^j\ind_{\{\sum_{j\in J'}\sfx^j<L\}}$, where $J'$ is obtained from $J$ by removing one arbitrary element $j'$, by the mapping $(\sfx^j,j\in
J')\mapsto (\sfx^j,j\in J)$ where $\sfx^{j'}=1-\sum_{j\in J'}\sfx^j$. 

Next, let $I(\bs)$ be an orientation of $\vec I(\bs)$, that is, a set
containing exactly one element from $\{e,\br e\}$ for every  $e\in \vec
I(\bs)$. We let $\cL^+_{I(\bs)}$ be the measure
$\bigotimes_{e\in I(\bs)}\d \sfh^e\ind_{\{\sfh^e\geq 0\}}$. Similarly, we let
$\cL_{V(\bs)}$ be the measure $\bigotimes_{v\in V'(\bs)}\d
\lambda^v$, where $V'(\bs)=V(\bs)\setminus \{v_0\}$.

Finally, the measure
$\cL_\bs$ is the image measure of
\[\Delta_{\vec E(\bs)}^1\otimes 
\cL^+_{I(\bs)}\otimes \bigotimes_{i=1}^{\bb}\Delta_{\vec
  B_i(\bs)}^{L_i}\otimes \cL_{V(\bs)}\] by the mapping
that associates with 
\[\Big((\sfa^e)_{e\in \vec E(\bs)}, (\sfh^e)_{e\in I(\bs)},\big((\sfl^e)_{e\in \vec B_i(\bs)},1\leq i\leq \bb\big),(\lambda^v)_{v\in V'(\bs)} \Big)\] 
the unique compatible element 
of~$\cT_\bs$, that is, such that $\sfh^e=\sfh^{\br e}$ for every $e\in \vec
I(\bs)$, and such that $\lambda^{v_0}=0$.

We let $\mathrm{Param}_\bL$ be the probability measure on $\bigcup_{\bs\in
  \RbS^\star}\{\bs\}\times \cT_\bs$ whose density with respect to the
measure $\sum_{\bs\in \RbS^\star}\delta_\bs\otimes \cL_\bs$ is 
\begin{equation}\label{eqlimvect}
\frac{1}{\ZZ_\bL}
	\prod_{e\in \vec I(\bs)} q_{\sfh^e}(\sfa^e)
	\prod_{e\in \vec B(\bs)} q_{\sfl^e}(\sfa^e)
	\prod_{e\in I(\bs)} p_{\sfh^e}(\smash{\delta\!\lambda^e})
	\prod_{e\in \vec B(\bs)} p_{3\sfl^e}(\delta\!\lambda^e),
\end{equation}
where
$p_t(x)=e^{-x^2/2t}/\sqrt{2\pi t}$ is the Gaussian density, 
$q_x(t)=(x/t)\,p_t(x)\ind_{\{t>0\}}$ is the (stable~$1/2$) density for the hitting time
of level~$-x$ by standard Brownian motion, $\delta\!\lambda^e=\lambda^{e^+}-\lambda^{e^-}$ for $e\in\vec E(\bs)$, and $\ZZ_\bL$
is a normalizing constant, equal to the integral of the remaining display. Beware that the third product is over~$I(\bs)$, not~$\vec I(\bs)$.

\paragraph{Scaling limits for size parameters.}
Next, let $(\bl_n)=(l_n^1,\ldots,l^\kk_n)$ and $Q_n$ be as in the statement of
Theorem~\ref{mainthm}. 
We let~$v^*_n$ be uniformly distributed over the set of internal vertices of~$Q_n$, whose cardinality given by~\eqref{nbintvert} only depends on the parameters. Consequently, $(Q_n,v^*_n)$ is uniformly distributed over the set of quadrangulations from~$\RbQ^{[g]}_{n,\bl_n0}$, seeing~$v_*$ as a $\kk+1$-th hole. The rooted labeled map $(M_n,\lambda_n)$ corresponding via the CVS correspondence is thus uniformly distributed
over~$\RbM^{[g]}_{n,\bl_n}$. We denote by~$S_n$ the scheme of the nonrooted map corresponding to~$M_n$, and we root~$S_n$ uniformly at random among its half-edges, incident to internal or external faces, but such that the corresponding edge does not belong to the boundary of~$\ch_{\bb+1}$, \dots, $\ch_\kk$. Note that we could have rooted~$S_n$ from the root of~$M_n$ by asking that the root of~$M_n$ belongs to the forest indexed by the root of~$S_n$ but this would have introduced an undesirable bias. Here instead, from the unrooted map corresponding to~$M_n$, the map~$M_n$ is rooted at a uniform corner incident to its internal face. Furthermore, the boundaries of the holes~$\ch_{\bb+1}$, \dots, $\ch_\kk$ are excluded from the possible rootings of~$S_n$ since they should be thought of as having null length in the limit.

For $i\in \{\bb+1,\ldots,\kk\}$, the hole~$\ch_i$ of~$M_n$ is called a \emph{vanishing face} if it is a face, that is, if $l_n^i>0$. The corresponding hole~$\ch_i$ of the scheme~$S_n$ is called a \emph{tadpole} if it is made of a single
self-loop edge incident to a single vertex of degree~$3$. We let
$S_n^\tp$ be the map~$S_n$ in which every tadpole corresponding to a vanishing face~$\ch_i$ has been shrunk into a single vertex, still denoted by~$\ch_i$, in the sense that the corresponding self-loop has been removed. Note that the root of~$S_n$ is never removed in this operation, so that~$S_n^\tp$ is always rooted.

Forgetting the root of~$Q_n$, we let $(\EP^e_n,e\in \vec E(S_n))$ be 
the collection of elementary pieces of~$(Q_n,v^*_n)$. For the half-edges $e\in\vec B(S_n)$, we let~$A_n^e$, $L_n^e$ be the area and width of the slice~$\EP^e_n$. For $e\in\vec I(S_n)$, we let~$A_n^e$, $H_n^e$ be the first half-area and width of the quadrilateral~$\EP^e_n$; note that~$A_n^{\bar e}$ is the second half-area of~$\EP^e_n$. Recall that the vertices of~$S_n$ are in one-to-one correspondence
with the nodes of~$M_n$: for every $v\in V(S_n)$, we denote
by~$\Lambda_n^v$ the label of the corresponding node, where we choose
for the labeling function~$\lambda_n$ the representative giving
label~$0$ to the root vertex of~$S_n$. 

\begin{prp}\label{sldata}
As $n\to \infty$, with probability tending to one, every vanishing face of~$M_n$ induces a tadpole in~$S_n$, and, on this likely event, for every $e\in\bigsqcup_{i=\bb+1}^\kk \vec B_i(S_n)$, it holds that $A^e_n+L_n^e=\Theta((L_n^e)^2)$ in probability.

Moreover, the following convergence in distribution holds: 
\begin{multline}\label{vectorscaling}
\left(S_n^\tp, \left(\frac{A_{n}^e}{n}\right)_{e\in \vec E(S^\tp_n)},
  \left(\frac{H_{n}^e}{\sqrt{2n}}\right)_{e\in \vec I(S^\tp_n)}, \left(\frac{L^e_{n}}{\sqrt{2n}}\right)_{e\in \vec B(S^\tp_n)},
  \left(\Big(\frac{9}{8n}\Big)^{1/4} \Lambda_n^v\right)_{v\in V(S^\tp_n)}
\right)\\
\tod \big(S, \left({A^e}\right)_{e\in \vec E(S)}, \left({H^e}\right)_{e\in
  \vec I(S)}, \left(L^e \right)_{e\in \vec B(S)},\left( \Lambda^v\right)_{v\in V(S)} \big),
\end{multline}
where the limiting random variable has the law $\mathrm{Param}_\bL$ described in the previous paragraph.
\end{prp}

This proposition is a generalization of \cite[Proposition~15]{Bet16geo}. Given its technical nature,
we postpone its proof to Appendix~\ref{appdata}.

\paragraph{Scaling limits of the elementary pieces.}
As $(M_n,\lambda_n)$ is uniformly distributed
over the set $\RbM^{[g]}_{n,\bl_n}$, conditionally given~\eqref{vectorscaling},
the random variables~$\EP^e_n$, $e\in \vec E(S_n)$, are only dependent
through the relations linking~$\EP^{\bar e}_n$ with~$\EP^e_n$ for
$e\in \vec I(S_n)$. Moreover,
\begin{itemize}
	\item if $e\in\vec B(S_n)$, then~$\EP^e_n$ is uniformly
distributed among slices with area~$A_n^e$,
width~$L_n^e$ and tilt $\Lambda_n^{e^+} - \Lambda_n^{e^-}$;
	\item if $e\in\vec I(S_n)$, then~$\EP^e_n$ is uniformly distributed among quadrilaterals with half-areas~$A_n^e$ and~$A_n^{\bar e}$, width~$H_n^e$ and tilt $\Lambda_n^{e^+} - \Lambda_n^{e^-}$.
\end{itemize}
Applying the Skorokhod representation theorem, we may and will assume
that the convergence of~\eqref{vectorscaling} holds almost surely. Since, by
Lemma~\ref{lemnbschemes}, there are finitely many possible schemes, this furthermore implies that $S_n^\tp=S$
for~$n$ sufficiently large. Together with Theorems~\ref{thmslslice}
and~\ref{thmslquad}, the above observations entail that the collection
of resclaed random metric spaces $(\Omega_n(\EP^e_n), e\in \vec
E(S^\tp_n))$, converge in distribution in the GHP topology toward a family
$\big(\cEP^e,e\in\vec E(S)\big)$ of continuum elementary pieces with
the following law conditionally given the right-hand side of~\eqref{vectorscaling}:
\begin{itemize}
	\item if $e\in\vec B(S)$, then~$\cEP^e$ is a continuum slice with area~$A^e$, width~$H^e$ and tilt $\Lambda^{e^+}-\Lambda^{e^-}$;
	\item if $e\in\vec I(S)$, then~$\cEP^e$ is a continuum quadrilateral with half-areas~$A^e$ and~$A^{\bar e}$, width~$H^e$ and
          tilt $\Lambda^{e^+}-\Lambda^{e^-}$.
\end{itemize}
We write $\gamma(\cEP^e)$, $\xi(\cEP^e)$, $\mu(\cEP^e)$, and either $\beta(\cEP^e)$, $\nu(\cEP^e)$ or $\br\gamma(\cEP^e)$, $\br\xi(\cEP^e)$ the marks and measures of~$\cEP^e$, with an obvious choice of notation. As continuum elementary pieces have only been defined as limits in the GHP topology of discrete elementary pieces so far, these marks and measures are the limits of the corresponding marks and measures of the discrete pieces.

\bigskip
We treat the elementary pieces corresponding to the vanishing faces
thanks to Corollary~\ref{slslice0}. 
We define, for every hole~$\ch_i$ of~$S_n$ with $\bb+1\le i \le \kk$, the
elementary piece~$\EP^{\ch_i}_n$ as~$\EP^{e}_n$ if $\vec
B_i(S_n)=\{e\}$ or the vertex map otherwise, marked five times at its
unique vertex and endowed twice with the zero measure (thinking of it as an ``empty slice''). By
Proposition~\ref{sldata} and Corollary~\ref{slslice0}, with probability tending
to~$1$, each one of the rescaled random metric spaces
$\Omega_n(\EP^{\ch_i}_n)$, $\bb+1\le i \le \kk$, converges in
distribution in the GHP topology toward the point space (note that the
tilt of $\EP^{\ch_i}_n$ is always equal to~$0$).

\subsection{Gluing pieces together}\label{sec:gluing-piec-togeth}

We can now complete the proof of Theorem~\ref{mainthm} by applying
Proposition~\ref{compgluing} at every step of the inductive
construction of Section~\ref{sec:gluing-quadr-from}.

We work on the event of asymptotic full probability of Proposition~\ref{sldata} and assume that~$n$ is large enough so that $S^\tp_n=S$. We denote by $\kappa=|\vec{E}(S)|+\pp$ and let~$e_1$, \dots, $e_\kappa$ be the sequence made of the half-edges of~$\vec{E}(S)$, as well as the external vertices of~$S$, listed in contour order. Since~$S$ is dominant, these external vertices are~$\ch_{\bb+1}$, \dots, $\ch_\kk$. Applying the construction of Section~\ref{sec:gluing-quadr-from} to the random
quadrangulation~$Q_n$, up to adding the gluing of the point space for each external vertex (not changing the markings and measures), we obtain a sequence $Q_{n,1}$, \dots, $Q_{n,\kappa}$ of marked measured metric spaces.

The limiting marked measured metric space~$\BS{g}{\bL}$
is obtained from $(S,(\cEP^e,e\in\vec{E}(S))$ by recursively defining a sequence of marked measured metric
spaces~$\sfS_0$, \dots, $\sfS_\kappa$, in the following way. For $0\le i \le \kappa$, the space~$\sfS_i$ will carry geodesic marks~$\gamma_i^j$, $\xi_i^j$, $0\leq j\leq u_i$, boundary marks~$\beta_i^1$, \dots, $\beta_i^\kk$, an area measure~$\mu_{i}$, and boundary measures~$\nu_i^1$, \dots, $\nu_i^\kk$.

\medskip
We initiate the construction by letting $u_0=0$, $\sfS_0\in\M^{(2,\kk+1)}$ be the point space with the two marks $\gamma_0^0$, $\xi_0^0$ being the unique point, and measures $\mu_0=0$, $\bnu_0=\bzero^\kk$. We also let all the boundary marks be empty.

Next, given $\sfS_{i-1}$ for some $i\in \{1,\ldots,\kappa\}$, consider the following cases. 
\begin{itemize}
\item If $e_i\in \vec B_r(S)$ for some $r\in \{1,\ldots,\kk\}$, set
\begin{align}
\sfS_i&=G\big(\sfS_{i-1},\cEP^{e_i};\xi_{i-1}^0,\gamma(\cEP^{e_i})\big),\notag\\
\beta_i^r&=\beta_{i-1}^r\cup \beta(\cEP^{e_i}),
	&\beta_i^{r'}&=\beta_{i-1}^{r'}\ \text{ for } r'\in \{1,\ldots,\kk\}\setminus\{r\},\notag\\
\gamma_i^0&=\gamma_{i-1}^0\cup \big(\gamma(\cEP^{e_i})\setminus\xi_{i-1}^0\big),
	&\xi_i^0&=\xi(\cEP^{e_i})\cup \big(\xi_{i-1}^0 \setminus \gamma(\cEP^{e_i})\big),\label{updategammaxi}
\end{align}
and, setting $u_i=u_{i-1}$, 
let $\gamma_i^j=\gamma_{i-1}^j$ and
$\xi_i^j=\xi_{i-1}^j$ for $1\leq j\leq u_i$. Finally, we let
$\mu_i=\mu_{i-1}+\mu(\cEP^{e_i})$,
$\nu_i^r=\nu_{i-1}^r+\nu(\cEP^{e_i})$ and
$\nu_i^{r'}=\nu_{i-1}^{r'}$ for $r'\neq r$.

\item
If $e_i\in \vec I(S)$, set
$\beta_i^r=\beta_{i-1}^r$, $\nu_i^r=\nu_{i-1}^r$ for
$1\leq r\leq \kk$, and 
consider the following two possible situations. 
\begin{itemize}
\item 
If $e_i\notin \{\bar{e}_j,\, 1\le j < i\}$, let 
\[\sfS_i=G\big(\sfS_{i-1},\cEP^{e_i};\xi_{i-1}^0,\gamma(\cEP^{e_i})\big),\]
update the  first two geodesic marks by~\eqref{updategammaxi}, 
and, setting $u_i=u_{i-1}+1$, let $\gamma_i^j=\gamma_{i-1}^j$ and
 $\xi_i^j=\xi_{i-1}^j$ for $1\leq j\leq u_i-1$, and
$\gamma_i^{u_i}=\bar{\gamma}(\cEP^{e_i})$,
$\xi_i^{u_i}=\bar{\xi}(\cEP^{e_i})$. Finally, set
$\mu_i=\mu_{i-1}+\mu(\cEP^{e_i})$. 
\item
if $e_i\in \{\bar{e}_j,\, 1\le j < i\}$ let 
\[\sfS_i=G\big(\sfS_{i-1};\xi_{i-1}^0,\gamma(\cEP^{e_i})\big),\]
update the first two geodesic marks by~\eqref{updategammaxi}, 
and, setting $u_i=u_{i-1}-1$, let $(\gamma_i^j,\xi_i^j,1\leq j\leq u_i)$ be the
the sequence $(\gamma_{i-1}^j,\xi_{i-1}^j,1\leq j\leq u_{i-1})$ from 
which the terms $\gamma(\cEP^{e_i})$ and
$\xi(\cEP^{e_i})$ have been removed. Finally, set $\mu_i=\mu_{i-1}$. 
\end{itemize}

\item If $e_i$ is an external vertex of~$S$, set $\sfS_i=\sfS_{i-1}$.
\end{itemize}
Finally, we let
$\BS{g}{\bL}=G(\sfS_\kappa;\xi_\kappa^0,\gamma_\kappa^0)$, seen as an element of~$\M^{(\kk,\kk+1)}$, equipped with the marking
$(\beta_\kappa^1,\dots,\beta_\kappa^\kk)$ and
measures~$\mu_\kappa$, $\bnu_\kappa$. 
An application of Proposition~\ref{compgluing} and of Lemma~\ref{lem1lip} at every step of the
construction shows that, for every $i\in \{1,\ldots,\kappa\}$, the rescaled marked measured metric space $\Omega_n(Q_{n,i})$ 
converges to~$\sfS_i$ in the marked
GHP topology, and finally $\Omega_n(Q_n)$ converges to~$\BS{g}{\bL}$ by a final application of Proposition~\ref{compgluing} and the observation regarding the measure supports in the final gluing mentioned at the end of Section~\ref{sec:gluing-quadr-from}. 
This completes the proof of Theorem~\ref{mainthm}.

\subsection{Topology and Hausdorff dimension}\label{sectopo}

In this section, we\MEMS{}{ essentially} derive from Proposition~\ref{proptopo} in the case $(g,\kk)\in\{(0,0),(0,1)\}$ an alternate proof of Proposition~\ref{proptopo} in the other cases. In the spherical case, Proposition~\ref{proptopo} was obtained by Le~Gall and Paulin~\cite{lgp} thanks to a theorem of Moore by seeing the Brownian sphere as a rather wild quotient of the sphere by some equivalence relation. The same result was later obtained in~\cite{miermontsph} through the theory of regularity of sequences developed by Begle and studied by Whyburn. The latter approach was generalized in~\cite{bettinelli11,bettinelli11b,Bet16geo} in order to obtain the general cases.

Our approach of decomposition into elementary pieces gives a rather direct and transparent proof of Proposition~\ref{proptopo} in the case $(g,\kk)\notin\{(0,0),(0,1)\}$ provided the following lemma, which will be obtained in Sections~\ref{seccvcs} and~\ref{seccvquad}, and which amounts to Proposition~\ref{proptopo} for the noncompact analogs of the cases $(g,\kk)\in\{(0,0),(0,1)\}$.

\begin{lmm}\label{topdimep}
Almost surely, a slice or a quadrilateral is homeomorphic to a disk and is locally of Hausdorff dimension~$4$. Its boundary as a topological manifold consists in the union of its marks, the intersection of any two marks being empty or a singleton. Furthermore, in the case of a slice, the base is locally of Hausdorff dimension~$2$.
\end{lmm}

Now, a quadrangulation from~$\bQnln$ is ``not far''
from being homeomorphic to~$\Sof{g}{\bb_n}$, where~$\bb_n$ is the number of external faces, in the sense that we can ``fill
in'' the internal faces with small topological disks and ``fill in'' the external faces by
thin topological annuli without altering the metric and in such a way
that the resulting object is homeomorphic to~$\Sof{g}{\bb_n}$ and at bounded
GH distance from the quadrangulation (see
\cite[Section~4.3.3]{Bet16geo} for more details about this
procedure). The decomposition of the quadrangulation into elementary
pieces gives a decomposition of this surface into pieces that are non
other than the elementary pieces of the quadrangulation with faces
filled in and a thin rectangle added on the bases of the slices; see Figure~\ref{topo}.

\begin{figure}[ht]
	\centering\MEMS{
	\includegraphics[height=4.3cm]{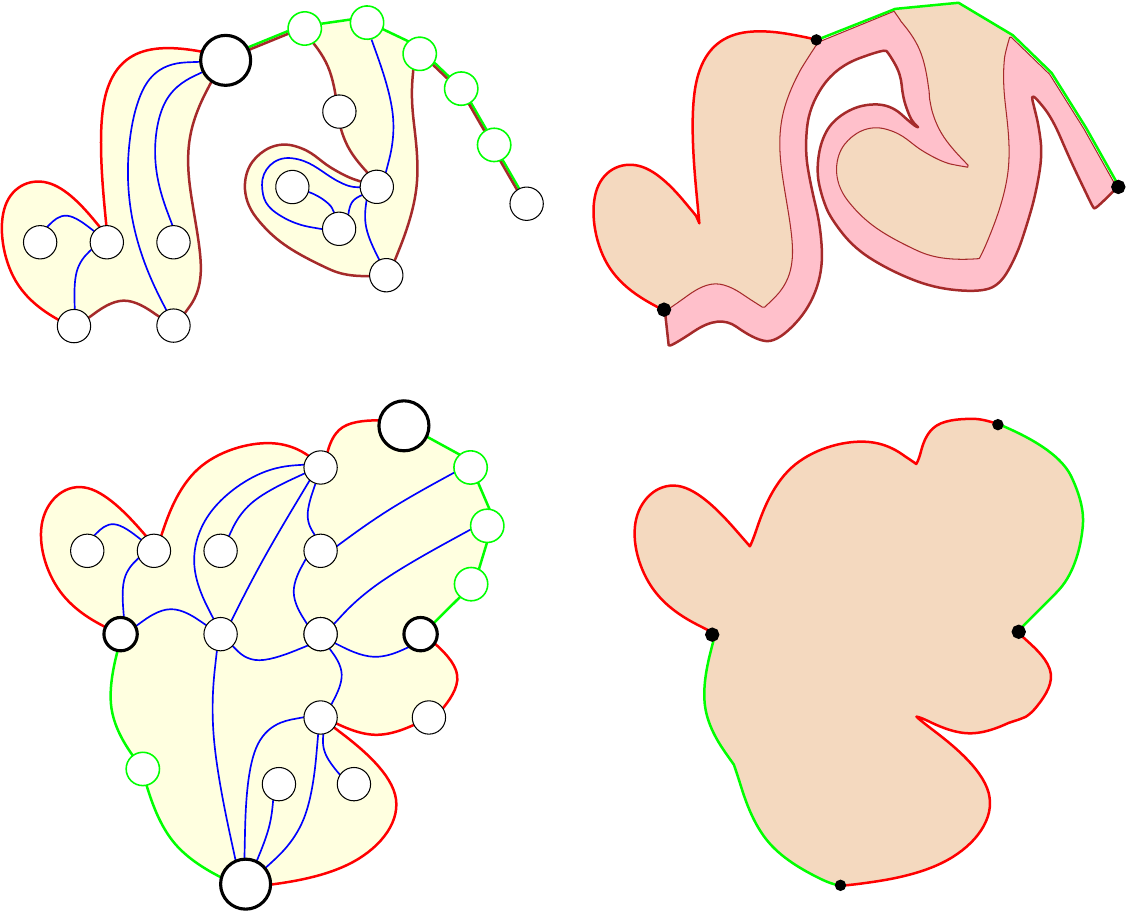}\hfill\includegraphics[height=4.3cm]{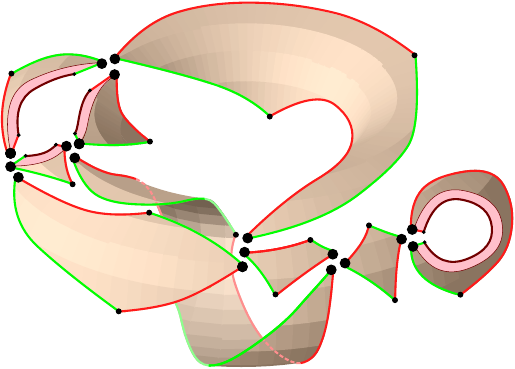}
	}{
	\includegraphics[height=5.3cm]{topo1}\hfill\includegraphics[height=5.3cm]{topo2}
	}
	\caption{\textup{\textgras{Left.}} Topological disk corresponding to an elementary piece of a quadrangulation. \textup{\textgras{Right.}}  Decomposition of the surface associated with a quadrangulation into surfaces homeomorphic to disks. Here also, we used the same scheme as in Figure~\ref{scheme} without~$\ch_3$.}
	\label{topo}
\end{figure}

With the notation of the previous section, $Q_n$ yields a surface homeomorphic to~$\Sof{g}{\bb_n}$ and its elementary pieces~$\EP^e_n$, $e\in \vec E(S_n)$, yield surfaces homeomorphic to disks.
\begin{itemize}
	\item If $e\in\vec I(S_n)$, then the boundary of the surface associated with~$\EP^e_n$ consists in the two maximal geodesics and  the two shuttles of~$\EP^e_n$.
	\item If $e\in\vec B(S_n)$, then the boundary of the surface associated with~$\EP^e_n$ consists in the maximal geodesic, the shuttle, as well as three sides of the added thin rectangle.
\end{itemize}
Then gluing back these disks along their boundaries in the same way as they were cut gives back the surface~$\Sof{g}{\bb_n}$. Forgetting the slice corresponding to a tadpole topologically amounts to fill in the corresponding vanishing face. Assuming that~$n$ is sufficiently large, all the vanishing faces correspond to tadpoles in the scheme. So the gluing of the elementary pieces without the slices corresponding to tadpoles yields~$\Sgb$. In the limit, we glue topological disk exactly in the same way (using markings that are topologically equivalent), so we obtain the same surface. The result about the topology follows.

The statement about the Hausdorff dimension is even more straightforward as we glue along geodesics a finite number of objects that are locally of dimension~$4$ and the boundary is the union of the bases of the slices, which are all locally of Hausdorff dimension~$2$.

\section{Convergence of composite slices}\label{seccvcs}

The goal of this section is to prove Theorem~\ref{thmslslice} on the convergence of slices to
their limiting slices. To this end, we are first going to derive a
``free'' version of this result by finding slices with a
free area and tilt within the uniform infinite half-planar
quadrangulation. The latter is known to converge to the Brownian
half-plane, which itself contains a ``flow'' of continuum slices with
free areas and tilts; these are shown to be the scaling limits of the
discrete slices. We conclude by a conditioning argument to pass
from free to fixed area and tilt. First, let us start with deterministic considerations.

\subsection{Metric spaces coded by real
  functions}\label{sec:metric-spaces-coded}

Here we borrow some material from \cite[Section~2.1]{BaMiRa}, with however several slight differences, in order to describe in a unified fashion the various random metric spaces we will use. 
Let~$\CC$ (resp.\ $\CC^{(2)}$) be the set of continuous functions of
one variable (resp.\ of two variables) defined on some nonempty closed interval: 
\[\CC=\bigsqcup_{\substack{I\text{ closed interval}\\I\neq
    \varnothing}}
\CC(I,\R)\qquad\text{ and }\qquad
\CC^{(2)}=\bigsqcup_{\substack{I\text{ closed interval}\\I\neq \varnothing}}\CC(I^2,\R).\]
For a function $f\in \CC(I,\R)$ we denote by $I(f)=I$ its interval of
definition and by $\br\tau(f)=\inf I$ and $\tau(f)=\sup I$ its
extremities.
The set~$\CC$ is naturally equipped with the topology of uniform
convergence over compact subsets of $\R$; more precisely, the topology
induced by the following metric: 
\begin{multline*}
  \dist_\CC(f,g)=\big|\arctan(\br\tau(f))-\arctan(\br\tau(g))\big|+\big|\arctan(\tau(f))-\arctan(\tau(g))\big|\\
  +\sum_{n\geq 1}\frac{1}{2^n}\sup_{t\in[-n,n]}\big|f\big(\br\tau(f)\vee
  t\wedge\tau(f)\big)-g\big(\br\tau(g)\vee t\wedge\tau(g)\big)\big|. 
\end{multline*}
We also equip~$\CC^{(2)}$ with a straightforward adaptation
$\dist_{\CC^{(2)}}$.

\subsubsection{$\R$-trees coded by functions}\label{Rtrees}

\paragraph{$\R$-trees.}
For $f\in \CC$ and $s$, $t\in I=I(f)$ with
$s\le t$, set
\begin{equation}\label{defunderline}
\un{f}(s,t)=\inf_{[s,t]}f
\end{equation}
and, for~$s$, $t\in I$, set
\begin{align}\label{df}
d_f(s,t)&=f(s)+f(t)-2\un{f}(s\wedge t, s\vee t).
\end{align}
This formula defines a pseudometric on~$I$, which is continuous as
a function from $I^2$ to $\Rp$, since $d_f(s,t)\leq
2\omega(f;[s\wedge t,s\vee t])$, where $\omega(f;J)=\sup_J f
-\inf_Jf$. 
We let
$\cT_f=(I/\{d_f=0\},d_f)$ be the associated quotient
space, and $\bp_f:I\to \cT_f$ be the canonical projection, which is
continuous since~$d_f$ is. 
The space~$\cT_f$ is a so-called \emph{$\R$-tree}, that is, satisfies the following.
\begin{itemize}
\item 
For every two points $a$, $b\in\cT_f$, there exists a geodesic from
$a$ to $b$, that is, an isometric
mapping $\chi_{a,b}:[0,d_f(a,b)]\to \cT_f$ with $\chi_{a,b}(0)=a$ and
$\chi_{a,b}(d_f(a,b))=b$.
\item The image of the path $\chi_{a,b}$, which we denote by
  $\llbracket a,b\rrbracket_f$, is the image of any injective path from~$a$ to~$b$. 
\end{itemize}
If~$I$ is compact, we let $a_*(f)=\bp_f(t_*)$, where $t_*$ is any
point at which $f$ attains its overall minimum. In this case, for
$t\in I$ and 
$a=\bp_f(t)$, the geodesic 
segment $\llbracket a,a_*(f)\rrbracket _f$ is given by
\[\llbracket a,a_*(f)\rrbracket_f=\bp_f\big(\{s\in [t\wedge t_*,t\vee t_*]:f(s)\leq f(u),\
  \forall u\in[s\wedge t,s\vee t]\}\big).\]
In the case where~$I$ is unbounded, we will systematically make the extra assumption that
\begin{equation}\label{aslimit}
\begin{cases}
	\text{when }\br\tau(f)=-\infty, 		&\displaystyle\inf_{t\leq 0}f(t)=-\infty\, \quad \text{ or }\quad  \lim_{t\to-\infty}f(t)=\infty\,;\\[4mm]
	\text{when }\tau(f)=\infty,\qquad	&\displaystyle\inf_{t\geq 0}f(t)=-\infty\, \quad \text{ or }\quad  \lim_{\rlap{\scriptsize$t\to\infty$}\hphantom{t\to-\infty}}f(t)=\infty.
\end{cases}
\end{equation}
In particular, it holds that
\begin{equation}
  \label{eq:26}
  \forall\, s\in I,\quad \lim_{|t|\to\infty,\,t\in I}d_f(s,t)=\infty,
\end{equation}
which
implies that $\cT_f$ is locally compact, as the reader may easily
check.

\paragraph{Gluing two $\R$-trees.}
Next, given two functions~$f$, $g\in \CC$ with common interval of
definition $I=I(f)=I(g)$ both satisfying~\eqref{aslimit}, we define another pseudometric on~$I$ as the quotient
pseudometric (defined by~\eqref{eq:quotient}) 
\begin{equation}\label{eqDfg}
  D_{f,g}(s,t)=d_g/\{d_f=0\},
\end{equation}
and equip the quotient set ${M}_{f,g}=I/\{{D}_{f,g}=0\}$ with the
metric~${D}_{f,g}$. 
Note that $D_{f,g}:I^2\to \Rp$ is
continuous since
\begin{align*}
\big|D_{f,g}(s,t)-D_{f,g}(s',t')\big|&\leq D_{f,g}(s,s')+D_{f,g}(t,t')\\
	&\leq 2\omega(g;[s\wedge s',s\vee s'])+2\omega(g;[t\wedge t',t\vee t']).
\end{align*}
For this reason, the canonical projection $\bp_{f,g}:I\to M_{f,g}$ is
continuous. 
We may view $(M_{f,g},D_{f,g})$ as gluing the $\R$-tree $\cT_g$ along
the equivalence relation defining the $\R$-tree $\cT_f$. In fact,
since either of $d_f(s,t)=0$ or $d_g(s,t)=0$ implies $D_{f,g}=0$,
the canonical projection $\bp_{f,g}$ factorizes as
\[\bp_{f,g}=\pi_f\circ \bp_f=\pi_g\circ \bp_g,\]
where $\pi_f:\cT_f\to M_{f,g}$ and $\pi_g:\cT_g\to M_{f,g}$ are two
surjective maps. Note that these functions are continuous: if
$a_n=\bp_f(t_n)$ converges to some point $a$, then, up to taking
extractions (and using~\eqref{eq:26} if $I$ is unbounded), we may assume
that $t_n$ converges to some limit $t$, and then $\bp_f(t)=a$ by
continuity of $\bp_f$, while $\pi_f(a_n)=\bp_{f,g}(t_n)$ converges to
$\bp_{f,g}(t)=\pi_f(a)$.  
As a consequence, every geodesic segment $\llbracket a,b\rrbracket _f$ in~$\cT_f$, and every
geodesic $\llbracket c,d\rrbracket _g$ in $\cT_g$ is ``immersed'' into $M_{f,g}$ via the
mappings~$\pi_f$, $\pi_g$.

\subsubsection{Composite slices coded by two functions}\label{secslicefg}

\paragraph{Slice trajectory.}
We say that $(f,g)$ is a \emph{slice trajectory} if~$f$, $g\in \CC$ have
common interval of definition~$I$,
\begin{align}
&\forall s, t\in I, \qquad d_f(s,t)=0 \ \implies\ g(s)=g(t),\label{eqslicetraj}\\
\shortintertext{if $\inf I=-\infty$, then}
&\lim_{t\to-\infty}f(t)=+\infty \qquad\text{ and }\qquad\inf_{t\leq 0}g(t)=-\infty,\notag\\
\shortintertext{and, if $\sup I=\infty$, then}
&\inf_{t\geq 0}f(t)=-\infty \qquad\text{ and }\qquad\inf_{t\geq 0}g(t)=-\infty.\notag
\end{align}
In particular, $f$ and~$g$ both satisfy~\eqref{aslimit} in the case where~$I$ is noncompact, and the quantity $f(\inf I)\in \R\cup \{+\infty\}$ is always
well defined. 

\medskip
In the remainder of this section, we fix a slice trajectory $(f,g)$, and call the metric space
\[\Sl_{f,g}=(M_{f,g},D_{f,g})\]
the \emph{slice coded by $(f,g)$}. For the moment, we focus on deterministic considerations; the functions~$f$, $g$ will be randomized in the following section.

\paragraph{Marks and measures.}
The slice $\Sl_{f,g}$ naturally comes with the following distinguished elements.

\bigskip\underline{Geodesics sides.}\quad
For every $t\in I$, we set
\begin{align*}
\Gamma_t(r)&=\inf\{s\geq t:g(s)=g(t)-r\}&\text{ for $r\in \Rp$ such that } \inf_{\substack{s\geq t\\s\in I}}g(s)\leq g(t)-r\,; \\
\Xi_t(r)&=\sup\{s\leq t:g(s)=g(t)-r\}	&\text{ for $r\in \Rp$ such that } \inf_{\substack{s\leq t\\s\in I}}g(s)\leq g(t)-r.
\end{align*}
In particular, we have $d_g(\Gamma_t(r),\Xi_t(r))=0$ for every $t\in
I$ and every~$r$ satisfying both inequalities above.

We extend the definition given in Section~\ref{sec:geod-metr-spac} of geodesics to paths $\chi:[0,\infty)\to\X$ that satisfy~\eqref{eqgeod} for every $s$, $t\in\Rp$. In this case, the point~$\chi(0)$ is called the \emph{origin} of~$\chi$, its length is set to $\lgth(\chi)=\infty$ by convention, and the range of~$\chi$ is called a \emph{geodesic ray}. The geodesic ray uniquely determines the geodesic~$\chi$ by the same argument as for finite length, since the origin of a geodesic ray is the unique point~$a$ such that, for any~$s>0$, the number of points in the ray at distance~$s$ from~$a$ is one.

We observe that~$\Gamma_t$ and~$\Xi_t$ are
geodesics (possibly of infinite length) from~$t$ for the pseudometrics~$d_g$ and 
$D_{f,g}$, in the sense that, for every $r$, $r'$ such that $\Gamma_t(r)$
and $\Gamma_t(r')$ are defined,
\begin{equation}\label{eq:28}
  d_g(\Gamma_t(r),\Gamma_t(r'))=D_{f,g}(\Gamma_t(r),\Gamma_t(r'))=|r'-r|,
\end{equation}
and the same holds with $\Xi_t$ in place of $\Gamma_t$. This fact is
immediate for $d_g$ by definition. In fact, when $I$ is compact, one checks that the images of $\bp_g\circ \Gamma_t$ and $\bp_g\circ\Xi_t$ are the geodesic segments $\llbracket \bp_g(t),a_*(g\rst_{[t,\sup I]})\rrbracket_g$ and $\llbracket \bp_g(t),a_*(g\rst_{[\inf I,t]})\rrbracket_g$ in~$\cT_g$. 

For $D_{f,g}$, this fact follows from the
bound
\[|g(s)-g(s')|\leq D_{f,g}(s,s')\leq d_g(s,s'),\qquad s,s'\in I,\]
where the first inequality is an easy consequence of the fact that $(f,g)$
is a slice trajectory. 

Therefore, for $t\in I$, the paths defined by
\begin{align*}
  \gamma_t(r)&=\bp_{f,g}(\Gamma_t(r)),\qquad  0\leq
r\leq g(t)-\un{g}(t,\sup I),\ r\in\R,\\
\xi_t(r)&=\bp_{f,g}(\Xi_t(r)),
\qquad 0\leq
r\leq g(t)-\un{g}(\inf I,t),\ r\in\R,
\end{align*}
are two geodesics (possibly of infinite length) from $\bp_{f,g}(t)$, sharing a common
initial part. As mentioned in Section~\ref{sec:geod-metr-spac}, we will often identify these paths with
the pairs formed by their origins and image sets, the latter being\MEMS{}{ the projections} $\pi_g(\llbracket \bp_g(t),a_*(g\rst_{[t,\sup I]})\rrbracket_g)$  and $\pi_g(\llbracket \bp_g(t),a_*(g\rst_{[\inf I,t]})\rrbracket_g)$ when $I$ is compact. 

The slice $M_{f,g}$ comes with zero, one, or two geodesic sides. If $\inf I>-\infty$, then the geodesic $\gamma=\gamma_{\inf I}$ 
is called the \emph{maximal geodesic} of $M_{f,g}$, and, if $\sup I<\infty$,
the geodesic $\xi=\xi_{\sup I}$ is called the \emph{shuttle} of
$M_{f,g}$. 
If $\inf I=-\infty$ (resp.\ $\sup I=\infty$), we let
$\gamma_{-\infty}$ (resp.\ $\xi_\infty$) be the empty set. 
If $I$ is a bounded interval, then the paths~$\gamma_{\inf I}$ and~$\xi_{\sup I}$ have a common endpoint at the  \emph{apex} $x_*=\bp_{f,g}(s_*)=\pi_g(a_*(g))$,
where $s_*$ denotes any point~$s$ in $I$ such
that $g(s)=\inf_I g$.

\bigskip\underline{Base.}\quad
For $x\in \R$, we define 
\[T_x=\inf\{t\in I: f(t)=-x\}\in \R\cup \{\infty\},\] 
the
hitting time of level $-x$ by the function $f$, with the convention
that $\inf \varnothing=\infty$. Note that, for $x\in\R$, $T_x\ne -\infty$ because of the fact that $(f,g)$ is admissible. 
By convention, we also set
$T_\infty=-T_{-\infty}=\infty$. 
The \emph{base} of $\Sl_{f,g}$ is the set
\[\beta=\bp_{f,g}\left(\left\{T_x:- f(\inf I)\leq x\leq -\inf_I f\right\}\cap \R\right).\]
Note that the set inside brackets projects via $\bp_f$ to a geodesic
in~$\cT_f$. When~$I$ is compact, the base is the path $\pi_f(\llbracket\bp_f(T_{-f(\inf I)}),\bp_f(T_{-\inf_I f})\rrbracket)$, and in general, it is the increasing union of the paths
\[\pi_f(\llbracket \bp_f(T_x),\bp_f(T_y)\rrbracket _f),\qquad
-f(\inf I)\leq x<y\leq
-\inf_I f,\ x,y\in \R.\]

\bigskip\underline{Measures.}\quad
Finally, denoting by $\Leb_J$ the Lebesgue measure on the interval~$J$, the slice $\Sl_{f,g}$ is endowed with the following measures:
\begin{itemize}
	\item the \emph{area measure} $\mu=(\bp_{f,g})_*\Leb_I$\,;
	\item the \emph{base measure}~$\nu$, defined as the pushforward of $\Leb_{[- f(\inf I), -\inf_I f]\cap\R}$ by the mapping $x\mapsto\bp_{f,g}(T_x)$.
\end{itemize}

\paragraph{Gluing slices.}
In what follows, we will make a slight abuse of notation and identify
intervals of the form $[a,\infty]$, $[-\infty,a]$ for $a\in\R$ and
$[-\infty,\infty]$ with the intervals $[a,\infty)$, $(-\infty,a]$
and~$\R$, respectively. For $L$, $L'$ in the extended line 
$\R\cup\{\pm\infty\}$ such that $-f(\inf I)\leq L\leq L'\leq -\inf_I
f$, we define the restrictions $f^{(L,L')}$ and $g^{(L,L')}$ of~$f$
and~$g$ to the interval $[T_L,T_{L'}]\cap I$, yielding also a slice
trajectory. 
We may therefore define the slice coded by $(f^{(L,L')},g^{(L,L')})$
and denote it by 
\[\Sl^{(L,L')}=\big(M^{(L,L')},D^{(L,L')}\big)=\big(M_{f^{(L,L')},g^{(L,L')}},D_{f^{(L,L')},g^{(L,L')}}\big).\]
We let $\bp^{(L,L')}:[T_L,T_{L'}]\to M^{(L,L')}$ be the canonical
projection, $\gamma^{(L,L')}$, $\xi^{(L,L')}$, $\beta^{(L,L')}$ be the maximal
geodesic, shuttle, and base, and $\mu^{(L,L')}$, $\nu^{(L,L')}$ be the area and base measures of~$\Sl^{(L,L')}$. 

This family of metric spaces is compatible with the gluing operation
in the following sense, illustrated in Figure~\ref{figgluslices}. 

\begin{prp}\label{sgslice}
Let $-f(\inf I)\leq L< L'< L''\leq -\inf_I f$ be in the extended real line. Then 
\[\Sl^{(L,L'')}=G\big(\Sl^{(L,L')},\Sl^{(L',L'')};\xi^{(L,L')},\gamma^{(L',L'')}\big).\]
Moreover, the marks and measures satisfy
\begin{align*}
\gamma^{(L,L'')}&=\gamma^{(L,L')}\cup 
	\big(\gamma^{(L',L'')}\setminus\xi^{(L,L')}\big),\\
\xi^{(L,L'')}&=\xi^{(L',L'')}\cup
	\big(\xi^{(L,L')}\setminus\gamma^{(L',L'')}\big),\\
\beta^{(L,L'')}&=\beta^{(L,L')}\cup \beta^{(L',L'')},\\[2mm]
\mu^{(L,L'')}&=\mu^{(L,L')}+\mu^{(L',L'')},\\
\nu^{(L,L'')}&=\nu^{(L,L')}+\nu^{(L',L'')},
\end{align*}
with the convention that, in the right hand-side, sets and measures are identified with their images and pushforwards by the canonical projections in $\Sl^{(L,L'')}$  . 
\end{prp}

\begin{figure}[ht]
\centering\includegraphics[width=11cm]{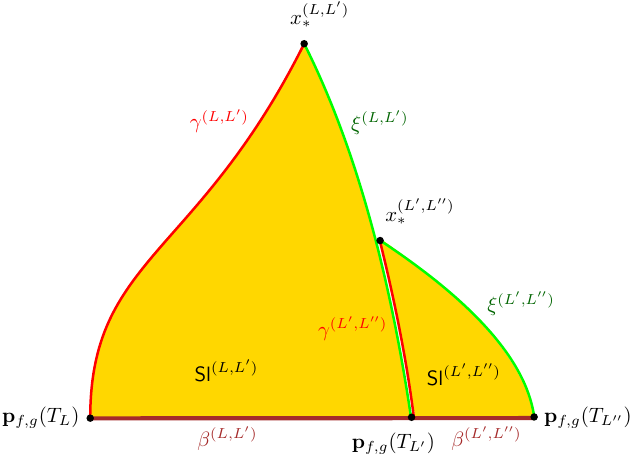}
\caption{Gluing slices encoded by a slice trajectory: the
gluing of $\Sl^{(L,L')}$ with $\Sl^{(L',L'')}$ results in
$\Sl^{(L,L'')}$. Here, $T_L>-\infty$ and
$T_{L''}<\infty$. We denoted by~$x_\sas^{(L',L'')}$ the apex
of~$\Sl^{(L',L'')}$ and~$x_\sas^{(L,L')}$ the apex
of~$\Sl^{(L,L')}$, which, on this example, is also the apex
of~$\Sl^{(L,L'')}$. Consequently, the shuttle
$\xi^{(L,L'')}$ is obtained by the union of $\xi^{(L',L'')}$
and the part of $\xi^{(L,L')}$ that is not glued to
$\gamma^{(L',L'')}$, whereas the maximal geodesic $\gamma^{(L,L'')}=\gamma^{(L,L')}$, as stated at the end of
Proposition~\ref{sgslice}. The bases and measures simply add up. The fact that the slices depicted
here are topological disks does not hold true in general; it
will, however, be the case for the random processes we will
consider in the upcoming sections.}
\label{figgluslices}
\end{figure}

\begin{proof}
In the disjoint union $[T_L,T_{L'}]\sqcup [T_{L'},T_{L''}]$, in order to avoid ambiguities due to the fact that the point
$T_{L'}$ belongs to both intervals (thus should be duplicated), we use a superscript~$0$ for points in the first interval and a superscript~$1$ for points in the second interval. We observe that $d_{g^{(L,L'')}}$ can be seen as a quotient pseudometric $d/R_1$
where~$d$ is the disjoint union pseudometric on $[T_L,T_{L'}]\sqcup [T_{L'},T_{L''}]$
given by $d(s,t)=d_g(s,t)$ if~$s$, $t$ belong to the same of the two intervals above and $d(s,t)=\infty$ otherwise, 
and $R_1$ is the coarsest equivalence
relation containing
\[\big\{\big(\Xi_{T_{L'}}(r)^0,\Gamma_{T_{L'}}(r)^1\big),\ 0\leq
r\leq g(T_{L'})-\un{g}(T_L,T_{L'})\vee \un{g}(T_{L'},T_{L''})\big\}.\]

Note also that, as~$T_{L'}$ is a hitting time, the equivalence relation $\{d_{f^{(L,L'')}}=0\}$ factorizes over these
two intervals, in the sense that if $d_{f^{(L,L'')}}(s,t)=0$ with $s\neq t$, then $s$, $t$ must
belong to the same interval $[T_L,T_{L'}]$ or $[T_{L'},T_{L''}]$. So
if~$R_2$ is the equivalence relation on the above disjoint union given by
$(s^i,t^j)\in R_2$ if and only if $d_f(s,t)=0$ and $i=j\in \{0,1\}$, using~\eqref{dR1R2},
we have
\[D^{(L,L'')}=(d/R_1)/R_2 = (d/R_2)/R_1 = \big(D^{(L,L')}\sqcup
D^{(L',L'')}\big)/R_1,\]
which is precisely the quotient metric of \MEMS{}{the gluing }
$G(\Sl^{(L,L')},\Sl^{(L',L'')};\xi^{(L,L')},\gamma^{(L',L'')})$. 

Checking the claimed formulas for the marks and measures of~$\Sl^{(L,L'')}$ is
straightforward.  
\end{proof}

We finish this paragraph with a very strong identity, saying that the distances in a slice $\Sl^{(L,L')}$ encoded by a restriction of the slice trajectory $(f,g)$ are in fact the restrictions of the distances in the ``whole'' slice $\Sl_{f,g}$.

\begin{crl}
  \label{sec:metric-spaces-coded-2}
  Let $(f,g)$ be a slice trajectory on the interval $I$,
  and $-f(\inf I)\leq L\leq  L'\leq -\inf_I f$. Then $D^{(L,L')}$ is
  the restriction of the function $D_{f,g}$ to $[T_L,T_{L'}]$. 
\end{crl}

\begin{proof}
This is a direct consequence of the preceding
proposition, which entails that $D_{f,g}$ is the pseudometric 
obtained by gluing $\Sl^{(L,L')}$ with $\Sl^{(L',\sup I)}$
along $\xi^{(L,L')}$ and $\gamma^{(L',\sup I)}$, and then by gluing
the resulting space $\Sl^{(L,\sup I)}$ with $\Sl^{(\inf I,L)}$
along $\gamma^{(L,\sup I)}$ and $\xi^{(\inf I,L)}$. Since at each
stage, the spaces that are glued together are isometrically embedded
in the resulting gluing, we obtain that $\Sl^{(L,L')}$ is
isometrically embedded in $\Sl^{(\inf I,\sup I)}=\Sl_{f,g}$. 
\end{proof}

\subsection{Random continuum composite slices}\label{sec:boundary-triangles}

We now randomize the functions $f$, $g$ considered in the preceding
section in various ways to construct random spaces of interest. For a fixed 
continuous function $f\in \CC$ with $0\in I(f)$, the \emph{snake}\footnote{Literally, 
  this is rather called the ``head of the snake driven by $f$''; see~\cite{legall99}.} driven by~$f$ is a random centered Gaussian process 
$(Z^f_t,t\in I(f))$ with $Z^f_0=0$ and with covariance function 
specified by
\begin{equation}\label{covZ}
\E\big[(Z^f_t-Z^f_s)^2\big]={d}_f(s,t),\qquad s,t\in I(f).
\end{equation}
As soon as~$f$ is Hölder continuous, 
which will always be the case in this paper, 
this process admits a continuous modification; we systematically 
consider this continuous modification of~$Z^f$. If 
now~$Y$ is a (a.s.\ Hölder continuous) random function, then the random snake 
driven by~$Y$ is defined as the Gaussian process $Z^Y$ conditionally given~$Y$. 

By~\eqref{covZ}, it holds that $Z_s^f=Z_t^f$ whenever $d_f(s,t)=0$, so that, provided~$f$ satisfies the required limit conditions if~$I(f)$ is noncompact, the pair $(f,Z^f)$ is a slice trajectory. In what follows, we will let $(X,W):(f,g)\mapsto (f,g)$ be the canonical 
process on~$\CC^2$. 

Below and throughout this work, we use, for any process~$Y$ defined on an interval~$I$, the piece of notation $\un{Y}_t=\inf_{s\le t,\, s\in I} Y_s$.

\medskip
Let us proceed to the definition of continuum slices, which arise in Theorem~\ref{thmslslice}. Fix~$A$, $L \in (0,\infty)$ and $\Delta\in \R$. We let
$\Slice_{A,L,\Delta}$ be the probability distribution under which 
\begin{itemize}
\item 
the process~$X$ is a first-passage bridge\footnote{A \emph{first-passage bridge of Brownian motion} can be defined from Brownian motion stopped
when first hitting $-L$ by absolute continuity; see \cite{BeMi17}.} of standard
Brownian motion from~$0$ to~$-L$ with duration~$A$;
\item
conditionally given~$X$, the process~$W$ has the same law as $(Z_t+\zeta_{-\un{X}_t},0\leq t\leq A)$, where~$Z$ is the
random snake driven by $X-\un{X}$, and $\zeta/\sqrt{3}$ is a standard Brownian bridge of duration~$L$ and terminal value $\Delta/\sqrt{3}$, independent of~$X$ and~$Z$. 
\end{itemize}

To be more precise, the process $\zeta=(\zeta_t,0\le t\le L)$ is
Gaussian, with $\E[\zeta_t]=t\Delta/L$ for $t\in [0,L]$ and 
\[\cov(\zeta_s,\zeta_t)=3\,\frac{s(L-t)}{L},\qquad 0\le s\leq
t\le L.\]
With this definition, it is simple to see that $\Slice_{A,L,\Delta}$
is indeed carried by slice trajectories defined on the
interval $I=[0,A]$.

\begin{defn}
  \label{sec:boundary-triangles-1}
The \emph{(composite) slice with area~$A$, width~$L$ and tilt~$\Delta$}, generically denoted by $\Sl_{A,L,\Delta}$,
is the $5$-marked\footnote{Recall from Section~\ref{sec:geod-metr-spac} that each geodesic mark counts for 2 marks, the first one being its marked extremity.} $2$-measured metric space $\Sl_{X,W}$ under\MEMS{ the law}{} $\Slice_{A,L,\Delta}$,
endowed with the marking
\[\partial\Sl_{A,L,\Delta}=(\beta,\gamma_0,\xi_A)\]
comprising its base and two geodesic marks, namely its maximal geodesic and its shuttle, as well as its area and base measures~$\mu$, $\nu$. 
\end{defn}

The piece of notation $\partial\Sl_{A,L,\Delta}$ for the marking comes from the fact that the union of the three marks gives the topological boundary of~$\Sl_{A,L,\Delta}$, as stated in Lemma~\ref{topdimep}.

\subsection{The Brownian half-plane, and its embedded slices}\label{secBHP}

There is a natural relation between slices and the
Brownian half-plane \cite{GwMi17,BaMiRa}, which we now introduce. Let $(B_t,t\ge 0)$, $(B'_t,t\ge 0)$ be
two independent standard Brownian motions, and let
$(\Pi_t=B'_t-2\inf_{\{0\leq s\leq t\}}B_s',t\ge 0)$ be the so-called \emph{Pitman
transform} of~$B'$, which is a three-dimensional Bessel process. Recall the piece of notation $\un{X}_t=\inf_{s\le t} X_s$.
We let $\Half$ be the probability distribution on~$\CC^2$ under which
\begin{itemize}
\item the process $X$ has same distribution as $(B_t\ind_{\{t\geq 0\}}+\Pi_{-t}\ind_{\{t<0\}},t\in \R)$, and
\item conditionally given $X$, the process $W$ has same distribution
as $(Z_t+\zeta_{-\un{X}_t},t\in \R)$, where~$Z$ is the
random snake driven by $X-\un{X}$, and $\zeta/\sqrt{3}$ is a two-sided
standard Brownian motion\footnote{This means that $(\zeta_x/\sqrt{3},x\geq 0)$ and $(\zeta_{-x}/\sqrt{3},x\geq 0)$ are independent
(one-dimensional) standard Brownian motions issued from~$0$.}, independent of~$X$ and~$Z$.
\end{itemize}

The measure $\Half$ is carried by slice trajectories defined on the interval~$\R$. 
We note that we can make this definition more symmetric using standard
excursion theory, in a way similar to the encoding triples of~\cite{LGRi21}. For this, we let $T_{L-}=\lim_{L' \uparrow L} T_{L'}$ and denote by 
\begin{align*}
X^{(L)}&=(L+X_{T_{L-}+t},0\leq t\leq T_L-T_{L-}),\\
W^{(L)}&=(W_{T_{L-}+t},0\leq t\leq T_L-T_{L-}),
\end{align*}
the excursion of~$X$ above its past infimum at level~$-L$, and the corresponding piece of~$W$. Note first that the process $\zeta_L=W_{T_L}$, $L\in \R$, is under $\Half$ a standard two-sided Brownian motion multiplied by~$\sqrt{3}$. Then, conditionally given~$\zeta$, 
the point measure on $\R\times \CC\times \CC$ given by 
\begin{equation}
  \label{eq:27}
  \mathcal{M}(\d L\, \d X\, \d W)=\sum_{L\in \R:\,T_L\neq
    T_{L-}}\delta_{(L,X^{(L)},W^{(L)})}
\end{equation}
is a
Poisson measure with intensity $2\d L\,  \N_{\zeta_L}(\d (X,W))$, where
$\N_x$ is the $\sigma$-finite ``law'' of the lifetime process and head of the
Brownian snake (started at~$x$) driven by the Itô measure of positive excursions of
Brownian motion. The process $(X,W)$ is then a measurable function of~$\zeta$ and~$\mathcal{M}$ by Itô's reconstruction theory of Brownian
motion from its excursions.

\begin{defn}\label{sec:bound-triangl-brown-3}
The \emph{Brownian half-plane}, generically denoted by~$\BHP$, is the $1$-marked $2$-measured metric space $\Sl_{X,W}$
considered under~$\Half$, endowed with the \MEMS{one }{}mark $\partial\BHP=\beta$, its area measure~$\mu$ and its base measure~$\nu$. 
\end{defn}

There is only one mark here, the base; there is no maximal geodesic nor shuttle since the interval of definition is~$\R$. The name comes from the fact that~$\BHP$ is homeomorphic to the
half-plane $\R\times \Rp$, its boundary as a topological manifold being equal to
the base; see~\cite[Corollary~3.8]{BaMiRa}. 

In the light of Proposition~\ref{sgslice}, the Brownian half-plane can
be seen to have a natural Markov property. First, let $\theta_t:f\mapsto f(t+\cdot)-f(t)$ be the translation
operator on~$\CC$.
We claim 
that
$\Half$ is invariant under $\theta_{T_L}$, since its action simply consists
in translating by~$L$  the time in process $\zeta$, and the first
coordinate of $\mathcal{M}$, which leaves their laws invariant. For
similar reasons, for every $L\in \R$, the processes
 $\big(X^{(0,L)},W^{(0,L)}\big)$, 
$(X^{(-\infty,0)},W^{(-\infty,0)})$ and  
$(\theta_{T_L}X^{(L,+\infty)},\theta_{T_L}W^{(L,+\infty)})$ are
independent under $\Half$, since they are respectively functionals of the independent random elements 
\begin{multline*}
(\zeta_{x},0\leq x\leq L), \mathcal{M}((0,L]\times \CC\times \CC),
\qquad (\zeta_{x},x\leq 0),
	\mathcal{M}((-\infty,0]\times\CC\times \CC),\\
\text{and }\quad (\zeta_{L+x}-\zeta_L,x\geq 0),
	\mathcal{M}((L,\infty)\times\CC\times \CC).
\end{multline*}

\paragraph{Free slices.}
Note that, under $\Half$, the process
$X^{(0,L)}$ is simply a standard Brownian motion killed
at its first hitting time of~$-L$, while the process
$(W^{(0,L)}_{T_x}/\sqrt{3},0\leq x\leq L)$ is a standard
Brownian motion killed at time~$L$. For this reason, the law of
$(X^{(0,L)},W^{(0,L)})$ under $\Half$ is the mixture
\begin{equation}
\label{eqFSlice}
\FSlice_L=\int_0^\infty q_L(A)\d A\int_\R p_{3L}(\Delta)\d
\Delta \, \Slice_{A,L,\Delta},
\end{equation}
where~$p_t$, $q_x$ are defined after~\eqref{eqlimvect}. In what follows, a random metric
space with same law as $\Sl^{(0,L)}$ under $\FSlice_L$ will be referred to as a \emph{free (composite) slice} of width~$L$. 
This, together with Proposition~\ref{sgslice}, yields the
following result. 

\begin{prp}\label{sec:brownian-half-plane}
Fix $L<L'<L''$ in the extended line. Then, under $\Half$, it holds that 
$\Sl^{(L,L'')}=G(\Sl^{(L,L')},\Sl^{(L',L'')};\xi^{(L,L')},\gamma^{(L',L'')})$,
where the spaces $\Sl^{(L,L')}$, $\Sl^{(L',L'')}$ are
independent. Moreover, if~$L$ and~$L'$ are finite, then $\Sl^{(L,L')}$
is a free slice of width $L'-L$. 
\end{prp}

Recall that this result is illustrated in Figure~\ref{figgluslices}, which can be completed by extending the brown segment into a line, letting the half-plane above be~$\BHP$, the line being its base $\beta=\beta^{(-\infty,\infty)}$. This also
suggests that $\Sl^{(L,L')}$ is the bounded connected component of the
complement of $\gamma^{(L,L')}\cup \xi^{(L,L')}$ in $\BHP$. More
precisely, the following holds.

\begin{prp}\label{sec:brownian-half-plane-1}
For every $L<L'$ in $\R$, almost surely under $\Half$, the geodesics
$\gamma^{(L,L')}$ and $\xi^{(L,L')}$ meet only at the apex
$x_\sas^{(L,L')}$, and meet the base $\beta$ only at their
respective origins $\bp_{X,W}(T_L)$ and $\bp_{X,W}(T_{L'})$. Moreover, $\Sl^{(L,L')}$
is the closure of the bounded connected component of the complement
of the union of these two paths in~$\BHP$. It is therefore
homeomorphic to the closed unit disk, with boundary given by the union of the 
three sets $\beta^{(L,L')}$, $\gamma^{(L,L')}$ and $\xi^{(L,L')}$,
which meet only at $\bp_{X,W}(T_L)$, $\bp_{X,W}(T_{L'})$ and $x_\sas^{(L,L')}$. 
\end{prp}

\begin{proof}
This proposition is proved in the same way as Lemma~6.15
in~\cite{BaMiRa}. Let us recall briefly the ideas. For any point
$t\in \R$, we let
\[\Sigma_t(r)=\inf\{s\geq t:X_s=X_t-r\}\qquad\text{ for }\qquad 0\leq
r\leq X_t-\un{X}_t,\]
so that the range of $\bp_X\circ\Sigma_t$ is the geodesic path
$\llbracket \bp_X(t),\bp_X(T_{-\un{X}_t}) \rrbracket_X$ in $\cT_X$. 
Its image by~$\pi_X$
defines a path~$\sigma_t$ starting at $\bp_{X,W}(t)$ and ending at the point
$\bp_{X,W}(T_{-\un{X}_t})$ of the base. Moreover, almost surely, any
path~$\sigma_t$, $t\in \R$, do not intersect a geodesic~$\gamma_s$,
$s\in \R$, except possibly at the starting point of either~$\bp_{X,W}(s)$ or~$\bp_{X,W}(t)$. This implies that any point $\bp_{X,W}(t)$ of $\Sl^{(L,L')}$ that is not
in the union $\gamma^{(L,L')}\cup \xi^{(L,L')}$ can be linked to the
bounded segment $\beta^{(L,L')}$ of the base of~$\BHP$
by the path~$\sigma_t$ without intersecting
$\gamma^{(L,L')}\cup \xi^{(L,L')}$ except perhaps at its 
endpoint. This latest possibility can be discarded by noting that, with
probability $1$, we have $T_{-\un{X}_t}\notin \{T_L,T_{L'}\}$.  
Similarly, a point $\bp_{X,W}(t)$ of $\BHP$ outside of $\Sl^{(L,L')}$ is linked to the unbounded set $\beta\setminus \beta^{(L,L')}$ of the
base of $\BHP$ by the path~$\sigma_t$, which does not intersect
$\gamma^{(L,L')}\cup \xi^{(L,L')}$. This means that $\Sl^{(L,L')}$ is
the closure of the bounded connected component of $\BHP$ minus
$\gamma^{(L,L')}\cup \xi^{(L,L')}$. 
\end{proof}

The above discussion shows that the Brownian half-plane contains a
natural ``flow'' of free slices. We can also
link directly the slices of Section~\ref{sec:boundary-triangles} with the Brownian half-plane via an
absolute continuity argument. Recall the definitions of~$p_t$ and~$q_x$ after~\eqref{eqFSlice}.

\begin{lmm}\label{lemaccontslice}
Fix $0< K<L$, as well as $A>0$ and $\Delta\in \R$. For every nonnegative function~$G$ that is measurable with respect to the $\sigma$-algebra generated by $(X^{(0,K)},W^{(0,K)})$, we have 
\begin{equation*}
	\Slice_{A,L,\Delta}[G]=\Half\big[\varphi_{A,L,\Delta}\big(T_{K},K,W_{T_{K}}\big)\cdot
	G\big],
\end{equation*}
where
\begin{equation}\label{eq:21}
	\varphi_{A,L,\Delta}(A',L',\Delta')=\frac{q_{L-L'}(A-A')}{q_L(A)}\frac{p_{3(L-L')}(\Delta-\Delta')}{p_{3L}(\Delta)}.
\end{equation}
\end{lmm}

\begin{proof}
This comes from similar statements for Brownian bridges and first-passage bridges; see for instance \cite[Equations~(18) and~(19)]{bettinelli10}. For bounded measurable functions~$f$, $g$ on~$\CC$, for $0<A'<A$ and $0<K<L$,
\begin{multline}\label{SliceHalf}
\Slice_{A,L,\Delta}\Big[f\big(X\rst_{[0,A']}\big)\cdot g\big(\zeta\rst_{[0,K]}\big)\Big]\\
	=\Half\Bigg[f\big(X\rst_{[0,A']}\big)\frac{q_{L-X_{A'}}(A-A')}{q_L(A)}\ind_{\{\un X_{A'}>-L\}}\cdot g\big(\zeta\rst_{[0,K]}\big)\frac{p_{3(L-K)}(\Delta-\zeta_K)}{p_{3L}(\Delta)}\Bigg].
\end{multline}
Here, the factor~$3$ in the index of the Gaussian density
function comes from the fact that $\zeta/\sqrt{3}$ is a bridge
of standard Brownian motion.  
We replace $A'$ with~$T_K$ by a standard argument, writing
\[
f\big(X^{(0,K)}\big) = \lim_{n\to\infty} \sum_{i\ge 0} \ind_{\{(i-1)2^{-n}<T_K\le i\, 2^{-n}\}}~f\big(X\rst_{[0,i\, 2^{-n}]}\big),
\]
using dominated convergence and applying the above equality~\eqref{SliceHalf} to $A'=i\, 2^{-n}$, noting that $\ind_{\{(i-1)2^{-n}<T_K\le i\, 2^{-n}\}}~f\big(X\rst_{[0,i\, 2^{-n}]}\big)$ is a function of~$X\rst_{[0,i\, 2^{-n}]}$.

The result follows by noting that $W^{(0,K)}$ is built in the same way from $X^{(0,K)}$ and~$\zeta\rst_{[0,K]}$ under $\Slice_{A,L,\Delta}$ as from $X^{(0,K)}$ and~$\zeta\rst_{[0,K]}$ under $\Half$.
\end{proof}

We may now prove the statement about the topology and Hausdorff dimension of a slice.

\begin{proof}[Proof of Lemma~\ref{topdimep} for slices]
First, almost surely, the Brownian half-plane is homeomorphic to the half-plane~\cite[Corollary~3.8]{BaMiRa}, is locally of Hausdorff dimension~$4$ and its boundary is locally of Hausdorff dimension~$2$. The latter facts are obtained from similar statements for Brownian disks~\cite{bettinelli11b} thanks to~\cite[Theorem~3.7]{BaMiRa} allowing to couple arbitrary balls of~$\BHP$ centered at the root~$\bp_{X,W}(0)$ with balls of large enough Brownian disks, centered at a point on the boundary.

\MEMS{Hence}{Consequently}, under the probability distribution $\Half$, for any
$L< L'$, by Lemma~\ref{sec:metric-spaces-coded-2}, the metric space
$\Sl^{(L,L')}$ is a.s.\ locally
of Hausdorff dimension~$4$ and its base~$\beta^{(L,L')}$ is
locally of Hausdorff dimension~$2$. Furthermore, it is homeomorphic to
the disk by Proposition~\ref{sec:brownian-half-plane-1} and its boundary is the union of its three marks~$\beta^{(L,L')}$, $\gamma^{(L,L')}$ and $\xi^{(L,L')}$, whose pairwise intersections are identified singletons. 

Now, arguing under $\Slice_{A,L,\Delta}$, we use the fact from
Proposition~\ref{sgslice} that
$\Sl^{(0,L)}=G(\Sl^{(0,L/2)},\Sl^{(L/2,L)};\xi^{(0,L/2)},\gamma^{(L/2,L)})$. 
Lemma~\ref{lemaccontslice} entails that, almost surely, under this
probability distribution, the law of $\Sl^{(0,L/2)}$ is absolutely
continuous with respect to that of the same random variable under $\Half$, and so
is homeomorphic to a disk. Now, we observe that, under
$\Slice_{A,L,\Delta}$, the process $\theta_{T_{L/2}}(X^{(L/2,L)},W^{(L/2,L)})$ has
same distribution as $(X^{(0,L/2)},W^{(0,L/2)})$, which we leave as an
exercise to the reader. Therefore, under this law, $\Sl^{(L/2,L)}$ has
same distribution as $\Sl^{(0,L/2)}$ and both are homeomorphic to a
disk. 
We conclude that the same is true for
$\Sl^{(0,L)}$ since it is obtained by gluing two topological disks
along two segments of
their boundaries. The identification of the marks given in
Proposition~\ref{sgslice} easily yields the desired property on the
marks of~$\Sl^{(0,L)}$. The facts on the local Hausdorff dimension are
obtained similarly.  
\end{proof}

\subsection{The uniform infinite half-planar quadrangulation}\label{secUIHPQ}

\paragraph{The UIHPQ.}
We now define a slight variant of the classical UIHPQ \cite{CuMiUIHPQ,CaCuUIHPQ,BaMiRa,BaRiUIHPQs}, the
half-planar version of the uniform infinite random planar
quadrangulation, in the following way. Let $F_\infty=(\bT^k,k\in \Z)$ be a two-sided sequence
of independent Bienaym\'e--Galton--Watson trees with a geometric offspring
distribution of parameter~$1/2$. Conditionally on~$F_\infty$, we let
$\lambda^0_\infty$ be a uniformly chosen well labeling function, meaning that every tree~$\bT^k$ is assigned a well labeling function giving label~$0$ to its root vertex, independently, uniformly at
random. Lastly, and independently of~$F_\infty$
and~$\lambda^0_\infty$, we let $(b_k,k\in \Z)$ be a doubly-infinite
walk with
shifted geometric steps, meaning that $b_0=0$ a.s., and that
$b_k-b_{k-1}$, $k\in \Z$, are independent and identically distributed
random variables with 
$\P(b_1=r)=2^{-r-2}$ for every $r\in \{-1,0,1,2,\ldots\}$. For
a vertex $v\in \bT^k$ we let $\lambda_\infty(v)=b_k+\lambda^0_\infty(v)$,
and call $(F_\infty,\lambda_\infty)$ the \emph{infinite random well-labeled forest}. We then embed~$F_\infty$ in the plane in such a way that all
trees are contained in the upper half-plane, and the root~$\rho^k$ of~$\bT^k$ is located at the point $(k,0)\in \R^2$. We also link
consecutive roots~$\rho^k$, $\rho^{k+1}$ by a line segment. 
We then let
$(c_i,i\in \Z)$ be the sequence of corners of the upper half-plane part of the resulting map, in
contour order from left to right, with origin the
first corner~$c_0$ incident to~$\rho^0$. The \emph{uniform infinite half-planar quadrangulation} (\emph{UIHPQ} for short) is then the infinite map~$Q_\infty$ obtained by applying the CVS construction to $(F_\infty,\lambda_\infty)$, that is, by linking every corner to its successor as defined in Section~\ref{sec:basic-construction}, and removing all edges of the forest afterward. The root of~$Q_\infty$ is defined as the corner preceding the arc from~$c_0$ to its successor. Note that, in this case, there is no need to add an extra vertex with a corner~$c_\infty$.

\begin{rem}
The difference between this definition of the UIHPQ and the one appearing in the mentioned references is a slight rooting bias. Indeed, the simplest way to obtain the usual definition is to consider a two-sided simple random walk $(z_i,i\in \Z)$ and construct the sequence $(b_k,k\in \Z)$ from it as follows. Let $S^\downarrow=\{i\in\Z\,:\, z_{i+1}-z_i=-1\}$ be the set of descending steps of $(z_i,i\in \Z)$ and $i_0=\sup (S^\downarrow\cap\Zm)$ the index of the descending step preceding~$0$. Then we define the sequence $(b_k,k\in \Z)$ by reindexing $(z_i-z_{i_0},i\in S^\downarrow)$ with~$\Z$ in such a way that~$i_0$ corresponds to the index~$0$. The UIHPQ is then constructed as above with this bridge but rooted at the corner preceding the arc linking $s^{-i_0}(c_0)$ to its successor $s^{-i_0+1}(c_0)$ instead of the convention we presented. Apart from this slight root shift, the resulting law of $(b_k,k\in \Z)$ is not exactly that of a doubly-infinite bridge with shifted geometric steps. The first step gets a size-biased distribution $\P(b_1=r)=(r+2)2^{-r-3}$, $r\ge -1$, whereas all other steps get the desired shifted geometric distribution. See the discussion in~\cite[Section~4.5.2]{BaMiRa} for more information.

The construction we use here has the advantage of making the law of the slices invariant by translation. 
\end{rem}

\paragraph{Convergence toward the Brownian half-plane.}
Denoting by~$v_i$ the vertex of~$F_\infty$ incident to~$c_i$ and by~$\Upsilon(i)\in\Z$ the index of the tree to which~$v_i$ belongs, we define the \emph{contour} and \emph{label processes} on~$\R$ by
\[C(i) = d_{\bT^{\Upsilon(i)}}\big(v_i,\rho^{\Upsilon(i)}\big)-\Upsilon(i)\qquad\text{ and }\quad\Lambda(i)=\lambda_\infty(v_i),\qquad i \in\Z,\]
and by linear interpolation between integer values; see Figure~\ref{contour}. As is well known, the part of the contour process
corresponding to~$\bT^k$ (counting the edge linking~$\rho^k$ to~$\rho^{k+1}$) has the same distribution as a simple random walk started at~$-k$ and killed when
first hitting~$-k-1$. Finally, for $k\ge 1$, we denote by~$\tau_k$ the hitting time of~$-k$ by~$C$; its value is thus also equal to~$k$ plus twice the number of edges in the first~$k$ trees~$\bT^0$, \dots, $\bT^{k-1}$.

\begin{figure}[htb!]
	\centering\includegraphics[width=.95\linewidth]{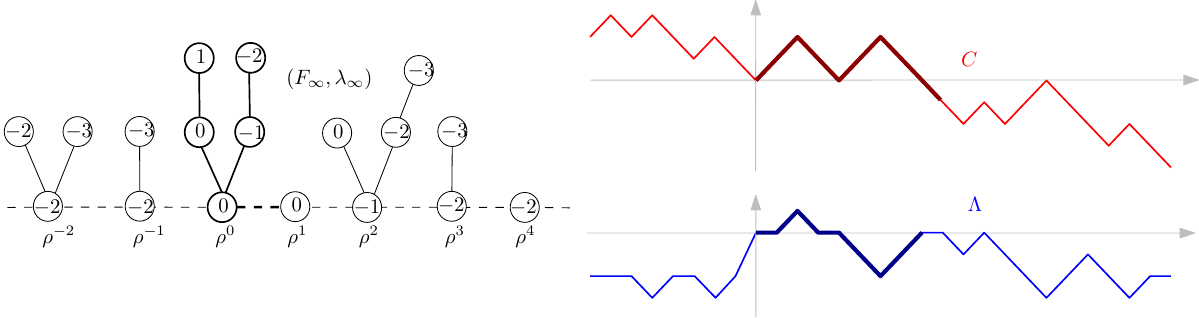}
	\caption{Contour and label processes associated with $(F_\infty,\lambda_\infty)$. The edges of the floor are represented with dashed lines. The tree~$\bT^0$ and the corresponding part of the encoding processes are highlighted (the corresponding floor edge and the final descending step of the contour function are also highlighted). For instance, $\tau_3=17$ on this example. The contour process can be thought of as recording the height of a particle moving at speed one around the forest. In this point of view, the root~$\rho^k$ should be at height~$-k$ for each $k\in\Z$; this can be achieved for instance by vertically translating each tree in such a way that~$\rho^k$ is mapped to location $(k,-k)$ instead of $(k,0)$.}
	\label{contour}
\end{figure}

As the vertices of the encoding objects are preserved through the
CVS bijection, the vertex~$v_i$ can also be seen as a vertex of~$Q_\infty$. Let us define
\begin{align*}
 D_\infty(i,j)&=d_{Q_\infty}(v_i,v_j),\qquad i,j\in \Z.
\end{align*}
We extend~${D}_\infty$ to a continuous function on~$\R^2$ by ``bilinear interpolation,'' writing
$\{s\}=s-\lfloor s\rfloor$ for the fractional part of $s$ and then setting
\begin{multline}\label{bilininterpo}
    D_\infty(s,t)=(1-\{s\}) (1-\{t\}) D_\infty(\lfloor s\rfloor,\lfloor
    t\rfloor)+\{s\}(1-\{t\}) D_\infty(\lfloor s\rfloor+1,\lfloor t\rfloor) \\
     + (1-\{s\})\{t\}D_\infty(\lfloor s\rfloor,\lfloor
    t\rfloor+1)+\{s\}\{t\} D_\infty(\lfloor s\rfloor+1,\lfloor t\rfloor+1).
\end{multline}

We then define the renormalized versions of $C$, $\Lambda$, and~$D_\infty$\,: for every $s$, $t\in\R$, we set
\begin{equation}\label{rescale}
C_{(n)}(s)=\frac{C(2ns)}{\sqrt{2n}},\qquad
\Lambda_{(n)}(s)=\frac {\Lambda(2ns)} {(8n/9)^{1/4}},\qquad
D_{(n)}(s,t)=\frac{D_\infty(2ns, 2nt)}{(8n/9)^{1/4}}.
\end{equation}

The next result can be seen as a reformulation of~\cite[Theorem~1.11]{GwMi17} or~\cite[Theorem~3.6]{BaMiRa}, proving the
convergence of the UIHPQ to the Brownian half-plane defined in
Section~\ref{secBHP}.

\begin{prp}\label{cvencbhp}
On $\CC\times\CC\times\CC^{(2)}$, it holds that
\begin{equation}
  \label{eqencbhp}
  \big(C_{(n)},\Lambda_{(n)},D_{(n)}\big)\tod\big(X,W,D_{X,W}\big),
\end{equation}
where the limiting triple is understood under~$\Half$.
\end{prp}

This statement does not appear in this exact form in the aforementioned
references, which do not explicitly focus on the processes
$C_{(n)}$, $\Lambda_{(n)}$, $X$, $W$. In \cite[Remark~6.16]{BaMiRa}, it was
however mentioned how to extend the results therein in order to take into
account these processes, so we will follow the line of reasoning
sketched in that work. 

\begin{proof}
The proof proceeds via  established convergence results
for random quadrangulations with one external face to Brownian disks. Fix some number $K>0$. We will
sample a quadrangulation with one external face, whose areas and perimeters
are so large that, in a neighborhood of~$0$ of amplitude~$K$, this
rescaled large quadrangulation and its limit, a Brownian disk of large area and
perimeter, are indistinguishable from the rescaled
UIHPQ and the Brownian half-plane, in a sense to be made precise. In the following, we will use for all the objects related to the quadrangulation with one external face or the limiting Brownian disk a similar notation as for those related to the UIHPQ or the Brownian half-plane, only with a superscript prime symbol~$\prime$.

Fix $L>0$, which should be thought of as being large. For $n\geq 1$, we sample the aforementioned quadrangulation~$Q'_n$ with one external face as follows. First, consider a uniform random element $(M'_n,\lambda'_n)$
of $\RbM^{[0]}_{a_n,(l_n)}$, where $a_n=\lfloor n L\rfloor$ and  
$l_n=\lfloor L \sqrt {2n}\rfloor$.
We can view this as a labeled forest $(F'_n,\lambda'_n)$
with $l_n$ trees arranged in a circle, and rooted at
$\rho^0$, \ldots, $\rho^{l_n-1}$ where~$\rho^0$ is the root of the tree
containing the root corner of $f_*$. We let $C'_n$, $\Lambda'_n$ be the
contour and label process of this forest, defined as above, starting from the tree
rooted at~$\rho^0$. We let $Q'_n=\CVS(M'_n,\lambda'_n;f_*)$ be
the rooted quadrangulation encoded by $(M'_n,\lambda'_n)$, and we let
$D'_n(i,j)=d_{Q_n'}(v'_i,v'_j)$ for $0\leq i,j\leq 2a_n+l_n$, where $v'_i$
is the $i$-th visited vertex of~$F_n'$ in contour order, viewed as a
vertex of $Q'_n$. As usual, we extend $D'_n$ into a continuous
function on $[0,2a_n+l_n]^2$. Finally, we extend the definition of
these processes to the interval $[-2a_n-l_n,2a_n+l_n]$ by the simple
translation formulas
\MEMS{
\begin{multline}\label{eq:2}
C'_n(t)=C'_n(t+2a_n+l_n)+l_n,\quad 
\Lambda'_n(t)=\Lambda'_n(t+2a_n+l_n),\\ t\in [-2a_n-l_n,0],
\end{multline}
and
\begin{multline}\label{eq:3}
D'_n(s,t)=D'_n(s+(2a_n+l_n)\ind_{\{s<0\}},
t+(2a_n+l_n)\ind_{\{t<0\}}),\\ s,t\in [-2a_n-l_n,2a_n+l_n].
\end{multline}
}{
\begin{equation}\label{eq:2}
C'_n(t)=C'_n(t+2a_n+l_n)+l_n,\quad
\Lambda'_n(t)=\Lambda'_n(t+2a_n+l_n),\quad t\in [-2a_n-l_n,0],
\end{equation}
and
\begin{equation}\label{eq:3}
D'_n(s,t)=D'_n(s+(2a_n+l_n)\ind_{\{s<0\}},
t+(2a_n+l_n)\ind_{\{t<0\}}),\quad s,t\in [-2a_n-l_n,2a_n+l_n].
\end{equation}
}
The idea behind this extension is that we are going to consider these
processes in neighborhoods of $0$, so that
we are really interested in the behavior of these processes when the
argument is close from $0$ or from $2a_n+l_n$. 

Define their rescaled versions: for $s$, $t\in [-2a_n-l_n,2a_n+l_n]$,
\begin{equation}
  \label{rescaleprime}
  C'_{(n)}(s)=\frac{C'_n(2ns)}{\sqrt{2n}},\qquad
\Lambda'_{(n)}(s)=\frac {\Lambda'_n(2ns)} {(8n/9)^{1/4}},\qquad
D'_{(n)}(s,t)=\frac{D'_n(2ns, 2nt)}{(8n/9)^{1/4}}.
\end{equation}
Then by \cite[Equation~(26) and Theorem~20]{BeMi17}, one has
the joint convergence 
\begin{equation}
  \label{eq:5}
  \big(C'_{(n)},\Lambda'_{(n)},D'_{(n)}\big)\tod(X',W',D')
\end{equation}
in distribution in $\CC([0,L])\times \CC([0,L])\times
\CC([0,L]^2)$, where $(X',W',D')$ is an explicit limiting
process, which is the encoding process of the Brownian disk of area
$L$ and width~$\sqrt{L}$.
In particular, the process $D'$ is
a measurable function of the pair $(X',W')$. Due to the formulas in~\eqref{eq:2} and~\eqref{eq:3}, this\MEMS{}{ easily} implies the convergence of these processes
on $\CC([-L,L])\times \CC([-L,L])\times 
\CC([-L,L]^2)$, where $(X',W',D')$ are extended to functions on
$[-L,L]$ or $[-L,L]^2$ in a similar way as above. Note that we choose to omit the dependence of
$(X',W',D')$ on $L$ for lighter notation, but we will need later to
choose~$L$ appropriately. 

Now recall that $K>0$ is a fixed number. 
The first crucial observation is that we may choose
$L$ large enough, so that with high probability, the laws of the restrictions of
$(X',W',D')$ and $(X,W,D_{X,W})$ to the interval $[-K,K]$ are very close. More
precisely, given $\eps\in (0,1)$, 
fix $r>0$ and $A>0$ such that
\[\P\left(\max_{-K\leq t\leq
  K}D_{X,W}(0,t)>r\right)<\eps/3,\qquad \P(T_{-A}<-K<K<T_A)\geq 1-\eps/3.\]
Then \cite[Proposition~6.6]{BaMiRa} and its proof (Lemmas~6.7 and~6.8) show that there exists
$L_0>0$ such that, for $L>L_0$, 
the two processes $(X,W)$ and $(X',W')$ can be coupled in such a
way that on some event $\mathcal{F}$ of probability
$\P(\mathcal{F})\geq 1-\eps/3$, we have
\begin{equation}
  \label{eq:4}
  X_t=X'_t,\quad W_t=W'_t,\quad D_{X,W}(s,t)=D'(s,t),
\end{equation}
for every $s$, $t\in [T_{-A},T_A]$ such that $\max( D_{X,W}(0,t), D_{X,W}(0,s))\leq
r$. Given our choice of~$r$, $A$, we see that with probability at least
$1-\eps$, \eqref{eq:4} holds for every $s,t\in [-K,K]$. 

Our second important observation is that, still with $K$
and $\eps$ fixed, and possibly
up to choosing $L$ even larger than the above, albeit in a way that does not depend on
$n$,
the laws of $(C_{(n)},\Lambda_{(n)},D_{(n)})$ and
$(C'_{(n)},\Lambda'_{(n)},D'_{(n)})$ in restriction to the
interval $[-K,K]$ are also close, in the sense that they can be
coupled in such a way that these restrictions coincide with
probability at least $1-\eps$.
This follows
from the proof of \cite[Theorem 3.6]{BaMiRa}, a minor difference
being that this proposition establishes that the balls of radius $(8n/9)^{1/4}r$
centered at the root in $Q_\infty$ and $Q'_n$ are isometric, rather
than giving a statement on $D_{(n)}$ and $D'_{(n)}$. Therefore, in
order to show that the latter coincide on $[-K,K]$, one 
again has to choose in the first place a radius $r>0$ so that
uniformly over $n$, with
probability at least $1-\eps/3$, the vertices $v'_i$ for integers $i$ lying
in $[-2Kn,2Kn]$ (where we naturally let $v'_i=v'_{i+2a_n+l_n}$ for
$i\leq 0$), all
belong to this ball. The existence of such an~$r$ is guaranteed by the
convergence~\eqref{eq:5} and the continuity of~$D'$.
Finally, we see that both sides of~\eqref{eq:5} can be coupled in such
a way that with probability at least $1-\eps$, they coincide with 
both sides of~\eqref{eqencbhp}. Since $\eps$ was arbitrary, we
conclude that~\eqref{eqencbhp} holds in restriction to $[-K,K]$. Since
$K$ was arbitrary, this concludes the proof. 
\end{proof}

\paragraph{Seeing a slice as part of the UIHPQ.}
We consider a fixed $L>0$ and a sequence $(l_n)\in\N^\N$ such that
\[
\frac{l_n}{\sqrt{2n}}\ton L
\]
and, for each~$n$, we let $(F_n,\lambda_n)$ be the random well-labeled forest obtained by keeping only the labeled trees~$\bT^0$, \ldots, $\bT^{l_n-1}$ of the infinite random well-labeled forest $(F_\infty,\lambda_\infty)$, as well as the root~$\rho^{l_n}$ of the tree~$\bT^{l_n}$. 
In particular, the forest~$F_n$ has~$l_n$ independent  
Bienaym\'e--Galton--Watson trees with Geometric($1/2$) offspring
distribution, and the labels of the root vertices of the trees (including~$\rho^{l_n}$) follow a
random walk of length~$l_n$ whose step distribution is a shifted
Geometric($1/2$) given by $\P(\,\cdot=r)=2^{-r-2}$ for $r\geq -1$. 

Recall that $(c_i,i\in \Z)$ denotes the sequence of corners of the
infinite random well-labeled forest $(F_\infty,\lambda_\infty)$ and
that~$v_i$ is the vertex of~$F_\infty$ incident to~$c_i$. According to
the construction of Section~\ref{sec:comp-sli}, $(F_n,\lambda_n)$ encodes a slice~$Q_n$, which is part of the UIHPQ~$Q_\infty$ constructed from the whole infinite forest
$(F_\infty,\lambda_\infty)$. More precisely, the maximal geodesic (resp.\ shuttle) can be read inside the UIHPQ as the chain of arcs linking~$c_0$ (resp.\ $c_{\tau_{l_n}}$) to its subsequent successors\footnote{Recall Section~\ref{sec:basic-construction}.} and the edges of the slice are given by the arcs from~$c_i$ to~$s(c_i)$, for $0\leq i< \tau_{l_n}$. As a consequence, the vertex~$v_i$ can be seen both as a vertex of~$Q_\infty$ and as a vertex of~$Q_n$, for $0\le i\le \tau_{l_n}$.

Furthermore, we can check that~$Q_n$ is in fact isometrically embedded in~$Q_\infty$ in the sense that, whenever $0\le i,j\le \tau_{l_n}$, it holds that $d_{Q_n}(v_i,v_j)=D_\infty(i,j)$. Indeed, similarly to Proposition~\ref{sgslice}, $Q_\infty$ can be obtained as the gluing of the infinite quadrangulation corresponding to the trees~$\bT^k$, $k<0$, of $(F_\infty,\lambda_\infty)$, with~$Q_n$ and then with the infinite quadrangulation corresponding to the trees~$\bT^k$, $k\ge l_n$, of $(F_\infty,\lambda_\infty)$ along the proper shuttles and maximal geodesics. Alternatively, one may also argue that there are no shortcuts outside~$Q_n$\,: for $0\le i,j\le \tau_{l_n}$, any path linking~$v_i$ to~$v_j$ in~$Q_\infty$ may be shorten to a path that stays within~$Q_n$ since the maximal geodesic and shuttle are geodesics and since the path $c_0\to s(c_0)\to s^2(c_0)\to\dots$ is a geodesic ray that disconnects~$Q_\infty$. 

The contour function, label function and pseudometric corresponding to~$Q_n$ are thus obtained by restricting to $[0,\tau_{l_n}]$ the analog functions corresponding to~$Q_\infty$. After rescaling, their joint limit is a direct consequence of Proposition~\ref{cvencbhp}.

\begin{crl}\label{cvCLD0L}
On $\CC\times\CC\times\CC^{(2)}$, it holds that
\begin{equation}\label{eqcvCLD0L}
  \big(C_{(n)}\rst_{[0,\tau_{l_n}/2n]},\Lambda_{(n)}\rst_{[0,\tau_{l_n}/2n]},D_{(n)}\rst_{[0,\tau_{l_n}/2n]^2}\big)\tod\big(X^{(0,L)},W^{(0,L)},D^{(0,L)}\big),
\end{equation}
where we used the notation of Section~\ref{sec:metric-spaces-coded}, that is,
\begin{enumerate}[label=(\textit{\alph*})]
	\item the pair $(X^{(0,L)},W^{(0,L)})$ is the restriction to the interval $[0,T_{L}]$ of $(X,W)$ distributed under~$\Half$,
	\item $D^{(0,L)}=D_{X^{(0,L)},W^{(0,L)}}$ is the random pseudometric on~$\R$ defined by~\eqref{eqDfg}.
\end{enumerate}
\end{crl}

\begin{proof}
By the Skorokhod representation theorem, we may and will
assume that the convergence~\eqref{eqencbhp} holds almost surely.
Classically, the a.s.\ path properties of~$X$ at time~$T_L$, namely,
the fact that~$X$ immediately visits the interval $(-L-\eps,-L)$ after time~$T_L$, yield that
$\tau_{l_n}/2n$ a.s.\ converges to~$T_L$. Proposition~\ref{cvCLD0L}
and Corollary~\ref{sec:metric-spaces-coded-2} then yield the result.
\end{proof}

\subsection{Scaling limit of conditioned slices}\label{secslcond}

We now derive Theorem~\ref{thmslslice} from the results of the
previous section by standard conditioning arguments. 

\paragraph{Convergence of the encoding processes.}
First, without loss of generality, we may assume that 
the contour and label processes $(C,\Lambda)$ of the infinite random
well-labeled forest defined in Section~\ref{secUIHPQ} are the
canonical processes, considered under the probability distribution~$\P_\infty$ on the canonical space. Next, for~$a$, $l\in \N$ and
$\delta\in \Z$, we denote by $\P_{a,l,\delta}$ the distribution of
$\big(C\rst_{[0,2a+l]},\Lambda\rst_{[0,2a+l]}\big)$ where
$(C,\Lambda)$ is distributed under
$\P_\infty[\;\cdot\mid\tau_l=2a+l,\,\Lambda(\tau_l)=\delta]$. The
corresponding forest encoded by this random process is thus composed of~$l$ Bienaym\'e--Galton--Watson trees with Geometric($1/2$) offspring distribution and uniform admissible labels, conditioned on the fact that the total number of edges in the trees is~$a$ and the label of the root of the last vertex-tree is~$\delta$. Similarly to the slice it encodes, we will say that the forest has \emph{tilt}~$\delta$.

For every measurable nonnegative functional~$G$, it thus holds that
\[\E_{a,l,\delta}[G]=\E_\infty\big[G\big((C(k),\Lambda(k)),0\le k\le \tau_l\big)\bigm\vert \tau_l=2a+l,\,\Lambda(\tau_l)=\delta\big].\]
Let $(\FF_k,k\ge 0)$ be the natural filtration associated with the canonical process $(C,\Lambda)$. 
Note that $((C(k),\Lambda(k)),0\le k\le \tau_l)$ is the pair of contour and label processes of the first~$l$ trees in the forest, and 
that $\FF_{\tau_l}$ is the $\sigma$-algebra generated by these
first~$l$ trees (with their labels and that of the root~$\rho^l$). Recall from Proposition~\ref{countslice} the definitions of~$Q_\ell$, $P_\ell$.

\begin{lmm}\label{lemacslice}
Fix $0< k<l$, as well as $a\in\N$ and $\delta\in \Z$. For every nonnegative functional~$G$ that is $\FF_{\tau_k}$-measurable, we have 
\[\E_{a,l,\delta}[G]=\E_\infty\big[\Phi_{a,l,\delta}(\tau_k,k,\Lambda(\tau_k))\cdot G\big],\]
where
\[\Phi_{a,l,\delta}(t,l',j)=\frac{Q_{l-l'}(2a+l-t)}{Q_l(2a+l)}\frac{P_{l-l'}(\delta-j)}{P_l(\delta)}
.\]
\end{lmm}

\begin{proof}
It suffices to prove the result when~$G$ is the indicator of the
contour and label processes of a given well-labeled forest with~$l'$
trees, $(t-l')/2$ edges, and tilt~$j$. In this case, 
$\E_{a,l,\delta }[G]$ is equal to
the number of ways in
which one can complete this labeled forest into a well-labeled forest
with~$l$ trees, $a$ edges and tilt~$\delta$, which is the number of
forests with $l-l'$ trees, $a+(l'-t)/2$ edges and tilt
$\delta-j$, divided by the number of well-labeled forests with~$l$
trees, $a$ edges and tilt~$\delta$. We conclude by Proposition~\ref{countslice}. 
\end{proof}

In addition to the already fixed sequence~$(l_n)$, we consider two\MEMS{}{ more} sequences~$(a_n)$, $(\delta_n)$ satisfying~\eqref{anlndn}. We will need the following direct consequence of the local limit
theorem~\cite[Theorem~8.4.1]{BGT}. Recall the definition of $\varphi_{A,L,\Delta}$ given
in~\eqref{eq:21}.

\begin{lmm}\label{sec:conditioned-version-5}
If the integer-valued sequence $(l'_n)$
satisfies $l'_n/\sqrt{2n}\to L'\in (0,L)$, it holds that
\begin{equation}
  \label{eq:32}
  \sup_{0\leq t\leq
  a_n,\,j\in \Z}\left|\Phi_{a_n,l_n,\delta_n}(t,l'_n,j)-
  \varphi_{A,L,\Delta}\Big(\frac tn,L',\Big(\frac 9{8n}\Big)^{\frac14}j\Big)\right|\ton 0.
\end{equation}
\end{lmm}

We start with the following conditioned version of Corollary~\ref{cvCLD0L}.

\begin{prp}\label{cvencslice}
On $\CC\times\CC\times\CC^{(2)}$, the triple $\big(C_{(n)},\Lambda_{(n)},D_{(n)}\rst_{[0,\tau_{l_n}/2n]^2}\big)$ considered under $\P_{a_n,l_n,\delta_n}$ converges in distribution to $\big(X,W,D_{X,W}\big)$, considered under $\Slice_{A,L,\Delta}$.
\end{prp}

\begin{proof}
The joint convergence of the first two coordinates is standard; see e.g.~\cite[Corollary~16]{bettinelli10}. Let us fix $\eps\in (0,L)$, define $l_n^\eps=l_n-\lfloor
\eps\sqrt{2n}\rfloor$, so that $l_n^\eps/\sqrt{2n}\to L-\eps$, and set
\[D_{(n)}^\eps=D_{(n)}\rst_{[0,\tau_{l_n^\eps}/2n]^2}\qquad\text{ and }\qquad D_{(n)}^0=D_{(n)}\rst_{[0,\tau_{l_n}/2n]^2}.\]
By the usual bound~\eqref{dd0bound}, for every~$i$, $j\in [0,\tau_{l_n}]$, 
\begin{align*}
  |D_{\infty}(i\wedge\tau_{l_n^\eps},j\wedge\tau_{l_n^\eps})-D_{\infty}(i,j)|&\leq D_{\infty}(i,i\wedge
\tau_{l_n^\eps})+D_{\infty}(j,j\wedge \tau_{l_n^\eps})\\
&\leq 4\big(\omega(\Lambda_n;\tau_{l_n}-\tau_{l_n^\eps})+1\big), 
\end{align*}
where $\omega(f;\cdot)$ denotes the modulus of continuity of~$f$. This implies that
\[\dist_{\CC^{(2)}}\big(D_{(n)}^\eps,D_{(n)}^0\big) \leq \frac{\tau_{l_n}-\tau_{l_n^\eps}}{2n} + 4\,\omega\big(\Lambda_{(n)},(\tau_{l_n}-\tau_{l_n^\eps})/2n\big)+\bO\big(n^{-1/4}\big).\]
From the joint convergence of the first two coordinates, we have, for every $\eta>0$, 
\begin{multline*}
  \limsup_{n\to\infty}\P_{a_n,l_n,\delta_n}\Big(\dist_{\CC^{(2)}}\big(D_{(n)}^\eps,D_{(n)}^0\big)\geq \eta\Big)\\
	\leq \Slice_{A,L,\Delta}\big(A-T_{L-\eps}+4\,\omega(W;A-T_{L-\eps})\geq \eta\big), 
\end{multline*}
which tends to~$0$ as $\eps\to 0$ since $T_{L-\eps}\to T_L=A$ a.s.\
under $\Slice_{A,L,\Delta}$. We next show that $D^\eps_{(n)}$ under
$\P_{a_n,l_n,\delta_n}$ converges in distribution to $D^{(0,L-\eps)}$
under $\Slice_{A,L,\Delta}$, and use the principle of accompanying laws \cite[Theorem 9.1.13]{Stroock11} to
conclude that, jointly with the convergence of
$(C_{(n)},\Lambda_{(n)})$ to $(X,W)$, the process $D_{(n)}^0$ converges to the
distributional limit of $D^{(0,L-\eps)}$  as
$\eps\to 0$, which is none other than
$D^{(0,L)}$, due to Corollary~\ref{sec:metric-spaces-coded-2}.

To prove the claimed convergence of $D^{\eps}_{(n)}$ to
$D^{(0,L-\eps)}$, we denote by $C_{(n)}^{\eps}$ and~$\Lambda_{(n)}^{\eps}$ the restrictions of~$C_{(n)}$ and~$\Lambda_{(n)}$
to $[0,\tau_{l_n^\eps}/2n]$ and let~$F$ be a nonnegative bounded
continuous function.
Using Lemma~\ref{lemacslice}, then Corollary~\ref{cvCLD0L} (for the choice of $L-\eps$ instead of~$L$) and Lemma~\ref{sec:conditioned-version-5} gives
\begin{multline*}
  \E_{a_n,l_n,\delta_n}\big[F\big(C_{(n)}^{\eps},\Lambda_{(n)}^{\eps},D^\eps_{(n)}\big)\big]
	=\E_\infty\left[\Phi_{a_n,l_n,\delta_n}\big(\tau_{l_n^\eps},l_n^\eps,\Lambda(\tau_{l_n^\eps})\big)
		F\big(C_{(n)}^{\eps},\Lambda_{(n)}^{\eps},D^\eps_{(n)}\big)\right]\\
	\ton\Half\big[\varphi_{A,L,\Delta}\big(T_{L-\eps},L-\eps,W_{T_{L-\eps}}\big)
		F\big(X^{(0,L-\eps)},W^{(0,L-\eps)},D^{(0,L-\eps)}\big)\big],
\end{multline*}
the latter being equal to
$\Slice_{A,L,\Delta}\big[F\big(X^{(0,L-\eps)},W^{(0,L-\eps)},D^{(0,L-\eps)}\big)\big]$
by Lemma~\ref{lemaccontslice}.
\end{proof}

\paragraph{GHP convergence.}
We infer from Proposition~\ref{cvencslice} the GHP convergence of Theorem~\ref{thmslslice} by a standard method. First, by Skorokhod's
representation theorem, we may assume that we are working on a
probability space on which the convergence of Proposition~\ref{cvencslice}
is almost sure. We let $\Sl_{A,L,\Delta}$ be the continuum
slice coded by the limiting process, and~$\Sln$ be the slice encoded
by the forest whose rescaled contour and label processes make up the pair 
$(C_{(n)},\Lambda_{(n)})$. As mentioned before
Corollary~\ref{cvCLD0L}, $\Sln$ is isometrically embedded
in~$Q_\infty$, so that the process $D_{(n)}\rst_{[0,\tau_{l_n}/2n]^2}$
under $\P_{a_n,l_n,\delta_n}$ projects into the metric
of~$\Omega_n(\Sln)$.

Then, from this almost sure convergence, we easily deduce
that the distortion of the correspondence~$\RR_n$ given by 
\begin{equation}\label{correspSln}
\RR_n=\big\{\big(v_{\lfloor (2a_n+l_n)s \rfloor },\bp_{X,W}(A s)\big)\,:\, s\in[0,1]\big\}
\end{equation}
between $\Omega_n(\Sln)$ minus its shuttle and $\Sl_{A,L,\Delta}$
tends to~$0$ as $n\to\infty$. Forgetting the marks and measures, this\MEMS{}{ already} gives the desired convergence in the $0$-marked Gromov--Hausdorff topology.

\medskip
In order to include the marking and measures,
we use the technique of enlargement of correspondences already used in the
proof of Lemma~\ref{lemananasplit}. 
Namely, we fix $\eps>0$ and let $\RR_n^\eps$ be the set of points of the
form $(v,x)$ in $\Sln\times \Sl_{A,L,\Delta}$
such that there exists $(w,y)\in \RR_n$ satisfying
$d_{\Sln}(v,w)< (8n/9)^{1/4}\eps$ and $D(x,y)< \eps$. As before, the distortion of
$\RR_n^\eps$ is at most $\dis(\RR_n)+4\eps$. Let us start with the marks.

\paragraph{Marks.}
For a function $f\in \CC$ defined over the interval~$I$, we say that $s\in I$ is a \emph{left-minimum} of~$f$ if
$f(t)\geq f(s)$ for every $t\leq s$ in~$I$, and we call it strict if
$f(t)>f(s)$ for $t<s$ in~$I$. 
Note that the points of
the form~$v_i$ and~$\bp_{X,W}(s)$ where $i$ and $s$ are left-minimums of~$\Lambda_n$
and~$W$ respectively belong to the
maximal geodesics of~$\Sln$ and~$\Sl_{A,L,\Delta}$, and that
all points in these sets are in fact of this form, where we can even
require the stronger property that~$i$ and~$s$ are strict
left-minimums. 

By the uniform convergence of $\Lambda_{(n)}$ toward~$W$, for every $\eta>0$, the following holds provided $n\ge n_0$ for some~$n_0$\,: every strict left-minimum of~$\Lambda_{(n)}$
is at distance at most~$\eta/2$ from some (not necessarily strict)
left-minimum of~$W$, and vice-versa, exchanging the roles of~$\Lambda_{(n)}$ and~$W$. Up to increasing~$n_0$, we furthermore assume that $|({2a_n+l_n})/{2n}-A|<\eta/2$ as soon as $n\ge n_0$. Choosing $\eta$ small enough so
that $|D_{(n)}(s,t)-D_{(n)}(s',t')|\leq \eps$ for every $n$ and
$|s-s'|\leq \eta$, $|t-t'|\leq \eta$, we deduce that the extended correspondence
$\RR_n^\eps$ is compatible with the maximal geodesics for $n\ge n_0$.

The argument is similar for the shuttles. This time, we note that
elements of the shuttle of $\Sl_{A,L,\Delta}$ are of the form $\bp_{X,W}(s)$
where~$s$ is a right-minimum of the
function~$W$ (with an obvious definition), while elements of the
shuttle of $\Sln$ are at distance~$1$ away from points of the
form~$v_i$  where~$i$ is a right-minimum
of the function~$\Lambda_n$. 

The mark corresponding to the base is also treated similarly. Recall from
Section~\ref{sec:comp-sli} that vertices
of the base are at distance at most $B_n=\max_{1\leq i\leq
  l_n}|\Lambda_n(\rho^i)-\Lambda_n(\rho^{i-1})|+1$ from some element of
the floor $\{\rho^0,\ldots,\rho^{l_n}\}$ of the forest coding the
slice. The process of labels $(\Lambda_n(\rho^i),0\leq i\leq l_n)$ forms a
random walk with shifted geometric(1/2) increments conditioned to be
equal to~$\delta_n$ at time~$l_n$, so, under our assumptions, it converges,
after rescaling by~$\sqrt{2n}$ in time and $(8n/9)^{1/4}$ in space, to a continuous
process (which is easily checked to be the Brownian bridge
$\zeta=(W_{T_x},0\leq x\leq L)$), so that $B_n=\sO(n^{1/4})$
a.s. Therefore, the base of $\Sln$ is at Hausdorff
distance $\sO(n^{1/4})$ from the floor $\{\rho^i,0\leq i\leq
l_n\}$. In turn, these vertices are exactly those of the form~$v_i$ where~$i$ is a left-minimum of the contour process~$C_n$. Moreover, by definition, the base of $\Sl_{A,L,\Delta}$
consists of the points $\bp_{X,W}(s)$ where~$s$ is a left-minimum of the
process~$X$. 
Therefore, the same argument as for the maximal geodesic -- replacing
the processes~$\Lambda_n$ and~$W$ by~$C_n$ and~$X$ -- shows that, a.s., for every~$n$
large enough, the correspondence~$\RR_n^\eps$ is also compatible with the bases of~$\Sln$ and of~$\Sl_{A,L,\Delta}$.

\paragraph{Measures.}
Finally, let us deal with the convergence of the measures, starting
with the area measure. To this end,
note that, for~$t$ in $[0,2a_n+l_n]$, the contour process~$C_n$ at
time~$t$ has either a left derivative equal to~$+1$ or to~$-1$. Letting
$i^n_t=\lceil t\rceil$ in the former case and $i^n_t=\lfloor
t\rfloor$ in the latter case, the image of $\Leb_{[0,(2a_n+l_n)/2n]}$ by $t\mapsto v_{i^n_{2nt}}$ is the counting measure on the set
of all non-floor vertices of the encoding forest, divided by~$n$. 
Since the number of floor vertices is $\bO(\sqrt{n})$, the
counting measure on all vertices of~$\Sln$ (except on the
shuttle) divided by~$n$ is at vanishing Prokhorov distance from the
counting measure on non-floor vertices of the forest, divided by~$n$. Let~$\omega_n$
be the image of the Lebesgue measure on $[0,A\wedge ((2a_n+l_n)/{2n})]$
by the mapping $t\mapsto (v_{i^n_{2n t}},\bp_{X,W}(t))$. Then~$\omega_n$ is 
carried by the correspondence~$\RR_n^\eps$ for every $n$ large
enough, and its image measures on $\Sln$ and $\Sl_{A,L,\Delta}$ by
the coordinate projections are at vanishing Prokhorov distances from
$\mu_{\Sln}/n$ and $\mu^{(0,L)}$ respectively.

For the base measure, we let $\omega'_n$ be the image of the
Lebesgue measure on $[0,L\wedge l_n/\sqrt{2n}]$ by the mapping
$t\mapsto (v_{\tau_{\lfloor \sqrt{2n}\,t\rfloor}},\bp_{X,W}(T_t))$. Then
$\omega'_n$ is carried by $\RR_n^\eps$, by the above discussion on the
mark corresponding to the base. Moreover, the coordinate
projections of $\omega'_n$ are at vanishing Prokhorov distance,
respectively, from the counting measure on
$\{\rho^0,\ldots,\rho^{l_n}\}$ divided by $\sqrt{2n}$, and
$\nu^{(0,L)}$. We now observe that, in turn, the counting measure on  
$\{\rho^0,\ldots,\rho^{l_n}\}$ divided by $\sqrt{2n}$, is at
vanishing Prokhorov distance from the renormalized counting measure
(with multiplicities)
$\nu_{\beta_n}/\sqrt{8n}$ of the base. 
To justify this, observe from Section~\ref{sec:comp-sli} and the
definition of the interval CVS bijection that the sequence
$\Lambda_n(w_0)$, \dots, $\Lambda_n(w_{2l_n+\delta_n})$ of labels of the vertices~$w_0$, 
\dots, $w_{2l_n+\delta_n}$ of the base, taken in contour order,
forms a simple random walk 
starting with a~$-1$ step, and 
conditioned on hitting~$\delta_n$ at time $2l_n+\delta_n$. Moreover,
if we write the set $\{j\in \{0,\ldots,
2l_n+\delta_n-1\}:\Lambda_n(w_{j+1})-\Lambda_n(w_{j})=-1\}$ of down steps of this walk as
$\{j_0,j_1,\ldots,j_{l_n-1}\}$ with $0=j_0<j_1<j_2<\dots<j_{l_n-1}$, then the
$i$-th root~$\rho^i$ is equal to~$w_{j_i}$ for $0\leq i<l_n$.
Now consider a uniform random variable~$U$ in $[0,1)$. Then
$w_{j_{\lfloor l_n U\rfloor}}$ is a uniformly chosen forest root, while $w_{\lfloor (2l_n+\delta_n)U\rfloor}$
is a vertex of the base chosen with probability proportional to its
multiplicity (and excluding~$\rho^{l_n}$ in
both cases). Moreover, a standard large deviation estimate entails
that $\max_{0\leq k<l_n}|j_k-2k|=\bO(\log n)$ in probability. In turn, this easily implies that
$d_{\Sln}(w_{j_{\lfloor l_n U\rfloor}}, w_{\lfloor
  (2l_n+\delta_n)U\rfloor})=\bO(\log n)$ in probability, showing that the uniform measure on the
$l_n\sim L\sqrt{2n}$ elements of $\{\rho^0,\ldots,\rho^{l_n-1}\}$ is at vanishing Prokhorov distance
from the law of the vertex incident to a corner uniformly chosen among
the $2l_n+\delta_n\sim L\sqrt{8n}$ corners incident to the base.

\paragraph{Conclusion.}
By Lemma~\ref{lemcorresP}, we finally obtain that
\[\limsup_{n\to\infty}\dGHP^{(5,2)}\big(\Omega_n(\Sln),\Sl_{A,L,\Delta}\big)\leq\eps.\]
Since $\eps>0$ was arbitrary, this concludes the proof of Theorem~\ref{thmslslice}.

\section{Convergence of quadrilaterals with geodesic sides}\label{seccvquad}

The general method to prove Theorem~\ref{thmslquad} is the same as for slices. We start by seeing a discrete quadrilateral as part of a discrete map that is known to converge to a
Brownian surface, which in this case is the Brownian plane rather than
the Brownian half-plane. However, the lack
of an analog of Corollary~\ref{sec:metric-spaces-coded-2}, namely that
quadrilaterals are only \emph{locally} isometrically embedded in the Brownian
plane, makes matters considerably
more delicate. For this reason, we adapt the strategy we used
in~\cite[Section~4]{BeMi17} when treating the case of noncomposite
slices. Beware that, in this section, part of the notation we will be
using is slightly conflicting with that of Section~\ref{seccvcs}:
in particular, the 
random times~$T_x$ will be re-defined.

\subsection{Quadrilaterals coded by two functions}\label{seccontquad}

In contrast with slices, which were coded by a pair of functions defined on a common interval~$I$, a quadrilateral will be coded by a pair of functions defined on a common union of two intervals $I_-\cup I_+$, each interval accounting for one ``half'' of the quadrilateral. This leads to similar but slightly more intricate definitions. We start with the most convenient setting, asking that $\sup I_-=\inf I_+=0$.

Recall the notation of Section~\ref{Rtrees}. We adapt Section~\ref{secslicefg} to quadrilaterals\MEMS{}{ instead of slices} as follows. We now
say that a pair $(f,g)\in \CC^2$ of functions with common closed interval of definition~$I$ is a \emph{quadrilateral trajectory} if they satisfy~\eqref{eqslicetraj} and the following:
\begin{itemize}
\item the interval~$I$ contains $0$ in
  its interior and is either bounded or equal to the whole real line
  $\R$, and, letting $I_+=I\cap \Rp$ and $I_-=I\cap \Rm$, 
  \item we have $\inf_{I_-}f=\inf_{I_+}f$, and,
\item if $I=\R$, then 
  $\inf_{t\geq 0}f(t)=\inf_{t\leq 0}f(t)=\inf_{t\geq 0}g(t)=\inf_{t\leq
      0}g(t)=-\infty$. 
\end{itemize}
We may observe that $(f\rst_{I_+},g\rst_{I_+})$ is a slice trajectory, a fact that will not be used here. For a quadrilateral trajectory $(f,g)$, we set
\begin{equation}
  \label{eq:23}
  \dov{d}_g(s,t)
=\begin{cases}
	d_g(s,t) & \text{ for }\  s,t\in I_+ \mbox { or } s,t\in  I_- \\
	\infty & \text{ for }\ st<0
\end{cases},
\end{equation}
and
\begin{equation}
  \label{eq:30}
  {\dov{{D}}}_{f,g}=\dov{d}_g/\{d_f=0\}.
\end{equation}
Note that $\dov{d}_g$ is the disjoint union pseudometric of the two
$\R$-tree pseudometrics~$d_{g\rsts_{I_+}}$ and~$d_{g\rsts_{I_-}}$.   
Let $\big(\dov{M}_{f,g},{\dov{D}}_{f,g}\big)$ be the quotient space
$I/\{\dov{D}_{f,g}=0\}$ equipped with the metric induced by~${\dov{D}}_{f,g}$, still denoted by the same symbol.
We call the metric space
\[\Qd_{f,g}=(\dov{M}_{f,g},\dov{D}_{f,g})\]
the \emph{quadrilateral} coded by $(f,g)$.

\medskip
We extend the above constructions to unions of two closed intervals $I=I_- \cup I_+$ where $I_-\subseteq \Rm$ and $I_+\subseteq \Rp$, as follows. First, a pair of functions $(f,g)$ defined on~$I$ is a \emph{quadrilateral trajectory} if the pair $(f',g')$ defined by
\[f'(t)=\begin{cases}
	f(t+\inf I_+)& \text{ for } t\in I_+-\inf I_+\\
	f(t+\sup I_-)& \text{ for } t\in I_--\sup I_-
\end{cases},\]
and similarly for~$g'$, is a quadrilateral trajectory as defined above. Note that the continuity hypothesis on~$f'$ implies in particular that $f(\sup I_-)=f(\inf I_+)$, and similarly for~$g$. We then define the quadrilateral coded by $(f,g)$ using the exact same definitions as above. Note that the mapping $t\mapsto(t-\inf I_+)\ind_{t\in I_+}+(t-\sup I_-)\ind_{t\in I_-}$ induces an isometry 
from $(\dov{M}_{f,g},\dov{D}_{f,g})$ onto
$(\dov{M}_{f',g'},\dov{D}_{f',g'})$. 

From now on, we work in
this extended framework and consider a fixed quadrilateral trajectory $(f,g)$.

\paragraph{Geodesic sides and area measure. }
For every $t\in I\setminus\{0\}$, we let $I_t=I_-$ if $t<0$ or $I_t=I_+$ if $t> 0$, and set
\begin{align*}
\Gamma_t(r)&=\inf\{s\geq t:g(s)=g(t)-r\}&\text{ for $r\in \Rp$ such that } \inf_{\substack{s\geq t\\s\in I_t}}g(s)\leq g(t)-r\,; \\
\Xi_t(r)&=\sup\{s\leq t:g(s)=g(t)-r\}	&\text{ for $r\in \Rp$ such that } \inf_{\substack{s\leq t\\s\in I_t}}g(s)\leq g(t)-r.
\end{align*}
If $0\in I$, we also define~$\Gamma_0$ and~$\Xi_0$ with the same definition, using $I_0=I_+$ in the definition of~$\Gamma_0$, while using $I_0=I_-$ in the definition of~$\Xi_0$. Observe that, in contrast with the definition for slices, the infimum of~$g$ is now taken on a subset of~$I_t$. In particular, this implies that the ranges of~$\Gamma_t$, $\Xi_t$ are included in~$I_t$. From the same discussion as the one around~\eqref{eq:28}, we see that~$\Gamma_t$, $\Xi_t$ are geodesics for the pseudometrics $\dov{d}_g$ and $\dov{D}_{f,g}$. In the case where
$\sup I_+=\infty$, then, for every $t\in I_+$, the range of the path~$\Gamma_t$ is a geodesic ray, and, in the case where $\inf I_-=-\infty$, then the same goes for~$\Xi_t$ for every $t\in I_-$. This allows to define
geodesic paths in~$\Qd_{f,g}$ by the formulas
\begin{align*}
  \gamma_t(r)& =\dov\bp_{f,g}(\Gamma_t(r)),\qquad 0\leq r\leq
  g(t)-\un{g}(t,\sup I_t),\ r\in\R,\\
   \xi_t(r)&=\dov\bp_{f,g}(\Xi_t(r)),\qquad 0\leq r\leq 
  g(t)-\un{g}(\inf I_t,t),\ r\in\R,
\end{align*}
where $\dov\bp_{f,g}:I\to \dov{M}_{f,g}$ is the canonical projection and, as above, if $0\in I$, $I_0=I_+$ in the definition of~$\gamma_0$ and $I_0=I_-$ in that of~$\xi_0$. Note
that the geodesics $\gamma_t$, $\xi_t$ share a common initial part.

The quadrilateral~$\Qd_{f,g}$ comes with four or two geodesic sides, defined as follows. If~$I$ is bounded, the particular geodesics $\gamma=\gamma_{\inf I_+}$ and $\br{\gamma}=\gamma_{\inf I_-}$
are called the \emph{maximal geodesics} of~$\Qd_{f,g}$, while
$\xi=\xi_{\sup I_+}$ and $\br\xi=\xi_{\sup I_-}$ are called the
\emph{shuttles} of~$\Qd_{f,g}$. In this case, $\gamma$, $\xi$ (resp.\ $\br{\gamma}$, $\br\xi$) have a common
endpoint $x_*=\dov\bp_{f,g}(s_*)$ (resp.\ $\br{x}_*=\dov\bp_{f,g}(\br s_*)$) where $s_*\in
I_+$ is such that $g(s)=\inf_{I_+} g$ (resp.\ $\br s_*\in I_-$ is such
that $g(\br s_*)=\inf_{I_-}g$). The points~$x_*$, $\br x_*$ are called
the \emph{apexes} of~$\Qd_{f,g}$. If~$I$ is unbounded, then~$\Qd_{f,g}$ has one \emph{maximal geodesic} $\gamma=\gamma_{\inf I_+}$ and one \emph{shuttle} $\br\xi=\xi_{\sup I_-}$; we set $\xi_\infty=\gamma_{-\infty}=\varnothing$.

Finally, the \emph{area measure} is defined
as $\mu=(\dov\bp_{f,g})_*\Leb_I$.

\paragraph{Gluing quadrilaterals.}
For $x\in \R$, we let
\begin{align*}
T_x&=\inf\{t\in I_+ :f(t)=-x\}\in\Rp\cup\{+\infty\},\\
\br T_{x}&=\sup\{t\in I_-:f(t)=-x\}\in\Rm\cup\{-\infty\},
\end{align*}
as well as $T_\infty=-\br T_\infty=\infty$. Note that, here again, there is a slight difference with the definition of Section~\ref{seccvcs} since, now, $\Rm$ and~$\Rp$ play different roles. Recall that $\inf_{I_-} f=\inf_{I_+} f=\inf_I f$, and let~$H$, $H'\in
\Rp\cup \{\infty\}$ be such that $0\leq H\leq H'\leq -\inf_I f$. We
may define the restrictions $f^{(H,H')}$, $g^{(H,H')}$ of~$f$ and~$g$ to
the union of intervals $I^{(H,H')}=[\br T_{H'},\br T_H]\cup [T_H,T_{H'}]$, which
is a subset of~$I$. The pair $(f^{(H,H')},g^{(H,H')})$ is another quadrilateral trajectory. The associated quadrilateral is defined as
\[\Qd^{(H,H')}=\big(\dov{M}^{(H,H')},\dov{D}^{(H,H')}\big)=\Qd_{f^{(H,H')},g^{(H,H')}}.\]
We let $\dov\bp^{(H,H')}:I^{(H,H')}\to \dov{M}^{(H,H')}$ be the
canonical projection, $\mu^{(H,H')}$ be the area measure of~$\Qd^{(H,H')}$, and, whenever they exist, 
$\gamma^{(H,H')}$, $\br \gamma^{(H,H')}$ be the maximal geodesics, $\xi^{(H,H')}$, $\br \xi^{(H,H')}$ be the shuttles.

We refer to Figure~\ref{figsgquad} for an illustration of the following proposition in the upcoming context of random quadrilaterals in the Brownian plane.

\begin{prp}\label{gludetquad}
Let $0\leq H< H'<H''\leq -\inf_I f$ be in the extended positive real
line. Then
\begin{equation}
\label{eq:29}
 \Qd^{(H,H'')}=G\big(G\big(\Qd^{(H,H')},\Qd^{(H',H'')};\xi^{(H,H')},\gamma^{(H',H'')}\big);\br\gamma^{(H,H')},\br
\xi^{(H',H'')}\big),
\end{equation}
and it holds that
\begin{align*}
\gamma^{(H,H'')}&=\gamma^{(H,H')}\cup \big(\gamma^{(H',H'')}\setminus
\xi^{(H,H')}\big),\\
\xi^{(H,H'')}&=\xi^{(H',H'')}\cup \big(\xi^{(H,H')}\setminus
		\gamma^{(H',H'')}\big),\\
\br\gamma^{(H,H'')}&=\br\gamma^{(H',H'')}\cup \big(\br\gamma^{(H,H')}\setminus
\br\xi^{(H',H'')}\big),\\
\br\xi^{(H,H'')}&=\br\xi^{(H,H')}\cup \big(\br\xi^{(H',H'')}\setminus
		\br\gamma^{(H,H')}\big).
\end{align*}
\end{prp}

Observe that, after the first gluing operation is performed, the marks~$\br \gamma$ and~$\br \xi$ remain geodesic, as observed in Section~\ref{secgluing}. 
Note also that the order of the gluings in~\eqref{eq:29} is not important, due to~\eqref{dR1R2}. 

\begin{proof}
The proof is similar to that of Proposition~\ref{sgslice}, and we
only sketch the argument. Again, we view $I^{(H,H'')}$ as a disjoint
union $I^{(H,H')}\sqcup I^{(H',H'')}$ (denoting elements of these
sets with superscripts~$0$, $1$ respectively) where the extremities
$T_{H'}^0$, $T_{H'}^1$ and $\br T_{H'}^0$, $\br T_{H'}^1$ are
identified. We then observe that the pseudometric $\dov{d}_g$ can be
viewed as a quotient $d/R_1$, where $d$ is the disjoint union metric on 
$I^{(H,H')}\sqcup I^{(H',H'')}$ whose restriction to each interval
composing this set 
equals the restriction of $d_g$ to that interval, and $R_1$ is the
coarsest equivalence relation containing
\begin{multline*}
\big\{\big(\Xi_{T_{H'}}(r)^0,\Gamma_{T_{H'}}(r)^1\big),\ 0\leq r\leq
g(T_{H'})-\un{g}(T_H,T_{H'})\vee \un{g}(T_{H'},T_{H''})\big\}\\
	\text{ and }\big\{\big(\Gamma_{\br T_{H'}}(r)^0, \Xi_{\br T_{H'}}(r)^1\big),\ 0\leq r\leq
	g(\br T_{H'})-\un{g}(\br T_H,\br T_{H'})\vee \un{g}(\br T_{H'},\br
	T_{H''})\big\}.
\end{multline*}
Moreover, the equivalence relation $\{d_f=0\}$ factorizes in the sense
that, if $d_f(s,t)=0$ with $s$, $t\in I^{(H,H'')}$, then it must hold that
$s$, $t$ belong either both to $I^{(H,H')}$ or both to
$I^{(H',H'')}$. Therefore, setting $R_2$ as the equivalence
relation on $I^{(H,H')}\sqcup I^{(H',H'')}$ defined by $s^i\binR_2 t^j$ if
and only if $d_f(s,t)=0$ and $i=j\in \{0,1\}$, we obtain
\[\dov{D}^{(H,H'')}=(d/R_1)/R_2=(d/R_2)/R_1.\]
We\MEMS{}{ now} recognize that $d/R_2$ is the pseudometric of the disjoint
union of $(I^{(H,H')},\dov{D}^{(H,H')})$ and
$(I^{(H',H'')},\dov{D}^{(H',H'')})$, while $R_1$ can be seen as the
coarsest equivalence relation obtained by first gluing $\xi^{(H,H')}$
with $\gamma^{(H',H'')}$, and then $\br \gamma^{(H,H')}$ with
$\br\xi^{(H',H'')}$. 
\end{proof}

Due to the fact that the second gluing operation in Proposition~\ref{gludetquad} involves two geodesics belonging to the same
space, there is no direct analog of Corollary~\ref{sec:metric-spaces-coded-2}. However, we have the following
alternative, which is an immediate consequence of Lemma~\ref{lemdistglue} and a crude estimate of the length of the
path $\br\xi^{(H',H'')}$. 

\begin{prp}\label{boundgluquad}
Under the same assumptions as in Proposition~\ref{gludetquad},
it holds that for every $s$, $t\in I^{(H,H')}$, 
\[\dov{D}^{(H,H'')}(s,t)\leq \dov{D}^{(H,H')}(s,t)\leq
\dov{D}^{(H,H'')}(s,t)+\omega(g;I^{(H',H'')}).\]
\end{prp}

Finally, we observe that the metric space $(M_{f,g},D_{f,g})$
obtained by metric gluing of the pseudometric $d_g$ along the
relation $\{d_f=0\}$, rather than using $\dov d_g$ as in the
definition of~$\Qd_{f,g}$, is related to the latter by a final gluing
operation. The proof is analog to that of Proposition~\ref{gludetquad}, noting
that~$d_g$ is the gluing of~$\dov d_g$ along the coarsest equivalence
relation containing $\{(\Gamma_0(r),\Xi_0(r)),r\geq 0\}$.

\begin{lmm}\label{sec:gluing-quadr-}
One has $\big(M_{f,g},D_{f,g}\big)=G\big(\Qd_{f,g};\gamma,\br\xi\big)$.
\end{lmm}

\subsection{Random continuum quadrilaterals}

Let us now describe the limiting continuum quadrilaterals that appear in Theorem~\ref{thmslquad}, by suitably
randomizing the quadrilateral trajectory $(f,g)$. We let $(X,W)$ be the canonical
process defined on quadrilateral trajectories. We introduce, for any process~$Y$ defined on an interval containing~$0$, the piece of notation $\dun{Y}_t=\un{Y}(0\wedge t,0\vee t)$.

Let us fix~$A$, $\br A$, $H\in (0,\infty)$ and
$\Delta\in \R$. We let $\Quad_{A,\br A,H,\Delta}$ be the probability
distribution under which
\begin{itemize}
\item $(X_t,0\leq t\leq A)$ and $(X_{-t},0\leq t\leq \br A)$ are
independent first-passage bridges of standard Brownian motion from~$0$ to~$-H$, with durations~$A$ and~$\br A$\,;
\item conditionally given~$X$, the process~$W$ has same law as $(Z_t+\zeta_{-\dun X_t},-\br A \le t \le A)$, where~$Z$ is the
random snake driven by $X-\dun{X}$, and~$\zeta$ is a standard Brownian bridge of duration~$H$ and terminal value~$\Delta$, independent of~$X$ and~$Z$. 
\end{itemize}

In this way, the probability distribution $\Quad_{A,\br A,H,\Delta}$ is carried by quadrilateral
trajectories on the interval $[-\br A,A]$. We remark that, in fact, we
can view~$W$ more directly as the random snake driven by~$X$,
conditioned on the event $\{W_A=\Delta\}$, a fact that we leave to the
interested reader. 

\begin{defn}\label{sec:cont-quadr}
The \emph{quadrilateral with half-areas $A$, $\br A$,
width~$H$ and tilt~$\Delta$}, generically denote by $\Qd_{A,\br A,H,\Delta}$, is the $6$-marked $1$-measured metric space
$\Qd_{X,W}$ under the law $\Quad_{A,\br A,H,\Delta}$, endowed with its area measure~$\mu$, as well as the marking
\[\partial \Qd_{A,\br A,H,\Delta}=\big(\gamma,\xi,\br\gamma,\br\xi\big),\]
where~$\gamma$, $\br\gamma$ are geodesic marks as usual, while~$\xi$, $\br\xi$ are seen as (nonoriented) geodesic segments, that is, given without their origins.
\end{defn}

As for slices, the piece of notation $\partial \Qd_{A,\br A,H,\Delta}$ comes from\MEMS{}{ the result of} Lemma~\ref{topdimep}, which we will prove at the end of the upcoming section. The boundary of the topological disk $\Qd_{A,\br A,H,\Delta}$ is the union of~$\gamma$, $\xi$, $\br{\gamma}$, $\br{\xi}$, which intersect only at
the points~$\gamma(0)=\br\xi(0)=\dov\bp_{f,g}(0)$, $x_*$,
$\br\gamma(0)=\xi(0)=\dov\bp_{f,g}(A)=\dov\bp_{f,g}(-\br A)$, and~$\br{x}_*$.

\subsection{The Brownian plane, and its embedded quadrilaterals}\label{secBP}

Similarly to the fact that (free) slices can be found in
the Brownian half-plane, one can obtain quadrilaterals from the Brownian plane, as we now explain. We let $\Plane$ be the probability distribution on~$\CC^2$ under which
\begin{itemize}
\item the process $X$ is a two-sided standard Brownian motion\footnote{This means that $(X_t,t\geq 0)$ and $(X_{-t},t\geq 0)$ are independent standard
Brownian motions.}, and
\item the process~$W$ is the random snake driven by~$X$.
\end{itemize}
The measure $\Plane$ is carried by quadrilateral trajectories defined over~$\R$. 

\begin{defn}\label{defBP}
The \emph{Brownian plane}, generically denoted by~$\BP$, is the metric space $(M_{X,W},D_{X,W})$ defined by~\eqref{eqDfg}, considered under~$\Plane$. Letting $\bp:\R\to \BP$ be the canonical projection,
it is 
endowed with the area
measure $\mu=\bp_*\Leb_\R$. 
\end{defn}

In this definition, beware that the metric is indeed defined by~\eqref{eqDfg} rather than~\eqref{eq:30}, which would produce the
metric space $\Qd_{X,W}=\Qd^{(0,\infty)}=(\dov{M}_{X,W},\dov{D}_{X,W})$. Observe
that, by Lemma~\ref{sec:gluing-quadr-}, 
\begin{equation}
  \label{eq:31}
  \BP = G\big(\Qd^{(0,\infty)};\gamma^{(0,\infty)},\br \xi^{(0,\infty)}\big)\, ;
\end{equation}
see Figure~\ref{figsgquad} below for an illustration. Alternatively, the space
$\Qd^{(0,\infty)}$ can be seen as cutting the Brownian plane along the
geodesic ray $\gamma_0=\xi_0$; we do not go into further details as we will not explicitly need this property. 

Note also that, despite the similarity between this definition and that of 
 the Brownian half-plane, there is no marking now because, as its name suggests, the
Brownian plane is homeomorphic to~$\R^2$ and therefore has an empty boundary
as a topological surface.

One should finally mind that this definition is different from the original one given in~\cite{CuLG12Bplane}, which will be recalled in Appendix~\ref{appBP}; in a nutshell, one goes from a definition to the other by changing~$X$ into the process obtained by taking its Pitman transform both on~$\Rp$ and on~$\Rm$.

\paragraph{Free quadrilaterals.}
Similarly to the discussion of Section~\ref{secBHP}, the Brownian
plane satisfies a Markov property which can be interpreted as a
``flow'' of continuum quadrilaterals. Fixing $0\leq H\leq H'\leq
\infty$, and 
denoting by 
\[\vartheta_H:t\in[\br T_{H'}-\br T_H,T_{H'}-T_H]\mapsto (t+\br
T_H)\ind_{t<0}+(t+T_H)\ind_{t\ge 0},\] we see that the process
$\big(X^{(H,H')}\circ\vartheta_{H}+H,W^{(H,H')}\circ\vartheta_{H} -W_{T_H}\big)$ is
independent of $(X^{(0,H)},W^{(0,H)})$,
$(X^{(H',+\infty)},W^{(H',+\infty)})$, and has same distribution as \MEMS{the pair }{}
$(X^{(H'-H)},W^{(H'-H)})$. This can be proved by excursion
theory of $(X,W)$ separately in positive and negative times; we
omit the details, which are similar to those presented in Section~\ref{secBHP}. 

Under $\Plane$, the process
$(X^{(0,H)})$ is a two-sided Brownian motion killed at its first
hitting times~$T_H$, $\br T_H$ of~$-H$ respectively in positive and negative
times, while the process $\big(W^{(0,H)}_{T_x},0\leq x\leq H\big)$ is a
standard Brownian motion killed at time $H$. This implies that the law
of $(X^{(0,H)},W^{(0,H)})$ under $\Plane$ equals
\[\FQuad_H=\int_{(0,\infty)^2}q_H(
A)q_H( \br A)\,\d A\,\d \br A\int_{\R}p_H(\Delta)\,\d \Delta\,
\Quad_{A,\br A,H,\Delta},\]
where the densities~$p_t$, $q_x$ are defined after~\eqref{eqFSlice}. A
random metric space with same law as~$\Qd^{(0,H)}$ under $\FQuad_H$
will be referred to as a \emph{free (continuum) quadrilateral} of width~$H$. From these considerations and Proposition~\ref{gludetquad}, we
obtain the following result. 

\begin{prp}
\label{sec:brownian-plane-its}
Let $0\leq H<H'<H''\leq \infty$. Then, under $\Plane$, it holds that
\[\Qd^{(H,H'')}=G\big(G\big(\Qd^{(H,H')},\Qd^{(H',H'')};\xi^{(H,H')},\gamma^{(H',H'')}\big);\br\gamma^{(H,H')},\br\xi^{(H',H'')}\big),\]
where the glued spaces $\Qd^{(H,H')}$ and $\Qd^{(H',H'')}$ are
independent. Moreover, $\Qd^{(H,H')}$ is a free continuum
quadrilateral of width $H'-H$. 
\end{prp}

\begin{figure}[ht]
	\centering\includegraphics[width=.95\linewidth]{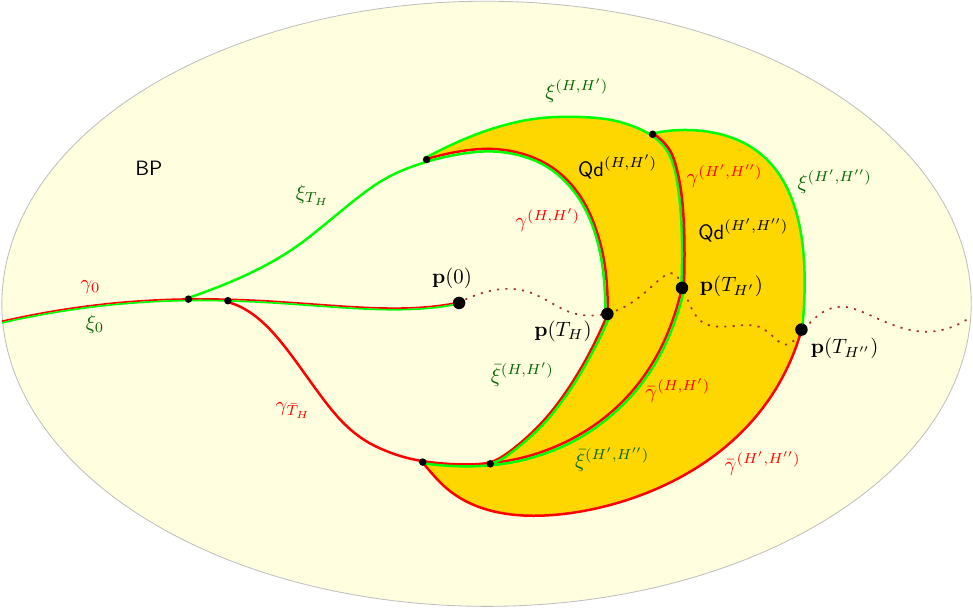}
	\caption{Seeing free quadrilaterals in the Brownian plane. The
union of the dark yellow regions forms $\Qd^{(H,H'')}$. The
dotted brown line is $\{\bp(T_h)\,:\, h\ge 0\}$. Note how~$\BP$ itself is obtained by gluing $\Qd^{(0,\infty)}$ along
the geodesics $\gamma^{(0,\infty)}$ and $\xi^{(0,\infty)}$, resulting
in the geodesic $\gamma_0=\xi_0$.}
	\label{figsgquad}
\end{figure}

We refer to Figure~\ref{figsgquad} for an illustration, which suggests, as is
proved in the following proposition, that
quadrilaterals are topological disks bounded by their geodesic sides. In contrast with our treatment of slices, a difficulty
arises from the fact that the quadrilaterals $\Qd^{(H,H')}$ are not isometrically
embedded in~$\BP$, and, in general, not even locally isometrically embedded (think
of a point of~$\BP$ lying on the geodesic~$\gamma_0$). 

\begin{prp}\label{sec:brownian-plane-its-1}
For every $H\in (0,\infty)$, almost surely under $\FQuad_H$, the
quadrilateral $\Qd^{(0,H)}$ is a topological disk with boundary
given by the geodesics $\gamma^{(0,H)}$, $\xi^{(0,H)}$,
$\br \gamma^{(0,H)}$ and $\br \xi^{(0,H)}$, which pairwise meet only
at the points $\gamma(0)=\br\xi(0)$, $\xi(0)=\br{\gamma}(0)$, and the
apexes~$x_\sas^{(0,H)}$ and~$\br x_\sas^{(0,H)}$.
\end{prp}

In order to prove this proposition and for later use, it will be
important to characterize the set $\{D_{X,W}=0\}$. 

\begin{lmm}\label{sec:brownian-plane-its-2}
The following holds almost surely under $\Plane$. For every~$s$, $t\in \R$ such that $s\neq t$, it holds that
$D_{X,W}(s,t)=0$ if and only if either $d_X(s,t)=0$ or $d_W(s,t)=0$, these
two cases being mutually exclusive.
\end{lmm}

\begin{proof}
By \cite[Proposition~11]{CuLG12Bplane}, it holds that $D_{X,W}(s,t)=0$
implies that $d_X(s,t)=0$ or $d_W(s,t)=0$. The fact that these two
properties are mutually exclusive is a consequence of the fact from
Lemma 2.2 in \cite{legall06} that
almost surely,
if $s$ is a point such that $X_u\geq X_s$ for every $u\in
[s,s+\eps]$ for some $\eps>0$, then it must hold that $\inf_{u\in
[s,s+\delta]}W_u<W_s$ for every $\delta\in (0,\eps)$. In fact,
\cite[Lemma~2.2]{legall06} is proved when the process~$X$ is distributed as a standard Brownian excursion, and~$W$ as a random snake~$Z$
driven by this excursion. However, being a local property of the
processes at hand, it extends easily to our setting by an absolute
continuity argument. Details are left to the reader. 
\end{proof}

To the terminology of Section~\ref{Rtrees}, we add\MEMS{}{ the piece of notation} $\llbracket a,b\cllbracket_f=\llbracket a,b\rrbracket_f\setminus\{b\}$ for~$a$, $b\in \cT_f$. The important consequence of this lemma for \MEMS{us}{our purposes} is the following. Almost surely, if
$a$, $b\in \cT_X$ and $c$, $d\in \cT_W$, then the paths $\pi_X(\llbracket a,b\rrbracket_X)$
and $\pi_W(\llbracket c,d\rrbracket_W)$ are simple paths. Furthermore, $\pi_X(\llbracket a,b\cllbracket_X)$ may intersect 
$\pi_W(\llbracket c,d\cllbracket_W)$ only if $\pi_X(a)=\pi_W(c)$, in which case these paths intersect
at this point only. In particular, if we denote the geodesic ray
$\bp_X(\{s\geq t:X_s= \un{X}(t,s)\})$ of~$\cT_X$ by $\llbracket \bp_X(t),\infty\cllbracket_X$, then
$\pi_X(\llbracket \bp_X(t),\infty\cllbracket_X )$ is a simple path in $\BP$. For instance, in Figure~\ref{figsgquad}, we represented the simple path $\pi_X(\llbracket \bp_X(0),\infty\cllbracket_X )$ with a dotted brown line.

\begin{proof}[Proof of Proposition~\ref{sec:brownian-plane-its-1}]
Let us depart slightly from the setting of the statement and fix for now
two numbers $0\leq H < H'<\infty$.

\begin{clm}
We assume that the geodesics
$\gamma^{(H,H')}$ and $\br{\xi}^{(H,H')}$ do not intersect
$\gamma_0$ in~$\BP$. Then the following holds.
\begin{enumerate}[label=(\textit{\roman*})]
	\item\label{Jordani} The geodesics $\gamma^{(H,H')}$, $\xi^{(H,H')}$, $\br \gamma^{(H,H')}$, $\br
\xi^{(H,H')}$ intersect only at the points
$x_\sas^{(H,H')}$, $\br\gamma^{(H,H')}(0)=\xi^{(H,H')}(0)$, $\br x_\sas^{(H,H')}$ and $\gamma^{(H,H')}(0)=\br\xi^{(H,H')}(0)$
in this cyclic order, and their union forms a Jordan curve~$C$.
	\item\label{Jordanii} The set
$\bp(I^{(H,H')})\subseteq \BP$  is the closure of the bounded connected
component of~$\BP\setminus C$.
\end{enumerate}
\end{clm}

Indeed, note that $\bp_X(T_H)$ and $\bp_X(T_{H'})$ are two distinct points of $\llbracket \bp_X(0),\infty\cllbracket_X$, so
that their images by $\pi_X$ are distinct in~$\BP$. Then the paths $\gamma^{(H,H')}$ and
$\xi^{(H,H')}$ are the images by $\pi_W$ of the two geodesic paths
$\llbracket \bp_W(T_H),a_*(W^{(H,H')})\rrbracket _W$ and $\llbracket \bp_W(T_{H'}),\allowbreak a_*(W^{(H,H')})\rrbracket _W$ in $\cT_W$,
which by definition meet only at $a_*(W^{(H,H')})$, and their union is
the geodesic $\llbracket \bp_W(T_H),\bp_W(T_{H'})\rrbracket _W$ in~$\cT_W$, which is thus projected
via~$\pi_W$ to a simple path in~$\BP$. Therefore, $\gamma^{(H,H')}$ and $\xi^{(H,H')}$ meet only at
$x_\sas^{(H,H')}=\pi_W(a_*(W^{(H,H')}))$. 
The same reasoning shows that $\br\gamma^{(H,H')}$ and
$\br\xi^{(H,H')}$ intersect only at $\br x_\sas^{(H,H')}$, and gives that
the points $x_\sas^{(H,H')}$ and $\br x_\sas^{(H,H')}$ are distinct points
(because they are distinct points in $\cT_W$ lying inside two
geodesics). 

Next, if the path $\gamma^{(H,H')}$ does not intersect $\gamma_0=\pi_W(\llbracket \bp_W(0),\infty\cllbracket_W)$, then necessarily the path
$\llbracket \bp_W(T_H),a_*(W^{(H,H')})\rrbracket _W$ must be disjoint from
$\llbracket \bp_W(0),\infty\cllbracket_W$, which  means that
\[\un W(T_H,T_{H'}) > \un W(0,T_H)\qquad \text{ and }\qquad \un W(\br T_{H'},\br T_{H})> \un W(\br T_H, 0).\]
This implies that $\llbracket \bp_W(T_{H'}),a_*(W^{(H,H')})\cllbracket_W$ is also
disjoint from $\llbracket \bp_W(0),\infty\cllbracket_W$, and by projecting by $\pi_W$,
that $\xi^{(H,H')}$ is disjoint from~$\gamma_0$. A similar
argument applies to $\br \gamma^{(H,H')}$ and
$\br\xi^{(H,H')}$. Therefore, under the conditions of the claim, the
paths $\llbracket \bp_W(T_H),\allowbreak a_*(W^{(H,H')})\rrbracket _W$ and $\llbracket \bp_W(\br T_{H}),\br
a_*(W^{(H,H')})\rrbracket _W$ are disjoint paths in $\cT_W$, and their
projections $\gamma^{(H,H')}$ and $\br\xi^{(H,H')}$ via $\pi_W$
intersect, if at all, only at their 
extremities. It is indeed the case that $\bp(T_H)=\bp(\br
T_H)$, while, as we already saw, $x_\sas^{(H,H')}\neq \br
x_\sas^{(H,H')}$. This proves~\ref{Jordani}. 

The argument for~\ref{Jordanii} is similar to that in the proof of
Lemma~\ref{sec:brownian-half-plane-1}, where the role of the base is now played by the infinite path $\pi_X(\llbracket \bp_X(0),\infty\cllbracket_X )=\{\bp(T_h)\,:\, h\ge 0\}$. For any $t\in \R$, we let
\[\Sigma_t(r)=\inf\{s\geq t:X_s=X_t-r\}\qquad\text{ for }\qquad 0\leq
r\leq X_t-\dun{X}_t,\]
where we recall that $\dun{X}_t=\un{X}(0\wedge t,0\vee t)$. The range of $\bp_X\circ\Sigma_t$ is the geodesic path
$\llbracket \bp_X(t),\bp_X(T_{-\dun{X}_t}) \rrbracket_X$ in $\cT_X$ and its image by~$\pi_X$
defines a path $\sigma_t=\bp\circ\Sigma_t$ from~$\bp(t)$ to~$\bp(T_{-\dun{X}_t})$. Moreover, by Lemma~\ref{sec:brownian-plane-its-2}, the paths~$\sigma_t$, $t\in\R$, do not intersect any of the geodesics~$\gamma_s$, $s\in \R$, except
possibly at their starting points. There are now the following possibilities.
\begin{itemize}
	\item If $t\in I^{(H,H')}$, then~$\sigma_t$ ends on the path $(\bp(T_h),H\leq h\leq H')$. This means that, if~$\bp(t)$ does not belong to the four geodesics~$\gamma^{(H,H')}$, $\xi^{(H,H')}$, $\br \gamma^{(H,H')}$, $\br\xi^{(H,H')}$, then we may connect it to, say, the point $\bp(T_{(H+H')/2})$ of the bounded set $\Qd^{(H,H')}$, without crossing the four mentioned geodesics.
	\item If $t\notin I^{(H,H')}$, we distinguish two cases.
	\begin{itemize}
		\item If $t\notin[\br T_{H'},T_{H'}]$, then~$\sigma_t$ ends on the unbounded path $\{\bp(T_h):h> H\}$.
		\item If $t\in(\br T_H,T_H)$, then~$\sigma_t$ ends on $\{\bp(T_h):0\leq h<H\}$.
	\end{itemize}
If~$\bp(t)$ does not belong to the four geodesics of interest, then it may be joined without crossing the four geodesics either to the unbounded path $\{\bp(T_h):h> H\}$ or to the unbounded path $\{\bp(T_h):0\leq h<H\}\cup\gamma_0$, by the assumption that~$\gamma_0$ does not intersect the four geodesics.
\end{itemize}
This completes the proof of the claim.

\medskip
Now fix $H>0$, and consider another positive number $H_0$ to be
thought of as large. Since 
we know that $\Qd^{(H_0,H_0+H)}$ under $\Plane$ has same 
distribution as $\Qd^{(0,H)}$, we may work with the former space 
rather than with the latter. For every $\eps>0$ it holds that there
exists some $H_0$ large enough such that with probability at least
$1-\eps$, 
the geodesics $\gamma^{(H_0,H_0+H)}$ and
$\br{\xi}^{(H_0,H_0+H)}$ do not intersect $\gamma_0$. Indeed, this happens whenever
$\un W(T_{H_0}, T_{H_0+H}) > \un W(0,T_{H_0})$ or, equivalently,
\begin{equation}\label{WHHo}
W_{T_{H_0}}-\un W(0,T_{H_0})>W_{T_{H_0}}-\un W(T_{H_0}, T_{H_0+H}),
\end{equation}
and similarly in negative times. The two sides of~\eqref{WHHo} are independent by the
Markov property stated above; the right-hand side has a
distribution that depends only on~$H$, while the left-hand side,
which has same distribution as $-\un W(0,T_{H_0})$ by a simple time-reversal argument,                   
converges to $\infty$ in probability as $H_0\to\infty$.

By the claim, we obtain that on an event happening with probability at least $1-\eps$, the set
$\bp(I^{(H_0,H_0+H)})$ is the closure of the connected component of
the complement in~$\BP$ of the paths
\[\gamma^{(H_0,H_0+H)},\quad \xi^{(H_0,H_0+H)},\quad \br \gamma^{(H_0,H_0+H)},\quad \br\xi^{(H_0,H_0+H)},\]
which all together form a Jordan curve. On this event, the
identity\MEMS{}{ mapping} on $I^{(H_0,H_0+H)}$ induces, by precomposition
with the projection mappings~$\bp$ and $\bp^{(H_0,H_0+H)}$, a bijective
mapping~$\phi$ from the compact space $\Qd^{(H_0,H_0+H)}$ to
$\bp(I^{(H_0,H_0+H)})$, which is $1$-Lipschitz since 
$D_{X,W}\leq \dov{D}^{(H_0,H_0+H)}$ by Lemma~\ref{lemdistglue}, \eqref{eq:31} and Proposition~\ref{sec:brownian-plane-its}. This shows that~$\phi$ is a
homeomorphism, and therefore, with probability at least $1-\eps$, the
space $\Qd^{(H_0,H_0+H)}$ has the properties claimed in the
statement. Using the fact that $\Qd^{(H_0,H_0+H)}$ has same
distribution as $\Qd^{(0,H)}$ and that $\eps$ was arbitrary, we conclude. 
\end{proof}

The continuum quadrilaterals of the preceding section can be linked to
the free quadrilaterals embedded in the Brownian plane by an absolute
continuity argument, whose proof is similar to that of Lemma~\ref{lemaccontslice} and is omitted. 

\begin{lmm}\label{sec:conditioned-version-4}
Fix $0<K<H$, as well as $A>0$, $\bar{A}>0$, and $\Delta\in \R$. Then, for every 
nonnegative function~$G$ that is measurable with respect to the 
$\sigma$-algebra generated by $(X^{(0,K)},W^{(0,K)})$, one has 
\[
\Quad_{A,\bar{A},H,\Delta}[G]=\Plane\big[\psi_{A,\bar{A},H,\Delta}(T_K,-\br
T_K,K,W_{T_K})\cdot G\big],\]
where 
\[\psi_{A,\bar{A},H,\Delta}(A',\bar{A}',H',\Delta')=\frac{q_{H-H'}(A-A')}{q_H(A)}\frac{q_{H-H'}(\bar{A}-\bar{A}')}{q_H(\bar{A})}
\frac{p_{H-H'}(\Delta-\Delta')}{p_H(\Delta)}.\]
\end{lmm}

This allows to obtain, as stated in Lemma~\ref{topdimep}, the topology of quadrilaterals.

\begin{proof}[Proof of Lemma~\ref{topdimep} for quadrilaterals]
The proof is similar to that for slices. We use the fact
that the Brownian plane is topologically a plane
\cite[Proposition~13]{CuLG12Bplane}, as well
as~\cite[Proposition~4]{CuLG12Bplane} to obtain that it is locally of
Hausdorff dimension~$4$ from the analog result about the Brownian
sphere~\cite{legall05}. We deduce from there the desired
properties for a free quadrilateral.
To extend this result to quadrilaterals $\Qd_{A,\br{A},H,\Delta}$,
which we view as $\Qd^{(0,H)}$ under the law $\Quad_{A,\br
  A,H,\Delta}$, we use the fact from Proposition~\ref{gludetquad} that it
can be seen as the gluing of $\Qd^{(0,H/2)}$ and $\Qd^{(H/2,H)}$ along
the boundaries $\xi^{(0,H/2)}$ and $\gamma^{(H/2,H)}$ on the one hand,
and $\br \gamma^{(0,H/2)}$ and $\br\xi^{(H/2,H)}$ on the other
hand. By the absolute continuity relation stated in Lemma~\ref{sec:conditioned-version-4}, we see that the law of $\Qd^{(0,H/2)}$ is
absolutely continuous with respect to that of a free quadrilateral
with width $H/2$, and the same is true for $\Qd^{(H/2,H)}$. Using
Proposition~\ref{sec:brownian-plane-its-1}, we obtain that
$\Qd^{(0,H)}$ is obtained by gluing two topological disks, both
locally of Hausdorff dimension 4, along part
of their boundaries, which allows to conclude. 
\end{proof}

\subsection{The uniform infinite planar quadrangulation}\label{secUIPQ}

The UIPQ is the whole plane pendant of the UIHPQ defined in
Section~\ref{secUIHPQ}. It is simpler to describe and was introduced
earlier \cite{ChDu06,Krikun,CuMIMeUIPQ}. Let
$(\bT^k,k\in \Z)$ be a two-sided sequence
of independent Bienaym\'e--Galton--Watson trees with a geometric offspring
distribution of parameter~$1/2$. We construct an infinite tree~$\bT_\infty$ embedded in the plane by mapping the roots of~$\bT^k$ and of~$\bT^{-k}$ to the point $\rho^k=(k,0)$ for every $k\geq 0$, in such a way
that, except for these roots, the trees $\bT^k$, $k\geq 0$ are embedded in the
open upper half-plane and the trees $\bT^k$, $k<0$ are embedded in the
open lower half-plane, without intersection. Lastly, we link the roots~$\rho^k$, $\rho^{k+1}$ with a horizontal segment for every $k\geq 0$.

Conditionally on~$\bT_\infty$, we assign to the edges random numbers, independent and uniformly distributed in $\{-1,0,1\}$, and let $\lambda_\infty:V(\bT_\infty)\to\Z$ be the labeling function whose increments along the edges are given by these numbers. Note that this uniquely defines~$\lambda_\infty$, up to the usual addition of a constant. We call $(\bT_\infty,\lambda_\infty)$ the \emph{infinite random well-labeled tree}. We then let
$(c_i,i\in \Z)$ be the sequence of corners of~$\bT_\infty$ in contour
order, with origin the corner~$c_0$ corresponding to the root
of~$\bT^0$. The \emph{uniform infinite planar quadrangulation} (\emph{UIPQ} for short) is then the infinite map~$Q_\infty$ obtained by applying the CVS
construction to $(\bT_\infty,\lambda_\infty)$, that is, by linking
every corner to its successor as defined in
Section~\ref{sec:basic-construction}, and removing all edges of the
tree afterward. The root of~$Q_\infty$ is defined as the corner
preceding the arc from~$c_0$ to its successor. As with the UIHPQ,
there is no need to add an extra vertex with a corner~$c_\infty$.

Similarly as before, we denote by~$v_i$ the vertex of~$\bT_\infty$ incident to~$c_i$ and by~$\Upsilon(i)\in\Z$ the index of the tree to which~$v_i$ belongs. We then define the \emph{contour} and \emph{label processes} on~$\R$ by
\[C(i) = d_{\bT^{\Upsilon(i)}}\big(v_i,\rho^{|\Upsilon(i)|}\big)-|\Upsilon(i)|\qquad\text{ and }\qquad\Lambda(i)=\lambda_\infty(v_i)-\lambda_\infty(v_0),\qquad i \in\Z,\]
and by linear interpolation between integer values; see Figure~\ref{fig:contour2}. Observe that, in contrast with the definition of Section~\ref{secUIHPQ} for an infinite forest, there is an absolute value in the definition of~$C$. In fact, changing the~$-$ into a~$+$ amounts to taking the so-called \emph{Pitman transform}, which is a one-to-one mapping, so this is just a matter of convention. We will come back to this in Appendix~\ref{appBP}. We can easily check that~$C$ is distributed as a two-sided random walk conditioned\footnote{See Remark~\ref{remUIPQ} for the explanation of this conditioning.} on $C(-1)=-1$.

\begin{figure}[ht!]
  \centering
  \includegraphics[width=.95\textwidth]{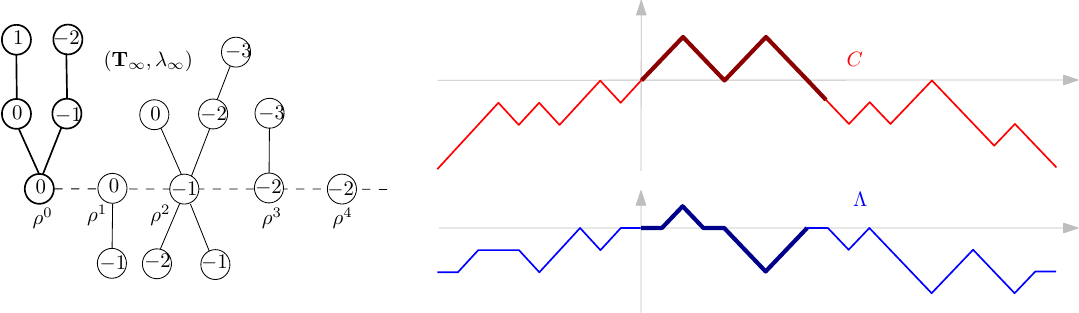}
  \caption{Contour and label processes associated with
    $(\bT_\infty,\lambda_\infty)$. The infinite dashed line is the so-called \emph{spine} of the tree. The tree~$\bT^0$ and the corresponding encoding processes are highlighted. Similarly as on Figure~\ref{contour}, one might see the contour process as recording the height of a particle moving at speed one around the infinite tree obtained by now letting~$\rho^k$ be located at $(0,-k)$ with~$\bT^k$ grafted on its right and~$\bT^{-k}$ on its left (both upright), for $k\ge 0$; see the left of Figure~\ref{fig:contourPit} for an illustration.}
  \label{fig:contour2}
\end{figure}

As before, we extend~$C$ and~$\Lambda$ to functions on~$\R$ by linear interpolation between integer values. For $k\ge 0$, we set
\[
\br\tau_k=\max\big\{i\leq 0\,:\,C(i)=-k\big\}\qquad\text{ and }\qquad \tau_k=\min\big\{i\geq 0\,:\,C(i)=-k\big\}.
\]
Note that, for a fixed $k\geq 0$, the process $(k+C(s+\tau_k), 0\leq s\leq \tau_{k+1}-\tau_k)$
is the contour process of~$\bT^k$, while, for $k\geq 1$, $(k+C(s+\br\tau_{k+1}+1),0\leq s\leq \br\tau_k-\br\tau_{k+1}-1)$ is the contour
process of $\bT^{-k}$ without the last descending step. Therefore, in
this notation, the forest composed of the~$k$ leftmost trees in the
upper half-plane is coded by the interval $[0,\tau_{k}]$,
while the forest composed of the~$k$ leftmost trees in the lower
half-plane is coded by the interval $[\br
\tau_{k+1}+1,0]$. This slightly annoying shift will appear later on,
in particular in the statement of Lemma~\ref{sec:conditioned-version}. 

\begin{rem}\label{remUIPQ}
As with the UIHPQ, the above definition gives a slight
variant of the usual UIPQ, which is similarly defined by
adding a further tree rooted at~$\rho^0$ embedded in the lower
half-plane, or equivalently, by removing 
the conditioning by $\{C(-1)=-1\}$. This bias is similar to the one we had for the UIHPQ. 
Here again, the reason for using this definition is that it will give the natural semigroup property for the discrete quadrilaterals.
\end{rem}

We set
\begin{equation}
  \label{eq:14}
  D_\infty(i,j)=d_{Q_\infty}(v_i,v_j),\qquad i,j\in \Z,
\end{equation}
and extend it to a function on~$\R^2$ by bilinear
interpolation between integer values, as in~\eqref{bilininterpo}. We
define the renormalized versions~$C_{(n)}$, $\Lambda_{(n)}$, $D_{(n)}$
of~$C$, $\Lambda$, $D_\infty$ by~\eqref{rescale}. The following proposition builds on the convergence obtained in~\cite{CuLG12Bplane} of the UIPQ to the Brownian plane. As was the case for the UIHPQ, it does not appear in this exact form in~\cite{CuLG12Bplane} and calls for a proof, which is postponed to Appendix~\ref{appBP}. Recall from Section~\ref{secBP} the definition of the distribution~$\Plane$.

\begin{prp}\label{cvencbp}
The following convergence in distribution holds on $\CC\times\CC\times\CC^{(2)}$:
\begin{equation*}
  \big(C_{(n)},\Lambda_{(n)},D_{(n)}\big)\tod\big(X,W,D_{X,W}\big),
\end{equation*}
where the limiting triple is understood under~$\Plane$.
\end{prp}

\subsection{Discrete quadrilaterals in the UIPQ}

We proceed as in the last paragraph of Section~\ref{secUIHPQ}. But, here, the lack of an analog of Corollary~\ref{sec:metric-spaces-coded-2} makes matter substantially more intricate. We consider a sequence $(h_n)\in\N^\N$ such that
\[\frac{h_n}{\sqrt{2n}}\ton H>0.\]
For each~$n$, we let $F_n$ be the random forest consisting of the~$h_n$ trees~$\bT^0$, $\bT^1$, \dots, $\bT^{h_n-1}$, and~$\rho^{h_n}$, as well as~$\br F_n$ be the random forest consisting of the~$h_n$ trees~$\bT^{-h_n}$, $\bT^{-h_n+1}$, \dots, $\bT^{-1}$ and~$\rho^{0}$. The pair $(F_n,\br F_n)$ is a double forest in the terminology of Section~\ref{sec:quadr-with-geod} and the map $F_n\cup \br F_n$ is well labeled by the restriction of~$\lambda_\infty$. We denote by~$Q_n$ the corresponding quadrilateral and by~$v_*$, $\br v_*$ its apexes; similarly to the previous section, we see it as part of the UIPQ~$Q_\infty$.

For each $i\in \Z$, the vertex~$v_i$ of~$\bT_\infty$ incident to~$c_i$ can still be seen as a vertex of~$Q_n$ when $\br\tau_{h_n+1}+1\le i\le \tau_{h_n}$. We set 
\begin{equation*}
\dov{D}_n(i,j)=d_{Q_n}(v_i,v_j),\qquad \br\tau_{h_n+1}+1\le i,j\le \tau_{h_n},
\end{equation*}
extend it to a function  on~$[\br\tau_{h_n+1}+1,\tau_{h_n}]^2$ by bilinear interpolation between integer values as in~\eqref{bilininterpo}, and define its renormalized version
\begin{equation}\label{eqdovDn}
\dov{D}_{(n)}(s,t)=\frac{\dov{D}_n(2ns, 2nt)}{(8n/9)^{1/4}},\qquad \frac{\br\tau_{h_n+1}+1}{2n}\le s,t\le \frac{\tau_{h_n}}{2n}.
\end{equation}

This section is devoted to the proof of the following result, which essentially amounts to stating that, jointly with the convergence of $\Omega_n(Q_\infty)$ to the Brownian plane, the properly rescaled quadrilateral $\Omega_n(Q_n)$ converges to~$\Qd^{(0,H)}$. 

\begin{thm}\label{thmddstar}
The following convergence in distribution holds in $\CC\times 
\CC\times \CC^{(2)}\times \CC^{(2)}$: 
\begin{equation*}
  \big(C_{(n)},\Lambda_{(n)},D_{(n)},\dov{D}_{(n)}\big)\tod\big(X,W,D_{X,W},\dov{D}^{(0,H)}\big), 
\end{equation*}
where the limiting quadruple is understood under~$\Plane$.
\end{thm}

The first step in the proof is the following tightness statement. 

\begin{lmm}\label{cvencquad}
From every increasing sequence of integers, one may extract a
subsequence along which the following convergence holds in~$\CC^{(2)}$, jointly with the convergence of Proposition~\ref{cvencbp}:
\begin{equation}
  \label{eqextracquad}
  \dov{D}_{(n)}\tod \wt{D},
\end{equation}
where $\wt D$ is a random pseudometric on~$[\br T_H,T_H]$.
\end{lmm}

\begin{proof}
The classical tightness argument from \cite[Proposition~3.2]{legall06} implies
that the laws of $\dov{D}_{(n)}$, $n\geq 1$, are 
tight in $\CC^{(2)}$. Together with Proposition~\ref{cvencbp}, this yields the tightness
of the laws of the sequence of the quadruples
$(C_{(n)},\Lambda_{(n)},D_{(n)},\dov{D}_{(n)})$, and therefore, by
Prokhorov's theorem, their joint convergence in distribution, at least
along some subsequence, to a limiting process $(X,W,D_{X,W},\wt{D})$,
where the law of the first three components is determined by
Proposition~\ref{cvencbp}. 
Since $\dov{D}_{(n)}$ is a
pseudometric on $[(\br\tau_{h_n+1}+1)/2n,\tau_{h_n}/2n]$, and because
of the convergence of~$C_{(n)}$ to~$X$ implying the joint convergence
of the bounds of this interval to~$\br T_H$, $T_H$, 
it is straightforward to check that all
subsequential limits of these laws are carried by functions that are
pseudometrics on the interval $[\br T_H,T_H]$. 
\end{proof}

From now on, we fix
a subsequence along which~\eqref{eqextracquad} holds, and only consider
for the time being values of~$n$ that belong to this particular
subsequence. By the Skorokhod representation theorem, we may and will
assume that this convergence furthermore holds almost surely.

We define $\wt{\Qd}$ as the set $[\br T_H,T_H]/\{\wt{D}=0\}$, endowed
with the metric~$\wt{D}$. Beware that it is not clear
at all that $\wt{\Qd}=\Qd^{(0,H)}$, and
this is precisely what we want to prove. More precisely, we aim at
showing that, almost surely, for every $s$, $t\in [\br T_H,T_H]$, it
holds that 
$\wt{D}(s,t)=\dov{D}^{(0,H)}(s,t)$, which will entail Theorem~\ref{thmddstar}. 

Since the real number~$H$ is fixed once and for all, we will use in the remainder of this section 
the shorthand pieces of notation
\[\dov{D}= \dov{D}^{(0,H)}\qquad\text{ as well as }\qquad D=D_{X,W}.\]
We let $\bp:\R\to \BP$ and $\wt{\bp}:[\br T_H,T_H]\to \wt{\Qd}$ be the
canonical projections, which are continuous since~$D$ and $\wt{D}$ are
continuous functions. Note that, clearly, for every~$n$, it holds that
$ D_\infty\le \dov D_n$ on $[\br\tau_{h_n+1}+1, \tau_{h_n}]^2$, so
that $ D\le \wt D$ on~$[\br T_H,T_H]$. As a result, there exists a unique continuous (even $1$-Lipschitz) projection
$\pi:\wt{\Qd} \to \bp([\br T_H,T_H])$ such that $\bp=\pi\circ\wt{\bp}$
on~$[\br T_H,T_H]$.  

The inequality $\wt{D}\le\dov{D}$ follows from the usual
following arguments. First we come back to discrete maps and observe
that, for integers~$i$, $j\in [\br\tau_{h_n+1}+1, \tau_{h_n}]$, we have $d_{C}(i,j)=0$ if and only if~$v_i$
and~$v_j$ are the same vertex of $F_n\cup\br F_n$, which implies that
$\dov D_n(i,j)=0$. Next, by considering the so-called \emph{maximal
  wedge path} consisting of the concatenation of the two geodesics
from~$c_i$ and from~$c_j$ obtained by following subsequent successors
up to the point where they coalesce, we obtain the classical upper
bound similar to~\eqref{dd0bound}:
\begin{equation}\label{dd0boundquad}
\dov D_n(i,j)\le  d_{\Lambda}(i,j)+2,\qquad i,j\in [\br\tau_{h_n+1}+1, \tau_{h_n}]\ \text{ with }\ ij\ge 0.
\end{equation}
Passing to the limit yields that $\{d_{X}=0\}\subseteq\{\wt{D}=0\}$ and that $\wt{D}\le \dov{d}_{W}$, which imply the inequality $\wt{D}\le\dov{D}$. The converse inequality is harder and is the focus of what follows.

Let us start with some key properties of the pseudometrics~$D$, $\wt{D}$,
and~$\dov{D}$. The following lemma is proved in the exact same way as \cite[Lemma~14]{BeMi17}.

\begin{lmm}\label{lemgeodspaces}
The spaces $\wt{\Qd}$ and $\Qd^{(0,H)}$ are compact geodesic metric spaces. 
\end{lmm}

We will need the
following identification of the set $\{\wt{D}=0\}$, analog to Lemma~\ref{sec:brownian-plane-its-2}. 

\begin{lmm}\label{lemddovd}
The following holds almost surely. For every~$s$, $t\in [\br T_H,T_H]$ with $s\neq t$, it
holds that $\wt{D}(s,t)=0$ if and only if either $d_X(s,t)=0$ or
$\dov{d}_W(s,t)=0$, these two cases being mutually
exclusive.
\end{lmm}

\begin{proof}
It follows very similar lines to that of
Proposition~3.1 in~\cite{legall11}, and we will only sketch the main
arguments. The fact that $d_X(s,t)=0$ or $\dov d_W(s,t)=0$ implies
$\wt{D}(s,t)=0$ is immediate from the inequality $\wt{D}\leq
\dov{D}$. Conversely, assume that $\wt{D}(s,t)=0$ for some $s\neq t$
in $[\br T_H,T_H]$. Then, in particular,
since $D\leq \wt{D}$, it holds that $D(s,t)=0$, so that either
$d_W(s,t)=0$ or $d_X(s,t)=0$, and these two cases are exclusive. If
we are in the case that~$s$, $t$ are of the same sign and that $d_W(s,t)=0$, this trivially
implies $\dov d_W(s,t)=0$, as wanted. And since $\dov d_W\ge d_W$, it cannot hold that $\dov d_W(s,t)=d_X(s,t)=0$ at the same time. Hence, the proof will be complete if we can show that the situation where $d_W(s,t)=0$ necessarily implies that~$s$ and~$t$ are of the same 
sign.

For this, we argue by contradiction, assuming that $t<0<s$ and
$d_W(s,t)=0$.
Note that this implies in particular that~$s$, $t$ lie on some point of
the geodesics~$\Gamma_0$ and~$\Xi_0$, respectively, meaning that
$W_s=\inf_{u\in [0,s]}W_u$ and $W_t=\inf_{u\in [t,0]}W_u$. 
Then, by the convergence of $\Lambda_{(n)}$ to $W$, 
there exist $i_n\in [0,\tau_{h_n}]$ and $j_n\in [\br
\tau_{h_n+1}+1,0]$ such that $i_n/2n\to s$ and $j_n/2n\to t$, with
the property that $\Lambda_n(i_n)=\min_{k\in [0,i_n]}\Lambda_n(k)$
and $\Lambda_n(j_n)=\min_{k\in [j_n,0]}\Lambda_n(k)$. This means
that $v_{i_n}$ lies on the maximal geodesic~$\gamma_n$ of~$Q_n$, and~$v_{j_n}$ lies at distance~$1$ from the shuttle~$\br\xi_n$ of~$Q_n$.

Now any geodesic path in~$Q_n$ from~$v_{i_n}$ to~$v_{j_n}$ will necessarily
intersect the spine of the tree $\bT_\infty$ at some tree root
$\rho^{l_n}$ with $0\leq l_n\leq h_n$. 
Let $k_n\in [\br\tau_{h_n+1}+1,\tau_{h_n}]$ be an integer such
that $v_{k_n}=\rho^{l_n}$. In terms of the contour process $C_n$,
this means that $C_n(k_n)\leq C_n(l)$ for every $l\in [0\wedge
k_n,0\vee k_n]$. Up to extracting along a further subsequence, we
may assume that $k_n/2n\to u\in [\br T_H,T_H]$ as $n\to\infty$, and
we observe that $u$ must be such that $X_u\leq X_t$ for every $t\in
[0\wedge u,0\vee u]$, and in particular, we observe that
$d_X(u,T_{H'})=d_X(u,\br T_{H'})=0$ where $H'=-X_u$. We may exclude
the case where $H'=0$ by noting that, necessarily, $W_s=W_t=W_u<0$. 

On the other hand, since $v_{k_n}$ lies on a
geodesic path from $v_{i_n}$ to $v_{j_n}$, which has length
$\sO(n^{1/4})$ because of our assumption that $\wt{D}(s,t)=0$, it
holds that $\wt{D}(s,u)=\wt{D}(u,t)=0$. We arrive at the wanted
contradiction since we have found four points $s\neq t$, $T_{H'}\neq \br
T_{H'}$ that are all identified by $D$ but such that $d_W(s,t)=0$ and
$d_X(T_{H'},\br T_{H'})=0$. 
\end{proof}

As $\wt{D}\le \dov{D}\le\dov{d}_W$ and $\{d_X=0\}\subseteq
\{\dov{D}=0\}$, Lemma~\ref{lemddovd} implies that the
equivalence relations $\{\wt{D}=0\}$ and $\{\dov{D}=0\}$ coincide, and that $\wt\bp=\dov{\bp}^{(0,H)}$. For this reason we may, and will, systematically
identify points of~$\wt{\Qd}$ with points of~$\Qd^{(0,H)}$. 
Moreover, the identity mapping $\Qd^{(0,H)}\to \wt{\Qd}$ is
continuous, and by compactness of these spaces, 
we conclude that $\wt{\Qd}$ is homeomorphic to~$\Qd^{(0,H)}$.

Theorem~\ref{thmddstar} will be obtained by compactness and continuity arguments from the following local version, stating that, locally and away either from both maximal geodesics or from both shuttles, the three distances under consideration are equal. The proof of the following lemma can straightforwardly be adapted from~\cite[Lemma~15]{BeMi17}, so that we only sketch it and refer the reader to the latter reference for the details. In an arbitrary pseudometric space $(M,d)$, we denote by $d(x,A)=\inf\{d(x,y):y\in A\}$ the distance from a point $x\in M$ to a subset $A\subseteq M$.

\begin{lmm}\label{lemloc}
The following holds almost surely. Fix $\eps>0$, and let $s$, $t\in
[\br T_H,T_H]$ be such that $\wt{D}(s,t)<\eps$ and
\begin{itemize}
	\item either $\wt{D}\big(s,\Gamma_0\cup\Gamma_{\br
            T_H}\big)\wedge \wt{D}\big(t,\Gamma_0\cup\Gamma_{\br T_H}\big)> \eps$\,;
	\item or $\wt{D}\big(s,\Xi_0\cup\Xi_{T_H}\big)\wedge \wt{D}\big(t,\Xi_0\cup\Xi_{T_H}\big)> \eps$.
\end{itemize}
Then, it holds that $D(s,t)=\wt{D}(s,t)=\dov{D}(s,t)$. 
\end{lmm}

\begin{proof}
Let $i_n$, $j_n$ be integers in $[\br\tau_{h_n+1}+1,\tau_{h_n}]$ such that
$i_n/2n\to s$ and $j_n/2n\to t$ as $n\to\infty$. From the assumption
that $\wt{D}(s,t)<\eps$ and the convergence of $D_{(n)}$ toward
$\wt{D}$, we deduce that $d_{Q_n}(v_{i_n},v_{j_n})<\eps(8n/9)^{1/4}$
for every $n$ large enough. 

Next, keeping the same notation, assume that we are in the first
alternative of the statement. Then we claim that
for every $n$ large enough, $v_{i_n}$ and $v_{j_n}$ must be at
$d_{Q_n}$-distance at least $\eps(8n/9)^{1/4}$ from the maximal geodesics
$\gamma_n$ and $\br\gamma_n$ of $Q_n$. Indeed, if we assume otherwise,
then up to taking an extraction along a further subsequence, we would
find a point $k_n\in [\br\tau_{h_n+1}+1,\tau_{h_n}]$ such that for every
$n$, $v_{k_n}$
belongs to (the same) one of these maximal geodesics, and is at $d_{Q_n}$-distance at most
$(8n/9)^{1/4}$ from (the same) one of two points $v_{i_n}$ or $v_{j_n}$. To fix the ideas,
assume that $v_{k_n}$ is on $\gamma_n$ and is close to $v_{i_n}$ in
the latter sense, the discussion being similar in the other cases.
Up to taking yet another subsequence if necessary, we may assume that
$k_n/2n$ converges to some $u\in [\br T_H,T_H]$. Note that $k_n$,
being a time of visit of the maximal geodesic $\gamma_n$, must be a
left-minimum for the label process $\Lambda_n$ restricted to nonnegative times, and, by passing to the limit,
$u$ must be a left-minimum of~$W$ restricted to nonnegative times, entailing
that $\wt{D}(u,\Gamma_0)=0$. Therefore, by passing to the limit in the
inequality $D_{(n)}(i_n/2n,k_n/2n)\leq \eps$, we would obtain that
$\wt{D}(s,\Gamma_0)\leq \eps$, a contradiction with our assumption.

Now observe that 
$Q_\infty$ is obtained by the following two gluing operations, from~$Q_n$ and the infinite quadrangulation~$Q_n^{\mathrm c}$, encoded by the labeled double forest with trees grafted above~$\rho^{h_n+i}$, $i\geq 0$, and below~$\rho^{h_n+i}$, $i\geq 1$.
 
\begin{itemize}
\item First, by gluing the geodesic sides~$\xi_n$ and~$\br\gamma_n$ of~$Q_n$ to the (unique) maximal geodesic and shuttle of~$Q_n^{\mathrm c}$. Note that the resulting infinite quadrangulation is also obtained by performing the interval CVS construction on~$\bT_\infty$ with the intervals $\{c_i,\, i\le 0\}$ and $\{c_i,\, i\ge 0\}$.

\item Second, by gluing together the (unique) maximal geodesic and shuttle of the infinite quadrangulation obtained at the first step. Note that the geodesic sides of this infinite map are prolongations of~$\gamma_n$ and~$\br\xi_n$.
\end{itemize}
Therefore, Lemma~\ref{lemdistglue}.\ref{distglueii} applied twice (once for each gluing
operation) shows that if $v$, $w\in V(Q_n)$ are such that
$d_{Q_n}(v,w)<K$ and either
\begin{itemize}
\item $d_{Q_n}(v,\gamma_n)\wedge d_{Q_n}(w,\gamma_n)>K$,
\item or $d_{Q_n}(v,\br\gamma_n)\wedge d_{Q_n}(w,\br\gamma_n)>K$, 
\end{itemize}
then $d_{Q_n}(v,w)=d_{Q_\infty}(v,w)$. Applying this to $v=v_{i_n}$, $w=v_{j_n}$,
and $K=(8n/9)^{1/4}\eps$ yields, after passing to the limit, that
$\wt{D}(s,t)=D(s,t)$. Since $\wt{D}\leq \dov D\leq D$, this yields the
result in the first alternative of the statement.  The second case,
with shuttles instead of maximal geodesics, is similar.
\end{proof}

We may finally prove Theorem~\ref{thmddstar}.

\begin{proof}[Proof of Theorem~\ref{thmddstar}]
We follow the same lines as in the proof of \cite[Theorem~11]{BeMi17}. As we observed before, the metric spaces $\wt{\Qd}$
and $\Qd^{(0,H)}$ are homeomorphic. Therefore, since the geodesics~$\gamma$ and~$\br\gamma$ do not intersect in $\Qd^{(0,H)}$, the same
is true in~$\wt{\Qd}$, and similarly, the geodesics~$\xi$ and~$\br\xi$
do not intersect in these spaces. Moreover, as we know, these four
geodesics intersect only at~$\gamma(0)=\br\xi(0)$, $\xi(0)=\br{\gamma}(0)$, $x_\sas^{(0,H)}$ and~$\br x_\sas^{(0,H)}$. Therefore, for every
$x\in \wt{\Qd}\setminus \big\{\gamma(0),\br\gamma(0),x_\sas^{(0,H)},\br
x_\sas^{(0,H)}\big\}$, there exists $\eps>0$ such that the open ball
$B_{\wt{D}}(x,\eps)$ of
radius~$\eps$ around~$x$ for the metric~$\wt{D}$ intersects neither $\gamma\cup \br\gamma$ nor $\xi\cup\br\xi$. By Lemma~\ref{lemloc}, this implies that the balls $B_{\wt{D}}(x,\eps)$,
and $B_{\dov{D}}(x,\eps)$ are isometric. Hence, $\wt{\Qd}$ and
$\Qd^{(0,H)}$ are two compact geodesic metric spaces that are locally
isometric except possibly around four points. Therefore, the lengths
of paths that do not go through these four points must be the same in
both spaces. It is then easy to see that the same is true for all
paths that visit each of these four points at most once, by splitting into
subpaths, and by standard properties of lengths of paths.
One concludes by observing that, given a path in~$\wt{\Qd}$, one may
construct another path of length smaller than or equal to that of the
initial path, and that visits each of the four distinguished points at
most once. Since a geodesic space is a length space \cite{burago01},
the distance between two points is given by the infimum of length of
paths between these points. Therefore, $\wt{\Qd}$ and $\Qd^{(0,H)}$
are isometric. 
\end{proof}

\subsection{Scaling limit of conditioned quadrilaterals}\label{sec:conditioned-version-2}

In this section, we finally prove Theorem~\ref{thmslquad}. As a preliminary result, we will need a simple estimate on distances
in quadrilaterals. We invite the reader to recall the combinatorial setting of
Section~\ref{sec:quadr-with-geod} and to consult Figure~\ref{trunc}. Let $((\bff,\bar{\bff}),\lambda)$ be a well-labeled double forest and let~$\bq$ be the corresponding quadrilateral. For $k\in\{1,2,\ldots,h-1\}$, keeping only the \textgras{first}~$k$ trees in~$\bff$ and
the \textgras{last}~$k$ trees in~$\bar{\bff}$ yields a submap of~$\bff\cup\br\bff$, well labeled by the restriction of~$\lambda$. We let~$\bq_k$ be the corresponding quadrilateral, which we naturally see as a submap of~$\bq$. We will need the following coarse comparison between distances in~$\bq$ and~$\bq_k$.

\begin{figure}[ht!]
	\centering\includegraphics[width=.95\linewidth]{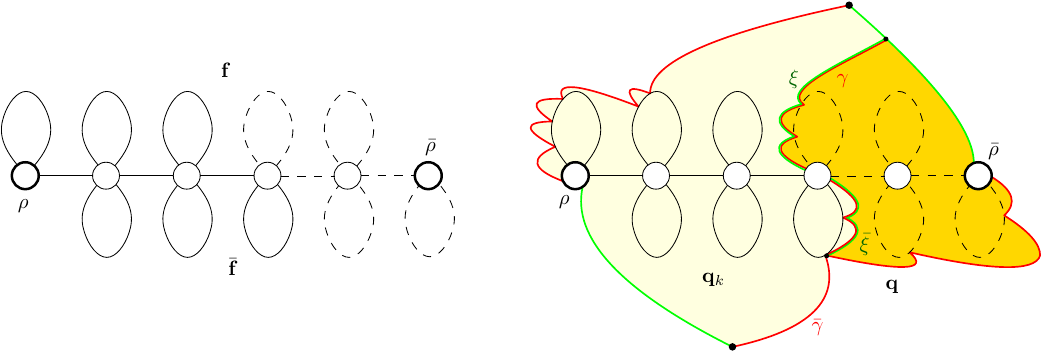}
	\caption{Here, $h=5$ and $k=3$. On the left, a schematic
          picture of a double forest $(\bff,\br\bff)$, assumed to be
          well labeled, and its ``truncation'' obtained after removing
          the dashed elements. On the right, a schematic picture of
          the corresponding quadrilaterals: the quadrilateral~$\bq$ is
          obtained by gluing~$\bq_k$ (in light yellow) along its
          sides~$\xi$ and~$\br\gamma$ with the quadrilateral (in dark yellow) coded by the dashed elements along its sides~$\gamma$ and~$\br\xi$ (only these four geodesic sides of interest are named in the picture).}
	\label{trunc}
\end{figure}

\begin{lmm}\label{sec:conditioned-version-1}
Let $\varpi=2+\max\{\lambda(u):u\in V(\bq)\setminus
V(\bq_k)\}-\min\{\lambda(u):u\in V(\bq)\setminus V(\bq_k)\}$. Then, for any~$v$, $w\in V(\bq_k)$, one has
\[d_\bq(v,w)\le d_{\bq_k}(v,w)\le d_\bq(v,w)+\varpi.\]
\end{lmm}

\begin{proof}
Observe that~$\bq$ may be obtained by
gluing~$\bq_k$ along its sides~$\xi$ and~$\br\gamma$ with the quadrilateral coded by the double forest obtained by taking the last $h-k$ trees in~$\bff$ and
the first $h-k$ trees in~$\bar{\bff}$, well labeled by the restriction of~$\lambda$, along its sides~$\gamma$ and~$\br\xi$. The lemma is then a straightforward consequence of Lemma~\ref{lemdistglue}.\ref{distgluei} since the lengths of the glued geodesics are bounded by the quantity~$\varpi$.
\end{proof}

We now prove Theorem~\ref{thmslquad} by proceeding similarly as in Section~\ref{secslcond}. Recall the
notation $(C(t),\Lambda(t),t\in \R)$, $\tau_k$, $\br\tau_k$ from Section~\ref{secUIPQ}. Let $\P_\infty$ be the law of $(C,\Lambda)$ and assume
without loss of generality that the latter is the canonical
process. Although we use the same notation~$\P_\infty$ as
in Section~\ref{secslcond}, we believe that there is
little risk of confusion. For $j\geq 1$, let~$\FF_j$ be the
$\sigma$-algebra generated by $(C(i),\Lambda(i),0\le i\le j)$, and
let~$\GG_j$ be the one generated by $(C(i),\Lambda(i+1),-j\le i\le -1)$. Note that
$\FF_{\tau_h}$ is the $\sigma$-algebra generated by the~$h$ leftmost
trees of~$\bT_\infty$ in the upper half-plane, together
with their labels, as well as the label of the root~$\rho^h$. Similarly, 
$\GG_{-\br\tau_{h+1}}$ is the $\sigma$-algebra generated by the~$h$
leftmost trees of~$\bT_\infty$ in the lower half-plane, together with their
labels, as well as the label of the root~$\rho^0$: the meaning of the shift by $+1$ in the process $\Lambda$ is
that we do not want to incorporate the information of the label of the
root $\rho^{h+1}$ in $\GG_{-\br\tau_{h+1}}$. 

Next, for~$a$, $\br a$, $h\in \N$ and
$\delta\in \Z$, we denote by $\P_{a,\br a,h,\delta}$ the distribution of
\[\big(C\rst_{[-2\br a-h,2a+h]},\Lambda\rst_{[-2\br a-h,2a+h]}\big)\]
under $\P_\infty[\;\cdot\mid\tau_h=2a+h,\,\br \tau_{h+1}+1=-2\br a-h,\,\Lambda(\tau_h)=\delta]$. The
corresponding double forest encoded by this\MEMS{}{ random} process is thus composed
of a spine~$\rho^0$, \dots, $\rho^h$ of length~$h$, on which are grafted $2h$ Bienaym\'e--Galton--Watson trees with Geometric($1/2$)
offspring distribution and uniform admissible labels, conditioned on
the fact that the total number of edges in the upper half-plane~$h$ trees is~$a$,
the number of edges in the lower half-plane~$h$ trees is $\br a$, and the
label of the last root $\rho^h$ is~$\delta$. The following lemma gives an absolute continuity relation between the
laws $\P_{a,\br a,h,\delta}$ and~$\P_\infty$. Its proof, which we
omit, is similar to that of Lemma~\ref{lemacslice}, using the enumeration results of Proposition~\ref{countquad}. Recall from Propositions~\ref{countslice} and~\ref{countquad} the definitions of~$Q_\ell$ and~$M_\ell$.

\begin{lmm}\label{sec:conditioned-version}
Fix the integers $0<k<h$, as well as positive integers $a$, $\bar{a}\in\N$
and $\delta\in \Z$. 
For every nonnegative functional~$G$ that is $\FF_{\tau_k}\vee \GG_{-\br\tau_{k+1}}$-measurable, we
have
\[\E_{a,\bar{a},h,\delta}[G]=\E_{\infty}\big[\Psi_{a,\bar{a},h,\delta}(\tau_k,-(\br\tau_{k+1}+1),k,\Lambda(\tau_k))\cdot
G\big],\]
where
\[\Psi_{a,\bar{a},h,\delta}(s,t,h',j)=
	\frac{Q_{h-h'}(2a+h-s)}{Q_h(2a+h)}
	\frac{Q_{h-h'}(2\bar{a}+h-t)}{Q_h(2\bar{a}+h)}
	\frac{M_{h-h'}(\delta-j)}{M_h(\delta)}
.\]
\end{lmm}

From now on, in addition to the sequence~$(h_n)$, we fix three sequences $(a_n)$, $(\bar{a}_n)$, $(\delta_n)$ as in~\eqref{2anhndn}. The following is a tedious but
straightforward consequence of the local limit theorem~\cite[Theorem~8.4.1]{BGT}.

\begin{lmm}\label{lltquad}
If the integer-valued sequence $(h'_n)$ satisfies
$h'_n/\sqrt{2n}\to H'\in (0,H)$, then
\[\sup_{\substack{0\leq s\leq a_n\\[2pt]0\leq t\leq\br a_n},\,j\in \Z}\left|\Psi_{a_n,\br
    a_n,h_n,\delta_n}(s,t,h'_n,j)-\psi_{A,\br
    A,H',\Delta}\Big(\frac sn,\frac tn,H',\Big(\frac 9{8n}\Big)^{\frac14}j\Big)\right|\ton 0.\]
\end{lmm}

We proceed to the conditioned version of Theorem~\ref{thmddstar}. Recall the definition of~$\dov{D}_{(n)}$ given in~\eqref{eqdovDn}.

\begin{prp}\label{cvencquadcond}
On $\CC\times\CC\times\CC^{(2)}$, the triple $\big(C_{(n)},\Lambda_{(n)},\dov{D}_{(n)}\big)$ considered under the distribution $\P_{a_n,\br a_n,h_n,\delta_n}$ converges in distribution to $\big(X,W,\dov{D}_{X,W}=\dov{D}^{(0,H)}\big)$, considered under~$\Quad_{A,\br A,H,\Delta}$.
\end{prp}

\begin{proof}
The arguments are very close to those used in the proof of Proposition~\ref{cvencslice} in Section~\ref{secslcond}, adding Lemma~\ref{sec:conditioned-version-1} and Proposition~\ref{boundgluquad} to cover the additional difficulty. The joint convergence of the first two coordinates is also standard. Then, fix $\eps\in (0,H)$ and set $h^\eps_n=h_n-\lfloor \eps
\sqrt{2n}\rfloor$. Let~$\dov{D}^\eps_{n}$ and~$\dov{D}^\eps_{(n)}$ be defined as in~\eqref{eqdovDn} and above, but with~$h_n^\eps$ instead of~$h_n$. For simplicity, for every $i\in \R$, let
\[i^\eps=(\br\tau_{h^\eps_n+1}+1)\vee i\wedge \tau_{h^\eps_n}\]
and define~$j^\eps$ similarly for any $j\in\R$. Define also
$\kappa^\eps_n=(\tau_{h_n}-\tau_{h^\eps_n})+ 
(\br\tau_{h^\eps_n+1}-\br\tau_{h_n+1})$. From~\eqref{dd0boundquad}, we obtain 
\begin{align*}
\big|\dov{D}_n(i,j)-\dov{D}_{n}(i^\eps,j^\eps)\big|&\leq
\dov{D}_n(i,i^\eps)+\dov{D}_n(j,j^\eps)
\leq 4\big(\omega(\Lambda_n;\kappa^\eps_n)+1\big).
\end{align*}
Using Lemma~\ref{sec:conditioned-version-1}, we have for every $i$, $j\in
[\br\tau_{h^\eps_n+1}+1,\tau_{h^\eps_n}]$,
\begin{equation*}
\big|\dov{D}_{n}(i,j)-\dov{D}_n^\eps(i,j)\big|\leq \omega(\Lambda_n; \kappa^\eps_n)+2.
\end{equation*}
These two facts together then imply that 
\begin{equation*}
\dist_{\CC^{(2)}}\big(\dov{D}^\eps_{(n)},\dov{D}_{(n)}\big)
  \leq \frac{\kappa^\eps_n}{2n}
  +5\omega\big(\Lambda_{(n)},\kappa^\eps_n/2n\big)
  +\bO\big(n^{-1/4}\big).
\end{equation*}
We now use the convergence of the first two coordinates, implying, for every $\eta>0$, 
\begin{equation}
  \label{eq:17}
  \limsup_{n\to\infty}\P_{a_n,\bar{a}_n,h_n,\delta_n}\left(\dist_{\CC^{(2)}}\big(\dov{D}^\eps_{(n)},\dov{D}_{(n)}\big)\geq
  \eta\right)\leq
\Quad_{A,\bar{A},H,\Delta}\big(\kappa^\eps+5\omega(W;\kappa^\eps)\geq \eta\big),
\end{equation}
where $\kappa^\eps=A-T_{H-\eps}+\br T_{H-\eps}+\bar{A}$. Since a.s.\ under
$\Quad_{A,\bar{A},H,\Delta}$, the quantity $\kappa^\eps$ tends to $0$ as
$\eps\to 0$, we
deduce that the left-hand side in~\eqref{eq:17} also converges to~$0$. It
remains to show that $\dov{D}^\eps_{(n)}$ under
$\P_{a_n,\bar{a}_n,h_n,\delta_n}$ converges toward
$\dov{D}^{(0,H-\eps)}$ under $\Quad_{A,\bar{A},H,\Delta}$ to
conclude, by the principle of accompanying laws, that
$\dov{D}_{(n)}$ converges to the distributional limit of
$\dov{D}^{(0,H-\eps)}$ as $\eps\to 0$, which is
$\dov{D}^{(0,H)}$ by Proposition~\ref{boundgluquad}. 
To this end, we consider the restrictions
$C^\eps_{(n)}$, $\Lambda^\eps_{(n)}$ of $C_{(n)}$, $\Lambda_{(n)}$ to the
intervals $[(\br\tau_{h^\eps_n+1}+1)/2n,\tau_{h^\eps_n}/2n]$ and, letting $F$ be a
nonnegative bounded continuous function, we observe that, using Lemma~\ref{sec:conditioned-version}, then Theorem~\ref{thmddstar} (for the choice of $H-\eps$ instead of $H$) and Lemma~\ref{lltquad}, and finally Lemma~\ref{sec:conditioned-version-4}, we have
\MEMS{
\begin{multline*}
\E_{a_n,\bar{a}_n,h_n,\delta_n}\big[F\big(C^\eps_{(n)},\Lambda^\eps_{(n)},D^\eps_{(n)}\big)\big]\\
\shoveright{=\E_\infty\big[\Psi_{a_n,\bar{a}_n,h_n,\delta_n}(\tau_{h^\eps_n},-1-\br\tau_{h^\eps_n+1},\Lambda(\tau_{h^\eps_n}))\,
	G\big(C^\eps_{(n)},\Lambda^\eps_{(n)},D^\eps_{(n)}\big)\big]}\\
\shoveleft{\ton \Plane\bigg[\psi_{A,\bar{A},H,\Delta}(T_{H-\eps},-\br T_{H-\eps},H-\eps,W_{T_{H-\eps}})}
	\\\shoveright{G\big(X^{(0,H-\eps)},W^{(0,H-\eps)},\dov{D}^{(0,H-\eps)}\big)\bigg]}\\
=\Quad_{A,\bar{A},H,\Delta}\big[G\big(X^{(0,H-\eps)},W^{(0,H-\eps)},\dov{D}^{(0,H-\eps)}\big)\big].
\end{multline*}
}{
\begin{multline*}
\E_{a_n,\bar{a}_n,h_n,\delta_n}\big[F\big(C^\eps_{(n)},\Lambda^\eps_{(n)},D^\eps_{(n)}\big)\big]\\
\shoveright{=\E_\infty\big[\Psi_{a_n,\bar{a}_n,h_n,\delta_n}(\tau_{h^\eps_n},-1-\br\tau_{h^\eps_n+1},\Lambda(\tau_{h^\eps_n}))\,
	G\big(C^\eps_{(n)},\Lambda^\eps_{(n)},D^\eps_{(n)}\big)\big]}\\
\ton \Plane\bigg[\psi_{A,\bar{A},H,\Delta}(T_{H-\eps},-\br T_{H-\eps},H-\eps,W_{T_{H-\eps}})\,
	G\big(X^{(0,H-\eps)},W^{(0,H-\eps)},\dov{D}^{(0,H-\eps)}\big)\bigg]\\
=\Quad_{A,\bar{A},H,\Delta}\big[G\big(X^{(0,H-\eps)},W^{(0,H-\eps)},\dov{D}^{(0,H-\eps)}\big)\big].
\end{multline*}
}
This concludes the proof.
\end{proof}

From there, we easily obtain the wanted GHP convergence by arguments similar as those developed in the proof of Theorem~\ref{thmslslice} at the end of Section~\ref{secslcond}.

\section{Construction from a continuous unicellular map}\label{sec:univ}

Our proof of Theorem~\ref{mainthm} gives a description of the limiting
Brownian surfaces as gluings of elementary pieces, which appear either 
in the Brownian plane or in the Brownian half-plane. Although this construction
has a clear geometric content, it can be arguably cumbersome to work
with, having in mind, for instance, the universal character that the
spaces $\BS{g}{\bL}$ are expected to bear.

Indeed, we believe that Brownian surfaces arise as universal limits
for many more classes of maps satisfying mild conditions (for instance
uniformly distributed maps) and a more direct description seems to be useful in order to show
such results. In particular, we believe that the Brownian torus is the
scaling limit of essentially 
simple triangulations, as considered in~\cite{BeHuLe19arX}. In fact, most of the known results of
convergence toward the Brownian sphere use a re-rooting technique due to
Le~Gall~\cite{legall11}, which, very roughly speaking, says that, if
maps in a given class are properly encoded by discrete objects
converging to the random snake driven by a normalized Brownian excursion and
if these maps and the limiting object exhibit a property of invariance
under uniform re-rooting, then the limiting space is the Brownian
sphere. We expect this approach to be generalizable to our context and
we now give a description of Brownian surfaces that is a direct
generalization of the classical definition of the Brownian
sphere. This can be thought of a continuum version of the
Cori--Vauquelin--Schaeffer bijection, building on a continuum version
of a unicellular map (a map with only one internal face). 

For a function $f\in \CC$ and $s$, $t\in I(f)$ with $t < s$, we extend~\eqref{defunderline} by setting
\begin{equation*}
\un{f}(s,t)= \inf_{I(f)\setminus[t,s]}f
\end{equation*}
and we set, for $s$, $t\in I(f)$,
\begin{equation}\label{tildedf}
\tilde d_f(s,t)=f(s)+f(t)-2\max\big\{\un{f}(s,t),\un{f}(t,s)\big\}.
\end{equation}
The difference with~\eqref{df} is that we now take into account the
minimum of~$f$ on the ``interval'' from $s\vee t$ to~$s\wedge t$ on
the ``circle'' $I(f)/\{\br\tau(f)=\tau(f)\}$. 

\paragraph{The Brownian sphere.}
As a warm-up, let us first recall the definition of the Brownian
sphere. It is the metric space $\BS{0}{\varnothing}=\big([0,1],\tilde d_Z\big)/\{d_\be=0\}$, where~$Z$ is the random snake driven by a normalized Brownian excursion~$\be$.

Recall that the Continuum Random Tree (CRT) introduced by Aldous~\cite{aldouscrt91,aldouscrt93} is the $\R$-tree\footnote{See Section~\ref{Rtrees}.} $\cT_\be=([0,1]/\{d_\be=0\},d_\be)$, so that the Brownian sphere $\BS{0}{\varnothing}$ may actually be seen as a quotient of the CRT. In fact, Le~Gall~\cite{legall06} showed that the pseudometric $\tilde{d}_Z/\{d_\be=0\}(s,t)=0$ if and only if $\tilde{d}_Z(s,t)=0$ or $d_\be(s,t)=0$, so that the topological space $\BS{0}{\varnothing}$ is obtained by a continuous analog to the Cori--Vauquelin--Schaeffer bijection.

\paragraph{The Brownian disk.}
Let us turn to the Brownian disk with perimeter $L\in(0,\infty)$. It is the metric space $\BS{0}{(L)}=\big([0,1],\tilde d_W\big)/\{d_X=0\}$, where $(X,W)$ is the pair encoding a slice with area~$1$, width~$L$ and tilt~$0$, that is, distributed according to $\Slice_{1,L,0}$ (defined in Section~\ref{sec:boundary-triangles}).

The most natural continuous object generalizing the CRT in the case of the disk is the gluing $\Mi^{[0]}_{(L)}=([0,1],d_X)/R$ where~$R$ is the coarsest equivalence relation containing $\{d_X=0\}$ and $\{(0,1)\}$. As $\tilde d_W(0,1)=0$, the Brownian disk is also $([0,1],\tilde d_W)/R$ and can be seen as a quotient of~$\Mi^{[0]}_{(L)}$. Visually, $\Mi^{[0]}_{(L)}$ is obtained by taking a circle of length~$L$ and gluing a Brownian forest of mass~$1$ and length~$L$ on it. The random snake~$W$ then assigns Brownian labels to it (with a Brownian bridge multiplied by~$\sqrt 3$ on the circle and standard Brownian motions everywhere else).

\paragraph{The general case.}
The CRT and the structure~$\Mi^{[0]}_{(L)}$ are the continuous
equivalent to the encoding objects of Section~\ref{secCMS} in the
particular cases of the sphere and the disk. In general, we have a
similar yet even more intricate construction, which we now
describe. Let $g\ge 0$ be fixed and $\bL=(L^1,\dots,L^\bb)$ be a 
$\bb$-tuple of positive real numbers. Let then 
$\big(S, \left({A^e}\right)_{e\in \vec E(S)}, \left({H^e}\right)_{e\in
  \vec I(S)}, \left(L^e \right)_{e\in \vec B(S)},\left( \Lambda^v\right)_{v\in V(S)} \big)$ be a random vector distributed according to the distribution $\mathrm{Param}_\bL$, defined around~\eqref{eqlimvect}. Conditionally
given this vector, we consider the following collection of
processes. For each $e\in\vec E(S)$,  
\begin{itemize} 
	\item the process $X^e$ is a first-passage bridge of standard Brownian motion from~$0$ to~$-H^e$ with duration~$A^e$;
	\item the process $Z^e$ is a random snake driven by the reflected process $X^e-\un{X}^e$;
\end{itemize}
the processes $(X^e,Z^e)$, $e\in\vec E(S)$, being independent. Independently, the process~$\zeta^e$ is a Brownian bridge 
\begin{itemize}
	\item of duration~$H^e$ from~$\Lambda^{e^-}$ to~$\Lambda^{e^+}$, with variance~$1$ if $e\in\vec I(S)$;
	\item of duration~$L^e$ from~$\Lambda^{e^-}$ to~$\Lambda^{e^+}$, with variance~$3$ if $e\in\vec B(S)$.
\end{itemize}
Furthermore, for $e\in\vec I(S)$, the bridges are linked through the relation 
\[\zeta^{\bar e}(s) = \zeta^e(H^e-s),\qquad 0\le s \le H^e,\]
and, except for these relations, are independent. We then set, for each $e\in\vec E(S)$,
\[W^e_t=Z^e_t+\zeta^e_{-\un{X}^e_t},\qquad 0\le t\le A^e.\]

In the end, we obtain a collection of processes $(X^e,W^e)$, $e\in\vec E(S)$, which are linked through the relations linking~$\zeta^e$ with~$\zeta^{\bar e}$, translating the fact that the labels of the floors of forests grafted on both sides of the same internal edge of the scheme should correspond.

We arrange the half-edges $e_1$, \ldots, $e_\kappa$ incident to the internal face of~$S$ according to the contour order, starting from the root, and we define the concatenation
\[
(W_s)_{0\le s\le 1} = W^{e_1} \bullet \dots \bullet W^{e_\kappa},
\]
which is a continuous process. We define~$\tilde d_W$ by~\eqref{tildedf} as above and now define the equivalence relation along which to glue.

Roughly speaking, we glue together Brownian forests coded by the $X^e$'s according to the scheme structure. For $s\in [0,1)$, we denote by~$[s]$ the integer in $\{1,\dots,\kappa\}$ such that
\[\sum_{i=1}^{[s]-1} A^{e_i} \le s < \sum_{i=1}^{[s]} A^{e_i}\qquad\text{ and }\qquad \langle s \rangle = s- \sum_{i=1}^{[s]-1} A^{e_i} \in \big[0,A^{e_{[s]}}\big).\]
By convention, we also set $[1]=1$ and $\langle 1 \rangle =0$. We define the relation~$R$ on $[0,1]$ as the coarsest equivalence relation for which $s\binR t$ if one of the following occurs:
\begin{subequations}
\begin{flalign}
\hspace{5mm}
&[s]=[t]\quad \text{ and }\quad  d_{X^{e_{[s]}}}(\langle s \rangle,\langle t \rangle)=0\, ;&\label{idsamefor}\\
&e_{[s]}=\bar e_{[t]},\ X^{{[s]}}\langle s \rangle = \un X^{{[s]}}\langle s \rangle,\ X^{{[t]}}\langle t \rangle = \un X^{{[t]}}\langle t \rangle\MEMS{}{\ }\text{ and }\MEMS{}{\ } X^{{[s]}}\langle s \rangle = H^{e_{[t]}} - X^{{[t]}}\langle t\rangle;&\label{idfacefor}
\end{flalign}
\end{subequations}
where we wrote $X^{[s]}\langle s \rangle$ instead of $X^{e_{[s]}}(\langle s \rangle)$ for short. Equation~\eqref{idsamefor} identifies numbers coding the same point in one of the Brownian forests, while Equation~\eqref{idfacefor} identifies the floors of forests ``facing each other'': the numbers~$s$ and~$t$ should code floor points (second and third equalities) of forests facing each other (first equality) and correspond to the same point (fourth equality). 

\begin{prp}\label{BSsingle}
The Brownian surface~$\BS{g}{\bL}$ has same distribution as $([0,1],\tilde d_W)/R$.
\end{prp}

Let us give a similar interpretation as in the case of the disk. Let first $(X_s)_{0\le s\le 1}$ be the continuous process obtained by shifting and concatenating~$X^{e_1}$, \ldots, $X^{e_\kappa}$. Then~$\BS{g}{\bL}$ may be seen as a quotient of $\Mi^{[g]}_{\bL}=([0,1],d_X)/R$, which can be pictured as follows. Starting from the random vector $\big(S, \left({A^e}\right)_{e\in \vec E(S)}, \left({H^e}\right)_{e\in\vec I(S)}, \left(L^e \right)_{e\in \vec B(S)} \big)$, we first construct the metric graph obtained from~$S$ by assigning either the length~$H^e$ or~$L^e$ to the edge corresponding to~$e$. For every half-edge~$e$ incident to the internal face of~$S$, we then glue a Brownian forest of mass~$A^e$ and length~$H^e$ or~$L^e$ on~$e$. We equip this space~$\Mi^{[g]}_{\bL}$ with Brownian labels (with variance~$\sqrt 3$ on the boundary edges) and define $\BS{g}{\bL}$ from there by the same process as in the case of the Brownian disk.

\begin{proof}[Proof of Proposition~\ref{BSsingle}]
First of all, recall from Section~\ref{secslEP} that the Brownian surface~$\BS{g}{\bL}$ is defined as the gluing along geodesic sides of a collection of continuum elementary pieces distributed as follows. Conditionally given
\[\big(S, \left({A^e}\right)_{e\in \vec E(S)}, \left({H^e}\right)_{e\in\vec I(S)}, 
	\left(L^e \right)_{e\in \vec B(S)},\left( \Lambda^v\right)_{v\in V(S)} \big),\]
the elementary pieces $\cEP^e$, $e\in\vec E(S)$, are only dependent through the relation linking~$\cEP^e$ with~$\cEP^{\bar e}$ and, setting $\Delta^e=\Lambda^{e^+}-\Lambda^{e^-}$,
\begin{itemize}
	\item if $e\in\vec B(S)$, then~$\cEP^e$ is a slice with area~$A^e$, width~$L^e$ and tilt~$\Delta^e$;
	\item if $e\in\vec I(S)$, then~$\cEP^e$ is a quadrilateral with half-areas~$A^e$ and~$A^{\bar e}$, width~$H^e$ and tilt~$\Delta^e$.
\end{itemize}

Furthermore, it is straightforward from the definition of the pairs $(X^e,W^e)$, $e\in\vec E(S)$, that, if $e\in\vec B(S)$, then the pair $(X^e,W^e-\Lambda^{e^-})$ is distributed as $\Slice_{A^e,L^e,\Delta^e}$.
When $e\in\vec I(S)$, we denote by
\[
\big(\overl{X}^e,\overl W^e\big)=\big(X^{e}_{s+A^{e}}-2\un X^{e}_{s+A^{e}}-H^{e},W^{e}_{s+A^{e}}\big)_{-A^{e}\le s \le 0}
\]
the process obtained by shifting the Pitman transform of~$X^e$ in order to obtain a process from $-H^{e}$ to~$0$, as well as changing the time range to $[-A^{e}, 0]$. By standard results on Brownian motion and random snakes, the pair obtained by concatenating $(\overl X^{\bar e},\overl W^{\bar e}-\Lambda^{\bar e^-})$ with $(X^e,W^e-\Lambda^{e^-})$ has the law of a process distributed as $\Quad_{A^e,A^{\bar e},H^e,\Delta^e}$.

As a result, we may assume that the elementary piece $\cEP^e$ is encoded by
\begin{itemize}
	\item the pair $(X^e,W^e-\Lambda^{e^-})$ if $e\in\vec B(S)$;
	\item the concatenation of $(\overl X^{\bar e},\overl W^{\bar e}-\Lambda^{\bar e^-})$ with $(X^e,W^e-\Lambda^{e^-})$ if $e\in\vec I(S)$.
\end{itemize}
This yields a collection of elementary pieces with the proper laws and dependencies; the fact that, for $e\in\vec I(S)$, $\cEP^e$ and~$\cEP^{\bar e}$ are the same with exchanged shuttles and maximal geodesics is a simple application of the Pitman transform.

\medskip
For $s\in [0,1]$, we denote by $\bpi(s)$ the projection in the gluing~$\BS{g}{\bL}$ of the point~$\langle s \rangle$ of the elementary piece~$\cEP^{e_{[s]}}$. We claim that $\bpi:[0,1]\to \BS{g}{\bL}$ is onto. Indeed, for each half-edge $\epsilon\in\vec E(S)$, recall that the elementary piece~$\cEP^{\epsilon}$ is defined as a quotient of $[0,A^{\epsilon}]$ and observe that $\{\langle s \rangle\,: s \text{ such that } e_{[s]}=\epsilon\}=[0,A^{\epsilon})$; furthermore, the ``missing point'' $A^\epsilon$ of~$\cEP^{\epsilon}$ is glued to a point~$0$ of some elementary piece, which is~$\bpi(s)$ for some~$s$ satisfying $\langle s \rangle=0$. Writing $\d_{\sfS}$ the distance in~$\BS{g}{\bL}$ and $\d_R=\tilde d_W/R$, it is sufficient to show that, for $s$, $t\in[0,1]$,
\[
\d_R(s,t)=\d_{\sfS}\big(\bpi(s),\bpi(t)\big).
\]

As the pseudometric~$d_f$ defined in~\eqref{df} is unchanged by the addition of an additive constant, setting
\[
\d_{\epsilon} = \begin{cases}
		d_{W^\epsilon} &\text{ if } \epsilon\in\vec B(S)\\
		\dov d_{\overl W^{\bar\epsilon}\bullet W^\epsilon} &\text{ if } \epsilon\in\vec I(S)
	\end{cases},
\]
the quantity $\d_{\sfS}\big(\bpi(s),\bpi(t)\big)$ is the infimum of sums of the form $\displaystyle\sum_{i=1}^\ell \d_{\epsilon_i}(s_i,t_i)$ where
\begin{itemize}
	\item $\epsilon_1=e_{[s]}$, $s_1=\langle s \rangle$, $\epsilon_\ell=e_{[t]}$, $s_\ell=\langle t \rangle$;
	\item for all $i$, it holds that $s_i$, $t_i \in \begin{cases}
	[0,A^{\epsilon_i}]	&\text{ if } \epsilon_i\in\vec B(S)\\
	[-A^{\bar\epsilon_i},A^{\epsilon_i}]	&\text{ if } \epsilon_i\in\vec I(S)
	\end{cases}$;
	\item for all $i$,
	\begin{enumerate}[label=\alph*)]
		\item either $\epsilon_i=\epsilon_{i+1}\in\vec B(S)$ and $d_{X^{\epsilon_i}}(t_i,s_{i+1})=0 $\,;\label{identa}
		\item or $\epsilon_i=\epsilon_{i+1}\in\vec I(S)$ and $d_{\overl X^{\bar\epsilon_i}\bullet X^{\epsilon_i}}(t_i,s_{i+1})=0 $\,;\label{identb}
		\item or the point~$t_i$ of~$\cEP^{\epsilon_i}$ is glued to the point~$s_{i+1}$ of~$\cEP^{\epsilon_{i+1}}$.\label{identc}
	\end{enumerate}
\end{itemize}
As $\d_{\epsilon}(u,v)=\infty$ whenever $uv<0$, we may furthermore assume that, for all $i$, $s_i t_i \ge 0$. Now, for each~$i$, we set
\[
\tilde s_i = \sum_{j=1}^{[\epsilon_i]-1} A^{e_j} + s_i \ \text{ if } s_i \ge 0,\qquad
\tilde s_i = \sum_{j=1}^{[\bar\epsilon_i]} A^{e_j} + s_i \ \text{ if } s_i < 0,
\]
where we wrote $[\epsilon]$ the index of the half-edge~$\epsilon$ in the ordering~$e_1$, \ldots, $e_\kappa$ of~$\vec E(S)$. We define~$\tilde t_i$ similarly. It is easy to check that $\tilde s_1=s$, $\tilde t_\ell=t$ and that, for each~$i$, we have $\d_{\epsilon_i}(s_i,t_i)=d_W(\tilde s_i,\tilde t_i)$. Furthermore, for each~$i$, we have the following.
\begin{itemize}
	\item[\ref{identa}] If $\epsilon_i=\epsilon_{i+1}\in\vec B(S)$ and $d_{X^{\epsilon_i}}(t_i,s_{i+1})=0 $, then, unless $t_i=s_{i+1}=A^{\epsilon_i}$ (in which case $\tilde t_i=\tilde s_{i+1}$), it holds that $s_i<A^{\epsilon_i}$ and $t_i<A^{\epsilon_i}$, which yields that $\tilde t_i \binR \tilde s_{i+1}$ by~\eqref{idsamefor}.
	\item[\ref{identb}] If $\epsilon_i=\epsilon_{i+1}\in\vec I(S)$ and $d_{\overl X^{\bar\epsilon_i}\bullet X^{\epsilon_i}}(t_i,s_{i+1})=0 $, then,
	\begin{itemize}
		\item if $t_i s_{i+1} \ge0$, then $\tilde t_i \binR \tilde s_{i+1}$ by~\eqref{idsamefor} as above;
		\item if $t_i s_{i+1} < 0$, then $\tilde t_i \binR \tilde s_{i+1}$ by~\eqref{idfacefor}.
	\end{itemize}
	\item[\ref{identc}] If the point~$t_i$ of~$\cEP^{\epsilon_i}$ is glued to the point~$s_{i+1}$ of~$\cEP^{\epsilon_{i+1}}$, then it implies that $\tilde d_W(\tilde t_i,\tilde s_{i+1})=0$ (recall the situation depicted in Figure~\ref{gluings}).
\end{itemize}
As a result, since $\tilde d_W\le d_W$, it holds that $\d_R(s,t)\le \d_{\sfS}\big(\bpi(s),\bpi(t)\big)$. The converse inequality is very similar, noting that~$R$ identifies points in the elementary piece~$\cEP^\epsilon$ as does $d_{X^{\epsilon}}=0$ or $d_{\overl X^{\bar\epsilon}\bullet X^{\epsilon}}=0$, and that~$\tilde d_W$ encodes all the functions~$\d_{\epsilon}$, together with the gluings of the elementary pieces. The use of~$\tilde d_W$ and not~$d_W$ takes into account the gluings of shuttles with maximal geodesics of elementary pieces ``overflying'' the root, as, for instance in Figure~\ref{gluings}, the shuttle of~$\cEP^{e_{14}}$ with part of the maximal geodesic of~$\cEP^{e_1}$, or part of the shuttle of~$\cEP^{e_{12}}$ with part of the maximal geodesic of~$\cEP^{e_7}$. The details are left to the reader.
\end{proof}

\appendix

\section{Technical lemmas on the Brownian plane}\label{appBP}

We now recall the definition of the Brownian plane from~\cite{CuLG12Bplane}, then show that it is equivalent to the one we gave in Section~\ref{secBP}, and we finally prove Proposition~\ref{cvencbp}.

\subsection{Equivalence of definitions of the Brownian plane}

The original definition goes as follows. Let $(\fX_t,t\in \R)$ be such that $(\fX_t,t\geq 0)$ and
$(\fX_{-t},t\geq 0)$ are two independent three-dimensional Bessel processes. Since the overall minimum of~$\fX$ is reached at~$0$, the maximum in the definition of~$\tilde{d}_\fX$ -- given in~\eqref{tildedf} -- is equal to
\[\max\big(\un\fX(s,t),\un\fX(t,s)\big)=\begin{cases}
	\inf\{\fX_u, u\in [s\wedge t,s\vee t]\}& \text{ if }st\geq 0\\
	\inf\{\fX_u,u\notin[s\wedge t,s\vee t]\}& \text{ if }st<0
\end{cases}.\]
Next, define $(\fW_t,t\in \R)$ to be a centered Gaussian process
conditionally given~$\fX$, with covariance function specified by 
\[\E[|\fW_s-\fW_t|^2\mid \fX]=\tilde{d}_\fX(s,t).\] 
The Brownian plane was defined in~\cite{CuLG12Bplane} as 
\[\big(\wt{M}_{\fX,\fW},\wt{D}_{\fX,\fW}\big)= \big(\R/\{\wt{D}_{\fX,\fW}=0\},\wt{D}_{\fX,\fW}\big)\qquad\text{ where }\qquad
\wt{D}_{\fX,\fW}=d_\fW/\big\{\tilde{d}_\fX=0\big\}.\]

\medskip
The following proposition shows that the definition given in Section~\ref{secBP} is equivalent to the one above. Recall the piece of notation $\dun{X}_t=\un{X}(0\wedge t,0\vee t)$ and define the process $(\Pi_t=X_t-2\dun{X}_t, t\in\R)$ by taking the Pitman transform of~$X$ on~$\Rp$ and on~$\Rm$.

\begin{prp}\label{PitmanBP}
The process $(\Pi,W)$ considered under $\Plane$ has same distribution as $(\fX,\fW)$ defined above. Moreover, as metric spaces, 
$\big(\wt{M}_{\Pi,W},\wt{D}_{\Pi,W}=d_W/\big\{\tilde{d}_\Pi=0\big\}\big)$ and
$(M_{X,W},D_{X,W})$ are a.s.\ equal.
\end{prp}

\begin{proof}
We claim that $\tilde d_{\Pi}=d_X$. This entails that, conditionally
given~$X$, the process~$W$ is also Gaussian with $\E[(W_s-W_t)^2\mid
X]=\tilde d_{\Pi}(s,t)$ and, since $\Pi$ has same distribution as~$\fX$ by
Pitman's $2M-X$ theorem \cite[Theorem~1.3]{Pi2M-X}, we see that $(\Pi,W)$ and $(\fX,\fW)$ have same
distribution. We then have
\[\wt{D}_{\Pi,W}=d_W/\big\{\tilde{d}_\Pi=0\big\}=d_W/\{d_X=0\}=D_{X,W}.\] 

Checking that $\tilde d_{\Pi}=d_X$ is a classical exercise, based on the fact that, for $0\le s< t$,
\begin{align}\label{dXXp}
\un\Pi(s,t)&=\un{X}(s,t)-\dun{X}_s-\dun{X}_t,\qquad\text{ and }\qquad\inf_{u\geq s}\Pi_u=-\dun{X}_s.
\end{align}
The right equation is obtained from the left one by letting
$t\to\infty$, noting that, for~$t$ large enough,
$\un{X}(s,t)=\dun{X}_t$. The left equation comes from a straightforward
case analysis. If $\dun{X}_s=\dun{X}_t$, then, for all $u\in[s,t]$,
$\dun{X}_u=\dun{X}_s=\dun{X}_t$ and so $\Pi_u=X_u-\dun{X}_s-\dun{X}_t$\,;
taking the infimum on $u\in[s,t]$ gives the result. If
$\dun{X}_s>\dun{X}_t$, then $\un{X}(s,t)=\dun{X}_t$ so the right-hand
side is $-\dun{X}_s$. Let $r\in[s,t]$ be such that
$X_r=\dun{X}_r=\dun{X}_s$. We have $\Pi_r=X_r-2\dun{X}_r=-\dun{X}_s$ and,
for $u\ge s$, $\Pi_u=X_u-2\dun{X}_u\ge -\dun{X}_u\ge -\dun{X}_s$.

For $0\le s< t$, the left equation of~\eqref{dXXp} entails
\begin{align*}
\tilde d_{\Pi}(t,s)&=\Pi_s+\Pi_t-2\un\Pi(s,t)\\
	&=X_s-2\dun{X}_s+X_t-2\dun{X}_t-2\big(\un{X}(s,t)-\dun{X}_s-\dun{X}_t\big)=d_X(s,t).
\end{align*}
For $s<0<t$, we have that 
\[\un\Pi(t,s)=\inf_{u\geq t}\Pi_u\,\wedge\, \inf_{u\leq s}\Pi_u=- \big(\un X(0,t)\vee \un X(s,0)\big)=\un{X}(s,t)-\un X(s,0)-\un X(0,t)\]
and we conclude as above. The remaining case $s<t<0$ is treated similarly.
\end{proof}

\subsection{Convergence of the UIPQ to the Brownian plane}

We use here the setting of Section~\ref{secUIPQ}. The proof of Proposition~\ref{cvencbp} will follow similar lines as that of Proposition~\ref{cvencbhp}, using the coupling results of~\cite{CuLG12Bplane}. As the law of~$\fX$ is obtained from that of~$X$ by taking the Pitman transform on~$\Rp$ and on~$\Rm$, the same should be done for the contour process~$C$ of the tree~$\bT_\infty$. We thus define the process $(\fC(t)=C(t)-2\dun{C}(t),t\in\R)$.

Note that this gives an alternate natural contour process since, for $i\in \Z$, it holds that
\[\fC(i) = d_{\bT^{\Upsilon(i)}}\big(v_i,\rho^{|\Upsilon(i)|}\big)+|\Upsilon(i)|=d_{\bT_\infty}(v_i,\rho^0).\]

\begin{figure}[ht!]
  \centering
  \includegraphics[width=.9\textwidth]{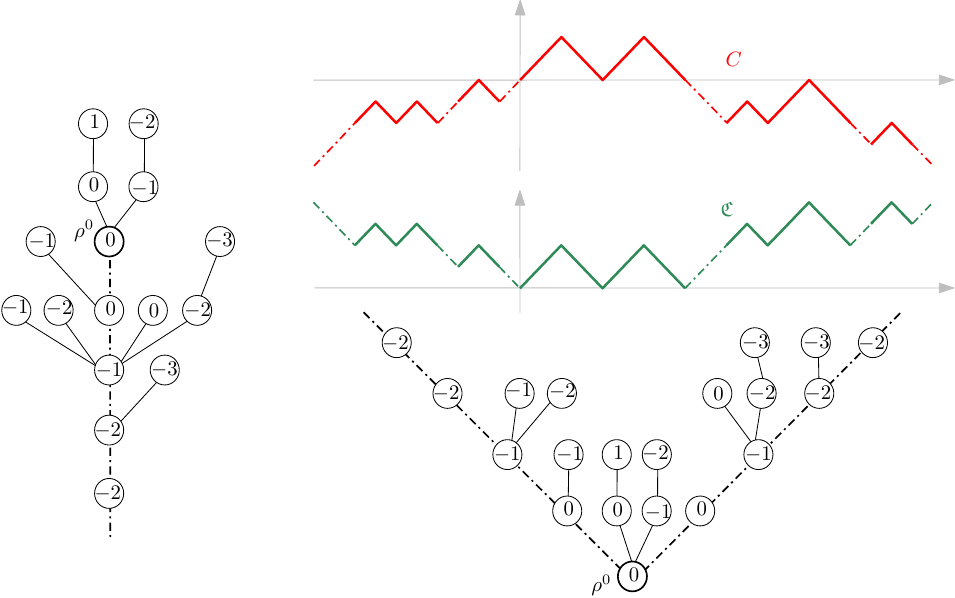}
  \caption{\textup{\textgras{Left.}} Representation of the infinite tree from Figure~\ref{fig:contour2} after moving the trees in such a way that, for $k\ge 0$, $\rho^k$ is located at $(0,-k)$ with~$\bT^k$ grafted on its right and~$\bT^{-k}$ on its left. \textup{\textgras{Top right.}} Taking the Pitman transform of the contour process on~$\Rp$ and on~$\Rm$ yields an alternate contour process. The differing parts are represented with dot-dashed lines; they correspond to the edges of the infinite spine of the tree. Visually, the process~$C$ records the height of a particle moving at speed one around the tree when represented as on the left. \textup{\textgras{Bottom right.}} Representation of the tree from Figure~\ref{fig:contour2} where the root of~$\bT^k$ is now located at $(k,|k|)$ for each $k\in\Z$. Note that, in this representation, the roots of~$\bT^{-k}$ and~$\bT^k$ differ so that the spine is duplicated. The process~$\fC$ records the height of a particle moving at speed one around this bi-infinite tree.}
  \label{fig:contourPit}
\end{figure}
In this setting of discrete trees, the Pitman transform on the contour process is very visual: it merely consists of going from reading the trees while moving \emph{down} between trees to reading them while moving \emph{up} between trees; see Figure~\ref{fig:contourPit} for an illustration. We may now proceed to the proof of the convergence of the UIPQ to the Brownian plane.

\begin{proof}[Proof of Proposition~\ref{cvencbp}]
Similarly as in the proof of Proposition~\ref{cvencbhp}, we fix some number $K>0$ and will sample a large plane quadrangulation such that, its properly scaled version and its limit, the Brownian sphere, are indistinguishable from the rescaled UIPQ and the Brownian plane, in a neighborhood of~$0$ of amplitude~$K$. We use again a superscript prime symbol~$\prime$ for the objects related to the plane quadrangulation and its limit. Here, some care will also be needed when taking an inverse Pitman transform, since this operation a priori involves more than just a neighborhood of~$0$.

We fix $L>0$ and $n\geq 1$, and consider a
uniform random element 
$(M'_n,\lambda'_n)$ of $\RbM^{[0]}_{a_n,\varnothing}$, where
$a_n=\lfloor n L\rfloor$, that is, $M'_n$ is a uniform rooted plane
tree with~$a_n$ edges, which we view as a map with a unique face
$f_*$, and $\lambda'_n$ is a labeling function uniformly distributed among those yielding a well-labeled tree. We let
$(\fC'_n,\fL'_n)$ be the contour and label function of this tree,
we let $Q'_n=\CVS(M'_n,\lambda'_n;f_*)$ be the
quadrangulation encoded by $(M'_n,\lambda'_n)$, and we set
$\fD'_n(i,j)=d_{Q_n'}(v_i,v_j)$ for $0\leq i,j\leq 2a_n$, where $v_i$ is
the $i$-th visited vertex in $M'_n$ in contour order, starting from
the root corner, and viewed as a vertex of $Q_n'$. We extend $\fD'_n$
into a continuous function on $[0,2a_n]^2$ by bilinearity, and all
processes $\fC'_n$, $\fL'_n$, $\fD'_n$ to $[-2a_n,2a_n]$ by the same
formulas as~\eqref{eq:2} and~\eqref{eq:3} but with $l_n=0$. We also
define the rescaled versions $\fC'_{(n)}$, $\fL'_{(n)}$, $\fD'_{(n)}$
exactly as in~\eqref{rescaleprime}. The joint convergence
\begin{equation}
  \label{eq:12}
  \big(\fC'_{(n)},\fL'_{(n)},\fD'_{(n)}\big)\tod\big(\fX',\fW',\fD'\big)
\end{equation}
on the space $\CC([-L,L])\times \CC([-L,L])\times \CC([-L,L]^2)$ is
then a consequence of \cite[Theorem~3]{miermont11}, where the limit is as follows. Restricted to $[0,L]$, the process~$\fX'$ is a
Brownian excursion of duration~$L$ and~$\fW'$ is the random snake driven by~$\fX'$, while $\fD'$ is a random pseudometric, which
is an explicit function of $(\fX',\fW')$. Moreover, all these processes are extended to $[-L,L]$ by a simple 
translation of their argument by~$L$. 

Let us now recall the relevant aspects of the coupling results of
\cite{CuLG12Bplane}, between the pairs $(\fX,\fW)$ and
$(\fX',\fW')$. It will be convenient to let
\[\fS_x=\inf\{t\leq 0:\fX_t=x\},\qquad \fT_x=\sup\{t\geq 0:\fX_t=x\}.\]
Fix $r>0$ and $\eps>0$. Then by \cite[Lemmas~5 and~6]{CuLG12Bplane}, it is possible to find $A>1$ and then $\alpha>0$ and
$L_0>0$ large, such that for $L>L_0$
the two processes $(\fX,\fW)$ and $(\fX',\fW')$, can be
coupled in such a way that on some event $\FF$ of probability
$\P(\FF)\geq 1-\eps$, the following properties hold.
\begin{itemize}
\item
For every $s$, $t\in [-\alpha,\alpha]$, one has
\begin{equation}\label{eq:7}
\fX_t=\fX'_t,\qquad \fW_t=\fW'_t.
\end{equation}
\item 
It holds that
\begin{equation}\label{eq:11}
-\alpha<\fS_{A^4} \qquad \text{ and }\qquad \fT_{A^4}<\alpha.
\end{equation}
\item 
For every $s$, $t\in [\fS_A,\fT_A]$, the two conditions
$\max\big(\wt{D}_{\fX,\fW}(0,t),
\wt{D}_{\fX,\fW}(0,s)\big)\leq r$ and $\max(\fD'(0,t),
\fD'(0,s))\leq r$ are equivalent, and, if these are
satisfied, one has
\[\fD'(s,t)=\wt{D}_{\fX,\fW}(s,t).\]
\end{itemize}

This choice of coupling being fixed, let us now define $(X_t,t\in \R)$
as the unique process such that $\fX_t=X_t-2\dun{X}_t$ is the Pitman
transform of~$X$ on~$\Rp$ and on~$\Rm$; more explicitly
\[X_t=
\begin{cases} \fX_t-2\inf_{s\geq t}\fX_s & \text{ if }t\geq 0\\
 \fX_t -2\inf_{s\leq t}\fX_s & \text{ if }t< 0
\end{cases}.\]
Let us also define $W=\fW$. Then $X$ indeed has the law of a two-sided
Brownian motion, and~$W$ is the random snake driven by~$X$, so that $(X,W)$ has law $\Plane$. 
Note that, in this particular coupling, we have $\br T_x=\fS_x$ and $T_x=\fT_x$ for every $x\geq
0$, and also $D_{X,W}=\wt{D}_{\fX,\fW}$, by the observation in the proof of Proposition~\ref{PitmanBP}. 
Moreover, on the event~$\FF$, the restriction $X\rst_{[\br T_{A^2},T_{A^2}]}$ is actually a function of $\fX'\rst_{[-\alpha,\alpha]}$. Indeed, by~\eqref{eq:7} and~\eqref{eq:11},
\[X_t=
\begin{cases}
   \fX'_t-2\inf_{t\leq s\leq \alpha}\fX'_s & \text{ if }0\leq t\leq T_{A^2}\\
  \fX'_t -2\inf_{-\alpha\leq s\leq t}\fX'_s & \text{ if }\br T_{A^2}\leq t< 0
\end{cases},\]
since, for $0\leq t\leq T_{A^2}$, one has $\un\fX(t,\infty)=\un\fX(t,T_{A^2})=\un\fX(t,\alpha)$, and similarly in negative times.

By choosing appropriately the values of $r$, and enlarging the
values of $A$ and $\alpha$ if necessary, then, similarly to the proof of
Proposition~\ref{cvencbhp}, we obtain that~\eqref{eq:7} holds on
$[-K,K]$, and that the restrictions to $[-K,K]^2$ of~$\fD'$ and $D_{X,W}$ coincide 
with probability at least $1-\eps$.

Next, keeping $K$, $\eps$ fixed, and possibly up to choosing $L$ even
larger, we need to couple the processes
$(C_{(n)},\Lambda_{(n)},D_{(n)})$ and
$(\fC'_{(n)},\fL'_{(n)},\fD'_{(n)})$ appropriately. 
To this end, we use the techniques of \cite[Proposition~9]{CuLG12Bplane}. The latter states that for $\eps>0$, there exists
$\alpha>0$ (independent of the choice of $L$ arising in the definition
of the scaling constant~$a_n$) such that for every $n$ large enough, one may
couple the quadrangulations
$Q'_n$ and $Q_\infty$ in such a way that, with probability at
least $1-\eps$, the balls of radius $\alpha a_n^{1/4}$ around the root
of~$Q'_n$ and~$Q_\infty$ are isometric. The proof proceeds by
coupling the encoding labeled trees $(M'_n,\lambda'_n)$,
and $(\bT_\infty,\lambda_\infty)$ in such a way
that, with even larger probability, the first $\lfloor \delta
a_n^{1/2}\rfloor$ generations of these trees coincide, for some
$\delta>0$, and the minimal value of $\lambda_\infty$ taken on the
vertices $\rho^0$, $\rho^1$, \ldots, $\rho^{\lfloor \delta a_n^{1/2}\rfloor}$
of $\bT_\infty$
is less than $-4\alpha a_n^{1/4}$. By choosing $R$ and then $L$ large enough in the
first place, for our choice of $K$, we may also require that with probability at least
$1-\eps$, 
\begin{itemize}
\item 
the contour and label processes $\fC'_n$, $\fL'_n$ of $(M'_n,\lambda'_n)$
and $\fC$, $\fL$ of $(\bT_\infty,\lambda_\infty)$ on the interval $[-2nK,2nK]$ involve
only vertices of generations less than $\lfloor 
Rn^{1/2}\rfloor$, and
\item
the most recent common ancestor of the vertices at generation $\lfloor
\delta a_n^{1/2}\rfloor$ has generation at least $\lfloor 
Rn^{1/2}\rfloor$.
\end{itemize}
In particular, on this event, the restriction of the
process $C_n'$ to $[-2nK,2nK]$ is equal to the restriction of the
process $C$ on this same event -- in words, the second itemized event
means that the spine of~$\bT_\infty$ is determined by the data of
$M'_n$ up to generation $Rn^{1/2}$. Since the process $C$ is the
inverse Pitman transform of~$\fC$, it is then a simple exercise to
conclude that $(C'_{(n)},\Lambda'_{(n)},D'_{(n)})$, which coincides
with $(C_{(n)},\Lambda_{(n)},D_{(n)})$ on $[-K,K]$ with high probability,
converges to some $(X',W',D')$, which coincides with $(X,W,D_{X,W})$ on $[-K,K]$
with high
probability. 
\end{proof}

\section{Scaling limit of size parameters in labeled maps}\label{appdata}

\subsection{Preliminaries}\label{secB1}

In this appendix, we prove Proposition~\ref{sldata}, following the
method of \cite[Proposition~7]{bettinelli10} and
\cite[Proposition~7]{bettinelli11b}. In the meantime, we obtain an
asymptotic enumeration result for $\RbQnln$ in
Proposition~\ref{cardRbQ} below, which will also allow us to
deduce Theorem~\ref{thmboltz} and Corollary~\ref{corboltz} from Theorem~\ref{mainthm}.

Recall that $(g,\kk)\notin \{(0,0),(0,1)\}$, that $\bL=(L^1,\dots,L^\kk)$ is a fixed $\kk$-tuple such that $L^1$, \dots, $L^{\bb}>0$, while~$L^{\bb+1}$, \dots, $L^{\kk}=0$, and that we consider a fixed sequence of $\kk$-tuples $\bl_n=(l_n^1,\dots,l_n^\kk) \in (\Zp)^\kk$, $n\ge 1$, such that $l_n^i/{\sqrt{2n}} \to L^i$ as $n\to\infty$, for $1\le i \le \kk$.

We furthermore assume that~$n$ is sufficiently large so that $l_n^i\ge 1$ for each $i\le\bb$. We denote by~$\RbS$ the set of rooted genus~$g$ schemes with~$\kk$ holes, such that~$\ch_1$, \dots, $\ch_{\bb}$ are faces. Note that our assumption on~$n$ ensures that $S_n\in\RbS$.

\paragraph{``Free'' parameters and notation.}
For every scheme $\bs\in\RbS$, not necessarily dominant, we arbitrarily fix, once and for all, half-edges $\epsilon_0\in \vec I(\bs)$, and $\epsilon_i\in\vec B_i(\bs)$, for $1\le i \le \bb$. We fix an orientation~$I(\bs)$ of $\vec I(\bs)$ that contains~$\epsilon_0$ and we set $I'(\bs)=I(\bs)\setminus\{\epsilon_0\}$. We also let~$v_0$ be the root vertex of~$\bs$, and $V'(\bs)=V(\bs)\setminus \{v_0\}$, as in Section~\ref{secslEP}. Finally, we set
\begin{align*}
\vec B_0(\bs)&=\bigsqcup_{\bb+1\leq i\leq \kk}\vec B_i(\bs),
&\vec B_+(\bs)&=\bigsqcup_{1\leq i\leq \bb}\vec B_i(\bs),
\\
\vec B'_i(\bs)&=\vec B_i(\bs)\setminus\{\epsilon_i\},\text{ for } 1\le i\le \bb,
&\vec B'_+(\bs)&=\bigsqcup_{1\leq i\leq \bb}\vec B'_i(\bs) .
\end{align*}
The motivation for introducing~$\vec B_+$ and~$\vec B_0$ is that we need a different treatment depending whether the hole perimeters are in the scale~$\sqrt n$ or $\sO(\sqrt n)$. The sets with a prime symbol should be thought of as the sets containing the parameters on which there is a ``degree of freedom.'' (\textit{The reason for removing one element from~$I$ will become clear in a moment. We will not need a~$\vec B'_0$ since the corresponding perimeters are all asymptotically null in the scale~$\sqrt n$ of interest.}).

From now on, we use the shorthand piece of notation $\bx^{\mathcal E}$ for a family $(x^j)_{j\in\mathcal E}$ indexed by a set~$\mathcal E$. For any subset~$\mathcal F\subseteq\mathcal E$, we also denote by $\bx^{\mathcal F}=(x^j)_{j\in\mathcal F}$ the subfamily indexed by~$\mathcal F$, and, in the case of real nonnegative numbers, by $\|\bx\|_{\mathcal F}=\sum_{j\in\mathcal F} x^j$ (note in particular that $\|\bx\|_{\varnothing}=0$).

\paragraph{Counting scheme-rooted labeled maps with given size parameters.}
For the time being, we do not take the areas parameters into account. We fix a rooted scheme $\bs\in\RbS$, and size parameters $\bh^{\vec I(\bs)}$, $\bl^{\vec B(\bs)}$ and $\blambda^{V(\bs)}$. We say that a labeled map is \emph{scheme-rooted on~$\bs$} if its scheme carries an extra root and the scheme rooted at this extra root is~$\bs$. We consider the elements of $\smash{\bM^{[g]}_{n,\bl}}$ scheme-rooted on~$\bs$ whose size parameters are $\bh^{\vec I(\bs)}$, $\bl^{\vec B(\bs)}$ and $\blambda^{V(\bs)}$. Reasoning as in Lemmas~\ref{countslice} and~\ref{countquad}, we can express the number of such elements as
\[
12^{n-\frac{\|\bh\|}2}2^{\|\bh\|+\|\bl\|}\,Q_{\|\bh\|+\|\bl\|}(2n+\|\bl\|)\,
	\prod_{e\in I(\bs)}3^{h^e}M_{h^e}(\delta\!\lambda^e)
	\prod_{e\in \vec B(\bs)}2^{2l^e+\delta\!\lambda^e}P_{l^e}(\delta\!\lambda^e),\]
the products over~$I(\bs)$ and~$\vec B(\bs)$ respectively counting the number of ways to label the vertices along the edges of~$I(\bs)$ and~$\vec B(\bs)$, and the remaining term counting the labeled forests, which can be seen as one big labeled forest obtained by concatenating all the labeled forests indexed by the half-edges of~$\vec E(\bs)$. After recalling that $\sum_{e\in \vec B_i(\bs)}\delta\!\lambda^e=0$ and that $\|\bl\|_{\vec B_i(\bs)}=l^i$ for every $i\in \{1,2,\ldots,\kk\}$ corresponding to an external face of~$\bs$, we may
recast this quantity as 
\[12^n\, 8^{\|\bl\|}	\,Q_{\|\bh\|+\|\bl\|}(2n+\|\bl\|)
	\prod_{e\in I(\bs)}M_{h^e}(\delta\!\lambda^e)
	\prod_{e\in \vec B(\bs)}P_{l^e}(\delta\!\lambda^e).
\]

Consequently, the number of elements of $\smash{\RbM^{[g]}_{n,\bl}}$ scheme-rooted on~$\bs$ (these labeled maps are thus rooted twice) whose size parameters are $\bh^{\vec I(\bs)}$, $\bl^{\vec B(\bs)}$ and $\blambda^{V(\bs)}$ is equal to
\begin{equation}\label{sahll}
\SR_n^\bs(\bh,\bl,\blambda)=(2n+\|\bl\|)\,
12^n\, 8^{\|\bl\|}	\,Q_{\|\bh\|+\|\bl\|}(2n+\|\bl\|)
	\prod_{e\in I(\bs)}M_{h^e}(\delta\!\lambda^e)
	\prod_{e\in \vec B(\bs)}P_{l^e}(\delta\!\lambda^e),
\end{equation}
since there are $2n+\|\bl\|$ possible rootings of the map.

\paragraph{Counting rooted labeled maps.}
Next, for $n$, $h\in\N$, $\bs\in\RbS$, we set
\begin{equation}\label{Z1shn}
\ZZ_1^\bs(h,n)=\sum_{\cT_\bs(h,n)}
		\prod_{e\in I(\bs)}M_{h^e}(\delta\!\lambda^e)
	\prod_{e\in \vec B(\bs)}P_{l^e}(\delta\!\lambda^e),
\end{equation}
where the sum is taken over the set~$\cT_\bs(h,n)$ of all size parameters from labeled maps in~$\bM^{[g]}_{n,\bl_n}$ scheme-rooted on~$\bs$, having~$h$ edges in total on the internal edges of~$\bs$. More precisely, it is the set of tuples
\[\Big(\bh^{\vec I(\bs)}, \bl^{\vec B(\bs)}, \blambda^{V(\bs)}\Big)\in \N^{\vec I(\bs)}\times \N^{\vec B(\bs)} \times \Z^{V(\bs)}\]
such that
\begin{multicols}{2}\begin{itemize}
	\item $\|\bh\|=2h$,
	\item $h^{\bar e}=h^e$, for all $e \in \vec I(\bs)$,
	\item $\|\bl\|_{\vec B_i(\bs)}=l_n^i$, for $1\le i \le \kk$, 
	\item $\lambda^{v_0}=0$.
\end{itemize}\end{multicols}\noindent
Note that the conditions $\|\bl\|_{\vec B_i(\bs)}=l_n^i$ may only be satisfied if~$\bl_n$ is \emph{compatible} with~$\bs$ in the sense that $l_n^i=0 \iff \vec B_i(\bs)=\varnothing$ for all~$i$. As a result, $\cT_\bs(h,n)=\varnothing$ and thus $\ZZ_1^\bs(h,n)=0$ whenever~$\bl_n$ is not compatible with~$\bs$.

By double counting the elements of $\RbM^{[g]}_{n,\bl_n}$ scheme-rooted on~$\bs$, we thus obtain that
\begin{equation}\label{RbMnln}
\big|\RbM^{[g]}_{n,\bl_n}\big|
	=12^n 8^{\|\bl_n\|}
		\sum_{\bs\in\RbS} \frac{2n+\|\bl_n\|}{2|E(\bs)|}
		\sum_{h\in\N} Q_{2h+\|\bl_n\|}(2n+\|\bl_n\|)\, \ZZ_1^\bs(h,n),
\end{equation}
since we sum over $\bigcup_{\bs\in\RbS,h\in\N} \{\bs\}\times \cT_\bs(h,n)$ the number $\SR_n^\bs(\bh,\bl_n,\blambda)$ given by~\eqref{sahll}, divided by the number $2\,|E(\bs)|$ of possible extra rootings on the scheme.

\subsection{Asymptotics of the scheme}

When we work with a fixed scheme, which will be the case in all but the fourth paragraph, we drop the argument from the sets in the notation in order to ease the reading, thus writing $I$ instead of $I(\bs)$ for instance.

\paragraph{Law of the scheme.} 
Recall that the triple $(M_n,\lambda_n,S_n)$ is a rooted, scheme-rooted, labeled map, where $(M_n,\lambda_n)$ is uniformly distributed over~$\RbM^{[g]}_{n,\bl_n}$, while, conditionally given it, $S_n$ is rooted by uniformly choosing its root among $\big\{e,\bar e\,:\, e\in \vec I(S_n)\cup\vec B_+(S_n)\big\}$. Let us fix a rooted scheme $\bs\in\RbS$ whose root or its reverse belongs to $\vec I\cup\vec B_+$. Writing~$\un\bs$ the nonrooted scheme corresponding to~$\bs$, observe that the set of rooted labeled maps in~$\RbM^{[g]}_{n,\bl_n}$ with scheme~$\un\bs$ is in bijection with the set of rooted labeled maps in~$\RbM^{[g]}_{n,\bl_n}$ scheme-rooted on~$\bs$. Then,
\begin{align*}
\P(S_n=\bs)
	&=\sum_{\substack{(\bm,\lambda)\in\RbM^{[g]}_{n,\bl_n}\\\text{with scheme }\un\bs}}
		\P\big((M_n,\lambda_n)=(\bm,\lambda),\,S_n=\bs\big)\\
	&=\sum_{\substack{(\bm,\lambda)\in\RbM^{[g]}_{n,\bl_n}\\\text{scheme-rooted on }\bs}}
		\P\big((M_n,\lambda_n)=(\bm,\lambda)\big)\,
		\P\big(S_n=\bs\mid (M_n,\lambda_n)=(\bm,\lambda)\big)\\
	&=\sum_{h\in\N} \sum_{\cT_\bs(h,n)} \SR_n^\bs(\bh,\bl_n,\blambda)
		\frac1{\big|\RbM^{[g]}_{n,\bl_n}\big|}
		\frac{1}{|\vec{I}\,|+2\,|\vec B_+|}
	=\frac{\ZZ_1^\bs(n)}{\ZZ_1(n)},
\end{align*}
where
\begin{equation}\label{Zphi}
\ZZ_1^\bs(n)=\frac{1}{|\vec I(\bs)|+2\,|\vec B_+(\bs)|}\sum_{h\in\N} Q_{2h+\|\bl_n\|}(2n+\|\bl_n\|)\, \ZZ_1^\bs(h,n)
\end{equation}
and $\ZZ_1(n)=\sum_{\bs\in\RbS} \ZZ_1^\bs(n)$ is the proper normalization constant.

\paragraph{Schemes with tadpoles.}
Here, we fix a scheme~$\bs$ whose external faces among~$\ch_i$, $\bb+1\le i\le \kk$, are all tadpoles. Equivalently, each $\vec B_i$, $\bb+1\le i \le \kk$, is either empty or a singleton. In this case, by the Euler characteristic formula,
\begin{equation}\label{Eulertad}
|V'|-|I|-|\vec B'_+|=-2g,
\end{equation}
since~$\bs$ has $|I|+|\vec B'_+|+\bb+|\vec B_0|$ edges and $1+\bb+|\vec B_0|$ faces.

Assuming that~$\bl_n$ is compatible with~$\bs$, we write the sum over~$\cT_\bs(h,n)$ in~\eqref{Z1shn} as an integral under the Lebesgue measure
\[\d\rmL_\bs=\d\bh^{I'}\otimes\d\bl^{ \vec B'_+}\otimes\d\blambda^{V'}
	\quad\text{ over }\quad (\Rp)^{I'}\times (\Rp)^{ \vec B'_+} \times \R^{V'}\,\]
and obtain
\begin{align}\label{Z1hn}
\ZZ_1^\bs(h,n)&=
	\prod_{e\in \vec B_0}P_{l^e}(0)
	\int \d\rmL_\bs\,
		\prod_{e\in I}M_{\un h^e}(\delta\!\un\lambda^e)
		\prod_{e\in \vec B_+}P_{\underline l^e}(\delta\!\un\lambda^e),
\end{align}
where $l^e=l_n^i$ if $e$ is the unique element of~$\vec B_i$, for $i>\bb$, and
\begin{multicols}{2}\begin{itemize}[itemsep=10pt]
	\item $\un h^e=\lceil h^e \rceil$, for $e\in I'$,
	\item $\underline l^e=\lceil l^e\rceil$, for $e\in \vec B'_+$, 
	\item $\un\lambda^v=\lceil\lambda^v\rceil$, for $v\in V'$,
	\item $\smash{\un h^{\epsilon_0}=h-\sum\limits_{e\in I'} \un h^e}$,
	\item $\smash{\underline l^{\epsilon_i}=l_n^i-\sum\limits_{e\in\vec B'_i} \underline l^e}$, for $1\le i \le \bb$,
	\item $\un\lambda^{v_0}=0$.
\end{itemize}\end{multicols}\vspace*{2mm}\noindent
(\textit{Note that the ceiling function is superfluous for integer parameters; we kept it for notational simplicity.}) In order to deal with the cases where $\un h^{\epsilon_0}\le 0$ or $\underline l^{\epsilon_i}\le 0$, we simply declare\footnote{This is just a convenience. Note that we set $M_0(0)=P_0(0)=0$ here, although it would be more natural from a combinatorial point of view to set both these quantities to~$1$.} $M_\ell(j)=P_\ell(j)=0$ whenever $\ell\le 0$.

Observe that~$\bl_n$ compatible with~$\bs$ means that~$\vec B_0$ corresponds to $\{ i >\bb : l^i_n \ge 1\}$. 
We then make the changes of variables in the natural scales to obtain
\begin{multline}\label{Z1shnint}
\ZZ_1^\bs(h,n)=
		3^{\frac{\bb}{2}-g}\,
		2^{\frac{|V'|}{2}-\frac g2-\frac34\bb}\,
		n^{\frac{|V'|}{2}+\frac g2-\frac{\bb}4-\frac12}
		\prod_{i > \bb\, :\, l^i_n \ge 1} P_{l^i_n}(0)
	\\\times
		\int \d\rmL_\bs\,
		\prod_{e\in I} \Big(\frac{8n}9\Big)^{\frac14} M_{\dun h^e}(\delta\!\dun\lambda^e)
		\prod_{e\in \vec B_+} \Big(\frac{8n}9\Big)^{\frac14} P_{\dunderline l^e}(\delta\!\dun\lambda^e),
\end{multline}
where
\begin{multicols}{2}\begin{itemize}[itemsep=10pt]
	\item $\dun h^e=\lceil \sqrt{2n}\, h^e \rceil$, for $e\in I'$,
	\item $\dunderline l^e=\lceil \sqrt{2n}\,l^e\rceil$, for $e\in \vec B'_+$, 
	\item $\smash{\dun\lambda^v=\lceil(\tfrac{8n}9)^{\frac14} \lambda^v\rceil}$, for $v\in V'$,
	\item $\smash{\dun h^{\epsilon_0}=h-\sum\limits_{e\in I'} \dun h^e}$,
	\item $\smash{\dunderline l^{\epsilon_i}=l_n^i-\sum\limits_{e\in\vec B'_i} \dunderline l^e}$, for $1\le i \le \bb$,
	\item $\dun\lambda^{v_0}=0$.
\end{itemize}\end{multicols}\vspace*{2mm}\noindent

We finally use the same method to treat the summation over $h\in\N$ in~\eqref{Zphi}, that is, we see it as an integral and do the proper change of variables. We write $\bl_n\comp\bs$ to mean ``$\bl_n$ compatible with~$\bs$'':
\begin{equation}\label{intZ1s}
\ZZ_1^\bs(n)=\frac{\ind_{\bl_n\comp\bs}}{|\vec I\,|+|\vec B_+|}\sqrt{\frac2n}\int_{\Rp} \d\sfh\, n Q_{2\lceil \sqrt{2n}\, \sfh \rceil+\|\bl_n\|}(2n+\|\bl_n\|)\, \ZZ_1^\bs\big(\lceil \sqrt{2n}\, \sfh \rceil,n\big) .
\end{equation}

Setting $h^{\epsilon_0}=\sfh-\sum_{e\in I'} h^e$, $l^{\epsilon_i}=L^i-\sum_{e\in\vec B'_i} l^e$ for $1\le i \le \bb$, and $\lambda^{v_0}=0$, by the local limit theorem \cite[Theorem~VII.1.6]{petrov75}, it holds that, when $h\sim \sqrt{2n}\,\sfh$,
\begin{equation}\label{equiPet1}
\Big(\frac{8n}9\Big)^{\frac14} M_{\dun h^e}(\delta\!\dun\lambda^e)\ton p_{h^e}(\delta\!\lambda^e),\qquad\qquad
\Big(\frac{8n}9\Big)^{\frac14} P_{\dunderline l^e}(\delta\!\dun\lambda^e)\ton p_{3l^e}(\delta\!\lambda^e),
\end{equation}
and
\begin{equation}\label{equiPet2}
n Q_{2\lceil \sqrt{2n}\, \sfh \rceil+\|\bl_n\|}(2n+\|\bl_n\|) \ton q_{2\sfh+\|\bL\|}(1).
\end{equation}
Consequently, provided the domination hypothesis obtained in the following paragraph, we get the following equivalent:
\begin{equation}\label{equivZ1s}
\ZZ_1^\bs(n){\build{\sim}{n\to\infty}{}} c_1^\bs(\bL)\,\ind_{\bl_n\comp\bs}\,
		n^{\frac{|V'|}{2}+\frac g2-\frac{\bb}4-1}
		\prod_{i > \bb\, :\, l^i_n \ge 1}P_{l^i_n}(0)
\end{equation}
where the constant~$c_1^\bs(\bL)$ is given by
\MEMS{
\begin{multline*}
c_1^\bs(\bL)=\frac{1}{|\vec I\,|+2\,|\vec B_+|}
	3^{\frac{\bb}{2}-g}\,
	2^{\frac{|V'|}{2}-\frac g2-\frac34\bb+\frac12}\,
		\\\times\int_{\Rp}\! \d\sfh\, q_{2\sfh+\|\bL\|}(1)\, \int\! \d\rmL_\bs\,
		\prod_{e\in I} p_{h^e}(\delta\!\lambda^e)
		\prod_{e\in \vec B_+} p_{3l^e}(\delta\!\lambda^e)
		.
\end{multline*}
}{
\begin{equation*}
c_1^\bs(\bL)=\frac{1}{|\vec I\,|+2\,|\vec B_+|}
	3^{\frac{\bb}{2}-g}\,
	2^{\frac{|V'|}{2}-\frac g2-\frac34\bb+\frac12}\,
		\int_{\Rp}\! \d\sfh\, q_{2\sfh+\|\bL\|}(1)\, \int\! \d\rmL_\bs\,
		\prod_{e\in I} p_{h^e}(\delta\!\lambda^e)
		\prod_{e\in \vec B_+} p_{3l^e}(\delta\!\lambda^e)
		.
\end{equation*}
}

\paragraph{Domination hypothesis.}

In order to show that the convergence is dominated, we use the bounds of Petrov \cite[Theorem~VII.3.16]{petrov75}, stating that there exist a constant~$C$ such that, for any $\ell\in\N$, $j\in\Z$, $i\in\N$, and $r\in\N$,
\begin{equation}\label{PetBound}
M_\ell(j) \vee P_\ell(j) \le C\, \frac{1}{\sqrt\ell}\qquad\text{ and }\qquad Q_i(\ell)\le C\, \frac{i}{\ell^{3/2}}\frac1{1+({i^2}/{\ell})^r}.
\end{equation}

We fix an arbitrary spanning tree of~$\bs$, that is, a tree with vertex-set $V$ and edge-set a subset of~$E$. We associate with any vertex $v\neq v_0$ the first edge of the unique path in the tree from~$v$ to~$v_0$ and we denote by~$e_v$ the unique half-edge of $ I\cup\vec B$ that corresponds to this edge.

We bound the integrand in~\eqref{Z1shnint} as follows. First, by~\eqref{PetBound}, we have, for $e\in I'$,
\[\Big(\frac{8n}9\Big)^{\frac14} M_{\dun h^e}(\delta\!\dun\lambda^e)
	\le \Big(\frac{8n}9\Big)^{\frac14}\frac{C}{\sqrt{\dun h^e}}
	\le \Big(\frac{8n}9\Big)^{\frac14}\frac{C}{\sqrt{\sqrt{2n}\, h^e}}
	=\sqrt{\frac23}\frac{C}{\sqrt{h^e}}
	\le \frac{C}{\sqrt{h^e}}.\]
For $h=\lceil \sqrt{2n}\, \sfh \rceil$ and $h^{\epsilon_0}=\sfh-\sum_{e\in I'} h^e$, a similar bound holds for $e=\epsilon_0$, up to possibly enlarging the constant. Indeed, it suffices to show that $\dun h^{\epsilon_0}$ is bounded from below by a constant times $\sqrt{2n}\,h^{\epsilon_0}$ in order to complete the computation. We may assume that $\dun h^{\epsilon_0}\ge 1$ as otherwise the left-hand side is null. Then, if $\sqrt{2n}\,h^{\epsilon_0}\le 2{|I'|}$, it immediately holds that $\dun h^{\epsilon_0}\ge \frac1{2|I'|} \sqrt{2n}\,h^{\epsilon_0}$. Otherwise, 
\[\dun h^{\epsilon_0}=h-\sum\limits_{e\in I'} \dun h^e \ge \sqrt{2n}\, h^{\epsilon_0}- |I'|\ge \frac12 \sqrt{2n}\, h^{\epsilon_0}.\]
In conclusion, up to changing the constant~$C$, it holds that, for all $e\in I$,
\[\Big(\frac{8n}9\Big)^{\frac14} M_{\dun h^e}(\delta\!\dun\lambda^e)
	\le \frac{C}{\sqrt{h^e}}.\]
Similarly, up to enlarging the constant~$C$ even more, setting $l^{\epsilon_i}=L^i-\sum_{e\in\vec B'_i} l^e$ for $1\le i \le \bb$, it holds that, for $e\in \vec B_+$,
\[\Big(\frac{8n}9\Big)^{\frac14} P_{\dunderline l^e}(\delta\!\dun\lambda^e)
	\le \frac{C}{\sqrt{l^e}}.\]
We use these bounds whenever $e\notin\vec E_V=\big\{ e_v: v\in V\setminus\{v_0\}\big\}$ and then, we operate the integral with respect to~$\d\blambda^{V'}$ vertices by vertices, starting from a leaf of the fixed spanning tree, then from a leaf of the tree remaining after removing the first vertex, and so on until only~$v_0$ remains. Since, for any $\ell\in\N$,
\[\int \d x\, \Big(\frac{8n}9\Big)^{\frac14} M_{\ell}\big(\lceil(\tfrac{8n}9)^{\frac14} x\rceil\big)
=1,\]
and similarly with~$P_\ell$ instead of~$M_{\ell}$, we obtain that, for $n$ sufficiently large and after integration with respect to~$\d\blambda^{V'}$, the integrand in~\eqref{Z1shnint} is bounded by
\begin{equation}\label{eq:6}
\ind_{\{\|\bh\|_{ I}\le 2\sfh\}}
	\ind_{\{\|\bl\|_{ \vec B'_+}\le 2\|\bL\|\}}
		\prod_{e\in I\setminus \vec E_V} \frac{C}{\sqrt{h^e}}
		\prod_{e\in \vec B_+\setminus \vec E_V}
			\frac{C}{\sqrt{l^e}}.
\end{equation}
This is integrable with respect to $\d\bh^{I'}\otimes\d\bl^{ \vec B'_+}$ and is bounded, after integration, by some constant times some power of~$\sfh$. Taking $r$ sufficiently large in~\eqref{PetBound} yields that this quantity multiplied by $ n Q_{2\lceil \sqrt{2n}\, \sfh \rceil+\|\bl_n\|}(2n+\|\bl_n\|)$ is integrable with respect to $\d\sfh$. The claimed dominated convergence follows.

\paragraph{Dominant schemes.}
We will now see which schemes are such that $\ZZ_1^\bs(n)$ has the highest possible order in~$n$. The exponent of~$n$ in the equivalent~\eqref{equivZ1s} is maximal when $|V(\bs)|$ is the largest; in this case,
\[|V(\bs)|=2(2g+\kk-1).\]
This equality is obtained as in the proof of Lemma~\ref{lemnbschemes}, since $|V(\bs)|$ being the largest means that the vertices have the lowest possible degrees, namely~$3$ for the internal vertices and~$1$ for the external vertices. More precisely, denoting by~$v$, $e$, $f$ the numbers of vertices, edges and faces of~$\bs$, as well as $t$ the number of tadpoles among~$\ch_{\bb+1}$, \dots, $\ch_\kk$, we obtain $f=\bb+t+1$, $2e= 3(v-\kk+\bb+t)+\kk-\bb-t$, and the result from the Euler characteristic formula $v-e+f=2-2g$.

Next, the local limit theorem \cite[Theorem~VII.1.6]{petrov75} yields the existence of a compact set $K\subset (0,\infty)$ such that, for all $\ell\in\N$, $\sqrt \ell\, P_{\ell}(0)\in  K$. Finally, for any $\bs\in\RbS$, we denote by~$\bs^\tp$ the scheme obtained by shrinking every tadpole among~$\ch_{\bb+1}$, \dots, $\ch_{\kk}$ into a vertex. For any fixed dominant scheme $\bd\in\RbS^\star$, observe that there exist exactly one scheme among $\{\bs\in\RbS:\bs^\tp=\bd\}$ that is compatible with~$\bl_n$, namely the one whose tadpoles among~$\ch_{\bb+1}$, \dots, $\ch_\kk$ are the holes indexed by $\{i>\bb:l_n^i\ge 1\}$. Furthermore, if $\bs\in\RbS$ is such that $\bs^\tp\in\RbS^\star$, then the external faces among~$\ch_i$, $\bb+1\le i\le \kk$, of~$\bs$ are all tadpoles. We may thus use the equivalent~\eqref{equivZ1s} for these schemes. Consequently, as $n\to\infty$,
\begin{equation}\label{domorder}
\sum_{\substack{\bs\in\RbS\\\bs^\tp=\bd}} \ZZ_1^\bs(n) = \Theta \Bigg(
		n^{\frac{5(g-1)}2 +\kk -\frac{\bb}4}
		\prod_{i > \bb\, :\, l^i_n \ge 1} (l^i_n)^{-\frac12}\Bigg).
\end{equation}

In particular, if~$\bs$ has only tadpoles among its external faces indexed by $\bb+1\le i\le \kk$ but is such that~$\bs^\tp$ is not dominant, then $\ZZ_1^\bs(n)$ is negligible with respect to this sum.

\paragraph{Nondominant schemes.}
We will now see that the above is the highest order in~$n$ and that it is only obtained for the schemes that are dominant after the tadpoles shrinkage. To this end, we fix an arbitrary scheme $\bs\in\RbS$. As above, we consider an arbitrary spanning tree of~$\bs$ and still denote by $e_v\in  I\cup\vec B$ the half-edge corresponding to $v\in V'$, as well as $\vec E_V=\big\{ e_v: v\in V'\big\}$.

In~\eqref{Z1shn}, we bound $M_{h^e}(\delta\!\lambda^e)$ or $P_{l^e}(\delta\!\lambda^e)$ thanks to~\eqref{PetBound} if $e\notin\vec E_V$ and we operate the sum over $\blambda^{V'}$ leaf by leaf as we did in the previous paragraph. Since, for any $\ell\in\N$, $\sum_{j\in\Z}M_\ell(j)=\sum_{j\in\Z}P_\ell(j)=1$, we obtain the bound
\[\ZZ_1^\bs(h,n)\le \sum_{\bh^{\vec I},\, \bl^{\vec B}}
	\prod_{e\in I\setminus \vec E_V} \frac{C}{\sqrt{h^e}}
	\prod_{e\in\vec B\setminus \vec E_V} \frac{C}{\sqrt{l^e}}
	,\]
where the sum is over the tuples $\bh^{\vec I}$, $\bl^{\vec B}$ satisfying the conditions of~$\cT_\bs(h,n)$. Seeing the sums as integrals under the simplex Lebesgue measures $\Delta_{I}^h$, and $\Delta_{\vec B_i}^{l^i_n}$ whenever $\vec B_i\neq\varnothing$ yields, after renormalization by~$h$ or~$l^i_n$, integrals of Dirichlet distributions (with parameter vectors containing only $\frac12$'s and $1$'s). As a result,
\begin{align*}
\ZZ_1^\bs(h,n)&\lesssim h^{|I| - 1 - \frac12 | I\setminus \vec E_V|}\,
	\prod_{\substack{1\le i \le \kk \\ \vec B_i\neq\varnothing}} 
		(l^i_n)^{|\vec B_i| - 1 - \frac12 |\vec B_i\setminus \vec E_V|}\,\\
	&\quad= h^{\frac12|I| + \frac12 | I\cap \vec E_V| - 1}\,
	\prod_{\substack{1\le i \le \kk \\ \vec B_i\neq\varnothing}} 
		(l^i_n)^{\frac12|\vec B_i| + \frac12 |\vec B_i\cap \vec E_V| - 1},
\end{align*}
where we used the symbol $\lesssim$ to mean bounded up to a constant independent\footnote{Recall from Lemma~\ref{lemnbschemes} that the number of edges in the schemes from~$\RbS$ is bounded.} of~$\bs$, $h$, and~$n$. Since~$l^i_n$ is in the scale $\sqrt n$ for $i \le \bb$, the part of the product concerning these indices is bounded by a constant times
\[n^{\frac14| \vec B'_+| +\frac14 |\vec B_+\cap \vec E_V| - \frac{\bb}4}.\]
Recall that $ B_i\neq\varnothing \iff l^i_n \ge 1$ when $\bl_n$ is compatible with~$\bs$. Using~\eqref{intZ1s}, which is valid for any scheme, as well as the bound~\eqref{PetBound} as above to get integrability, we obtain
\begin{align*}
\ZZ_1^\bs(n)
		&\lesssim\ind_{\bl_n\comp\bs}\,
		n^{\frac14(|I| + | \vec B'_+|) +\frac14 |(I\cup \vec B_+)\cap \vec E_V| - \frac{\bb}4 - 1}\,
		\prod_{i > \bb\, :\, l^i_n \ge 1} 
			(l^i_n)^{\frac12|\vec B_i| + \frac12 |\vec B_i\cap \vec E_V| - 1}\\
	&\lesssim\ind_{\bl_n\comp\bs}\,
		n^{\frac14(|I| + | \vec B'_+|+ |V'|) - \frac{\bb}4 - 1}\,
		\prod_{i > \bb\, :\, l^i_n \ge 1} 
			(l^i_n)^{\frac12|\vec B_i| - 1},
\end{align*}
since $l_n^i=\bO(\sqrt n)$ for all~$i$, and $|(I\cup \vec B_+\cup\vec B_0)\cap \vec E_V|=|V'|$. Using again the Euler characteristic formula, as well as the bound $|V|\le 2(2g+\kk-1)$, we obtain
\[\ZZ_1^\bs(n)
	\lesssim\ind_{\bl_n\comp\bs}\,
	n^{\frac{5(g-1)}2 +\kk -\frac{\bb}4}
		\prod_{i > \bb\, :\, l^i_n \ge 1} (l^i_n)^{-\frac12} \Big(\frac{l^i_n}{\sqrt n} \Big)^{\frac12(|\vec B_i| - 1)},\]
which gives an order lower than that of~\eqref{domorder} as soon as there exists $i>\bb$ such that $|\vec B_i| \ge 2$ since $l^i_n=\sO(\sqrt n)$. As a result, the normalization constant $\ZZ_1(n)$ is of the order appearing in~\eqref{domorder} and $\ZZ_1^\bs(n)=\sO(\ZZ_1(n))$ whenever $\bs^\tp\notin\RbS^\star$. In particular, $\P(S_n^\tp\in\RbS^\star)\to 1$ as $n\to\infty$ and we obtain the first statement of Proposition~\ref{sldata}: with asymptotic probability~$1$, every vanishing face of~$M_n$ induces a tadpole in~$S_n$.

\subsection{Asymptotics of the size parameters}

\paragraph{Limiting distribution of the size parameters.}
Given a bounded continuous function~$\phi$ on the set
\[\bigcup_{\bs\in\RbS\,:\,\bs^\tp=\bs} \{\bs\}\times (\Rp)^{\vec I(\bs)}\times (\Rp)^{\vec B(\bs)}\times \R^{V(\bs)},\]
we set, for $n$, $h\in\N$, $\bs\in\RbS$,
\begin{equation*}
\ZZ_\phi^\bs(h,n)=\sum_{\cT_\bs(h,n)}
	\phi\left(\bs^\tp, \frac{\bh^{\vec I(\bs^\tp)}}{\sqrt{2n}}, \frac{\bl^{\vec B(\bs^\tp)}}{\sqrt{2n}},
			\frac{\blambda^{V(\bs^\tp)}}{(8n/9)^{1/4}}\right)
	\prod_{e\in I(\bs)}M_{h^e}(\delta\!\lambda^e)
	\prod_{e\in \vec B(\bs)}P_{l^e}(\delta\!\lambda^e),
\end{equation*}
and
\begin{equation*}
\ZZ_\phi^\bs(n)=\frac{1}{|\vec I(\bs)|+2\,|\vec B_+(\bs)|}\sum_{h\in\N} Q_{2h+\|\bl_n\|}(2n+\|\bl_n\|) \ZZ_\phi^\bs(h,n) ,
\end{equation*}
so that 
\begin{equation*}
\E\left[\phi\left(S_n^\tp, 
	\frac{\bH_{n}^{\vec I(S_n^\tp)}}{\sqrt{2n}}, \frac{\bL^{\vec B(S_n^\tp)}_{n}}{\sqrt{2n}},
	\frac{\bLambda_n^{V(S_n^\tp)}}{(8n/9)^{1/4}}\right)\right]
=\frac{1}{\ZZ_1(n)}\sum_{\bs\in\RbS} \ZZ_\phi^\bs(n).
\end{equation*}

Conducting with $\ZZ_\phi^\bs(n)$ exactly the same computations as the ones we did with $\ZZ_1^\bs(n)$, we obtain the same domination (up to $\sup|\phi|$) when $\bs^\tp$ is not dominant and a similar equivalent when $\bs^\tp$ is dominant, namely~\eqref{equivZ1s} where $c_1^\bs(\bL)$ is replaced with
\begin{multline*}
c_\phi^\bs(\bL)=\frac{1}{|\vec I(\bs)|+2\,|\vec B_+(\bs)|}\,
	3^{\frac{\bb}{2}-g}\,
	2^{\frac{|V'(\bs)|}{2}-\frac g2-\frac34\bb+\frac12}
		\MEMS{}{\\\times}
		\int_{\Rp} \d\sfh\, q_{2\sfh+\|\bL\|}(1)\, 
		\MEMS{\\\times}{}
		\int \d\rmL_\bs\,
		\phi\Big(\bs^\tp, \bh^{\vec I(\bs^\tp)}, \bl^{\vec B(\bs^\tp)}, \blambda^{V(\bs^\tp)}\Big)
		\prod_{e\in I(\bs)} p_{h^e}(\delta\!\lambda^e)
		\prod_{e\in \vec B_+(\bs)} p_{3l^e}(\delta\!\lambda^e)
		.
\end{multline*}

From the Euler characteristic formula, we obtain\MEMS{}{ that} $|\vec I(\bs)|+2\,|\vec B_+(\bs)|=2|E(\bs^\tp)|=2(6g+2p+\bb-3)$ does not depend on~$\bs$, and we remind that $|V'(\bs)|=4g+2p-3$ does not either. Let us consider a dominant scheme $\bd\in\RbS^\star$ and an integer $n\in\N$. We let $\bd_n\in\RbS$ be the unique scheme compatible with~$\bl_n$ and such that $\bd_n^\tp=\bd$. Recall that this is the scheme obtained from~$\bd$ by making into tadpoles the external vertices indexed by $\{i>\bb:l_n^i\ge 1\}$. Since the above integral only involves~$\bs^\tp$, we have $c_\phi^{\bd_n}(\bL)=c_\phi^\bd(\bL)$, and then
\[
\ZZ_1(n)=\sum_{\bs\in\RbS} \ZZ_1^\bs(n)
	\sim \sum_{\bd\in\RbS^\star} \ZZ_1^{\bd_n}(n)
	\sim n^{\frac{5(g-1)}2 +\kk -\frac{\bb}4}
		\prod_{i > \bb\, :\, l^i_n \ge 1}P_{l^i_n}(0)\,
		\sum_{\bd\in\RbS^\star} c_1^\bd(\bL),
\]
and 
\begin{equation}\label{limEphi}
\E\left[\phi\left(S_n^\tp, 
	\frac{\bH_{n}^{\vec I(S_n^\tp)}}{\sqrt{2n}}, \frac{\bL^{\vec B(S_n^\tp)}_{n}}{\sqrt{2n}},
	\frac{\bLambda_n^{V(S_n^\tp)}}{(8n/9)^{1/4}}\right)\right]
\ton\frac{1}{\sum\limits_{\bd\in\RbS^\star} c_1^\bd(\bL)}\,\sum_{\bd\in\RbS^\star} c_\phi^\bd(\bL).
\end{equation}

In passing, we obtain the following asymptotic formula for the cardinality of~$\RbQnln$, which readily yields Proposition~\ref{enumQnln} (corresponding to the case $\kk=\bb$), the unrooting giving a factor $1/4n$ coming from~\eqref{eqnrbQ}. We define the continuous function in $\bL\in (0,\infty)^\bb$
\[t_g(\bL) = \sum_{\bd\in\RbS^\star} c_1^\bd(\bL),\]
and we add the excluded cases $(g,\kk)=(0,0)$ and $(g,\kk,\bb)=(0,1,1)$, which are needed in Section~\ref{secboltz}: we set 
\[t_0(\varnothing)=\frac1{2\sqrt\pi}\qquad\text{ and }\qquad
	t_0(L)=\frac{2^{-\frac94}}{\pi\sqrt L} e^{-\frac{L^2}{2}}.\]

\begin{prp}\label{cardRbQ}
As $n\to\infty$, it holds that
\[\big|\RbQnln\big| \sim 4\,t_g(\bL)\, 12^n \, 8^{\|\bl_n\|} \,
		n^{\frac{5(g-1)}2 +\kk -\frac{\bb}4}
		\frac{e^\star}{e^\star+p_n^\lozenge}
		\prod_{i > \bb\, :\, l^i_n \ge 1}P_{l^i_n}(0),
\]
where $e^\star=6g+2p+\bb-3$ is the common number of edges of all dominant schemes, and $p_n^\lozenge=|\{i>\bb:l_n^i\ge 1\}|$ is the number of external faces among~$\ch_{\bb+1}$, \dots, $\ch_\kk$ in the maps of~$\bM^{[g]}_{n,\bl_n}$.

In the excluded cases $(g,\kk)=(0,0)$ and $(g,\kk,\bb)=(0,1,1)$, a similar formula holds:
\[\big|\RbQnzero\big| \sim 4\,t_0(\varnothing)\, 12^n \, n^{-\frac{5}2}
		\qquad\text{ and }\qquad
		\big|\RbQnof{(l_n)}{0}\big| \sim 4\,t_0(L)\, 12^n \, 8^{l_n} \, n^{-\frac{7}4}
\]
for $L>0$ and $l_n\sim \sqrt{2n}\,L$ as $n\to\infty$.
\end{prp}

\begin{proof}
Recall from Section~\ref{secCMS} that $\RbM^{[g]}_{n,\bl_n}$ is in 1-to-2 correspondence with~$\smash{\RbQ^{[g]}_{n,\bl_n0}}$, and that~\eqref{RbMnln} gives its cardinality. Using~\eqref{nbintvert} then~\eqref{RbMnln} and~\eqref{Zphi}, we obtain that
\begin{align*}
\big|\RbQ^{[g]}_{n,\bl_n}\big|&=\frac2{n+\|\bl_n\|+2-2g-\kk}\, \big|\RbM^{[g]}_{n,\bl_n}\big|\\
	&=2\,\frac{2n+\|\bl_n\|}{n+\|\bl_n\|+2-2g-\kk}\, 12^n 8^{\|\bl_n\|}
		\sum_{\bs\in\RbS} \frac{|\vec I(\bs)|+2|\vec B_+(\bs)|}{2\,|E(\bs)|} \ZZ_1^\bs(n)\\
	&\sim 4\times 12^n \, 8^{\|\bl_n\|} \,
		\sum_{\bd\in\RbS^\star} \frac{|E(\bd)|}{|E(\bd_n)|} \ZZ_1^{\bd_n}(n)\\
	&\sim 4\times 12^n \, 8^{\|\bl_n\|} \, 
		n^{\frac{5(g-1)}2 +\kk -\frac{\bb}4}
		\frac{e^\star}{e^\star+p_n^\lozenge}
		\prod_{i > \bb\, :\, l^i_n \ge 1}P_{l^i_n}(0)\,
		\sum_{\bd\in\RbS^\star} c_1^\bd(\bL),
\end{align*}
which gives the desired first statement.

The excluded cases $(g,\kk)=(0,0)$ and $(g,\kk,\bb)=(0,1,1)$ are standard; they are obtained similarly, by computing $|\RbM^{[0]}_{n,\varnothing}\big|$ and $|\RbM^{[0]}_{n,(l_n)}\big|$. More precisely, it is well known that
\[\big|\RbM^{[0]}_{n,\varnothing}\big|= 3^n \frac{(2n)!}{n! (n+1)!}\sim \frac{12^n}{\sqrt{\pi}}\, n^{-\frac32} ,.\]
In order to compute the remaining cardinality, we proceed as in Section~\ref{secB1} and obtain
\[\big|\RbM^{[0]}_{n,(l_n)}\big| = \frac{2n+l_n}{l_n}\,12^n\, 8^{l_n}\, Q_{l_n}(2n+l_n)
	P_{l_n}(0),
\]
the division by~$l_n$ taking into account the fact that seeing an element of~$\RbM^{[0]}_{n,(l_n)}$ as a forest amounts to choose a first tree among~$l_n$. From the equivalents~\eqref{equiPet1} and~\eqref{equiPet2}, this yields
\[\big|\RbQnof{(l_n)}{0}\big| \sim \frac2n\, \frac{2n}{\sqrt{2n}\,L}\,12^n\, 8^{l_n}\, 
	\frac1n q_L(1)\, \Big(\frac{8n}9\Big)^{-\frac14} p_{3L}(0),
\]
which gives the desired result.
\end{proof}

\paragraph{Limiting distribution of the areas.}
We finally take into account the areas. To this end, observe that, conditionally given 
\[\left(S_n, {\bH_{n}^{\vec I(S_n)}}, {\bL^{\vec B(S_n)}_{n}}\right),\]
the area vector ${\bA_{n}^{\vec E(S_n)}}$ is distributed as follows. We arrange the half-edges~$e_1$, \dots, $e_\kappa$ incident to the internal face of~$S_n$ according to the contour order, starting arbitrarily, and let $x_i=\sum_{j=1}^i \ell_j$, where $\ell_j=H_{n}^{e_j}$ if $e_j\in \vec I(S_n)$ or $\ell_j=L_{n}^{e_j}$ if $e_j\in \vec B(S_n)$. Then, $A_n^{e_1}$, $A_n^{e_1}+A_n^{e_2}$, $A_n^{e_1}+A_n^{e_2}+A_n^{e_3}$, \dots, are distributed as the hitting times of the successive levels $-x_1$, $-x_1-x_2$, $-x_1-x_2-x_3$, \dots, by a simple random walk conditioned on hitting the final level $-\sum_{j=1}^\kappa x_j=-\|\bH_n\|-\|\bL_n\|$ at time $2n+\|\bl_n\|$. The desired convergence~\eqref{vectorscaling} easily follows from this together with~\eqref{limEphi}, as well as the fact that, for every $e\in \vec B_0(S_n)$, we have $A^e_n+L_n^e=\Theta((L_n^e)^2)$ in probability.

\subsection{Boltzmann quadrangulations}\label{secBoltzproof}

We finally prove Theorem~\ref{thmboltz}; in its setting,
\begin{align*}
\WW\left(F\left(\Omega_{a^{-1}}(Q)\right)
		\ind_{\bQ^{[g]}_{\bl_a\bzero^\pp}}	\right)\MEMS{\hspace{-25mm}\\}{}
&=\sum_{n\in [a^{-1}\!/K,a^{-1}K]\cap\Zp}
	\WW\left(\bQ^{[g]}_{n,\bl_a\bzero^\pp}\right)
	\WW\big[F\big(\Omega_{a^{-1}}(Q)\big)\bigm\vert \bQ^{[g]}_{n,\bl_a\bzero^\pp}\big]\\
&=a^{-1}\int_{a\lfloor a^{-1}\!/K\rfloor}^{a\lfloor a^{-1}K\rfloor}\d A\:
	\WW\left(\bQ^{[g]}_{\lfloor A/a\rfloor,\bl_a\bzero^\pp}\right)
	\WW\big[F\big(\Omega_{a^{-1}}(Q)\big) \bigm\vert \bQ^{[g]}_{\lfloor A/a\rfloor,\bl_a\bzero^\pp}\big].
\end{align*}

By Theorem~\ref{mainthm} and the definition of
$\sfS^{[g]}_{A,\bL}$, it holds that
\[\WW\big[F\big(\Omega_{a^{-1}}(Q)\big)\bigm\vert
	\bQ^{[g]}_{\lfloor A/a\rfloor,\bl_a\bzero^\pp}\big]
\build{\longrightarrow}{a\downarrow 0}{}
	\E\big[F\big(\sfS^{[g]}_{A,\bL\bzero^\pp}\big)\big],\]
while Proposition~\ref{enumQnln} yields that
\[\WW\left(\bQ^{[g]}_{\lfloor A/a\rfloor,\bl_a\bzero^\pp}\right)
\build{\sim}{a\downarrow 0}{}
	t_g\big(\bL/\sqrt{A}\big)
	(A/a)^{\frac{5g-7}{2}+\frac{3\bb}{4}+\pp}.\]
Hence, Theorem~\ref{thmboltz} will be proved if we can show that the
convergence in the last integral expression is dominated. However,
this is a direct consequence of the discussion of the domination
hypothesis around~\eqref{eq:6}. 
Corollary~\ref{corboltz} is proved is a very similar way, this time
summing over all possible values of the perimeters, which results
in the integral with respect to~$\d\bL$ on $(0,\infty)^\bb$.

\bibliographystyle{alphaabbr}
\bibliography{biblio}

\newcommand{\noopsort}[1]{} \def\polhk#1{\setbox0=\hbox{#1}{\ooalign{\hidewidth
  \lower1.5ex\hbox{`}\hidewidth\crcr\unhbox0}}}
\begin{thebibliography}{GKRV21}

\bibitem[ABA17]{ABAl17}
L.~Addario-Berry and M.~Albenque.
\newblock The scaling limit of random simple triangulations and random simple
  quadrangulations.
\newblock {\em Ann. Probab.}, 45(5):2767--2825, 2017.

\bibitem[ABA21]{ABAl21}
L.~Addario-Berry and M.~Albenque.
\newblock Convergence of non-bipartite maps via symmetrization of labeled
  trees.
\newblock {\em Ann. H. Lebesgue}, 4:653--683, 2021.

\bibitem[Abr16]{abr14}
C.~Abraham.
\newblock Rescaled bipartite planar maps converge to the {B}rownian map.
\newblock {\em Ann. Inst. Henri Poincar\'e Probab. Stat.}, 52(2):575--595,
  2016.

\bibitem[ABW17]{ABW17}
L.~Addario-Berry and Y.~Wen.
\newblock Joint convergence of random quadrangulations and their cores.
\newblock {\em Ann. Inst. Henri Poincar\'{e} Probab. Stat.}, 53(4):1890--1920,
  2017.

\bibitem[ADH13]{AbDeHo13}
R.~Abraham, J.-F. Delmas, and P.~Hoscheit.
\newblock A note on the {G}romov-{H}ausdorff-{P}rokhorov distance between
  (locally) compact metric measure spaces.
\newblock {\em Electron. J. Probab.}, 18:no. 14, 21, 2013.

\bibitem[AHS23]{AnHoSu23}
M.~Ang, N.~Holden, and X.~Sun.
\newblock The {SLE} loop via conformal welding of quantum disks.
\newblock {\em Electron. J. Probab.}, 28:Paper No. 30, 20, 2023.

\bibitem[AKM17]{AnKoMi17}
O.~Angel, B.~Kolesnik, and G.~Miermont.
\newblock Stability of geodesics in the {B}rownian map.
\newblock {\em Ann. Probab.}, 45(5):3451--3479, 2017.

\bibitem[Ald91]{aldouscrt91}
D.~J. Aldous.
\newblock The continuum random tree. {I}.
\newblock {\em Ann. Probab.}, 19(1):1--28, 1991.

\bibitem[Ald93]{aldouscrt93}
D.~J. Aldous.
\newblock The continuum random tree. {III}.
\newblock {\em Ann. Probab.}, 21(1):248--289, 1993.

\bibitem[AP15]{AlPo15}
M.~Albenque and D.~Poulalhon.
\newblock A generic method for bijections between blossoming trees and planar
  maps.
\newblock {\em Electron. J. Combin.}, 22(2):Paper 2.38, 44, 2015.

\bibitem[ARS22]{ARS22mod}
M.~Ang, G.~Remy, and X.~Sun.
\newblock The moduli of annuli in random conformal geometry.
\newblock {\em Preprint,
  \href{http://arxiv.org/abs/2203.12398}{\textup{\nolinkurl{arXiv:2203.12398}}}},
  2022.

\bibitem[AS03]{AnSc03}
O.~Angel and O.~Schramm.
\newblock Uniform infinite planar triangulations.
\newblock {\em Comm. Math. Phys.}, 241(2-3):191--213, 2003.

\bibitem[BBI01]{burago01}
D.~Burago, Y.~Burago, and S.~Ivanov.
\newblock {\em A course in metric geometry}, volume~33 of {\em Graduate Studies
  in Mathematics}.
\newblock American Mathematical Society, Providence, RI, 2001.

\bibitem[BC86]{BenCan86}
E.~A. Bender and E.~R. Canfield.
\newblock The asymptotic number of rooted maps on a surface.
\newblock {\em J. Combin. Theory Ser. A}, 43(2):244--257, 1986.

\bibitem[BCK18]{BeCuKo18}
J.~Bertoin, N.~Curien, and I.~Kortchemski.
\newblock Random planar maps and growth-fragmentations.
\newblock {\em Ann. Probab.}, 46(1):207--260, 2018.

\bibitem[BDG04]{BdFGmobiles}
J.~Bouttier, P.~D{i F}rancesco, and E.~Guitter.
\newblock Planar maps as labeled mobiles.
\newblock {\em Electron. J. Combin.}, 11:Research Paper 69, 27 pp.
  (electronic), 2004.

\bibitem[Bet10]{bettinelli10}
J.~Bettinelli.
\newblock Scaling limits for random quadrangulations of positive genus.
\newblock {\em Electron. J. Probab.}, 15:no. 52, 1594--1644, 2010.

\bibitem[Bet12]{bettinelli11}
J.~Bettinelli.
\newblock The topology of scaling limits of positive genus random
  quadrangulations.
\newblock {\em Ann. Probab.}, 40:no. 5, 1897--1944, 2012.

\bibitem[Bet15]{bettinelli11b}
J.~Bettinelli.
\newblock Scaling limit of random planar quadrangulations with a boundary.
\newblock {\em Ann. Inst. Henri Poincar\'e Probab. Stat.}, 51(2):432--477,
  2015.

\bibitem[Bet16]{Bet16geo}
J.~Bettinelli.
\newblock Geodesics in {B}rownian surfaces ({B}rownian maps).
\newblock {\em Ann. Inst. Henri Poincar\'e Probab. Stat.}, 52(2):612--646,
  2016.

\bibitem[Bet22]{Bet22}
J.~Bettinelli.
\newblock A bijection for nonorientable general maps.
\newblock {\em Ann. Inst. Henri Poincar\'{e} D}, 9(4):733--791, 2022.

\bibitem[BG09]{BoGu09}
J.~Bouttier and E.~Guitter.
\newblock Distance statistics in quadrangulations with a boundary, or with a
  self-avoiding loop.
\newblock {\em J. Phys. A}, 42(46):465208, 44, 2009.

\bibitem[BG12]{BoGu12}
J.~Bouttier and E.~Guitter.
\newblock Planar maps and continued fractions.
\newblock {\em Comm. Math. Phys.}, 309(3):623--662, 2012.

\bibitem[BGT89]{BGT}
N.~H. Bingham, C.~M. Goldie, and J.~L. Teugels.
\newblock {\em Regular variation}, volume~27 of {\em Encyclopedia of
  Mathematics and its Applications}.
\newblock Cambridge University Press, Cambridge, 1989.

\bibitem[BH99]{BrHa99}
M.~R. Bridson and A.~Haefliger.
\newblock {\em Metric spaces of non-positive curvature}, volume 319 of {\em
  Grundlehren der Mathematischen Wissenschaften [Fundamental Principles of
  Mathematical Sciences]}.
\newblock Springer-Verlag, Berlin, 1999.

\bibitem[BHL19]{BeHuLe19arX}
V.~Beffara, C.~B. Huynh, and B.~L{\'e}v{\^e}que.
\newblock Scaling limits for random triangulations on the torus.
\newblock {\em Preprint,
  \href{http://arxiv.org/abs/1905.01873}{\textup{\nolinkurl{arXiv:1905.01873}}}},
  2019.

\bibitem[BJM14]{BeJaMi14}
J.~Bettinelli, E.~Jacob, and G.~Miermont.
\newblock The scaling limit of uniform random plane maps, \textit{via} the
  {A}mbj{\o}rn--{B}udd bijection.
\newblock {\em Electron. J. Probab.}, 19:no. 74, 1--16, 2014.

\bibitem[BLG13]{BeLG}
J.~Beltran and J.-F. Le~Gall.
\newblock Quadrangulations with no pendant vertices.
\newblock {\em Bernoulli}, 19(4):1150--1175, 2013.

\bibitem[BM17]{BeMi17}
J.~Bettinelli and G.~Miermont.
\newblock Compact {B}rownian surfaces {I}. {B}rownian disks.
\newblock {\em Probab. Theory Related Fields}, 167:555--614, 2017.

\bibitem[BMR19]{BaMiRa}
E.~Baur, G.~Miermont, and G.~Ray.
\newblock Classification of scaling limits of uniform quadrangulations with a
  boundary.
\newblock {\em Ann. Probab.}, 47(6):3397--3477, 2019.

\bibitem[Bou19]{BouttierHDR}
J.~Bouttier.
\newblock {\em Planar maps and random partitions}.
\newblock Habilitation {\`a} diriger des recherches, {Universit{\'e} Paris XI},
  2019.

\bibitem[BR18]{BaRiUIHPQs}
E.~Baur and L.~Richier.
\newblock Uniform infinite half-planar quadrangulations with skewness.
\newblock {\em Electron. J. Probab.}, 23:Paper No. 54, 43, 2018.

\bibitem[CC18]{CaCuUIHPQ}
A.~Caraceni and N.~Curien.
\newblock Geometry of the uniform infinite half-planar quadrangulation.
\newblock {\em Random Structures Algorithms}, 52(3):454--494, 2018.

\bibitem[CD06]{ChDu06}
P.~Chassaing and B.~Durhuus.
\newblock Local limit of labeled trees and expected volume growth in a random
  quadrangulation.
\newblock {\em Ann. Probab.}, 34(3):879--917, 2006.

\bibitem[CD17]{ChapuyDolega}
G.~Chapuy and M.~Do{\l}{\k{e}}ga.
\newblock A bijection for rooted maps on general surfaces.
\newblock {\em J. Combin. Theory Ser. A}, 145:252--307, 2017.

\bibitem[CLG14]{CuLG12Bplane}
N.~Curien and J.-F. Le~Gall.
\newblock The {B}rownian plane.
\newblock {\em J. Theoret. Probab.}, 27(4):1249--1291, 2014.

\bibitem[CLG19]{CuLG19}
N.~Curien and J.-F. Le~Gall.
\newblock First-passage percolation and local modifications of distances in
  random triangulations.
\newblock {\em Ann. Sci. \'{E}c. Norm. Sup\'{e}r. (4)}, 52(3):631--701, 2019.

\bibitem[CM15]{CuMiUIHPQ}
N.~Curien and G.~Miermont.
\newblock Uniform infinite planar quadrangulations with a boundary.
\newblock {\em Random Structures Algorithms}, 47(1):30--58, 2015.

\bibitem[CMM13]{CuMIMeUIPQ}
N.~Curien, L.~M\'{e}nard, and G.~Miermont.
\newblock A view from infinity of the uniform infinite planar quadrangulation.
\newblock {\em ALEA Lat. Am. J. Probab. Math. Stat.}, 10(1):45--88, 2013.

\bibitem[CMS09]{ChMaSc}
G.~Chapuy, M.~Marcus, and G.~Schaeffer.
\newblock A bijection for rooted maps on orientable surfaces.
\newblock {\em SIAM J. Discrete Math.}, 23(3):1587--1611, 2009.

\bibitem[CS04]{CSise}
P.~Chassaing and G.~Schaeffer.
\newblock Random planar lattices and integrated super{B}rownian excursion.
\newblock {\em Probab. Theory Related Fields}, 128(2):161--212, 2004.

\bibitem[Cur19]{CurienStFlour}
N.~Curien.
\newblock Peeling random planar maps.
\newblock {\em Saint-Flour lecture notes}, 2019.

\bibitem[CV81]{CoVa}
R.~Cori and B.~Vauquelin.
\newblock Planar maps are well labeled trees.
\newblock {\em Canad. J. Math.}, 33(5):1023--1042, 1981.

\bibitem[Dav85]{david85}
F.~David.
\newblock Planar diagrams, two-dimensional lattice gravity and surface models.
\newblock {\em Nuclear Phys. B}, 257(1):45--58, 1985.

\bibitem[DDDF20]{DiDuDuFa20}
J.~Ding, J.~Dub\'{e}dat, A.~Dunlap, and H.~Falconet.
\newblock Tightness of {L}iouville first passage percolation for {$\gamma \in
  (0,2)$}.
\newblock {\em Publ. Math. Inst. Hautes \'{E}tudes Sci.}, 132:353--403, 2020.

\bibitem[DDG23]{DiDuGw23}
J.~Ding, J.~Dub\'{e}dat, and E.~Gwynne.
\newblock Introduction to the {L}iouville quantum gravity metric.
\newblock In {\em I{CM}---{I}nternational {C}ongress of {M}athematicians.
  {V}ol. 6. {S}ections 12--14}, pages 4212--4244. EMS Press, Berlin, [2023]
  \copyright 2023.

\bibitem[DKRV16]{DaKuRhVa16}
F.~David, A.~Kupiainen, R.~Rhodes, and V.~Vargas.
\newblock Liouville quantum gravity on the {R}iemann sphere.
\newblock {\em Comm. Math. Phys.}, 342(3):869--907, 2016.

\bibitem[Eyn16]{eynard16}
B.~Eynard.
\newblock {\em Counting surfaces}, volume~70 of {\em Progress in Mathematical
  Physics}.
\newblock Birkh\"{a}user/Springer, [Cham], 2016.
\newblock CRM Aisenstadt chair lectures.

\bibitem[GHS23]{GwHoSu23}
E.~Gwynne, N.~Holden, and X.~Sun.
\newblock Mating of trees for random planar maps and {L}iouville quantum
  gravity: a survey.
\newblock In {\em Topics in statistical mechanics}, volume~59 of {\em Panor.
  Synth\`eses}, pages 41--120. Soc. Math. France, Paris, 2023.

\bibitem[GKRV21]{GuKuRhVa21}
C.~Guillarmou, A.~Kupiainen, R.~Rhodes, and V.~Vargas.
\newblock Segal's axioms and bootstrap for {L}iouville theory.
\newblock {\em Preprint,
  \href{http://arxiv.org/abs/2112.14859}{\textup{\nolinkurl{arXiv:2112.14859}}}},
  2021.

\bibitem[GM17]{GwMi17}
E.~Gwynne and J.~Miller.
\newblock Scaling limit of the uniform infinite half-plane quadrangulation in
  the {G}romov-{H}ausdorff-{P}rokhorov-uniform topology.
\newblock {\em Electron. J. Probab.}, 22:Paper No. 84, 47, 2017.

\bibitem[GM19]{GwMi19a}
E.~Gwynne and J.~Miller.
\newblock Metric gluing of {B}rownian and {$\sqrt{8/3}$}-{L}iouville quantum
  gravity surfaces.
\newblock {\em Ann. Probab.}, 47(4):2303--2358, 2019.

\bibitem[GM21a]{GwMi21a}
E.~Gwynne and J.~Miller.
\newblock Convergence of the self-avoiding walk on random quadrangulations to
  {$\rm SLE_{8/3}$} on {$\sqrt{8/3}$}-{L}iouville quantum gravity.
\newblock {\em Ann. Sci. \'{E}c. Norm. Sup\'{e}r. (4)}, 54(2):305--405, 2021.

\bibitem[GM21b]{GwMi21b}
E.~Gwynne and J.~Miller.
\newblock Existence and uniqueness of the {L}iouville quantum gravity metric
  for {$\gamma\in(0,2)$}.
\newblock {\em Invent. Math.}, 223(1):213--333, 2021.

\bibitem[GMS20]{GwMiSh20}
E.~Gwynne, J.~Miller, and S.~Sheffield.
\newblock The {T}utte embedding of the {P}oisson-{V}oronoi tessellation of the
  {B}rownian disk converges to {$\sqrt{8/3}$}-{L}iouville quantum gravity.
\newblock {\em Comm. Math. Phys.}, 374(2):735--784, 2020.

\bibitem[GMS22]{GwMiSh22}
E.~Gwynne, J.~Miller, and S.~Sheffield.
\newblock An invariance principle for ergodic scale-free random environments.
\newblock {\em Acta Math.}, 228(2):303--384, 2022.

\bibitem[GRV19]{GuRhVa19}
C.~Guillarmou, R.~Rhodes, and V.~Vargas.
\newblock Polyakov's formulation of {$2d$} bosonic string theory.
\newblock {\em Publ. Math. Inst. Hautes \'{E}tudes Sci.}, 130:111--185, 2019.

\bibitem[HS23]{HoSu23}
N.~Holden and X.~Sun.
\newblock Convergence of uniform triangulations under the {C}ardy embedding.
\newblock {\em Acta Math.}, 230(1):93--203, 2023.

\bibitem[KPZ88]{KnPoZa88}
V.~G. Knizhnik, A.~M. Polyakov, and A.~B. Zamolodchikov.
\newblock Fractal structure of {$2$}{D}-quantum gravity.
\newblock {\em Modern Phys. Lett. A}, 3(8):819--826, 1988.

\bibitem[Kri05]{Krikun}
M.~Krikun.
\newblock Local structure of random quadrangulations.
\newblock {\em Preprint,
  \href{http://arxiv.org/abs/0512304}{\textup{\nolinkurl{arXiv:0512304}}}},
  2005.

\bibitem[KRV20]{KuRhVa21}
A.~Kupiainen, R.~Rhodes, and V.~Vargas.
\newblock Integrability of {L}iouville theory: proof of the {DOZZ} formula.
\newblock {\em Ann. of Math. (2)}, 191(1):81--166, 2020.

\bibitem[LG99]{legall99}
J.-F. Le~Gall.
\newblock {\em Spatial branching processes, random snakes and partial
  differential equations}.
\newblock Lectures in Mathematics ETH Z\"urich. Birkh\"auser Verlag, Basel,
  1999.

\bibitem[LG06]{legall05}
J.-F. Le~Gall.
\newblock A conditional limit theorem for tree-indexed random walk.
\newblock {\em Stochastic Process. Appl.}, 116(4):539--567, 2006.

\bibitem[LG07]{legall06}
J.-F. Le~Gall.
\newblock The topological structure of scaling limits of large planar maps.
\newblock {\em Invent. Math.}, 169(3):621--670, 2007.

\bibitem[LG10]{legall08}
J.-F. Le~Gall.
\newblock Geodesics in large planar maps and in the {B}rownian map.
\newblock {\em Acta Math.}, 205(2):287--360, 2010.

\bibitem[LG13]{legall11}
J.-F. Le~Gall.
\newblock Uniqueness and universality of the {B}rownian map.
\newblock {\em Ann. Probab.}, 41(4):2880--2960, 2013.

\bibitem[LG19a]{LGa19dis}
J.-F. Le~Gall.
\newblock Brownian disks and the {B}rownian snake.
\newblock {\em Ann. Inst. Henri Poincar\'{e} Probab. Stat.}, 55(1):237--313,
  2019.

\bibitem[LG19b]{LeGallBrownianGeometry}
J.-F. Le~Gall.
\newblock Brownian geometry.
\newblock {\em Jpn. J. Math.}, 14(2):135--174, 2019.

\bibitem[LG22a]{legall22}
J.-F. Le~Gall.
\newblock The {B}rownian disk viewed from a boundary point.
\newblock {\em Ann. Inst. Henri Poincar\'{e} Probab. Stat.}, 58(2):1091--1119,
  2022.

\bibitem[LG22b]{LGa22sta}
J.-F. Le~Gall.
\newblock Geodesic stars in random geometry.
\newblock {\em Ann. Probab.}, 50(3):1013--1058, 2022.

\bibitem[LGM11]{LGMi09}
J.-F. Le~Gall and G.~Miermont.
\newblock Scaling limits of random planar maps with large faces.
\newblock {\em Ann. Probab.}, 39(1):1--69, 2011.

\bibitem[LGP08]{lgp}
J.-F. Le~Gall and F.~Paulin.
\newblock Scaling limits of bipartite planar maps are homeomorphic to the
  2-sphere.
\newblock {\em Geom. Funct. Anal.}, 18(3):893--918, 2008.

\bibitem[LGR20]{LGRi20}
J.-F. Le~Gall and A.~Riera.
\newblock Growth-fragmentation processes in {B}rownian motion indexed by the
  {B}rownian tree.
\newblock {\em Ann. Probab.}, 48(4):1742--1784, 2020.

\bibitem[LGR21]{LGRi21}
J.-F. Le~Gall and A.~Riera.
\newblock Spine representations for non-compact models of random geometry.
\newblock {\em Probab. Theory Related Fields}, 181(1-3):571--645, 2021.

\bibitem[LZ04]{LaZv04}
S.~K. Lando and A.~K. Zvonkin.
\newblock {\em Graphs on surfaces and their applications}, volume 141 of {\em
  Encyclopaedia of Mathematical Sciences}.
\newblock Springer-Verlag, Berlin, 2004.

\bibitem[Mar22]{marzouk22}
C.~Marzouk.
\newblock On scaling limits of random trees and maps with a prescribed degree
  sequence.
\newblock {\em Ann. H. Lebesgue}, 5:317--386, 2022.

\bibitem[Mie08]{miermontsph}
G.~Miermont.
\newblock On the sphericity of scaling limits of random planar
  quadrangulations.
\newblock {\em Electron. Commun. Probab.}, 13:248--257, 2008.

\bibitem[Mie09]{miertess}
G.~Miermont.
\newblock Tessellations of random maps of arbitrary genus.
\newblock {\em Ann. Sci. \'Ec. Norm. Sup\'er. (4)}, 42(5):725--781, 2009.

\bibitem[Mie13]{miermont11}
G.~Miermont.
\newblock The {B}rownian map is the scaling limit of uniform random plane
  quadrangulations.
\newblock {\em Acta Math.}, 210(2):319--401, 2013.

\bibitem[MM03]{mm01}
J.-F. Marckert and A.~Mokkadem.
\newblock The depth first processes of {G}alton--{W}atson trees converge to the
  same {B}rownian excursion.
\newblock {\em Ann. Probab.}, 31(3):1655--1678, 2003.

\bibitem[MQ21]{MiQi21arX}
J.~Miller and W.~Qian.
\newblock Geodesics in the {B}rownian map: {S}trong confluence and geometric
  structure.
\newblock {\em Preprint,
  \href{http://arxiv.org/abs/2008.02242}{\textup{\nolinkurl{arXiv:2008.02242}}}},
  2021.

\bibitem[MS20]{MiSh21I}
J.~Miller and S.~Sheffield.
\newblock Liouville quantum gravity and the {B}rownian map {I}: the
  {$\mathrm{QLE}(8/3,0)$} metric.
\newblock {\em Invent. Math.}, 219(1):75--152, 2020.

\bibitem[MS21a]{MiSh210}
J.~Miller and S.~Sheffield.
\newblock An axiomatic characterization of the {B}rownian map.
\newblock {\em J. \'{E}c. polytech. Math.}, 8:609--731, 2021.

\bibitem[MS21b]{MiSh21II}
J.~Miller and S.~Sheffield.
\newblock {L}iouville quantum gravity and the {B}rownian map {II}: {G}eodesics
  and continuity of the embedding.
\newblock {\em Ann. Probab.}, 49(6):2732--2829, 2021.

\bibitem[MS21c]{MiSh21III}
J.~Miller and S.~Sheffield.
\newblock Liouville quantum gravity and the {B}rownian map {III}: the conformal
  structure is determined.
\newblock {\em Probab. Theory Related Fields}, 179(3-4):1183--1211, 2021.

\bibitem[Pet75]{petrov75}
V.~V. Petrov.
\newblock {\em Sums of independent random variables}.
\newblock Springer-Verlag, New York, 1975.
\newblock Translated from the Russian by A. A. Brown, Ergebnisse der Mathematik
  und ihrer Grenzgebiete, Band 82.

\bibitem[Pit75]{Pi2M-X}
J.~W. Pitman.
\newblock One-dimensional {B}rownian motion and the three-dimensional {B}essel
  process.
\newblock {\em Advances in Appl. Probability}, 7(3):511--526, 1975.

\bibitem[Pol81]{Pol81}
A.~M. Polyakov.
\newblock Quantum geometry of bosonic strings.
\newblock {\em Phys. Lett. B}, 103(3):207--210, 1981.

\bibitem[RW95]{RiWo95}
L.~B. Richmond and N.~C. Wormald.
\newblock Almost all maps are asymmetric.
\newblock {\em J. Combin. Theory Ser. B}, 63(1):1--7, 1995.

\bibitem[Sch98]{schaeffer98}
G.~Schaeffer.
\newblock {\em Conjugaison d'arbres et cartes combinatoires al{\'e}atoires}.
\newblock PhD thesis, Universit{\'e} Bordeaux {I}, 1998.

\bibitem[She23]{sheffieldICM}
S.~Sheffield.
\newblock What is a random surface?
\newblock In {\em I{CM}---{I}nternational {C}ongress of {M}athematicians.
  {V}ol. {II}. {P}lenary lectures}, pages 1202--1258. EMS Press, Berlin, [2023]
  \copyright 2023.

\bibitem[Str11]{Stroock11}
D.~W. Stroock.
\newblock {\em Probability theory}.
\newblock Cambridge University Press, Cambridge, second edition, 2011.
\newblock An analytic view.

\bibitem[Vil09]{villani09}
C.~Villani.
\newblock {\em Optimal transport, old and New}, volume 338 of {\em Grundlehren
  der Mathematischen Wissenschaften [Fundamental Principles of Mathematical
  Sciences]}.
\newblock Springer-Verlag, Berlin, 2009.

\bibitem[Wu22]{Wu22}
B.~Wu.
\newblock Conformal bootstrap on the annulus in {L}iouville {CFT}.
\newblock {\em Preprint,
  \href{http://arxiv.org/abs/2203.11830}{\textup{\nolinkurl{arXiv:2203.11830}}}},
  2022.

\end{thebibliography}
\end{document}